\documentclass[a4paper,12pt]{report}
\usepackage{fbe_tez}
\usepackage[utf8x]{inputenc} 
\usepackage[T1]{fontenc}

\usepackage{amsmath, amsthm, amssymb}
\usepackage[bottom]{footmisc}
\usepackage{cite}
\usepackage{graphicx}
\usepackage{longtable}

\graphicspath{{figures/}} 

\usepackage{tikz}
\usepackage{tkz-tab}
\usepackage{tikz-qtree}
\usepackage{enumerate}
\usepackage{float}

\usetikzlibrary{arrows}
\usetikzlibrary{positioning}

\usepackage{bbm}
\usepackage{breqn}

\usepackage{multirow}
\usepackage{subfigure}
\usepackage{algorithm}
\usepackage{algorithmic}

\newtheorem{thm}{Theorem}[chapter]
\newtheorem{prop}[thm]{Proposition}
\newtheorem{lem}[thm]{Lemma}
\newtheorem{cor}[thm]{Corollary}

\theoremstyle{definition}
\newtheorem{mydef}[thm]{Definition}
\newtheorem{rem}[thm]{Remark}
\newtheorem{ex}[thm]{Example}

\newcommand{\E}{\mathbb{E}}
\newcommand{\Pro}{\mathbb{P}}
\newcommand{\Var}{\mathrm{Var}}
\newcommand{\Cov}{\mathrm{Cov}}
\newcommand{\Po}{\operatorname{Po}}
\newcommand{\1}{\mathbf{1}}
\newcommand{\mult}{\operatorname{mult}}

\usepackage{verbatim}

\title{ON ASYMPTOTICS OF TWO NON-UNIFORM RECURSIVE TREE MODELS}
\turkcebaslik{İKİ FARKLI DÜZGÜN DAĞILIMA SAHİP OLMAYAN YİNELİ AĞAÇ MODELİ ÜZERİNE}
\degree{B.S., Mathematics, University of Vienna, 2013}
\author{Ella Veronika Hiesmayr}
\program{Mathematics}
\subyear{2017}

\supervisor{Assist. Prof. Ümit Işlak}
\examinerii{Assist. Prof. Fatih Ecevit}
\examineri{Assoc. Prof. Mine Çağlar}
\dateofapproval{24.08.2017}

\begin{document}

\pagenumbering{roman}
\makemstitle 
\makeapprovalpage
\begin{acknowledgements}
First of all I want to thank my supervisor Ümit Işlak. I feel extremely lucky to have a supervisor that is passionate about his research as well as excited to transfer this enthusiasm to his students. What and how he taught me in the comparatively short time we have known each other changed my view on mathematics and probability.

I also want to thank Serdar Altok for introducing me to random graphs and sparking my interest in probability theory.

Moreover I want to thank all my professors who supported my wish to understand in the course of my mathematics studies. I especially want to thank Fatih Ecevit and Mine Çağlar for being part of my thesis commitee.

I also want to thank my fellow mathematics students, in particular Arda, Beyza, Can, Çınar, Deniz, Doğa, Eray, Erol, Gözde, Mert and Turan for the mathematical as well fun atmosphere they created during the last weeks of my writing.

I thank Anna, Carmen, Ece and Emre for their friendship and for providing me with a different perspective when I struggled during my studies.

I also want to thank Efe for his support and understanding and for doing everything he could to provide me with a good working environment.

I am thankful that Judith and Fritz, my sister and brother, always keep my interest in different areas of mathematics going and for diverting me when things did not go as expected.

Finally I want to thank my parents, Hildegard and Michael, for their support during my studies and for trusting in my decisions.

\end{acknowledgements}
\begin{abstract}
In this thesis the properties of two kinds of non-uniform random recursive trees are studied. In the first model weights are assigned to each node, thus altering the attachment probabilities. We will call these trees weighted recursive trees. In the second model a different distribution rather than the uniform one is chosen on the symmetric group, namely a riffle shuffle distribution. These trees will be called biased recursive trees. For both of these models the number of branches, the number of leaves, the depth of nodes and some other properties are studied. The focus is on asymptotic results and the comparison with uniform random recursive trees. It will be shown that the studied properties of weighted recursive trees are close to uniform recursive trees in many cases when the number of nodes increases. In contrast biased recursive trees show a different behaviour but approach uniform recursive trees depending on the parameters of the riffle shuffle distribution.
\end{abstract}

\begin{ozet}
Bu tezde iki çeşit düzgün dağılıma sahip olmayan yineli ağaç modelinin özellikleri incelenmektedir. İlk modelde her köşeye ağırlık vermek suretiyle bağlanma olasıklıkları değişkenlik göstermektedir. İkinci modelde ise altta  yatan düzgün permütasyon dağılımı özel bir kart karma modeli ile değiştirilerek yine düzgün dağılmayan bir ağaç tipi oluşturulmaktadır.  Her iki modelde de dalların sayısı, yaprakların sayısı, köşelerin derinliği gibi pek çok özellik incelenecektir. Odak noktamız asimptotik sorular ve oluşan ağaçların düzgün yineli ağaçlar ile kıyaslanması üzerinedir. Köşelerin sayısı arttıkça ağırlıklı yineli ağaçların özelliklerinin düzgün dağılmış ağaçlara benzediği gösterilecektir. Benzer şekilde, ikinci modelimizde de altta yatan parametrelere bağlı olarak ortaya çıkan dağılım  düzgün dağılıma yakın olabilmektedir. 

\end{ozet}
\tableofcontents
\listoffigures
\listoftables
\begin{symbols}
\sym{$\1(A)$}{Indicator function of the event $A$}
\sym{$|A|$}{Cardinality of a set $A$}
\sym{$\mathcal{B}(\mathbb{R})$}{Borel sigma algebra of $\mathbb{R}$}
\sym{$\Cov(X,Y)$}{Covariance of two random variables $X$ and $Y$}
\sym{$\to_d$}{Convergence in distribution}
\sym{$=_d$}{Equality in distribution}
\sym{$d_{TV}$}{Total variation metric}
\sym{$d_K$}{Kolmogorov metric}
\sym{$d_W$}{Wasserstein metric}
\sym{$\E[X]$}{Expectation of a random variable $X$}
\sym{$\mathcal{G}$}{Standard normal distribution }
\sym{$H_n$}{Harmonic numbers of the first order: $\sum_{i=1}^{n} \frac{1}{i}$}
\sym{$H_n^{(2)}$}{Harmonic numbers of the second order: $\sum_{i=1}^{n} \frac{1}{i^2}$}
\sym{$\mult(a, \vec{p})$}{Multinomial distribution with parameters $a$ and \\ $\vec{p}=(p_1, p_2, \dots, p_a)$}
\sym{$[n]$}{First $n$ positive integers: $\{1,\ldots,n\}$}
\sym{$\mathcal{O}$}{Big O}
\sym{$\Pro(A)$}{Probability of an event $A$}
\sym{$\Po(\lambda)$}{Poisson distribution with expectation $\lambda$}
\sym{$\Var(X)$}{Variance of a random variable $X$}
%
%
%
\end{symbols}
\begin{abbreviations}
\sym{a-RT}{a-recursive tree}
\sym{a.s.}{almost surely}
\sym{BRT}{Biased recursive tree}
\sym{URT}{Uniform recursive tree}
\sym{URP}{Uniform random permutation}
\sym{WRT}{Weighted recursive tree}
\end{abbreviations}
\chapter{INTRODUCTION}
\label{chapter:introduction}
\pagenumbering{arabic}

Recursive trees are rooted labeled trees where  the nodes on a path from the root to any other node form an increasing sequence. Because of this property, many recursive trees can be considered to grow dynamically, by attaching every new node to one of the already present nodes.
Several distributions can be applied on the set of recursive trees of size $n$, the most common one being the uniform distribution, i.e. every recursive tree being equally likely. Given this distribution the number of leaves, internodal distances, the number of branches, the height of the tree and various other statistics are well studied. 

In this thesis we will study two non-uniform distributions on recursive trees. Our main motivation is to get a better understanding of  general random recursive tree models. More precisely the goal is to grasp the behaviour of inhomogeneous recursive trees, which are recursive trees that can be constructed by attaching at every step the new node $n$ to one node of a recursive tree of size $n-1$ according to some distribution on $[n-1]$ and independently of the structure of the tree. Our first model is a special kind of inhomogeneous tree, namely the distribution obtained when each node is assigned a fixed weight. The second model is not an inhomogeneous tree itself but provides an approximation for inhomogeneous trees as it asymptotically approaches the uniform model.

Recursive trees are interesting from a theoretical as well as a practical point of view. First of all recursive trees are connected to a wide range of other mathematical structures like permutations, branching processes and records. Because this diversity is well reflected in the history of the study uniform recursive trees, we will now give a short summary of that development. After that we will give some examples of applications of uniform recursive trees that illustrate well how profitable knowledge of more general recursive tree models is from a practical perspective. For the results that were established until  the early 90's the reader is referred to the comprehensive survey  \cite{Survey}. Unless otherwise mentioned, the results discussed below in this section can be found in the cited source. 

Properties that were considered early on are internodal distances. In particular the expectation and variance of the distance between two nodes and as a special case of the depth of node $n$ were solved in the seventies. Moon found a recursion for the distribution of the distance between two fixed nodes by using the fact that uniform recursive trees have a relatively simple growing rule \cite{Moon74}. This recursion can then be used to derive the expectation without solving the recursion for the exact distribution. In 1996  Dobrow gave the exact distribution of the distance between a fixed node and $n$ and proves asymptotic normality for the distance between certain sequences of nodes and node $n$ \cite{Dobrow96}. Asymptotic normality of the distance between a fixed node and $n$ was shown by Su, Liu and Feng in 2006 by using the decomposition of the distance into a sum of random variables and applying the Lindeberg-Feller central limit theorem \cite{Feng06}. They moreover proved a more general result, namely that the distance between any sequence of nodes and node $n$ is asymptotically normal.

As a special case of an internodal distance, the exact distribution of the depth of node $n$ was established by making use of a proper recursion by Szyma\'{n}ski in 1990 \cite{Survey}. Finally, asymptotic normality of the depth of node $n$ was proven by two different methods. Devroye used the theory of records \cite{DevroyeRecords} and Mahmoud calculated the limit of the moment generating function of the depth using its exact distribution which was already known  \cite{Survey}. The expectation and limiting distribution of the depth of node $n$ was proved by Feng, Su and Hu by yet another method in 2005 \cite{Feng05}. They were able to write the depth as a sum of independent indicator random variables, which easily yields expectation, variance and a central limit theorem.

Another statistic that was studied early on is the number of nodes of a certain degree and as a special case the number of leaves. The earliest result in this direction is about the expected number of nodes of a certain degree, which was proved using a recursion by Na and Rapoport \cite{NaRapoport}. Gastwirth could establish upper bounds by first writing the random variables as a sum of Bernoulli random variables and then using Poisson approximations \cite{Gastwirth}. For the exact distribution, expectation, variance and asymptotic distribution of the number of leaves, again several methods can be used. First of all a special case of Friedman's urn \cite{Friedman} can be used to derive the exact distribution \cite{Survey}. Friedman's urn contains balls of two colours and grows according to a replacement rule. When applied on the leaves of uniform recursive trees, the white balls represent the internal nodes and the black balls the leaves. When a black ball is drawn, this means that the parent of the next node was a leaf and has now become an internal node. The new node definitely is a leaf. We thus put the black ball back, together with an additional white ball. Similarly, if a white ball is drawn, we put it back together with a black ball, since the internal node is still internal and the new node is a leaf. Based on the urn model and recursions derived from it, a differential equation for the moment generating function can be established, which can be solved and turns out to be the generating function of the Eulerian numbers \cite{Survey}. Similarly Mahmoud and Smythe derived the expectation, variance and asymptotic normality of the number of nodes of degree 1 and 2 by a generalization of the above described urn model \cite{MahmoudSmythe92}.

By the use of a recursion Najock and Heyde also derived the distribution of the number of leaves and via some well-known results on permutations they could further prove asymptotic normality of the number of leaves \cite{NajockHeyde82}. The emergence of the use of  results on permutation statistics in leaf related problems   is no coincidence, as it was later discovered that there is a bijection between uniform random permutations and uniform recursive trees, see for example \cite{AltokIslak}.
Using this bijection, size-biased coupling and Stein's method, results on the convergence rate could recently be established by Zhang \cite{Zhang}. Altok and Işlak refined these results  and studied other leaf-related properties \cite{AltokIslak} .

Moving on to another statistic,  the number of nodes with a fixed number of descendants was studied by Devroye in  90's \cite{Devroye91}. He used a connection between recursive trees and binary search trees, which were well studied then. He also uses the idea of local counters which allows many local properties to be studied by writing them as sums of Bernoulli random variables and using central limit theorems for locally dependent random variables.

The properties that were historically investigated next were two global properties, and thus combinatorial methods turned out not to suffice \cite{Survey}. For these properties different probabilistic methods were developed. Concerning the expectation of the maximum degree, Szymanski proved an upper bound using the averages of the number of nodes of a certain degree developed earlier \cite{Survey}. Actually the maximum degree turned out to converge almost surely to that bound, which was proved by Devroye and Lu in 1995 \cite{DevroyeLu95}. 
Later results on the distribution of the maximum degree were obtained by Goh and Schmutz \cite{Goh01} and Addario-Berry and Eslava \cite{Addario15}. Moreover Eslava investigated the height of nodes of high degree \cite{Eslava16}. The last two papers use a connection between uniform recursive trees and Kingman's coalescence, which is a new approach yet to uniform recursive trees.  Kingman's coalescence is a Markov process that starts with all singletons of $[n]$. At every step subsets or blocks can merge to form new blocks, and all blocks merge with any other at the same rate, until all elements are in the same block \cite{Coalescence}
. 

The other global property about which results could only be obtained in the 90's  is the height of the tree, i.e. the longest path from the root to a leaf. In 1994 Pittel obtained almost sure convergence of the height of uniform recursive trees by using another way of constructing a uniform recursive tree: from a branching structure with population-dependent rate we can obtain a uniform recursive tree by attaching the $i$-th node to the parent of the $i$-th born child in the branching process \cite{Pittel94}. By using results about the connection between the time the $n$-th descendant is born and the birth time of the first member of the $k$-th generation, almost sure convergence could be obtained. Devroye proved this result without using branching processes, this time by making use of a second moment method argument \cite{Devroye11}. 
 Moreover convergence results for the minimum depth of the second half of the nodes are given in that paper.

Finally another important statistic of uniform recursive trees are their branches. By using the fact that subtrees of uniform recursive trees have, conditioned on their size, the same structure as the tree itself, results on the number of leaves and the size of the subtrees were derived in the 90's by Mahmoud and Smythe \cite{MahmoudSmythe91}. In that paper a generalized urn model, similar to the one described above, and known asymptotic results about urn statistics are used. 
The branching structure was further investigated in 2005 by Feng, Su and Hu, mainly by using recursive formulas for branches of a given size \cite{Feng05}. The number of branches is shown to be asymptotically normal while the number of branches of a fixed size converges to a Poisson random variable. Furthermore, results on the size of the largest branch were proven in the same reference.

It is clear that the literature of uniform recursive trees encompasses  a wide collection of methods. Recursions play an important role in various proofs, which is not surprising given that the considered structure is recursive. But the ways used to solve these recursions, or to use these recursions without solving them explicitly, differ considerably, as was described above. Moreover recursive trees can be constructed by at least 4 different processes: by dynamically building the tree step by step, by constructing it from a permutation, by building a genealogy tree for a branching process or by using Kingsman's coalescence. This makes uniform recursive trees very interesting from a mathematical point of view, since on the one hand methods from different areas can be applied in the investigation of uniform recursive trees and on the other hand connections between different mathematical structures can be discovered through the study of uniform recursive trees. Also the bijection between uniform recursive trees and permutations immediately suggests some applications. We will now describe some of them and demonstrate the importance of more general recursive tree models.

Regarding applications, first, recursive trees can  be used as a model for the spread of epidemics \cite{Feng05}. In that case the root stands for the first person infected, and in general node $i$ stand for the $i$-th person infected. Now the second person will definitely be infected by the first one, the third person can in turn be infected by the first or the second one and so on. In a uniform recursive tree, every node is attached to any of the previous nodes with the same probability, so in the model any of the previously infected persons is equally likely to infect the next one. In \cite{Moon74} a different way of modelling the spread of an infection is suggested: Given a uniform recursive tree an infection starts from a node $i$ and spreads to any node attached to $i$ with probability $p$. A uniform recursive tree can thus either model the infected people or the structure the infection spreads on.

Similarly uniform recursive trees are useful in order to determine the genealogy of ancient and medieval texts. As described in \cite{NajockHeyde82}, often the original sources of old texts are lost. By modeling the existent copies as  nodes of a uniform recursive tree whose root is the original text, it is possible to reconstruct the genealogy of these texts. 

Also, recursive trees are used as models for the pyramid scheme \cite{Gastwirth}. The pyramid scheme is a business model that is based on offering people a sales job where they have to pay an initial fee to participate and most of their revenue will come from recruiting new people. By letting the $i$-th node in a recursive tree denote the $i$-th person that participates, the pyramid scheme can be modeled by a uniform recursive tree, which gives estimations on the number of persons that will not even recruit enough new sellers to make up for their initial investment. More generally, a uniform recursive tree can represent a distribution network where the root is the producer, the internal nodes are suppliers and the leaves are retailers, i.e. sell the product to the consumer.

Uniform recursive trees are moreover used to model the spread of a fire in a tree \cite{Fire}. This is done by first determining each edge to be either fire-proof or to be set on fire. When an edge is set on fire it burns all edges connected to it, but cannot pass any edges previously determined as fire-proof. By removing all nodes that are connected to burnt edges, only some connected components of the tree remain. In \cite{Fire} these components and the number of remaining nodes are then investigated.

Having at hand different distributions than the uniform one entails much more flexibility when modeling real life problems. For many applications described above recursive trees seem to be a proper structure to represent the phenomenon in question. Whether the uniform distribution is the most appropriate one is more questionable. Using uniform recursive trees implies that all nodes are identical, or more specifically that every infected person is equally likely to infect the next one, that every book is equally likely to be copied or that every person is equally likely to recruit the next seller. This is obviously not the case in real world applications. Thus for recursive trees to be successfully used it is necessary to investigate the properties of non-uniform distributions on recursive trees.

Parallel to the development of the theory of uniform recursive trees some other recursive tree structures were already studied. One of the most common ones are probably binary recursive trees, which are described in \cite{Flajolet} and can also be represented bijectively by permutations. The binary recursive tree is also very well studied, see for example \cite{Drmota09}.
Binary trees can moreover be generalized  in a straightforward way to $m$-ary trees by specifying that every node can have at most $m$ children \cite{Dobrow}.

Another possibility is to consider plane-oriented recursive trees, i.e. recursive trees where the children of each node are ordered. By choosing each such tree with equal probability and subsequently ignoring different orderings of children, this leads to a non-uniform distribution on increasing trees \cite{Survey}. In this model the attachment probabilities depend on the out-degrees of the nodes \cite{Dobrow}. Plane-oriented recursive trees were first introduced in \cite{Szymanski87} where  results on node degrees were derived and compared to analogous results for the uniform model.

Yet another distribution on recursive trees recently introduced are scaled attachment random recursive trees. There the parent of each node $i$ is chosen as $\lfloor iX_i \rfloor$, where all $X_i$ are identically distributed on $[0,1)$. By choosing the uniform distribution for the $X_i$'s uniform recursive trees can be recovered. This model was introduced in \cite{Devroye11} and subsequently some depth properties were studied.

Another natural generalization of uniform recursive trees, Hoppe trees, was recently considered in \cite{Hoppe}. There, the root is assigned a weight $\theta$, all other nodes get weight 1. Node $i$ then attaches to the root with probability $\frac{\theta}{\theta+i-2}$ and to any other node with probability $\frac{1}{\theta+i-2}$. This model is associated to Hoppe's urn, which has an application in modelling the alleles of a gene with mutation rate $\theta>0$. Concerning many properties like the number of leaves, the height and the depth of node $n$, Hoppe trees behave similarly to uniform recursive trees.

The first model we have chosen to study in this thesis generalizes the idea of Hoppe trees: we  assign every node a weight $\omega_i$. Node $j$ then attaches to node $1 \leq i<j$ with probability $\frac{\omega_i}{\omega_1 + \cdots + \omega_{j-1}}$. 
For this model we first give a coupling construction from a uniform recursive tree on $n$ nodes. We then study the number of branches and the depth of node $n$ and give their expectation and variance, as well as some conditions under which asymptotic normality holds. We moreover derive explicit values for the expectation and the variance for some examples of weight sequences.

For the number of leaves we had to restrict ourselves to a model where the first $k$ nodes have weight $\theta$ and the rest weight 1. For this case we first use a martingale argument to get expressions for the expectation and the variance. We then introduce another coupling between uniform recursive trees and weighted  recursive trees like the one  just described with the additional restriction that $\theta \in \mathbb{N}$. Moreover we introduce a coupling between Hoppe trees and weighted recursive trees where the first $k$ nodes have weight $\omega_i \in \mathbb{R}^{+}$ and the rest have weight 1. This coupling easily allows inferences about the number of leaves of these trees, based on known results about uniform recursive trees and Hoppe trees.

In our model we choose fixed weights for each node, so every step is independent of the structure of the already present tree. This property is crucial in many arguments we use, such as couplings and martingales, as well as in applications of various central limit theorems. It is also worth noting that by defining the tree model via attachment probabilities, implies that we directly use some of the approaches on uniform recursive trees described above, but cannot use others.

Introducing weights is also interesting from the point of view of applications since it allows to introduce diversity among the nodes. In the other non-uniform distributions discussed above, all nodes have the same behaviour, or in other words attract nodes according to the same rule. When a recursive process does not satisfy such conditions, weighted recursive trees can be used to model it more precisely.  Moreover the properties of weighted recursive trees and how much they differ from the uniform model can be interpreted as an indicator for the stability of a process. It is reasonable to assume that it is in general more probable for some nodes to get children as others. For example some persons might be more likely to infect others, some copies of ancient texts are more probable to have been copied again and some people might be more likely to recruit new people. Thus it is interesting to see how much fluctuation in the attachment probabilities of a phenomenon can be tolerated when modelling with uniform recursive trees. If weighted recursive trees have the same asymptotic behaviour as uniform recursive trees under some conditions, we can still model processes that satisfy these conditions by the uniform model. 

The second model we consider, introduced in \cite{AltokIslak},  is based on a completely different approach based on the bijection between the symmetric group and recursive trees. Instead of the uniform distribution we choose a biased riffle shuffle distribution on the symmetric group and then consider the trees obtained from these permutations. We use the aforementioned connections between properties of permutations and recursive trees in order to derive results on the number of branches, the number of nodes with at least $k$ descendants and the depth of node $n$.

Among the non-uniform distributions on the symmetric group the biased riffle shuffle distribution was chosen for several reasons. First of all the possibility to construct inverse riffle shuffles from random variables makes the model tractable. Moreover riffle shuffle permutations are themselves theoretically and practically important \cite{AltokIslak}. Finally biased riffle shuffle distributions vary a lot depending on the chosen parameters. Accordingly, the properties of the corresponding trees can  differ more or less from the uniform case. For example the number of leaves and the number of branches can be limited if appropriate parameters are chosen. This allows us to model more diverse recursive phenomena with more precision.
The two models we have chosen hence have the advantage of reflecting the diversity of approaches there are to uniform recursive trees and also being interesting for applications. 

The rest of the thesis is organized as follows: Chapter \ref{chapter:preliminaries} is devoted to providing the necessary background on graph theory and probability theory techniques that are to be used below. In particular, we begin with a  review of some basic definitions and results from graph theory, especially about trees, and then  introduce the tree statistics that are to be investigated later on. Also,  we include some facts about permutations, and some relevant theorems and methods from probability theory that we will use subsequently.

In Chapter \ref{chapter:review} we then introduce different representations of uniform recursive trees and discuss the current literature on  them. This will allow us to evaluate the distance between statistics of our models and the ones from the uniform case. Here, some non-uniform recursive tree models from the literature are also discussed. 

Next we begin the actual topic of this thesis: In Chapter \ref{chapter:weight} we introduce the weighted recursive tree model and give results about its number of branches, the depth of node $n$ and its number of leaves. Furthermore, we introduce two couplings of uniform recursive trees and weighted recursive trees. 

We introduce the second model we will consider in Chapter \ref{chapter:brt}. Since it is based on a different distribution on the symmetric group we first define riffle shuffle permutations and based on them biased recursive trees. Subsequently we review known results about their number of leaves and then give our results about the number of branches, the number of nodes with at least $k$ descendants and the depth of node $n$. In the conclusion we will give an overview of the problems we could not solve and some further generalizations we think might be interesting.

\chapter{PRELIMINARIES}
\label{chapter:preliminaries}

\section{Graph Theory}
\subsection{First Definitions and Properties}

A \emph{graph} G consists of a set of \emph{vertices}, also called \emph{nodes} or \emph{points}, $V$, denoted by $V(G)$ and a set of \emph{edges}, also called \emph{lines}, $E$, denoted by $E(G)$, such that each element of $E$ is an unordered pair of elements of $V$. The \emph{size} of $G$, denoted by $|G|$ is the number of vertices of $G$. We will denote the number of edges of $G$ by $e(G)$. A graph of size $n$ can have between $0$ and ${n \choose 2}$ edges. A graph of size $n$ with ${n \choose 2}$ edges is called a \emph{complete n-graph}, and is denoted by $K_n$. A graph of size $n$ without edges is called an \emph{empty n-graph} and denoted by $E_n$. The graph $K_1 = E_1$ is called \emph{trivial}. Below, we will only focus on a special type of graphs, trees, which will always be of finite size. 

A graph is called \emph{labeled}, if all its vertices have a name. If a graph is labeled, not only the structure of the graph matters, but also between which nodes these edges exist. Thus, for each graph with vertex set $V$ and edge set $E$, there are $|V|$ labelings. The number of labeled graphs of size $n$ is $2^{\binom{n}{2}}$, since each edge can either be present or absent in the graph \cite{West}. From now on we will consider all graphs to be labeled.

If two vertices  $v$ and $w$ are \emph{joined} by an edge $e=\{v,w\}$, they are said to be   \emph{adjacent} or \emph{neighbouring}, and $v, w$ are called \emph{endpoints} of $e$. An edge will sometimes be denoted $vw$ for convenience. Edges are said to be \emph{adjacent} if they have a common vertex.  A graph without loops, i.e. no edges of the form $\{v,v\}$, and no multiple edges is called \emph{simple}. We will only deal with simple graphs in this thesis.

The set of vertices adjacent to a vertex $v$ is denoted by $A(v)$ and the \emph{degree} of a vertex is denoted by $d(v) := |A(v)|$. Here, $|\cdot|$ is used for the cardinality of the underlying set. A vertex of degree $0$ is called an \emph{isolated vertex}. Since every edge has two endpoints, the well-known hand shaking lemma says that $\sum_{i=1}^{n} d(v_i) = 2e(G)$, where $V(G)= \{ v_1, \dots, v_n\}$.

A \emph{path} $P$ is a sequence of vertices $(v_1, v_2,  \dots, v_i)$ with the property that \linebreak $\{v_1 v_2, v_2 v_3, \dots, v_{i-1}v_i \} \subseteq E$. We sometimes also write  $v_1v_2\dots v_i$ to mean the path through $v_1$ to $v_i$. $v_1$ and $v_i$ are called \emph{endpoints of $P$} and $P$ is said to \emph{join} $v_1$ and $v_i$. The \emph{length} of P is denoted by $\ell := e(P)$. A path is called \emph{simple} if no vertex occurs more than once. A graph is \emph{connected} if between any two vertices in $V$ there exists a path joining them, which, in particular, implies that there is no isolated vertex.

If the edges are ordered pairs of vertices, $G$ is said to be a \emph{directed} graph, and a given edge $e$ is  written as $(v,w)$ or $vw$. In the directed case, the edges $vw$ and $wv$ are not the same, and the edge $vw$ can only be used to go from $v$ to $w$ in a path. In a directed graph the edge $vw$ is said to \emph{start} at vertex $v$ and \emph{end} at vertex  $w$. A directed graph $G = (V,E)$ is called an \emph{oriented} graph if $E \cap E^{-1} = \varnothing$; i.e.  between any two vertices,  there can only be an edge in one direction. 
For a detailed account of graph theory, we refer to the texts  \cite{West, Bollobas, Golumbic, Diestel}.

\subsection{Basics about Trees}

Given a graph $G=(V,E)$, a  \emph{cycle} is a path $v_1v_2\dots v_{\ell}v_1$, with $v_i \in V$ for $i=1,\ldots,l$; i.e. a cycle is a  path with only one endpoint. A graph without any cycles is called a \emph{forest}, and a connected graph without any cycles is called a \emph{tree}, and is usually denoted by $\mathcal{T}$. In a tree all nodes of degree $1$ are called \emph{leaves}. Every tree with at least one edge has at least two leaves \cite{UmitComb}. This can be seen by considering the longest path in $\mathcal{T}$ with distinct nodes, and concluding by contradiction that its endpoints must be leaves. The following proposition summarizes some equivalent formulations of a tree structure. 
\newpage
\begin{prop}[\cite{Diestel, UmitComb}]
The following are equivalent: \begin{enumerate}
\item $\mathcal{T}$ is a tree.
\item Any two vertices in $\mathcal{T}$ are connected by a single path in $\mathcal{T}$.
\item $\mathcal{T}$ is minimally connected, i.e. $\mathcal{T}$ is connected but $\mathcal{T} \setminus e$ is disconnected for every $e \in E(\mathcal{T})$.
\item $\mathcal{T}$ is maximally acyclic, i.e. $\mathcal{T}$ contains no cycle but $\mathcal{T} \cup vw$ does, for any two non-adjacent vertices $v,w \in \mathcal{T}$.
\item $\mathcal{T}$ is connected and has $n-1$ edges.
\item $\mathcal{T}$ has no cycles and $n-1$ edges.
\end{enumerate}
\end{prop}

Cayley's theorem states that there are $n^{n-2}$ undirected labeled trees on $n$ vertices \cite{West}. 
A \emph{rooted} tree is a labeled tree where one node is specified as the \emph{root}.  Cayley's theorem, in particular, implies that there are $n^{n-1}$ rooted labeled trees on $n$ vertices. In a rooted \emph{plane} or \emph{planted} tree the children of each node have a left-to-right ordering \cite{West}. 

An oriented tree where all edges are directed outwards from the root is called \emph{branching tree} \cite{West}. 
In a branching tree, leaves are defined slightly differently, since the root should not be considered a leaf, even if it has degree $1$. Thus, for a directed graph, we define the \emph{outdegree}, $d_+(v):= \left| \{e \in E(\mathcal{T}) : e = (v,w), w \in V(\mathcal{T})\} \right|$, i.e. the number of edges starting in $v$. In a branching tree, a leaf is a vertex with outdegree $0$. Since the trees we will consider are all branching trees, we will from now one use the word leaf in this sense. Similarly to the outdegree, the indegree $d_-(v)$ is defined as the number of edges ending in $v$, i.e. $d_-(v) := \left| \{e \in E(\mathcal{T}) : e = (w,v), w \in V(\mathcal{T})\} \right|$. Hence, in a branching tree, the root has indegree 0 and all other nodes have indegree 1. By definition of the degree of a vertex, it is in general true in an oriented graph that $d(v) = d_+(v) + d_-(v)$.

In a branching tree, if $vw \in E(\mathcal{T})$, $w$ is called a \emph{child} of $v$ and $v$ is called \emph{parent} of $w$. Similarly, if there is a $y \in V(\mathcal{T})$ such that $vy$ and $yw$ are in $E(\mathcal{T})$, $w$ is called \emph{grandchild} of $v$ and $v$ is called \emph{grandparent} of $w$. Every node can only have one parent and one grandparent, since otherwise we would get a cycle. On the other hand, a node can have several children and all the nodes that have the same parent are called \emph{siblings}. There is a single path from the root $r$ to each vertex $v$ and all nodes in this path are called \emph{ancestors} of $v$. Similarly if $v$ lies on the path from $r$ to another node $w$, $w$ is called a \emph{descendant} of $v$ \cite{West}. 

There are many subcategories of trees. Some of these subcategories can be obtained by restricting the number of children a node can have, the most common one being the restriction to two children. Unfortunately there are different definitions for binary trees. In \cite{West} a \emph{binary tree} is defined as a rooted plane tree in which every node can have at most 2 children. 
However, a binary tree is also sometimes defined as a rooted tree where each node can have at most 2 children \cite{UmitComb 
, Diestel}. 
We will use the term binary tree is this latter sense, i.e. we will the not consider the children to be ordered.
In general, for $m\in \mathbb{N}$, an \emph{$m$-ary tree} is a rooted labeled tree where every node has at most $m$ children \cite{UmitComb, Dobrow}. 

The trees we will consider are another subcategory of trees, called \emph{increasing} trees. These are rooted trees such that on every path from the root to a node the labels of the nodes are increasing \cite{UmitComb}. It is also possible to consider the intersection of these two subcategories, i.e. increasing binary trees or increasing $m$-ary trees.

\subsection{Statistics of Interest}\label{sec:stats}

In this subsection, we define certain tree statistics that are studied under various branching structures. Most of these will be directly or indirectly handled in the following chapters for the non-uniform recursive tree models we discuss.  Let $\mathcal{T}$ be a tree. 

Recall that a vertex is called a leaf if it has no children. The \emph{number of leaves} of a branching tree is denoted by $\mathcal{L}_{\mathcal{T}}$. For $|E(\mathcal{T})| > 1$, $1 \leq \mathcal{L}_{\mathcal{T}} \leq n-1$, which can be seen by construction or by considering that for a tree $\sum_{v \in V(E)}d_-(v)= n-1$, because the total number of edges is $n-1$. Similarly, we can also consider the \emph{number of nodes with at least degree $k$}, i.e. the number of nodes $v \in \mathcal{T}$, with $d(v)\geq k$. 

The \emph{number of branches} $\mathcal{B}_{\mathcal{T}}$ of a rooted tree is the number of children of the root. Clearly, we have $1 \leq \mathcal{B}_{\mathcal{T}} \leq n-1$. Since every branch is again a branching structure, with the child of the root as the new root, all statistics defined for trees can also be considered on branches. In particular, we will be interested in the size of the branches, i.e. in the number of nodes that are descendants of a given child of the root. Let $w$ be a child of the root and $b_w$ the branch rooted at $w$. If we define $Anc(v)$ as the set of ancestors of node $v$, then $|b_w| = |\{v \in V: w \in Anc(v)\}|$.

The \emph{depth} $\mathcal{D}_{v}$ of node $v$ in a branching tree is the length of the path from the root to $v$ or equivalently the number of ancestors of $v$. For $v \neq r$, we have $1 \leq \mathcal{D}(v) \leq n-1$. Similarly, the \emph{distance} from one node to another, denoted by $\mathcal{D}_{vw}$, is the length of the path from $v$ to $w$. For general graphs the distance between two nodes is the length of the shortest path from $v$ to $w$ but since in a tree there is only one path between any two nodes, this amounts to the same. 

The \emph{height} $\mathcal{H}_{\mathcal{T}}$ of a rooted tree  is the length of the longest path from the root to a leaf. As for the other statistics $1 \leq \mathcal{H}_{\mathcal{T}} \leq n-1$, again with the same examples. When comparing the height of two trees, $\mathcal{T}$ and $\mathcal{T}'$, we say that $\mathcal{T}$ is \emph{taller} than $\mathcal{T}'$ and $\mathcal{T}'$ is \emph{flater} or \emph{shorter} than $\mathcal{T}$, when $\mathcal{H}_{\mathcal{T}'}<\mathcal{H}_{\mathcal{T}}$.

\section{Permutations}

There are several ways to define permutations, we will only give two here, see \cite{Stanley}. 
While permutations can be defined for any finite set, by giving each element of the set a label from 1 to $n$ we can consider that all permutations are defined on $[n]$.  First of all a \emph{permutation} of size $n$ is an linear ordering of the set $[n]$, say $\pi_1 \pi_2 \dots \pi_n$, where every integer from 1 to $n$ only appears once. In other words, a permutation is a word with letters from $[n]$, such that each letter appears exactly once. Equivalently a permutation can be considered as a bijective map $\pi: [n] \to [n]$, by defining $\pi(i) = \pi_i$.  

Corresponding to these two ways of defining permutations there are two different notations for permutations. Corresponding to the view of a permutation as a bijection, a permutation $\pi$ can be represented as 
\begin{figure}[H]
\begin{minipage}{1\textwidth}
\begin{center}
\begin{tabular}{ c c c c c c }
1 & 2 & 3 & \dots & $n-1$ & $n$\\ 
$\pi(1)$ & $\pi(2)$ & $\pi(3)$ & \dots & $\pi({n-1})$  & $\pi(n)$.
\end{tabular}
\end{center}
\label{fig:PermRepCauchy}
\end{minipage}
\end{figure}
We  call this the \emph{Cauchy representation}. Corresponding to the view of a permutation as a list of the numbers from 1 to $n$, a permutation also has a \emph{word} or \emph{one-line representation},  
\begin{center}
\begin{figure}[H]
\begin{minipage}{1\textwidth}
\begin{center}
\begin{tabular}{ c c c c c c }
$\pi(1)$ & $\pi(2)$ & $\pi(3)$ & \dots & $\pi(n-1)$  & $\pi(n)$.
\end{tabular}
\end{center}
\label{fig:PermRepWord}
\end{minipage}
\end{figure}
\end{center}
Since the word notation is more compact and thus easier to include in the text we will prefer it most of the time. 

The set of all permutations of size $n$ is called \emph{symmetric group} of size $n$ and is denoted by $S_n$.  The number of permutations of $[n]$ is $n!$, since $\pi(1)$ can be chosen among $n$ elements, $\pi(2)$ among $n-1$ elements, and so on. A \emph{uniform random permutation}, or URP of $[n]$ is a permutation chosen uniformly among all permutations of $[n]$.

There are several properties of permutations that will be important later, when we  use them as representations for trees. For a permutation $\pi$ of $[n]$, an \emph{inversion} is a pair $(i,j)\in [n]\times[n]$, such that $i<j$ and $\pi(i)>\pi(j)$. 
For $1 \leq i \leq n-1$, the permutation $\pi$ has a \emph{descent} in $i$ if $\pi(i)> \pi(i+1)$ and an \emph{ascent} in $i$ if $\pi(i) < \pi(i+1)$. For example the permutation $\pi=41562837$ has 3 descents: in 1, 4 and 6, and 4 ascents in 2,3,5 and 7.

Another concept we will often use  are records and anti-records. A record is an element that is greater than all previous ones and an anti-record an element that is smaller than all previous ones. More precisely: a permutation $\pi$ has a \emph{record} in $i$ if $\pi(i) > \{\pi(1), \dots, \pi(i-1)\}$ and an \emph{anti-record} in $i$ if $\pi(i) < \{\pi(1), \dots, \pi(i-1)\}$ \cite{Stanley}. 
Every permutation has a record and an anti-record in $\pi(1)$. 

Records and anti-records are equally distributed in uniform random permutations as we can see by the following bijection:  Consider the map $f: S_n \to S_n$ defined by $f: \pi \to \rho=n-\pi$, i.e. for all $i=1, \dots, n$, we have $\rho(i)=n-\pi(i)$. This map is bijective and if $\pi$ has an anti-record in $i$, then $\rho$ has a record in $i$ by construction. Thus the distributions of the number of anti-records and records in URPs are equal.

Nevzorov summarizes many results about records of sequences of random variables in \cite{Nevzorov01}. The theory of records is very rich and connected to several mathematical structures and also has many applications. The results in \cite{Nevzorov01} also apply for the records in a uniform random permutation by the following standard construction of a uniform recursive permutation from random variables, as described for instance in \cite{AltokIslak}. 
 Let $Y_1, \dots, Y_n$ be independent uniformly distributed random variables over $(0,1)$. Then the \emph{rank} $R_i$ of $Y_i$ is equal to $j$ if $Y_i$ is the $j$-th-largest among $Y_1, \dots, Y_n$. The sequence $(R_1, \dots, R_n)$ is distributed as a uniform random permutation of $[n]$, see \cite{DevroyeRecords}. 
In our case we will often only consider permutations where $\pi(1)=1$ and thus use permutations of $\{2, 3, \dots, n\}$. This means that we mostly use $Y_2, \dots, Y_{n}$ to construct the permutation we need. 

By considering a permutation as a bijection it can also be represented by its cycle structure. Since we only consider permutations of finite sets, for every $i\in [n]$, there will be a unique ${\ell} \in [n]$, such that $\pi^{\ell}(i) = i$. We can thus define a \emph{cycle} of \emph{length} ${\ell}$ of a permutation as a sequence $(i, \pi(i), \pi^2(i), \dots, \pi^{{\ell}-1}(i))$, where $\pi^{\ell}(i)=i$. For $k=0, 1, 2, \dots$, the cycles $(\pi^{k}(i) \pi^{k+1}(i) \dots \pi^{k+{\ell}-1}(i))$ are all representations of the same cycle. Since every element of $[n]$ can only be in one cycle, the different cycles of a permutation are distinct and thus partition $[n]$. Hence we can write every permutation as a product of at most $n$ distinct cycles, $C_1,  \dots, C_m$, i.e.  $\pi=C_1C_2 \dots C_m$. The cycles have different representations because they can start with any member. Moreover the cycles can also be ordered in several ways. Thus the cycle representation of a permutation is not unique. 

To guarantee uniqueness it is common to start every cycle with its smallest element and order the cycles according to this first element from largest to smallest. In this way, even if we remove the parenthesis, we know that a new cycle starts every time there is an anti-record in the permutation. Thus in this standard notation, every permutation has a unique cycle representation. For example, given the permutation $439782516$, we determine the first cycle by considering $\pi(1) =4$, then $\pi(4)=7$, $\pi(7) = 5$, $\pi(5) = 8$, $\pi(8)=1$, which gives the cycle $(14758)$ in standard notation. Then we take the smallest element not in the first cycle, which is in this case 2, and proceed similarly, and so on. Finally we get $\pi = (2396)(14758)$. The advantage of this notation is that one can recover the cycles even if the parenthesis are removed: every time an anti-record, i.e. a new smallest element appears, a new cycle starts.

\section{Probability Theory}

We start with a brief explanation of    the method of indicators as it will be used several times throughout the thesis. The basic idea is to write a discrete random variable $X$ as a sum of Bernoulli random variables. Often $X$ is the number of something, as  in the following example.
\begin{ex}[\cite{DersNotProb}]
At a party $n$ men throw their hat in the air, and then every man chooses one of the hats randomly. Let $X$ be the random variable denoting the number  of men who choose their own hat. In order to calculate $\E[X]$ we define the Bernoulli random variables $X_i:= \1(i\text{-th} \text{ man finds his own hat})$. Then we have
\begin{equation}\E[X] = \E \left [ \sum_{i=1}^n X_i \right ] = \sum_{i=1}^n \E[X_i] = \sum_{i=1}^n \frac{1}{n} = 1.\end{equation}
\end{ex}
When we use this method we will often consider the distribution of the limit of such sums of indicators. The asymptotic results we will prove are results of convergence in distribution.

\begin{mydef}[\cite{DeGroot}]
Let $(X_n)_{n \in \mathbb{N}}$ be a sequence of random variables with cumulative distribution function $F_n(x)$. Let $X$ be another random variable with cumulative distribution function $F(x)$. If for all $x$ at which $F$ is continuous
\begin{equation}\lim_{n\to \infty} F_n(x) = F(x)
\end{equation}
the sequence $(X_n)$ is said to \emph{converge in distribution} to $X$ and we write $X_n \xrightarrow[]{d} X$.
\end{mydef}
When comparing random variables we will use three different probability metrics. Let $\mu$ and $\nu$ be two probability measures. Then each of the metrics we will consider have the form
\begin{equation}d_{\mathcal{H}}(\mu,\nu) = \sup_{h \in \mathcal{H}}\left | \int h(x)d \mu(x) - \int h(x)d \nu(x)\right | \end{equation}
where $\mathcal{H}$ is some family of functions.
By extension this also gives a distance function for random variables: if $X$ and $Y$ are random variables with respective laws $\mu$ and $\nu$, then $d_{\mathcal{H}}(X,Y) = d_{\mathcal{H}}(\mu, \nu).$
Depending on the set $\mathcal{H}$, this form gives rise to different metrics.
\begin{mydef}
\begin{enumerate}
\item The \emph{Kolmogorov metric} is obtained by setting $\mathcal{H} = \{\1(x\leq a): a \in \mathbb{R}\}$, and is denoted by $d_K$.
\item We get the \emph{Wasserstein metric} if we set $\mathcal{H} = \{ h: \mathbb{R} \to \mathbb{R}: |h(x)-h(y)| \leq |x-y| \}$, and we denote it by $d_W$. This is the main metric used for approximations by continuous distributions.
\item The \emph{total variation metric} is obtained by choosing $\mathcal{H} = \{ \1(x \in A): A \in \mathcal{B}(\mathbb{R}) \}$ and is denoted by $d_{TV}$. This metric is commonly used for approximations by discrete distributions.
\end{enumerate}
\end{mydef}
These metrics have some important properties, which can be found in \cite{Ross11}. 
\begin{prop}[\cite{Ross11}]
\begin{enumerate}
\item For any two random variables $X$ and $Y$, \begin{equation}d_K(X,Y) \leq d_{TV}(X,Y).\end{equation}
\newpage
\item If the density function of a random variable $X$ is bounded by a constant $C$, and $Y$ is any random variable, \begin{equation}d_K(X,Y) \leq \sqrt{2Cd_W(X,Y)}.\end{equation}
\item If the random variables $X$ and $Y$ take values in a discrete space $A$, \begin{equation}d_{TV} (X,Y) = \frac{1}{2} \sum_{a \in A} |\Pro(X=a) - \Pro(Y=a)|.\end{equation}
\item The Kolmogorov metric gives the maximum distance between distribution functions. Thus, if for a sequence of random variables $X_1, X_2, \dots$ and a random variable $Y$, $d_W(X_n,Y) \to 0$, then $(X_n)_{n \in \mathbb{N}}$ converges to $Y$ in distribution.
\end{enumerate}
\end{prop}
For proofs of items $(i), (ii),$ and $(iv)$ see \cite{Ross11} 
and for $(iii)$ see \cite{MarkovMixing}. 
Using Stein's method several bounds on these metrics for different kinds of random variables are then given in \cite{Ross11}. While we will mainly use these bounds to determine rates of convergence, Stein's method is a general tool to determine the distance between two probability measures. 
It was introduced by Stein in \cite{Stein72} as a method to bound the error in approximations of sums of random variables with the normal distribution. Later the method was generalized and applied on bounds of approximations of more general random variables as well as with other distributions. The first part of the main  idea of Stein's method is to bound the distance between the random variable we want to approximate and a well-known distribution by the expectation of a functional of the random variable we want to approximate. The second part consists in methods to bound the expectation of that functional \cite{Ross11}.

In order to analyze the asymptotic behavior of the  statistics introduced in the previous section  we also need some limit theorems. In  subsequent theorems we  use the following notation: in general $Y_1, Y_2, \dots$ denote a sequence of random variables and we define $\mu_i:= \E[Y_i]$ and $\sigma_i^2 := \Var(Y_i)$. Also we set
 \begin{equation}W_n:= \frac{\sum_{i=1}^n  Y_i - \mu_i}{\left (\sum_{i=1}^n \sigma_i^2\right )^{\frac{1}{2}}}.\end{equation}
Note that we have $\E[W_n] = 0$ and $\Var(W_n) = 1$. The following theorems give conditions on when $W_n$ converges to a standard normal distribution.

\begin{thm}[Lindeberg-Feller's central limit theorem, \cite{Ash00}] 
Let $Y_1, Y_2, \dots$ be independent random variables such that  $\mu_i < \infty$ and $\sigma_i^2< \infty$ for all $i =1, 2, \dots$. Define $s_n^2 = \sum_{i=1}^{n} \Var(Y_i)$ and let $W_n$ be as above. Assume that for all $\epsilon>0$,
\begin{equation}\lim_{n \to \infty} \frac{1}{s_n^2} \sum_{i=1}^{n} \E \left [(Y_i-\mu_i)^2 \1(|Y_i-\mu_i| > \epsilon s_n) \right ]= 0.
\end{equation}
Then \begin{equation}
\lim_{n\to \infty} W_n =_d \mathcal{G}.
\end{equation}
\end{thm}
Another condition for sums of independent random variables to converge, that implies the Lindeberg-Feller condition, is Liapounov's condition.
\begin{thm}[Liapounov's central limit theorem, \cite{DeGroot}] 
Let $Y_1, Y_2, \dots$ be independent random variables and  $\mu_i$, $\sigma_i^2$ and $W_n$ be as above. Now assume that $\E\left[|Y_i-\mu_i|^3\right] < \infty$ for all $i=1,2,\dots$ and that 
\begin{equation}\lim_{n \to \infty} \frac{\sum_{i=1}^n\E[|Y_i- \mu_i|^3]}{\left ( \sum_{i=1}^n\sigma_i^2 \right )^{\frac{3}{2}}}=0.\end{equation}
Then
\begin{equation}\lim_{n\to \infty} W_n  =_d \mathcal{G}.\end{equation}
\end{thm} 
\newpage
The following case of the Liapounov theorem will suffice for our purposes. 
\begin{thm}[Liapounov's central limit theorem for sums of independent Bernoulli random variables, \cite{DeGroot}] 
\label{thm:LiapounovBernoulli}
Let $Y_1, Y_2, \dots$ be independent Bernoulli random variables with parameter $p_i$, $i= 1, 2, \dots$. Letting $W_n$ be as above we have
\begin{equation}W_n = \frac{\sum_{i=1}^n Y_i - \sum_{i=1}^n p_i}{\left (\sum_{i=1}^n p_i(1-p_i) \right )^{\frac{1}{2}}}.
\end{equation}
If the infinite series $\sum_{i=1}^\infty p_i(1-p_i)$ diverges, then \begin{equation} \lim_{n\to \infty} W_n =_d \mathcal{G}.\end{equation}
\end{thm}
In connection to asymptotic convergence to the normal random variable, the following result is also of interest.
\begin{thm}[\cite{Barbour92}]
\label{thm:PoisNorm}
Let $W_n$ be a sequence of random variables. If there  is a sequence of real numbers $\lambda_n$ such that $\lambda_n \to \infty$ and 
\begin{equation}d_{TV}(W_n, \Po(\lambda_n)) \xrightarrow[n \to \infty]{} 0 
\end{equation}
then 
\begin{equation}\lim_{n \to \infty} \frac{W_n-\lambda_n}{\sqrt{\lambda_n}} =_d \mathcal{G}.\end{equation}
\end{thm}
As a bound on the distance between sums of Bernoulli random variables and a Poisson distribution, we will use the following result.
\begin{thm}[Law of small numbers, \cite{Ross11}] 
\label{thm:smallnumbers}
Let $Y_1, Y_2, \dots, Y_n$ are independent indicator random variables with $\Pro(Y_i=1) = p_i$, $Y= \sum_{i=1}^n Y_i$ and $\mu = \E[Y] = \sum_{i=1}^{n} p_i $. Then
\begin{equation}d_{TV}(Y,\Po(\mu)) \leq \min\{1,\mu^{-1}\} \sum_{i=1}^n p_i^2.
\end{equation}
\end{thm}
As we will also consider several cases of sums of locally dependent random variables, we will now give two bounds for approximations of sums of locally dependent random variables by the normal distribution.

\begin{mydef}[\cite{Ross11}] 
\label{localdep}
Let $(Y_1, Y_2, \dots, Y_n)$ be a collection of random variables. For each $i$, we call $N_i$ the \emph{dependency neighbourhood} of $Y_i$, if $Y_i$ is independent of $\{Y_j\}_{j \notin N_i}$ and $i \in N_i$.
\end{mydef}
For random variables with such dependency neighbourhoods the following convergence theorem is true:
\begin{thm}[\cite{Ross11}]
 \label{thm:normalrate}
Let $Y_1, Y_2, \dots, Y_n$ be random variables for which $\E[Y_i^4] < \infty$ and $\E[Y_i]=0$ holds. Let moreover $N_i$ be the dependency neighbourhoods of $(Y_1, \dots, Y_n)$ and define $D:= \max_{1\leq i \leq n} \{|N_i|\}$. Finally set $\sigma^2 = \Var\left ( \sum_{i=1}^n Y_i\right )$ and define $W:= \sum_{i=1}^n \frac{Y_i}{\sigma}$. Then
\begin{equation}d_W(W,\mathcal{G}) \leq \frac{D^2}{\sigma^3} \sum_{i=1}^{n} \E[|Y_i|^3] + \frac{\sqrt{28}D^{\frac{3}{2}}}{\sqrt{\pi}\sigma^2} \sqrt{\sum_{i=1}^n \E[Y_i^4]}. \end{equation}
\end{thm}
Note that this theorem can also be applied to independent random variables by setting $D=1$, since $i \in N_i$ for all $i$.
There is another more restricted definition of local dependence, which will mostly suffice for our cases.
\begin{mydef} \label{mdep}
For $m \in \mathbb{N}_0$, a sequence $(Y_i)_{i \in \mathbb{N}}$ of random variables is called \emph{$m$-dependent} if for all $i \in \mathbb{N}$, the sets $\{Y_j, j \leq i\}$ and $\{Y_j, i+m< j\}$ are independent.
\end{mydef}
Note that a sequence of random variables is independent if it is $0$-dependent in the above sense. Definition \ref{mdep} is a special case of Definition \ref{localdep} where the dependency sets do not have to be sets of random variables with consecutive indices.
In order to derive asymptotic results for $m$-dependent random variables, we will need a special case of theorem 9.4 in \cite{NormalStein}. 
There, the random variables are indexed over $\mathbb{N}^d$, so we need to take $d=1$. This theorem was also proved by above mentioned Stein's method.
\begin{thm}
Let $Y_1, Y_2, \dots$ be a sequence of zero-mean $m$-dependent random variables and $W_n:= \frac{\sum_{i=1}^{n} Y_i}{\left ( \sum_{i=1}^{n} \sigma_i^2 \right )\frac{1}{2}}$. Then for all $p \in (2,3]$, 
\begin{equation}d_W(W_n,\mathcal{G}) \leq 75(10m+1)^{p-1} \sum_{i=1}^{n} \E[|Y_i|^2].
\end{equation}
\end{thm}

In the analysis of the number of leaves of a special kind of weighted recursive tree, we will moreover need the following concentration inequality for martingale difference sequences, which is Theorem 3.13 from \cite{McDiarmid}.
\begin{thm}[\cite{McDiarmid}]
\label{thm:McDiarmid}
 Let $Y_1, Y_2, \dots, Y_n$ be a martingale difference sequence with $a_i < Y_i <b_i$ for each $i$, for suitable constants $a_i, b_i$. Then for any $t\geq 0$, 
\begin{equation}\Pro\bigg ( \bigg| \sum_{i=1}^n Y_i \bigg | \geq t \bigg )  \leq 2 e^{-\frac{2t^2}{\sum_{i=1}^n (b_i-a_i)^2}}.
\end{equation} 
\end{thm}
By using a coupling construction we will moreover be able to bound some differences of random variables. In order to derive results about the asymptotic behaviour of the unknown random variable from the asymptotic distribution of the known one, we will use Slutsky's Theorem.
\begin{thm}[Slutsky's theorem\cite{Ash00}] \label{thm:Slutsky}
Let $X_n$ and $Y_n$ be sequences of random variables such that $X_n \to_d X$ and $Y_n \to_d c$ for $c \in \mathbb{R}$. Then
\begin{equation} \lim_{n\to \infty} X_n + Y_n  =_d  X +c.\end{equation}
\end{thm}

\chapter{REVIEW OF UNIFORM RECURSIVE TREES} \label{chapter:review}

\section{Definition of Uniform Recursive Trees}
In  following chapters we will consider different distributions on the space of recursive trees.
\emph{Increasing} or \emph{recursive} trees are branching trees whose vertices are labeled by $\{1, \dots, n\}$ such that on every path starting from the root the labels are increasing, which implies that node $1$ is the root. This property allows us to picture many increasing trees as growing dynamically, such that an increasing tree of size $n$ is obtained by joining node $n$ to an increasing tree of size $n-1$ according to some rule \cite{Dobrow}. We will see later that this is not the case for biased recursive trees. 
Figure \ref{inctrees4} shows all increasing trees on $4$ vertices. 
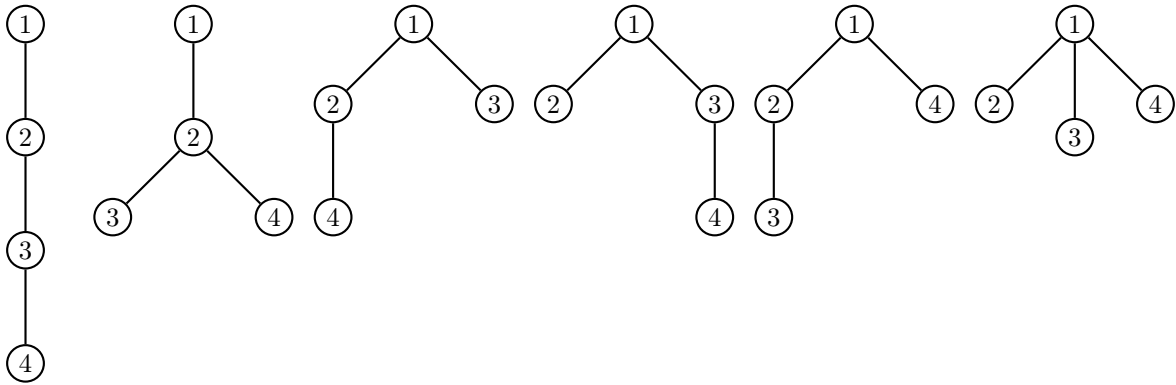
\begin{figure}[H]
  \tikzstyle{every node}=[circle,inner sep=2pt, draw]
\tikzset{ font={\fontsize{10pt}{12}\selectfont}}
\begin{tikzpicture}[scale=0.69, node distance=1.5cm,
  thick]
  \node (1) {1};
  \node (2)  [below of=1] {2};
  \node (3) [below of =2] {3};
  \node (4) [below of =3] {4};
  
\path[every node/.style={font=\sffamily\small}]
    (1) edge node {} (2)
    (2) edge node {} (3)
    (3) edge node {} (4);

\begin{scope}[xshift=3.2cm]
    
  \node (1) {1};
  \node (2)  [below of=1] {2};
  \node (3) [below left of =2] {3};
  \node (4) [below right of =2] {4};
  
\path[every node/.style={font=\sffamily\small}]
    (1) edge node {} (2)
    (2) edge node {} (3)
    (2) edge node {} (4);

\end{scope}

\begin{scope}[xshift=7.4cm]     
    
  \node (1) {1};
  \node (2)  [below left of=1] {2};
  \node (3) [below right of =1] {3};
  \node (4) [below of =2] {4};
  
\path[every node/.style={font=\sffamily\small}]
    (1) edge node {} (2)
    (1) edge node {} (3)
    (2) edge node {} (4);

\end{scope}

\begin{scope}[xshift = 11.6cm]
  \node (1) {1};
  \node (2)  [below left of=1] {2};
  \node (3) [below right of =1] {3};
  \node (4) [below of =3] {4};
  
\path[every node/.style={font=\sffamily\small}]
    (1) edge node {} (2)
    (1) edge node {} (3)
    (3) edge node {} (4);
    \end{scope}

    \begin{scope}[xshift = 15.8cm]
  \node (1) {1};
  \node (2)  [below left of=1] {2};
  \node (4) [below right of =1] {4};
  \node (3) [below of =2] {3};
  
\path[every node/.style={font=\sffamily\small}]
    (1) edge node {} (2)
    (2) edge node {} (3)
    (1) edge node {} (4);
    \end{scope}

        \begin{scope}[xshift = 20cm]
  \node (1) {1};
  \node (2)  [below left of=1] {2};
  \node (4) [below right of =1] {4};
  \node (3) [below of =1] {3};
  
\path[every node/.style={font=\sffamily\small}]
    (1) edge node {} (2)
    (1) edge node {} (3)
    (1) edge node {} (4);
    \end{scope}
   
\end{tikzpicture}
\caption{All increasing trees on $4$ vertices.} \label{inctrees4}
\end{figure}
A \emph{uniform recursive tree},  abbreviated by  URT, of size $n$ is a random tree that is chosen uniformly among all increasing trees of size $n$. Since attaching node $n$ to any of the nodes $\{1, \dots, n-1\}$ gives an increasing tree, there are $(n-1)!$ recursive trees of size $n$.

\subsection{Recursive Construction of a Uniform Recursive Tree}

This is equivalent to the following recursive construction principle. A URT $\mathcal{T}_n$ of size $n$ is obtained from a URT $\mathcal{T}_{n-1}$ of size $n-1$ by joining node $n$ to any of the nodes $\{1, \dots, n-1\}$ with equal probability. If we start from scratch, first node $1$ is added as the root, and node $2$ is attached to node $1$. Then, node $3$ is either attached to node $1$ or to node $2$  with equal probability $\frac{1}{2}$. In general node $i$ attaches to any of the nodes $\{1, \dots, i-1\}$ with probability $\frac{1}{i-1}$ \cite{AltokIslak}. It is important to note that every step is independent of the previous ones for URTs since the structure of $\mathcal{T}_{n-1}$ does not affect which node becomes the parent of $n$. We will call this way of constructing a URT \emph{construction principle} and refer to the probabilities for nodes to attach to other nodes as the \emph{attachment probabilities}.

\subsection{Construction of a Recursive Tree from a Permutation}

Besides the two alternative definitions of a URT described above, we may also make use of random  permutations to generate such trees. Given a permutation $\pi$ of $\{2, \dots, n\}$, we construct a URT as follows: First, $2$ is attached to $1$, then $3$ is connected to $1$ if it is to the left of $2$, otherwise to $2$. In general the node $i$ is attached to the rightmost node to the left of $i$ that is less than $i$. If there is no smaller number than $i$ to its left, $i$ is attached to 1. Thus $i= \pi(s)$ attaches to node $j=\pi(r)$ when $r = \max \{t \in \{1, \dots, s-1\}: \pi(t) < \pi(s) \}$, where we set $\pi(1) = 1$. We will call this way of constructing a URT the \emph{construction from a permutation}. Figure \ref{StepByStep} shows an example of the step by step construction of a recursive tree corresponding to a permutation. 
Similarly, given a URT $\mathcal{T}_n$, we can construct the corresponding permutation by writing $1$ to the very left, $2$ on its right and then step by step every node $i$ directly on the right of the node it is attached to.

\begin{figure}
  \tikzstyle{every node}=[circle, draw]
\begin{tikzpicture}[node distance=1.6cm,
  thick]
  
  \node[draw=none,rectangle]  (p) {\underline{1}6387254};
  
  \node (1) [below=3mm  of p] {1};

\begin{scope}[xshift=2cm]
         \node[draw=none,rectangle]  (p)  {\textbf{1}6387\underline{2}54};
  \node (1) [below=3mm of p] {1};
  \node (2)  [below of=1] {2};
  
\path[every node/.style={font=\sffamily\small}]
    (1) edge node {} (2);

\end{scope}

\begin{scope}[xshift=4.5cm]     

          \node[draw=none,rectangle]  (p) {\textbf{1}6\underline{3}87254};
  \node (1) [below= 3mm of p] {1};
  \node (2)  [below left of=1] {2};
  \node (3) [below right of =1] {3};
  
\path[every node/.style={font=\sffamily\small}]
    (1) edge node {} (2)
    (1) edge node {} (3);

\end{scope}

\begin{scope}[xshift = 8cm]
 \node[draw=none,rectangle]  (p) {16378\textbf{2}5\underline{4}};
   \node (1) [below=3mm of p] {1};
  \node (2)  [below left of=1] {2};
  \node (3) [below right of =1] {3};
  \node (4) [below of =2] {4};
  
\path[every node/.style={font=\sffamily\small}]
    (1) edge node {} (2)
    (1) edge node {} (3)
    (2) edge node {} (4);
    \end{scope}
    
    \begin{scope}[xshift = 12.5cm]
    \node[draw=none,rectangle]  (p)  {16378\textbf{2}\underline{5}4};
   \node (1) [below=3mm of p] {1};
  \node (2)  [below left of=1] {2};
  \node (3) [below right of =1] {3};
  \node (4) [below left of =2] {4};
  \node (5) [below right of=2] {5};
  
\path[every node/.style={font=\sffamily\small}]
    (1) edge node {} (2)
    (1) edge node {} (3)
    (2) edge node {} (4)
    (2) edge node {} (5);
    \end{scope}
    
        \begin{scope}[xshift = 2cm, yshift = -5cm]
          \node[draw=none,rectangle]  (p)  {\textbf{1}\underline{6}387254};
   \node (1) [below=3mm of p] {1};
  \node (2)  [below left of=1] {2};
  \node (3) [below of =1] {3};
  \node (4) [below left of =2] {4};
  \node (5) [below of=2] {5};
  \node (6) [below right of =1] {6};
  
\path[every node/.style={font=\sffamily\small}]
    (1) edge node {} (2)
    (1) edge node {} (3)
    (2) edge node {} (4)
    (2) edge node {} (5)
    (1) edge node {} (6);
    \end{scope}
    
            \begin{scope}[xshift = 7cm, yshift = -5cm]
             \node[draw=none,rectangle]  (p)  {16\textbf{3}8\underline{7}254};
   \node (1) [below= 3mm of p] {1};
  \node (2)  [below left of=1] {2};
  \node (3) [below of =1] {3};
  \node (4) [below left of =2] {4};
  \node (5) [below of=2] {5};
  \node (6) [below right of =1] {6};
  \node (7) [below of =3] {7};
  
\path[every node/.style={font=\sffamily\small}]
    (1) edge node {} (2)
    (1) edge node {} (3)
    (2) edge node {} (4)
    (2) edge node {} (5)
    (1) edge node {} (6)
    (3) edge node {} (7);
    \end{scope}
    
\begin{scope}[xshift = 12cm, yshift = -5cm]
  \node[draw=none,rectangle]  (p)  {16\textbf{3}\underline{8}7254};
   \node (1) [below=3mm of p] {1};
  \node (2)  [below left of=1] {2};
  \node (3) [below of =1] {3};
  \node (4) [below left of =2] {4};
  \node (5) [below of=2] {5};
  \node (6) [below right of =1] {6};
  \node (7) [below of =3] {7};
  \node (8) [below right of =3] {8};
  
\path[every node/.style={font=\sffamily\small}]
    (1) edge node {} (2)
    (1) edge node {} (3)
    (2) edge node {} (4)
    (2) edge node {} (5)
    (1) edge node {} (6)
    (3) edge node {} (7)
    (3) edge node {} (8);
    \end{scope}

\end{tikzpicture}
\caption{Step by step construction of the recursive tree corresponding to 16387254. The newly attached node is underlined and its parent bold.}   \label{StepByStep}
\end{figure}

That this relation gives a bijection between URTs and URPs can be seen by the symmetry of the recursive constructions. The tree and the permutation corresponding to it can be constructed simultaneously, since we can construct a uniform random permutation by inserting the numbers successively, see \cite{AltokIslak}. Given any permutation $\pi$ of $\{2, \dots, k\}$ we can construct a permutation $\pi'$ of $\{2, \dots, k+1\}$ by inserting $k+1$ at the ultimate left or right or between any $\pi(i)$ and $\pi(i+1)$ for $i = 2, \dots, k-1$. In total there are $k$ spots where we can insert $k+1$. If at each step we choose the spot for the next number uniformly, this process gives a uniform random permutation $\pi'$ on $\{2, \dots, k+1\}$. Note that if we insert $k+1$ between $\pi_i$ and $\pi_{i+1}$, $\pi'(s)= \pi(s)$ for all $s\leq i$, $\pi (i+1) = k+1$ and $\pi'(s) = \pi(s-1)$ for all $s>i$. 

\subsection{Simultaneous Construction of a URP in Cycle Notation and a URT} \label{BijURTCycle}

By using the cycle notation for permutations, we can define another way of defining a bijection between URPs and URTs. We will show how we can simultaneously construct a URP in cycle notation and a URT \cite{Drmota09}: 
We start with the node 1. At the first step we attach 2 to 1. The corresponding permutation is $\pi=(2)$. At each following step, we proceed as follows: if $j$ is attached to the root in the URT, we add a cycle of length 1 that only contains $j$. If $j$ is attached to a node $i>1$, we add $j$ to the cycle of $i$ right after $i$. Thus there are in total $j-1$ slots we can put $j$ in. If we choose uniformly among all possibilities, we get a URT and a URP in cycle notation simultaneously. 
\newpage
Since we want the permutation to have the standard cycle notation, when we start a new cycle, we put it in the leftmost place. This is the only thing we need to do, since when adding a number to an existent cycle, it will always be bigger than the already present ones, hence the condition that every cycle starts with its smallest element is automatically satisfied. It is easy to see that we simultaneously get a unique URT and a unique URP by following this process, so this is a bijection. 
See Figure \ref{fig:URTCycle} for an example.
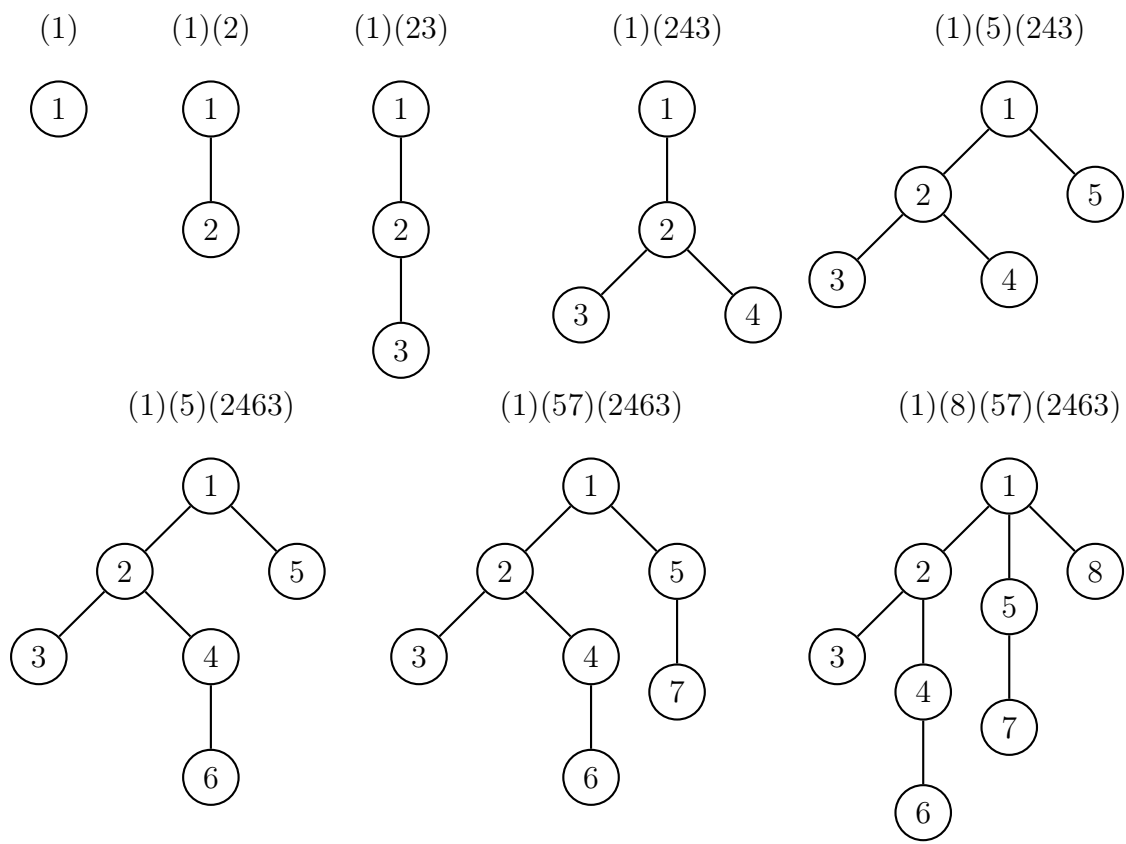
\begin{figure}[H]
   \tikzstyle{every node}=[circle, draw]
\begin{tikzpicture}[node distance=1.6cm,
  thick]
  
  \node[draw=none,rectangle]  (p) {(1)};
  
  \node (1) [below=3mm  of p] {1};

\begin{scope}[xshift=2cm]
         \node[draw=none,rectangle]  (p)  {(1)(2)};
  \node (1) [below=3mm of p] {1};
  \node (2)  [below of=1] {2};
  
\path[every node/.style={font=\sffamily\small}]
    (1) edge node {} (2);

\end{scope}

\begin{scope}[xshift=4.5cm]     

          \node[draw=none,rectangle]  (p) {(1)(23)};
  \node (1) [below= 3mm of p] {1};
  \node (2)  [below of=1] {2};
  \node (3) [below of =2] {3};
  
\path[every node/.style={font=\sffamily\small}]
    (1) edge node {} (2)
    (2) edge node {} (3);

\end{scope}

\begin{scope}[xshift = 8cm]
 \node[draw=none,rectangle]  (p) {(1)(243)};
   \node (1) [below=3mm of p] {1};
  \node (2)  [below of=1] {2};
  \node (3) [below left of =2] {3};
  \node (4) [below right of =2] {4};
  
\path[every node/.style={font=\sffamily\small}]
    (1) edge node {} (2)
    (2) edge node {} (3)
    (2) edge node {} (4);
    \end{scope}
    
    \begin{scope}[xshift = 12.5cm]
    \node[draw=none,rectangle]  (p) {(1)(5)(243)};
   \node (1) [below=3mm of p] {1};
  \node (2)  [below left of=1] {2};
  \node (3) [below left  of =2] {3};
  \node (4) [below right of =2] {4};
  \node (5) [below right of=1] {5};
  
\path[every node/.style={font=\sffamily\small}]
    (1) edge node {} (2)
    (2) edge node {} (3)
    (2) edge node {} (4)
    (1) edge node {} (5);
    \end{scope}
    
        \begin{scope}[xshift = 2cm, yshift = -5cm]
          \node[draw=none,rectangle]  (p)  {(1)(5)(2463)};

   \node (1) [below=3mm of p] {1};
  \node (2)  [below left of=1] {2};
  \node (3) [below left of =2] {3};
  \node (4) [below right of =2] {4};
  \node (5) [below right of=1] {5};
  \node (6) [below of=4] {6};
  
\path[every node/.style={font=\sffamily\small}]
    (1) edge node {} (2)
    (2) edge node {} (3)
    (2) edge node {} (4)
    (1) edge node {} (5)
    (4) edge node {} (6);
    \end{scope}
    
            \begin{scope}[xshift = 7cm, yshift = -5cm]
      
          \node[draw=none,rectangle]  (p)  {(1)(57)(2463)};

   \node (1) [below=3mm of p] {1};
  \node (2)  [below left of=1] {2};
  \node (3) [below left of =2] {3};
  \node (4) [below right of =2] {4};
  \node (5) [below right of=1] {5};
  \node (6) [below of=4] {6};
  \node (7) [below of=5] {7};
  
\path[every node/.style={font=\sffamily\small}]
    (1) edge node {} (2)
    (2) edge node {} (3)
    (2) edge node {} (4)
    (1) edge node {} (5)
    (4) edge node {} (6)
    (5) edge node {} (7);
    \end{scope}
    
\begin{scope}[xshift = 12.5cm, yshift = -5cm]
          \node[draw=none,rectangle]  (p)  {(1)(8)(57)(2463)};

   \node (1) [below=3mm of p] {1};
  \node (2)  [below left of=1] {2};
  \node (3) [below left of =2] {3};
  \node (4) [below  of =2] {4};
  \node (5) [below of=1] {5};
  \node (6) [below  of=4] {6};
  \node (7) [below of=5] {7};
  \node(8) [below right of=1] {8};
  
\path[every node/.style={font=\sffamily\small}]
    (1) edge node {} (2)
    (2) edge node {} (3)
    (2) edge node {} (4)
    (1) edge node {} (5)
    (4) edge node {} (6)
    (5) edge node {} (7)
    (1) edge node {} (8);
    \end{scope}

\end{tikzpicture}
\caption{Simultaneous construction of a uniform recursive tree and a permutation in cycle notation.} \label{fig:URTCycle}
\end{figure}

When comparing the two permutations one gets for the same tree by these two bijections, one can see that the numbers are actually in the same order, we just interpret the sequence in two different ways: either as a permutation in Cauchy notation or in cycle notation. Thus we get the cycle representation of a URT from the Cauchy representation by simply starting a new cycle at every anti-record.

By simultaneously considering a uniform recursive tree and  a uniform random permutation in cycle notation in this way, we can see that the construction is also equivalent to a \emph{Chinese restaurant process}: the cycles can stand for round tables and every node corresponds to a customer. When a new customer arrives, s/he can either join an already existing table, or sit at a new one \cite{Drmota09}. \label{ChineseRestaurant}

\section{A Brief Literature Review on Uniform Recursive Trees}

Various aspects of URTs are well studied, and in particular the statistics we introduced in Section \ref{sec:stats} are deeply understood in most cases.  In this section, we briefly go over some known results on URTs. For a more detailed survey on the subject see \cite{Survey} or the relevant chapters in \cite{Drmota09}.

As mentioned in the introduction there are several ways of studying the leaves of a URT. By using the bijection between URTs and URPs described above, the expectation and variance of the number of leaves in a URT can easily be derived. First we observe that $i$ is a leaf if and only if none of the nodes in  $\{i+1, \dots, n\}$ is attached to $i$. This is the case if and only if none of $\{i+1, \dots, n\}$ is inserted in the spot to the right of $i$. Thus $i = \pi(r)$ is a leaf if and only if $\pi({r}) > \pi({r+1})$, because this implies by the construction principle that $\pi(r)$ cannot be the closest smaller label to the left of any $j>i$. Moreover $\pi(n)$ definitely is a leaf.
Thus we can write $\mathcal{L}_n$, the number of leaves of the URT $\mathcal{T}_n$, as 
\begin{equation}\mathcal{L}_n =_d \sum_{r=2}^{n-1} \1({\pi(r) > \pi({r+1})}) +1.\end{equation}
Since the permutation is uniformly random, $\E(\1(\pi(r) > \pi(r+1)) =1) =1/2$, and we can immediately conclude that $\E[\mathcal{L}_n] = \frac{n}{2}$. Similar considerations yield $\Var(\mathcal{L}_n) = \frac{n}{12}$. 
Concerning the limiting distribution of $\mathcal{L}_n$, we moreover have the following theorem:
\newpage
\begin{thm}[\cite{AltokIslak}]
Let $\mathcal{T}_n$ be a URT of size $n$ and $\mathcal{L}_n$ the number of leaves of $\mathcal{T}_n$. Then
\begin{equation}d_K \left ( \frac{\mathcal{L}_n - n/2}{\sqrt{n/12}}, \mathcal{G} \right ) \leq \frac{C}{\sqrt{n}}
\end{equation}
and for any $x>0$
\begin{equation}\max \left \{ \Pro \left (\mathcal{L}_n-\frac{n}{2} \geq x \right), \Pro \left (\mathcal{L}_n - \frac{n}{2} \leq -x \right ) \right \} \leq e^{-\frac{2x^2}{n}}.
\end{equation}
\end{thm}

As mentioned in the introduction, the expectation of the number of nodes of degree $k$ was also established early on.

\begin{thm}[\cite{NaRapoport}]
Let $C_n^k$ denote the number of nodes of degree $k$ in a URT $\mathcal{T}_n$ of size $n$. Then 
\begin{equation} \frac{\E \left [C_n^k \right ]}{n} \xrightarrow {n \to \infty} \frac{1}{2^{k}}.\end{equation}
\end{thm}

There are several methods to analyze the distribution of the number of branches, $\mathcal{B}_n$ in a URT. We will introduce three of them here because we will need them later.  The first approach is to use the construction principle, see for example \cite{Feng05}. Since the number of branches is equal to the number of children of node $1$, we set $X_i = \1(\text{$i$  is a child of 1})$, and write $B_n$ as a sum of these indicator random variables:
\begin{equation}\mathcal{B}_n = \sum_{i=2}^{n}X_i = 1 + \sum_{i=3}^{n}X_i
.\end{equation}
Since it is equally probable for node $i$ to attach to any of the already present nodes, $\Pro(X_i = 1) = \frac{1}{i-1}$ for $i = 2,\ldots, n$. This implies $\E[X_i] = \frac{1}{i-1}$ and $\Var(X_i) = \frac{i-2}{(i-1)^2}$. 
Thus, we get 
\begin{equation}\E [\mathcal{B}_n] = \sum_{i =2}^{n} \E [X_i] = \sum_{i =2}^{n} \frac{1}{i-1} = \sum_{i =1}^{n-1} \frac{1}{i} = H_{n-1}
\end{equation}
and for large $n$ we have $\E [B_n]= \ln(n) + \mathcal{O}(1)$.
Also
\begin{equation}\Var (\mathcal{B}_n) = \sum_{i =2}^{n} \Var (X_i) = \sum_{i =2}^{n} \frac{i-2}{(i-1)^2} = \sum_{i=1}^{n-1} \frac{i-1}{i^2} = H_{n-1}- H_{n-1}^{(2)}
\end{equation}
and for large $n$ we have $\Var (\mathcal{B}_n) = \ln(n) + \mathcal{O}(1)$.
Moreover these results allow us to derive a central limit theorem for $\mathcal{B}_n$ by using the Lindeberg-Feller theorem.
\begin{thm}[\cite{Feng05}]
Let $\mathcal{B}_n$ denote the  number of branches of a URT $\mathcal{T}_n$ of size $n$. Then
\begin{equation} \frac{\mathcal{B}_n - \ln(n)}{\sqrt{\ln(n)}} \xrightarrow{d}\mathcal{G}.\end{equation}
\end{thm}
We can moreover use the representation of $\mathcal{T}_n$ as a random permutation in order to calculate the number of branches. We can observe that $\pi(2)$ attaches to 1 and is thus the first node of a branch of $T_n$. If $\pi(3) > \pi(2)$ it will attach to $\pi(2)$, if $\pi(4) > \pi(2)$ it will attach to $\pi(2)$ or $\pi(3)$ and so on. As long as no $\pi(r) < \pi(2)$ appears, all nodes will be part of the branch starting with node $\pi(2)$. Let $ r= \min \{ r = 3, \dots, n : \pi(r) < \pi(2) \} $ then $\pi(r)$ also attaches to 1 and is thus the start of the next branch. Similarly, as long as no smaller number comes up, all subsequent nodes will be part of this second branch. In general, $\pi(r)$ attaches to 1 if and only if $\pi(r) = \min \{\pi(2), \dots, \pi(r)\}$ which means that $\pi(r)$ must be an anti-record. Thus, the number of branches of $\mathcal{T}_n$ is equal to the number of antirecords of $\pi$ and we can write the number of branches $\mathcal{B}_n$ of $\mathcal{T}_n$ as another sum of indicator random variables:
\begin{equation}\mathcal{B}_n = \sum_{r=2}^{n} \1({\pi(r) = \min \{\pi(1), \dots, \pi(r)\}}).
\end{equation}

As expected the number of records in a URP and thus also the number of anti-records satisfy the same central limit theorem as $\mathcal{B}_n$, see \cite{DevroyeRecords, Glick}. In \cite{DevroyeRecords} the equality in distribution between the number of branches and the number of records of a permutation was discovered without refering to the permutation representation of a URT. This way of analysing the number of branches in a URT has the advantage that records of sequences of random variables and thus random permutations are well studied, see for example \cite{Nevzorov01}. Thus all properties true for records of uniform permutations are bijectively also true for the branches of uniform recursive trees.

The third possibility to calculate the number of branches of a URT is by using the representation of URTs as permutations in cycle representation.
This interpretation allows us to equate the number of branches with the number of cycles in a random permutation which is well known, see for example \cite{Drmota09}. 

We can also consider the number of branches of a given size. We have that:

\begin{thm}[\cite{Feng05}] 
For $m \in \mathbb{N}$, let $\beta_{m,n}$ denote the number of branches of size 
 $m$ in $\mathcal{T}_n$. Then 
 \begin{enumerate}
\item for $m \in \mathbb{N}$, $\beta_{m,n} \xrightarrow[]{n\to \infty} \Po\left (\frac{1}{m} \right)$ and
 \item for $n > m > \frac{n-1}{2}$, $\beta_{m,n}$ can only take the values $0$ or $1$ and $\Pro(\beta_{m,n}=1) = \frac{1}{m}$.
 \end{enumerate}
\end{thm}

Moreover the covariance matrix for $(\beta_{m,n})_{1\leq m\leq n}$ is derived 
and it is shown that the number of branches of different sizes are asymptotically independent for any sequence of integers $(m_1, \dots, m_\ell)$, see \cite{Feng05}. 

As a final note, these results allow some conclusions concerning the size of the largest branch. 
\begin{thm}[\cite{Feng05}]
Let $\nu_n$ denote the size of the largest branch of $\mathcal{T}_n$. Then
\begin{enumerate}
\item $\Pro\left (\nu_n > \frac{n-1}{2} \right ) \xrightarrow{n \to \infty} \ln(2)$ and
\item $\nu_n \xrightarrow[a.s.]{n \to \infty} \infty$.
\end{enumerate}
\end{thm}

Also in connection to the size of specific components the number of nodes with $k$ descendants can be studied.
In \cite{Devroye91}, Devroye uses a connection between URTs and binary trees and local counters to analyse this statistic. 

\begin{thm}[\cite{Devroye91}] \label{thm:URTkDes}
Let $X_{k,n}$ denote the number of leaves with exactly $k$ descendants in a URT. Moreover we define $\alpha_k=\frac{1}{(k+2)(k+1)}$, $\gamma_k= \frac{1}{(2k+3)(2k+2)(k+1)}$ and \linebreak
$\sigma_k= \alpha_k(1-\alpha_k)-2(k+1)\alpha_k^2 +2\gamma_k$. Then
\begin{enumerate}
\item $\frac{X_{k,n}}{n} \xrightarrow{n \to \infty} \alpha_k$ in probability and
\item $\frac{X_{k,n} - n \alpha_k}{\sqrt{n}}  \xrightarrow{n \to \infty} \mathcal{N}(0,\sigma_k^2)$ in distribution.
\end{enumerate}
\end{thm}

The depth of node $n$ can also be studied by several different methods, of which we will introduce three. Let $\mathcal{D}_n$ denote the depth of node $n$ in a URT $\mathcal{T}_n$ of size $n$. It turns out that $\mathcal{D}_n$ and $\mathcal{B}_n$ have the same distribution. Again there are several methods to show that this is the case.

The first possibility is to describe the depth of node $n$ is to consider the path $P_n$ from 1 to $n$ in $\mathcal{T}_n$. Every node in $P_n$ is an ancestor of $n$ and the number of nodes in the path, except $n$, thus gives the depth of the node $n$: 
\begin{equation} \mathcal{D}_n = \sum_{i=2}^{n} \1(i \in P_n) = \sum_{i=2}^{n-1} \1(i \in P_n) + 1.
\end{equation}
It can then be shown that $\mathbb{P}(i \in P_n) = \frac{1}{i}$ and that the indicator random variables $\1(i \in P_n)$ are mutually independent, which allows to conclude that $\mathcal{D}_n$ has the same distribution as $\mathcal{B}_n$ \cite{Feng05}. 

On the other hand, the depth of node $n$ can also be seen as follows: By construction, the parent of node $n$ is a node $a_1$ uniformly distributed among $\{1,\dots, n-1\}$. Its grandparent is then a node uniformly distributed among $\{1,\dots,a_1-1\}$ and so on until we reach $a_m=1$. The depth of node $n$ is then $m$. Similarly, the position of the last record in a random permutation, $b_1$, is uniformly distributed on $\{1,\dots, n-1\}$. The position of the 2nd to last record, $b_2$, is then uniformly distributed on $\{1,\dots, b_1-1\}$ and so on, until $b_m=1$. $m$ then gives the number of records in the permutation. Therefore, we get

\begin{thm}[\cite{DevroyeRecords}]
$\mathcal{D}_n$ is distributed as the number of records in a uniform random permutation of $\{2,\dots,n\}$.
\end{thm}
We already saw that the number of branches is distributed as the number of records in a uniform random permutation, implying that the depth and the number of branches are equally distributed as well.

There is a third method to determine the depth of node $n$ by recursions, which was historically the first one used, see \cite{Survey}. It is based on a recursion for the expected number of nodes in the $k$-th generation, which allows to derive the depth of node $n$ by the following relation. Let $\mu_{n}^k$ denote the number of nodes in the $k$-th generation, where the generation of the root is 0. Then $\Pro(\mathcal{D}_n=k) = \frac{\E \left [\mu_{n-1}^{k-1}\right ]}{n-1}$. The expectation and variance can then directly be computed from the distribution.

Finally there is another, easier way of determining the depth of node $n$ by the use of recursion without first determining the exact distribution  described in \cite{Survey}. It is a corollary of a result about internodal distances, which we will now give.

In \cite{Moon74} the expectation and variance of the distance between two nodes in a random recursive tree are derived via a simple recursion. Let us denote the distance between $i$ and $j$ by $\mathcal{D}_{i,j}$. If $i<j$, the distance between $i$ and $j$ is longer by 1 than the distance between the node $j$ attaches to and $i$. In other words if $j$ attaches to $k$ where $1 \leq k \leq j-1$, then we have $\mathcal{D}_{i,j} = \mathcal{D}_{i,k} +1$. This gives the recursion
\begin{equation}
\begin{split}
\Pro(\mathcal{D}_{i,j} = d) = \frac{1}{j-1} [\Pro(\mathcal{D}_{1,i} = d-1)& + \Pro(\mathcal{D}_{2,i} = d-1)  \\
& + \cdots + \Pro(\mathcal{D}_{i,j-1}=d-1) ].
\end{split}
\end{equation}

\newpage This recursion can be used to prove the following theorem.
\begin{thm}[\cite{Moon74}] Let $\mathcal{D}_{i,j}$ denote the distance between $i$ and $j$ where $1 \leq i < j \leq n$ in a random recursive tree. Then
\begin{equation}
\begin{split}
\E[\mathcal{D}_{i,j}] &= H_i + H_{j-1} -2 + \frac{1}{i} \text{ and } \\
\Var(\mathcal{D}_{i,j}) &= H_i + H_{j-1} - 3H_i^{(2)} - H_{j-1}^{(2)} + 4  - 4\frac{H_i}{i} + \frac{3}{i} - \frac{1}{i^2}.
\end{split}
\end{equation}
\end{thm} 
In \cite{Dobrow96} the exact distribution for $\mathcal{D}_{i,n}$ is given and asymptotic results for $\mathcal{D}_{i_n,n}$ for some special cases of $(i_n)_{n \in \mathbb{N}}$ are given.
The asymptotic distribution of the distance between a fixed node and $n$ is shown to be normal in \cite{Feng06}.
\begin{thm}[\cite{Feng06}]
Let $\mathcal{D}_{i,n}$ denote the distance between nodes $i$ and $n$ in a URT $\mathcal{T}_n$.
Then
\begin{equation} \frac{\mathcal{D}_{i,n}- \ln(n)}{\sqrt{ln(n)}} \xrightarrow[d]{n \to \infty} \mathcal{G} .\end{equation}
\end{thm}
Moreover, asymptotic normality of $\mathcal{D}_{i_n,n}$ for any sequence $(i_n)_{n \in \mathbb{N}}$ is shown.
\begin{thm}[\cite{Feng06}]
Let $\mathcal{D}_{i_n,n}$ denote the distance between nodes $i_n$ and $n$ in a URT $\mathcal{T}_n$, where $i_n \geq n-1$ for all $n$. Then
\begin{equation} \frac{\mathcal{D}_{i_n,n}- \ln(n)- \ln(i_n)}{\sqrt{\ln(n)+ \ln(i_n)}} \xrightarrow[d]{n \to \infty} \mathcal{G} .\end{equation}
\end{thm}

As already mentioned in the introduction there are two main global properties of URTs that were investigated: the maximum degree and the height. First let us consider the maximum degree.
Let the degree of node $i$ in $\mathcal{T}_n$ be denoted by $\delta_{n,i}$ and $\Delta_n$ be the maximum degree of any vertex in $\mathcal{T}_n$. Since $\E[\delta_{n,1}] > \ln(n)$, and $\Delta_n > \delta_{n,1}$, we get $\E[\Delta_n] > \ln(n)$. Moreover we have the following result.
\newpage
\begin{thm}[\cite{DevroyeLu95}]
Let $\mathcal{T}_n$ be a URT and let $\Delta_n$ denote the maximum degree among the vertices of $\mathcal{T}_n$. Then, 
\begin{equation}\frac{\Delta_n}{\ln(n)} \xrightarrow[a.s.]{n \to \infty} 1 \text{ \phantom{und} and \phantom{und} } \frac{\E[\Delta_n]}{\ln(n)} \xrightarrow{n \to \infty} 1.
\end{equation}
\end{thm}
For results on the distribution of the maximum degree, see \cite{Goh01, Addario15}.

Concerning the height of a URT the following is known.
\begin{thm}[\cite{Pittel94}] Let $\mathcal{T}_n$ be a URT and let $\mathcal{H}_n$ denote the height of $\mathcal{T}_n$ i.e. the longest path from the root to any other node. Then  with probability 1
\begin{equation}\frac{\mathcal{H}_n}{\ln(n)} \xrightarrow {n \to \infty} e.\end{equation}

\end{thm}
In \cite{Pittel94} this result is proved by using the connection between a continuous time branching process and URTs. In \cite{Devroye11} results about the height of a more general family of random recursive trees are obtained, implying in particular the above result, by using a second moment method.

Moreover we have a more precise estimation of the expectation of the height of a uniform recursive tree, again derived by the use of a branching process.
\begin{thm}[\cite{AddarioFord13}]
Let $\mathcal{T}_n$ be a URT and let $\mathcal{H}_n$ denote the height of $\mathcal{T}_n$. Then 
\begin{equation}
\E[\mathcal{H}_n] = e \ln(n) - \frac{3}{2} \ln \ln (n) + \mathcal{O}(1).\end{equation}
\end{thm}

\section{Non-uniform Recursive Trees Considered in the Literature}

There has been a tremendous effort to study uniform recursive trees as the results above imply. There also are a few papers considering different kinds of non-uniform recursive trees. We now give some examples of such distributions.

\subsection{Non-uniform Recursive Trees via External Nodes}

 In \cite{Dobrow} recursive trees are described by the use of external nodes. By variations on the number of external nodes each node has, one can thus change the attachment probabilities and the distribution on the recursive trees. In general, recursive trees can be constructed by attaching node $n$ according to some rule to a node of a recursive tree of size $n-1$. In order to specify that rule one can add external nodes at every place the new node can be attached to and then choose one of these according to some distribution. In the uniform model there is one external node at every node of $\mathcal{T}_{n-1}$ and one of them is chosen uniformly. This node then becomes the internal node $n$ and a new external node is created at $n$ and at the parent of $n$.

This model can be generalized by adding at each step $n$, after attaching node $n$, $\alpha$ external nodes at the parent of $n$ and $\beta$ external nodes at $n$. Among all the external nodes the place for $n+1$ is then again chosen uniformly. Using this principle the following distributions can for example be obtained.
For plane-oriented trees the number of external nodes of a node $i$ is equal to the out-degree of $i$ +1. This model is called plane-oriented because it corresponds to the distribution obtained by considering the children of each node as ordered, then choosing one of these trees uniformly and finally identifying all trees that only differ because of the ordering of the children in order to get a distribution on recursive trees \cite{Survey}. 
This model was introduced in \cite{Szymanski87} as one of the first non-uniform models. In that paper moreover properties related to the degrees of nodes where obtained and compared to the uniform case.
For $m \in \mathbb{N}$, $m$-oriented trees are a generalization of this model where every node has $m-1$ times its outdegree +1 external nodes. 

Binary trees are a tree model where every node can have exactly $2$ children. For $m \in \mathbb{N}$, $m$-ary trees are a generalization of binary trees where every node can have exactly $m$ children. These trees can also be defined via the use of external nodes. Table  \ref{tbl:treemodels} shows the $\alpha$ and $\beta$ values for these tree models \cite{Dobrow}.

\begin{table}[thbp]
\vskip\baselineskip 
\caption[Distributions on recursive trees for some values of $\alpha$ and $\beta$.]{Distributions on recursive trees for some values of $\alpha$ and $\beta$.} \label{tbl:treemodels}
\begin{center}
\begin{tabular}{| c | c| c |}
\hline
\textbf{Recursive tree model} & $\alpha$ & $\beta$ \\ \hline
Uniform recursive & 1 & 1 \\ \hline
Plane-oriented &2&1 \\ \hline
m-oriented &m&1 \\ \hline
Binary &0&2 \\ \hline
m-ary & 0 &m \\ \hline
\end{tabular} 
\end{center}
\end{table}
\newpage
Figure \ref{fig:3-ary} demonstrates the construction of a $2$-oriented tree with the use of external nodes. The nodes without label are external nodes. At each step one of the external nodes is chosen uniformly chosen as the place of the newly added node.

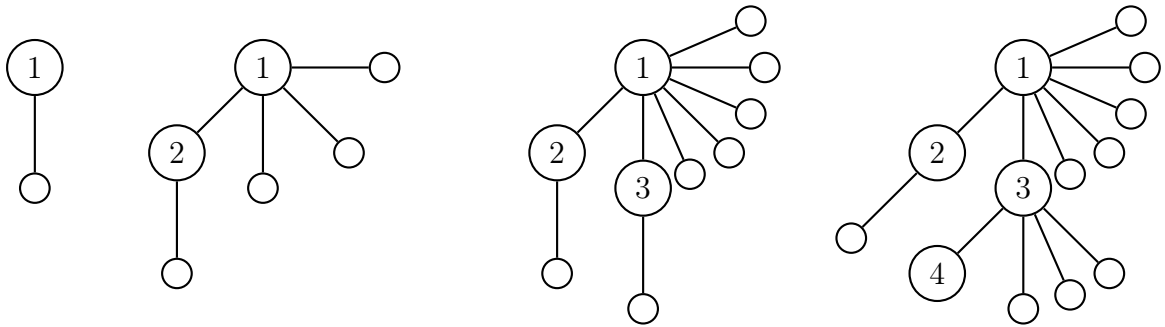
\begin{figure}[H]
\tikzstyle{every node}=[circle, draw]
\begin{tikzpicture}[node distance=1.6cm,
  thick]
  \node (1) {1};
  \node (2)  [below of=1] {};

\path[every node/.style={font=\sffamily\small}]
    (1) edge node {} (2);

\begin{scope}[xshift=3cm]
    
  \node (1) {1};
  \node (2)  [below left of=1] {2};
  \node (12) [below of = 1] {};
  \node (13) [below right of = 1] {};
  \node (14) [right of = 1] {};
  \node (21) [below of =2] {};

\path[every node/.style={font=\sffamily\small}]
    (1) edge node {} (2)
    (1) edge node {} (12)
    (1) edge node {} (13)
    (1) edge node {} (14)
    (2) edge node {} (21);
     \end{scope}

\begin{scope}[xshift=8cm]     
     
  \node (1) {1};
  \node (2)  [below left of=1] {2};
  \node (3) [below of = 1] {3};
  \node (13) [below right of = 1] {};
  \node (14) [right of = 1] {};
  \node (15) [below right=1cm and 0.2cm of 1] {};
  \node (16) [below right=0.2cm and 1cm of 1] {};
  \node (17) [above right=0.2cm and 1cm of 1] {};
  \node (21) [below of =2] {};
  \node (31) [below of =3] {};

\path[every node/.style={font=\sffamily\small}]
    (1) edge node {} (2)
    (1) edge node {} (3)
    (1) edge node {} (13)
    (1) edge node {} (14)
    (1) edge node {} (15)
    (1) edge node {} (16)
    (1) edge node {} (17)
    (2) edge node {} (21)
    (3) edge node {} (31);
\end{scope}

\begin{scope}[xshift = 13cm]

  \node (1) {1};
  \node (2)  [below left of=1] {2};
  \node (3) [below of = 1] {3};
  \node (13) [below right of = 1] {};
  \node (14) [right of = 1] {};
  \node (15) [below right=1cm and 0.2cm of 1] {};
  \node (16) [below right=0.2cm and 1cm of 1] {};
  \node (17) [above right=0.2cm and 1cm of 1] {};
  \node (21) [below left of =2] {};
  \node (4) [below left of = 3] {4};
  \node (32) [below of = 3] {};
  \node (33) [below right=1cm and 0.2cm of 3] {};
  \node (34) [below right of = 3] {};

\path[every node/.style={font=\sffamily\small}]
    (1) edge node {} (2)
    (1) edge node {} (3)
    (1) edge node {} (13)
    (1) edge node {} (14)
    (1) edge node {} (15)
    (1) edge node {} (16)
    (1) edge node {} (17)
    (2) edge node {} (21)
    (3) edge node {} (4)
    (3) edge node {} (32)
    (3) edge node {} (33)
    (3) edge node {} (34);
    \end{scope}
    
\end{tikzpicture}
\caption{Example of the construction of a $2$-oriented tree with external nodes.} \label{fig:3-ary}
\end{figure}

Dobrow and Smythe then give a recursive formula for the distribution of $\mathcal{D}_{i,n}$, the distance between node $i$ and node $n$, and the exact distribution of $\mathcal{D}_{1,n}$, the depth of node $n$, for general $\alpha$ and $\beta$. By using Poisson approximations they then prove asymptotic normality of $\mathcal{D}_{i,n}$ for all the models introduced above.  

\subsection{Introducing Choice}
Instead of altering the distribution on the set of recursive trees by introducing external nodes, it is possible to introduce choice according to some criterion. The idea is to choose at each step several potential parents one of which is then chosen according to some criterion. In \cite{DSouza07}, $k$ possible parents are selected for each node. The node is then attached to the one among these with the smallest or biggest distance to the root or the lowest or highest degree. These are called the \emph{smallest-depth}, \emph{highest-depth}, \emph{lowest-degree} and \emph{highest-degree model} respectively. The relevant properties of these models are then shown to qualitatively differ from the uniform one.

While for both $m$-oriented trees and the highest degree model the choice of the parent depends on the degree, the two models differ considerably. In the construction of plane-oriented trees knowledge about the degrees of all nodes is necessary. On the other hand in the highest-degree model only knowledge about the degree of $k$ nodes is necessary at each step. The authors of \cite{DSouza07} refer to this as global vs. local knowledge.

In \cite{Mahmoud10} a different choice criterion is introduced because the ones investigated in \cite{DSouza07} are intractable in many cases. They propose to either choose the node with the minimal or the maximal label among the $k$ potential parents and call this the \emph{label model}. In addition to giving the exact distribution of the depths in these trees, they also show that these label models can be used as bounds for the depth models introduced in \cite{DSouza07}.

\subsection{Scaled Attachment Random Recursive Trees} 
Another generalization of uniform random recursive trees, that goes into a very different direction are \emph{scaled attachment random recursive trees}, abbreviated  SARRT, which were introduced in \cite{Devroye11}. To construct an SARRT with $n+1$ nodes, we need $n$ identically distributed independent random variables $X_1, \dots, X_n$ on $[0,1)$. The root gets label 0. The parent of $i$  is then chosen as $\lfloor iX_i \rfloor$. Like the uniform recursive tree, the attachment probabilities do not depend on structural properties of the tree constructed before node $i$ is attached. This was not the case for the recursive tree models introduced in \cite{Dobrow}, since there the attachment probability depends on the number of children of the already present nodes. The models introduced in \cite{DSouza07} also relies on knowledge of the structure of $\mathcal{T}_{n-1}$ in order to choose the parent of $n$ since we need to know the depth or degree of the potential parents. SARRTs moreover include URTs: when we choose the uniform distribution for the $X_i$, the random variable $\lfloor iX_i \rfloor$ is uniformly distributed over $[i]$. 

The explicit example given in \cite{Devroye11} is the SARRT where $X_i = \max\{U_1, U_2, \dots, U_k\}$ or $X_i = \min\{U_1, U_2, \dots, U_k\}$ with $k \in \mathbb{N}$ and $U_\ell$, $\ell=1, \dots, k$, uniformly distributed over $[0,1)$. This is equivalent to choosing uniformly $k$ potential parents for each node $i$ and among these attaching $i$ to the youngest respectively oldest parent, where labels with higher label are considered to be younger. In this way SARRTs recover the model introduced in \cite{Mahmoud10}. In \cite{Devroye11} the authors then go on to prove asymptotic results on the height of the tree, the depth of node $n$ and the minimal depth of the nodes $\lceil\frac{n}{2}\rceil$ to $n$, for general $X_i$ and depending on the common distribution of the $X_i$.

Apart from different distributions on recursive trees, recursive structures other than trees are also studied. One possibility is to consider random recursive forests, i.e. allowing disconnected components. A way in the other direction is to allow every new node to join several previously present nodes.

\subsection{Random Recursive Forests}
Another generalization of random recursive trees are random recursive forests, which were introduced in \cite{forests}. A \emph{random recursive forest} is constructed in a similar way as a random recursive tree. The difference is that each node $n$ can either attach to any of the nodes $1$ to $n-1$ or become the root of a new tree. In the general model, at each step $n$, a number  $y_n \in \{0, \dots, n-1\}$ is chosen according to probability $\vec{p} = (p_n(y))_{y=0, \dots, n-1}$. If $y_n=0$, node $n$ is the root of a new tree. If $y_n \in \{ 1, \dots, n-1\}$, $y_n$ is the parent of $n$. See Figure \ref{fig:RRforest} for an example of a random recursive forest.

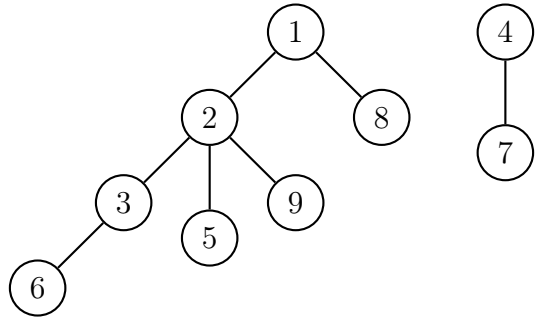
\begin{figure}
\centering
\tikzstyle{every node}=[circle, draw]
\begin{tikzpicture}[node distance=1.6cm,
  thick]
  
  \node (1) {1};
  \node (2)  [below left of=1] {2};
  \node (3) [below left of = 2] {3};
  \node (4) [right =2 cm of 1] {4};
  \node (5) [below  of = 2] {5};
  \node (6) [below left of = 3] {6};
  \node (7) [below of = 4] {7};
  \node (8) [below right of =1] {8};
  \node (9) [below right of =2] {9};
  
  \path[every node/.style={font=\sffamily\small}]
  
  (1) edge node {} (2)
  (2) edge node {} (3)
  (2) edge node {} (5)
  (3) edge node {} (6)
  (1) edge node {} (8)
  (4) edge node {} (7)
  (2) edge node {} (9);
  
\end{tikzpicture}
\caption{Example of a random recursive forest.} \label{fig:RRforest}
\end{figure}

For the general model the expectations and variances of the outdegrees of all nodes $i$ and as a corollary the expected number of components are given. Moreover the expected number of leaves and of nodes of out-degree $1$ are given. The rest of the results are given for uniform recursive forests, so forests where $p_n(y)= \frac{1}{n}$ for all $y=0, 1, \dots, n-1$. For the uniform model asymptotic normality of the number of components is shown and results about the expected number of nodes of out-degree $k$ and the maximum out-degree are derived. Moreover the distribution of the label of the root of the last component and of the label of the root of the component containing node $i$ are given.

A uniform recursive forest of size $n$ can easily be constructed from a uniform recursive of size $n+1$ by deleting the root. The nodes that are the start of the branches in the URT then become the roots of the components in the uniform recursive forests. Most results mentioned above thus follow more or less directly from results on URT.

\subsection{Uniform Random Recursive Directed Acyclic Graphs}
As already mentioned, we can also allow nodes to attach to several nodes. In \cite{DevroyeLu95} a model was introduced where initially $m$ roots are present and every additional node attaches uniformly to $r$ already present nodes, with $r$ and $m$ in $\mathbb{N}$. This yields a \emph{uniform random recursive directed acyclic graph}, short URRD, which can be used to model circuits \cite{Devroye11}. In \cite{DevroyeLu95} asymptotic results about the maximum degree of such structures are derived. Similarly as for recursive trees, various distributions different from the uniform one can be considered on recursive directed acyclic graphs. 
When letting the attachment probability depend on the degree of the present nodes, this in particular yields the preferential attachment model \cite{Devroye11}. 
Also see \cite{DevroyeJansen11} for results on paths in uniform random recursive acyclic graphs.

\subsection{Hoppe Trees} 
Another non-uniform distribution on random recursive trees considered in the literature are Hoppe trees. These trees were introduced by Leckey and Neininger on the basis of Hoppe's urn  \cite{Hoppe}. In Hoppe's urn one starts with a black ball of weight $\theta$. All other balls will have weight 1. We will call all balls that are not black coloured. At each step one of the balls is drawn from the urn with a probability proportional to its weight. The drawn ball is always put back. Moreover, when the black ball is drawn, a ball with a new colour is added to the urn. When a coloured ball is drawn, a new ball of the same colour is added to the urn.

On the basis of this process a random recursive tree called \emph{Hoppe tree} can be constructed: The black ball represents the root and gets label 1. The other balls get their label from the time they were added to the urn. If we consider the black ball as the first ball added to the urn, the $n$-th ball added to the urn thus gets label $n$. The label of the ball we draw at time $n$ is the label of the parent of node $n$ in the recursive tree. Each colour in the urn then corresponds to one branch of the tree.

At each time $n$ the probability to draw the black ball is $\frac{\theta}{\theta+ n-2}$ and the probability for any coloured ball to be drawn is $\frac{1}{\theta+n-2}$. We thus get a random recursive tree with the following construction principle: let $c_{i,n}$ denote the event that $i$ is the parent of $n$, then
\begin{equation}
\Pro(c_{i,n} = 1) = 
\begin{cases} \frac{\theta}{\theta+n-2} & \text{ for } i=1 \\
\frac{1}{\theta + n-2} & \text{ for } i =2, \dots, n-1.
\end{cases}
\end{equation}

The Hoppe distribution on recursive trees is equivalent to the $\theta$-biased distribution on permutations as described in\cite{Arratia03}.
 We will call these permutations \emph{Hoppe permutations}. \label{HoppePerm} $\theta$-biased permutations can be constructed by using a variation of the Chinese restaurant process  described in Section \ref{ChineseRestaurant}. At each step, $i$ either starts a new cycle with probability $\frac{\theta}{\theta+i-2}$ or joins any other cycle after any other integer with probability $\frac{1}{\theta+i-2}$. The number of cycles of $\theta$-biased permutations and the distribution of the cycle sizes are known and thus imply the corresponding results for branches of Hoppe trees, see \cite{Arratia03, HoppeMaster}. In particular we have the following theorem.

\begin{thm}[\cite{Arratia03, HoppeMaster}] \label{thm:HoppeBranches}
Let $\mathcal{B}_n^{\theta}$ denote the number of branches of a Hoppe tree of size $n$. Then 
\begin{equation}
\E\left [\mathcal{B}_n^{\theta}\right] =\theta \sum_{i=1}^{n} \frac{1}{\theta+i-1}
\end{equation}
and \begin{equation}
\Var\left (\mathcal{B}_n^{\theta}\right) = \theta \sum_{i=1}^{n}\frac{n-1}{(\theta + j-1)^2}.
\end{equation}
Moreover asymptically
\begin{equation}
\frac{\mathcal{B}_n^{\theta} - \theta \ln(n)}{\sqrt{\theta \ln(n)}} \xrightarrow[d]{n \to \infty} \mathcal{G}.
\end{equation}
\end{thm}

In \cite{Hoppe}, results about the depth of node $n$ and the height, the number of leaves and the internal path length of a Hoppe tree of $n$ nodes are given. We will now give the theorems from \cite{Hoppe} that concern statistics we will also investigate.
\begin{thm}[\cite{Hoppe}] \label{thm:HoppeDepth}
Let $\mathcal{D}_n^{\theta}$ denote the depth of the $n$-th node of a Hoppe tree and let, for $i=1, \dots, n-2$, $B_i$ be independent Bernoulli random variables with \linebreak $\Pro(B_i =1) = \frac{1}{\theta+1}$.  Then
\begin{equation} \mathcal{D}_n^{\theta} =_d 1 + \sum_{i=1}^{n-2} B_{i}.
\end{equation}
This easily gives
\begin{equation}
\begin{split}
& \E\left [\mathcal{D}_n^{\theta}\right] = \ln(n) + \mathcal{O}(1)  \\
& \Var\left (\mathcal{D}_n^{\theta}\right) = \ln(n) + \mathcal{O}(1) \\
& \frac{\mathcal{D}_n^{\theta} - \E[\mathcal{D}_n^{\theta}]}{\sqrt{\Var(\mathcal{D}_n^{\theta})}} \xrightarrow[d]{n \to \infty} \mathcal{G} \text{ and  } \\
& d_{TV} (\mathcal{D}_n^{\theta}, \operatorname{Poisson}(\E\left [\mathcal{D}_n^{\theta}\right]) ) = \mathcal{O}\left (\frac{1}{\ln(n)} \right ).
\end{split}
\end{equation}
\end{thm}

\begin{thm}[\cite{Hoppe}] \label{thm:HoppeHeight}
Let $\mathcal{H}_n^{\theta}$ denote the height of a Hoppe tree with $n$ nodes. Then
\begin{equation}
\begin{split}
&\E\left [\mathcal{H}_n^{\theta}\right] = e\ln(n) - \frac{3}{2}\ln \ln n + \mathcal{O}(1) \text{ and} \\
&\Var\left(\mathcal{H}_n^{\theta}\right) = \mathcal{O}(1).
\end{split}
\end{equation}
\end{thm}
\begin{thm}[\cite{Hoppe}] \label{thm:HoppeLeaves}
Let $\mathcal{L}_n^{\theta}$ denote the number of leaves of a Hoppe tree with $n \geq 2 $ nodes. Then
\begin{equation}
\begin{split}
&\E\left [\mathcal{L}_n^{\theta}\right] = \frac{n}{2} + \frac{\theta -1}{2} + \mathcal{O}\left (\frac{1}{n} \right) \\
&\Var\left (\mathcal{L}_n^{\theta}\right) = \frac{n}{12} + \frac{\theta -1}{12} + \mathcal{O}\left (\frac{1}{n} \right ) \\
& \Pro(|\mathcal{L}_n^{\theta} - \E\left [\mathcal{L}_n^{\theta}\right]| \geq t) \leq 2 e^{-\frac{6t^2}{n+\theta+1}} \text{ for all } t>0 \text{ and } \\
&\frac{\mathcal{L}_n^{\theta} - \E\left[\mathcal{L}_n^{\theta}\right]}{\sqrt{\Var\left (\mathcal{L}_n^{\theta}\right)}} \xrightarrow[d]{n \to \infty} \mathcal{G}.
\end{split}
\end{equation}
\end{thm}

The first distribution we will consider in this thesis is a generalization of Hoppe trees so this brings us to the next chapter.

\chapter{WEIGHTED RECURSIVE TREES}   \label{chapter:weight}

\section{Definition}

We now consider a recursive tree model obtained by varying the weights of each node. In a URT the weight of each node can be considered to be 1. We now give each node $i=1, \dots, n$ a weight $\omega_i \in \mathbb{R}$, thus every weighted recursive tree model can be characterized by its sequence of weights, denoted $(\omega_i)_{i \in \mathbb{N}}$. These weights affect the attachment probabilities: let $c_{i,j}$ denote the event that $j$ is attached to $i$ in the $j$-th construction step. Then for $i<j$ we have 
\begin{equation}
\Pro(c_{i,j}=1) = \frac{\omega_i}{\sum_{k=1}^{j-1} \omega_k}.
\end{equation}
 We will call these trees weighted recursive tree, or WRT. This is a generalization of Hoppe trees, which correspond to the case $\omega_1= \theta>0$ and $\omega_i =1$ for $i = 2, \dots, n$.

WRTs share with URTs the property that the construction steps are mutually independent: whether $j$ attaches to $i$ is not dependent on the structure of the tree at time $j$, but only on the weight sequence. We already mentioned this property above and saw that this is not true for several distributions considered in the literature. For example for the trees considered by Dobrow and Smythe in \cite{Dobrow} where $\alpha$, i.e. the number of external nodes created at the parent, was different from 1, the number of external nodes of node $i$, and thus the attachment probabilities, were dependent on the outdegree of $i$ at time $j$. The second model we will consider, biased recursive trees, does not have this independence property either. 

Because the construction steps in weighted recursive trees are mutually independent, the weighted recursive tree model is a kind of a more general model, which we call  \emph{inhomogeneous trees}, where $j$ attaches with probability $p_{j,i}$ to a node $i$, such that $ \sum_{i=1}^{j-1} p_{j,i}=1$. In some sense, in inhomogeneous trees, the weight of each node can change at each step, while in weighted recursive trees the weight is fixed.

Because of technicalities we will not be able to tackle all problems for a general sequence of weights. Where it will be necessary we will make some assumptions on the weight sequence $(\omega_i)_{i \in \mathbb{N}}$, which will be specified at the beginning of each section or subsection. For the number of leaves, for example, we will assume that for some $k\in \mathbb{N}$, $\omega_{i} = 1$ for $i >k$ and $\omega_i = \theta$ for $1 \leq i \leq k$. 

In the rest of this section we will proceed as follows: In Section \ref{coupling1} we will discuss a method of constructing WRTs from URTs. It can be applied to any weighted recursive tree, but is much easier in the case where only the first $k$ nodes get a weight  $\theta$ that is different from 1.
We then go on by giving results about the number of branches and the depth of node $n$  of general weighted recursive trees. In both cases we will be able to write the random variable in question as a sum of independent Bernoulli random variables. As we will see the depth of node $n$ and the number of branches are not equidistributed, as it is the case for URTs.
We then derive results on the number of leaves of a special case of weighted recursive trees, where only the first $k$ nodes get a weight different from 1. This is done by using a martingale approach. Finally we introduce a second coupling that can only be applied on weighted recursive trees where for some $k\in \mathbb{N}$, $\omega_{i} = 1$ for $i >k$. The counterpart of this restriction is that it is much simpler and deterministic and thus allows for easy conclusions concerning the asymptotic behaviour of the number of leaves of weighted recursive tree of this form.

\section{Constructing Weighted Recursive Trees from Uniform Recursive Trees} \label{coupling1}

We begin our discussion of weighted recursive trees by presenting a coupling construction that can be used to generate a WRT once we are given a URT. Given a URT we will relocate some nodes with a certain probability. By relocation we mean that a node $j$ is detached from its parent and attached to another node with all its descendants. In Figure \ref{relocation} such a relocation is illustrated.

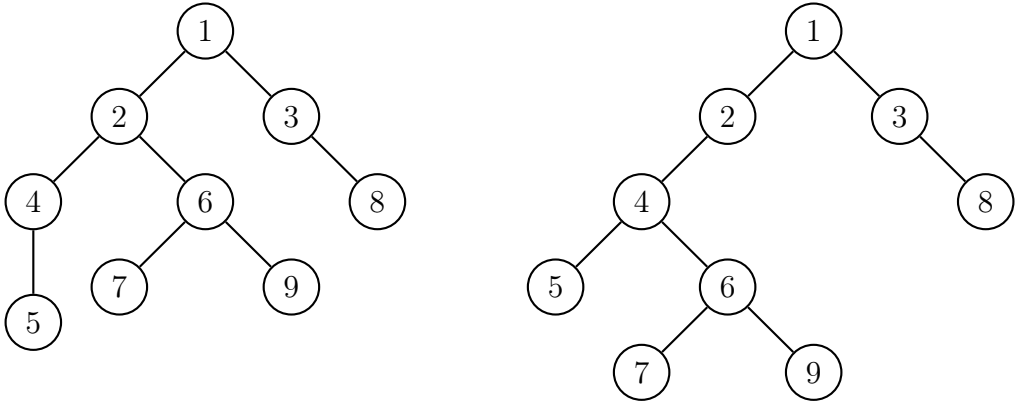
\begin{figure}
\tikzstyle{every node}=[circle, draw]

\begin{center}
\begin{tikzpicture}[node distance=1.6cm,
  thick]

  \node (1) {1};
  \node (2) [below left of=1] {2};
  \node (3) [below right of=1] {3};
  \node (4) [below left of=2] {4};
  \node (5) [below  of = 4] {5};
  \node (6) [below right of=2] {6};
  \node (7) [below left of=6] {7};
  \node (8) [below right of=3] {8};
  \node (9) [below right of = 6] {9};
 
\path[every node/.style={font=\sffamily\small}]
    (1) edge node {} (2)
    (1) edge node {} (3)
    (2) edge node {} (6)
    (2) edge node {} (4)
    (3) edge node {} (8)
    (4) edge node {} (5)
    (6) edge node {} (7)
    (6) edge node {} (9);
    
  \begin{scope}[xshift=8cm]
    \node (1) {1};
  \node (2) [below left of=1] {2};
  \node (3) [below right of=1] {3};
  \node (4) [below left of=2] {4};
  \node (5) [below left of = 4] {5};
  \node (6) [below right of=4] {6};
  \node (7) [below left of=6] {7};
  \node (8) [below right of=3] {8};
  \node (9) [below right of = 6] {9};
 
\path[every node/.style={font=\sffamily\small}]
    (1) edge node {} (2)
    (1) edge node {} (3)
    (4) edge node {} (6)
    (2) edge node {} (4)
    (3) edge node {} (8)
    (4) edge node {} (5)
    (6) edge node {} (7)
    (6) edge node {} (9);
  \end{scope}
\end{tikzpicture}
\end{center}
\caption{Node $6$ is relocated from node 2 to node 4.} \label{relocation}
\end{figure}
The coupling construction can be applied step by step, as the new edges appear, or at the end, after the whole URT is constructed. In other words, we can either at each step attach a new node uniformly to any of the present nodes and then relocate it according to the rule we will now describe, or we can construct a URT with $n$ nodes and then apply the relocation rule for every node. This amounts to the same result since the structure of the tree at any time does not affect which node is chosen as the parent of the next node. 

Let us further remark that in the latter case we can relocate the nodes in any order we like, since we always relocate the node with all its descendants. As we will see, the relocation process only depends on the weight of the old and the new parent, and is thus independent of any structural properties of the tree. Thus this coupling has the property that all adjustments we make, i.e. all relocations, are mutually independent. So we do not lose the important property of WRTs and URTs that each construction step is independent while rearranging one to get the other.
\subsection{General Case}
First construct a URT. We will rearrange it by relocating nodes in order to get a WRT with weight sequence $(\omega_i)_{i\in \mathbb{N}}$. For any two  nodes $3 \leq i <j$, node $j$ can only be either relocated to or from node $i$, not both. This depends on whether its weight is bigger or smaller than the average weight at time $j$, i.e. among the first $j-1$ nodes. More precisely, let us denote the average weight at time $j$ by $\overline{\omega_j}:= \frac{\sum_{k=1}^{j-1} \omega_k}{j-1}$. If $\omega_i < \overline{\omega_j}$ node $j$ can only be relocated \emph{from} $i$ and if $\omega_i > \overline{\omega_j}$, node $j$ can only be relocated \emph{to} node $i$. This restriction is suggested by the fact that $\omega_i > \overline{\omega_j}$ is equivalent to $\frac{1}{j-1} < \frac{\omega_i}{\omega_1 + \dots + \omega_{j-1}}$. We will show that under this condition we can construct a coupling producing a WRT with weight sequence $(\omega_i)_{i\in \mathbb{N}}$.

For $j \geq 3$, if in the URT $j$ is attached to a node $i$ with $\omega_i < \overline{\omega_j}$, we relocate it with probability $r_{i,j}$ to another node such that 
\begin{equation}
\begin{split}
&\Pro(j \text{ at } i \text{ in WRT}) = \Pro(j \text{ at } i \text{ in reconstructed tree})\\
&\Leftrightarrow \Pro(j \text{ at } i \text{ in WRT}) =  \Pro(j \text{ at } i \text{ in URT})(1-\Pro(j \text{ is relocated from } i)) \\
&\Leftrightarrow \Pro(j \text{ at } i \text{ in WRT})) =  \Pro(j \text{ at } i \text{ in URT})(1-r_{i,j}) \\
&\Leftrightarrow  \frac{\omega_i}{\sum_{k=1}^{j-1} \omega_k} = (1-r_{i,j}) \frac{1}{j-1} \\
&\Leftrightarrow r_{i,j} = 1 - \frac{(j-1)\omega_i}{\sum_{k=1}^{j-1} \omega_k} \\
&\Leftrightarrow r_{i,j} = \frac{\sum_{k=1}^{j-1} \omega_k - (j-1)\omega_i}{\sum_{k=1}^{j-1} \omega_k}.
\end{split}
\end{equation}

We only get a positive value for $r_{i,j}$ if $\omega_i < \frac{\sum_{k=1}^{j-1} \omega_k}{j-1}$, which is consist with our assumption that we only relocate from nodes whose weight is strictly less than the average at time $j$. It is important to remember that the average of the weights can change at each step. 

Now we want to relocate $j$ to a node that has a weight larger than the average in such a way that we get the right attachment probabilities. Let $p_{i,j,h}=\Pro(j \text{ relocated to } i |j \text{ relocated from } h)$. We will see that we get the right probabilities if we choose 
\begin{equation}p_{i,j,h} = \frac{(j-1)\omega_i-\sum_{k=1}^{j-1}{\omega_k}}{ \sum_{\substack{\ell=1 \\\omega_{\ell} > \overline{\omega_j} }}^{j-1} \left [(j-1)\omega_{\ell} - \sum_{k=1}^{j-1} \omega_k\right ]}.
\end{equation}
\newpage
 That $p_{i,j,h}$ does not depend on $h$ implies that given a node $j$ is relocated it does not matter which node it is relocated from. We thus get
\begin{equation}
\begin{split}
&\Pro(j \text{ at } i \text{ in reconstructed tree}) \\
&=  \Pro(j \text{ at } i \text{ in URT})+ \Pro(j\text{ relocated to } i) \\
&=  \Pro(j \text{ at } i \text{ in URT})+\sum_{\substack{h=1 \\ \omega_h < \overline{\omega_j}}}^{j-1}\Pro(j\text{ at } h \text{ in URT}) \Pro(j \text{ relocated }| j \text{ at } h \text{ in URT}) \\
&\hspace{16em} \cdot \Pro(j \text{ relocated to } i |j \text{ relocated from } h)\\
&= \frac{1}{j-1} + \sum_{\substack{h=1 \\ \omega_h < \overline{\omega_j}}}^{j-1} \frac{1}{j-1} r_{h,j} p_{i,j,h}
 \\
&= \frac{1}{j-1} + \sum_{\substack{h=1 \\ \omega_h < \overline{\omega_j}}}^{j-1} \frac{1}{j-1} \frac{\sum_{k=1}^{j-1} \omega_k-(j-1)\omega_h}{\sum_{k=1}^{j-1} \omega_k} \frac{ (j-1)\omega_i-\sum_{k=1}^{j-1}{\omega_k}}{ \sum_{\substack{\ell=1 \\\omega_{\ell} > \overline{\omega_j} }}^{j-1} \left [ (j-1)\omega_{\ell} - \sum_{k=1}^{j-1} \omega_k\right ]}
 \\
 &= \frac{1}{j-1} \left [ 1 +  \frac{\sum_{\substack{h=1 \\ \omega_h < \overline{\omega_j}}}^{j-1}\left [\sum_{k=1}^{j-1} \omega_k-(j-1)\omega_h\right ]}{\sum_{k=1}^{j-1} \omega_k} \frac{ (j-1)\omega_i-\sum_{k=1}^{j-1}{\omega_k}}{ \sum_{\substack{\ell=1 \\\omega_{\ell} > \overline{\omega_j} }}^{j-1} \left [ (j-1)\omega_{\ell} - \sum_{k=1}^{j-1} \omega_k\right ]}\right ]
 \\
&= \frac{1}{j-1} \left [ 1 +  \frac{1}{\sum_{k=1}^{j-1} \omega_k} \frac{ (j-1)\omega_i-\sum_{k=1}^{j-1}{\omega_k}}{ 1}\right ] \\
&= \frac{1}{j-1} \left [ 1 +  \frac{(j-1)\omega_i}{\sum_{k=1}^{j-1} \omega_k} -\frac{\sum_{k=1}^{j-1}{\omega_k}}{\sum_{k=1}^{j-1} \omega_k}\right ] \\
&= \frac{\omega_i}{\sum_{k=1}^{j-1} \omega_k} \\
&= \Pro(j \text{ at } i \text{ in WRT}).
\end{split}
\end{equation}
In the sixth line we used that
\begin{equation}
\begin{split}
& \sum_{\substack{h=1 \\ \omega_h < \overline{\omega_j}}}^{j-1}\left [\sum_{k=1}^{j-1} \omega_k-(j-1)\omega_h\right ] - \sum_{\substack{l=1 \\\omega_{\ell} > \overline{\omega_j} }}^{j-1} \left [ (j-1)\omega_{\ell} - \sum_{k=1}^{j-1} \omega_k\right ]\\
&= \sum_{\substack{h=1 \\ \omega_h < \overline{\omega_j}}}^{j-1}\left [\sum_{k=1}^{j-1} \omega_k-(j-1)\omega_h\right ] + \sum_{\substack{\ell=1 \\\omega_{\ell} > \overline{\omega_j} }}^{j-1}  \left [ \sum_{k=1}^{j-1} \omega_k -(j-1)\omega_{\ell} \right ] \\
& \hspace{4em}+ \sum_{\substack{h=1 \\ \omega_h = \overline{\omega_j}}}^{j-1}\left [\sum_{k=1}^{j-1} \omega_k-(j-1)\omega_h\right ] \\
&= \sum_{h=1}^{j-1}\left [\sum_{k=1}^{j-1} \omega_k-(j-1)\omega_h\right ] \\
&= (j-1) \sum_{k=1}^{j-1} \omega_k- \sum_{h=1}^{j-1} (j-1)\omega_h \\
&= 0.
\end{split}
\end{equation}
Thus the probabilities for the reconstructed tree correspond to the ones in the weighted recursive tree we wanted to get. The values of $r_{i,j}$ and $p_{i,j,h}$ imply that, if $\omega_i = \overline{\omega_j}$, node $j$ is neither relocated from nor to $i$.

\subsection{Special Case: When the First $k$ Nodes Have Weight $\theta$}

Let now $(\omega_i)_{i \in \mathbb{N}}$ be such that for some $k\in \mathbb{N}$, $\theta \in \mathbb{R}^+$, $\omega_i= \theta$ for $i\leq k$ and $\omega_i=1$ for $i>k$. As mentioned above we can either first construct a URT and then rearrange all nodes or relocate the new node at each step. Since the first $k$ nodes all have the same weight, for $1 \leq i < j < k+1$ we have 
\begin{equation}\Pro(j \text{ attaches to } i) = \frac{\theta}{(j-1)\theta} = \frac{1}{j-1}.
\end{equation}
Hence we don't need to change anything for the first $k+1$ nodes.

We now need to differentiate between the case $\theta>1$ and $\theta<1$, because in the first case the probability that a node attaches to the first $k$ nodes increases in the WRT compared to the URT and in the second case this probability decreases. So in the first case we will need to relocate nodes from the nodes $k+1, \dots, n-1$ to the nodes $1, \dots, k$ and in the second case the other way round. The exact rules of relocation will be specified below. 

\begin{enumerate}
\item Let first $\theta < 1$. We have $\Pro(j \text{ attaches to } i \text{ in the URT}) = \frac{1}{j-1}$ and for $j > k$ and $i \leq k$ we would like to have $\Pro(j \text{ at } i \text{ in reconstructed tree}) = \frac{\theta}{k\theta + j-k-1}.$ Let $r_{i,j}$ be the probability that $j$ is relocated given that it is attached to some $i \leq k$. Then we want
\begin{equation} 
\begin{split}& \Pro(j \text{ at } i \text{ in WRT})  = \Pro(j \text{ at } i \text{ in reconstructed tree}) \\
& \Leftrightarrow \frac{\theta}{k\theta + j-k-1} = \Pro(j \text{ at } i \text{ in URT}) \Pro(j \text{ is not relocated}) \\
& \Leftrightarrow \frac{\theta}{k\theta + j-k-1} = \frac{1}{j-1} (1-r_{i,j})\\
& \Leftrightarrow 1-r_{i,j} = \frac{\theta(j-1)}{k\theta + j-k-1} \\
& \Leftrightarrow r_{i,j} = \frac{k\theta +j-k-1 -\theta(j-1)}{k\theta + j-k-1}.
\end{split}
\end{equation}
If the node $j>k$ is relocated from a node $i \leq k$, we attach it with uniform probability to any node $k<h \leq j-1$. We now show that this gives the right probabilities.

Let $h >k$, then
\begin{equation}
\begin{split}
& \Pro(j \text{ at }h \text{ in reconstructed tree }) \\
&= \Pro(j \text{ at } h \text{ in URT}) \\
& \hspace{2em} + \sum_{\ell=1}^{k} \Pro(j \text{ at } \ell \text{ in URT})\Pro(j \text{ relocated } | j \text{ at } l) \Pro(j \text{ relocated to } h)  \\
&= \frac{1}{j-1} + \sum_{\ell=1}^k \frac{1}{j-1} \frac{k\theta + j -k -1 -(j-1)\theta}{k\theta + j -k -1} \frac{1}{j-1-k} \\
&= \frac{1}{j-1} \left ( 1 + \frac{k(k\theta + j -k -1) - k(j-1)\theta}{(k\theta + j -k -1) (j-1-k)} \right ) \\
&= \frac{(k\theta + j -k -1) (j-1-k) + k(k\theta + j -k -1) -k(j-1) \theta}{(j-1)(k\theta+j-k-1)(j-1-k)} \\
&= \frac{k\theta j - k\theta - k^2 \theta + (j-k-1)^2 + k^2 \theta + k(j-k-1) -kj\theta + k\theta}{(j-1)(k\theta+j-k-1)(j-1-k)} \\
&= \frac{(j-k-1)(j-1-k+k)}{(j-1)(k\theta+j-k-1)(j-1-k)} \\
&= \frac{1}{k \theta + j -k -1} \\
&= \Pro(j \text{ attaches to } h \text{ in WRT}).
\end{split}
\end{equation}

\item Let now $\theta>1$. In the URT $\Pro(j \text{ attaches to } i) = \frac{1}{j-1}$. But for $j > i$ and $i > k$ we want in the reconstructed tree $\Pro(j \text{ attaches to } i) = \frac{1}{k\theta + j-k-1}$. Let $r_{i,j}$ be the probability that $j$ is relocated given that it is attached to some $i > k$. Then we want
\begin{equation}
\begin{split}
& \Pro(j \text{ at } i \text{ in WRT}) = \Pro(j \text{ at } i \text{ in reconstructed tree}) \\
&\Leftrightarrow  \Pro(j \text{ at } i \text{ in WRT}) = \Pro(j \text{ at } i \text{ in URT}) \Pro(j \text{ is not relocated}) \\
& \Leftrightarrow \frac{1}{k\theta + j-k-1} = \frac{1}{j-1} (1-r_{i,j})\\
& \Leftrightarrow 1-r_{i,j} = \frac{(j-1)}{k\theta + j-k-1} \\
& \Leftrightarrow r_{i,j} = \frac{k(\theta -1)}{k\theta + j-k-1}.
\end{split}
\end{equation}
If node $j$ is relocated from a node $h>k$, we attach $j$ uniformly to a node $1 \leq i \leq k$. Again we check if we get the right probabilities.

Let $i\leq k$, then
\begin{equation}
\begin{split}
& \Pro(j \text{ at } i \text{ in reconstructed tree }) \\
&= \Pro(j \text{ at } i \text{ in URT}) \\
& \hspace{3em} + \sum_{\ell=k+1}^{j-1} \Pro(j \text{ at } \ell \text{ in URT})\Pro(j \text{ relocated } | j \text{ at } l) \Pro(j \text{ relocated to } i)  \\
&= \frac{1}{j-1} + \sum_{\ell=k+1}^{j-1} \frac{1}{j-1} \frac{k(\theta -1)}{k\theta + j -k -1} \frac{1}{k} \\
&= \frac{1}{j-1} \left ( 1 + \frac{(j-k-1)k(\theta-1)}{(k\theta + j -k -1) k} \right ) \\
&= \frac{1}{j-1} \frac{j-1+k(\theta -1) + (j-k-1)(\theta -1)}{k\theta +j-k-1} \\
&= \frac{1}{j-1} \frac{(j-1)\theta}{k\theta + j -k -1}\\
&= \frac{\theta}{k \theta + j -k -1} \\
&= \Pro(j \text{ at } i \text{ in WRT}).
\end{split}
\end{equation}
\end{enumerate}

We will now start to study some tree statistics for WRTs. All probabilities in the next sections refer to probabilities in WRTs.

\section{The Number of Branches} \label{sec:BranchesWRT}

As before the number of branches is denoted by $\mathcal{B}_n^\omega$ and is equal to the number of nodes attaching to 1.
We have the following theorem
\begin{thm} \label{thm:branchesWRT}
Let $\mathcal{B}_n^\omega$ denote the number of branches in a WRT with weight sequence $(\omega_i)_{i \in \mathbb{N}}$. Then 
\begin{equation} \label{ExBranchesWRT}
\mathbb{E}[\mathcal{B}_n^\omega] =  \omega_1 \sum_{i=1}^{n-1} \frac{1}{\sum_{k=1}^{i}\omega_k}
\end{equation}
and 
\begin{equation} \label{VarBranchesWRT}
\Var\left(\mathcal{B}_n^\omega\right) = \sum_{i=1}^{n-1} \frac{\omega_1}{\sum_{k=1}^{i}\omega_k} - \sum_{i=1}^{n-1} \frac{\omega_1^2}{(\sum_{k=1}^{i}\omega_k)^2} = \omega_1 \sum_{i=2}^{n-1} \frac{ \sum_{k=2}^{i}\omega_k}{(\sum_{k=1}^{i}\omega_k)^2}.
\end{equation}
\end{thm}
\begin{proof}
 Since the construction steps are independent we can write $\mathcal{B}_n^\omega$ as a sum of independent Bernoulli random variables $b_i^\omega = \1(\text{node $i$ attaches to node 1})$. Thus \begin{equation}
 \mathcal{B}_n^\omega = \sum_{i=2}^{n} b_i^\omega
 \end{equation} and
\begin{equation}
\begin{split}
\E\left [\mathcal{B}_n^\omega\right] &=  \sum_{i=2}^{n} \E[b_i^\omega] = \sum_{i=2}^{n} \Pro(\text{$i$ attaches to $1$} ) = \sum_{i=2}^{n} \frac{\omega_1}{\sum_{k=1}^{i-1}\omega_k} \\ &= \omega_1 \sum_{i=2}^{n} \frac{1}{\sum_{k=1}^{i-1}\omega_k} = \omega_1 \sum_{i=1}^{n-1} \frac{1}{\sum_{k=1}^{i}\omega_k}.
\end{split}
\end{equation}

By independence of the $b_i^\omega$ it is also easy to calculate the variance:
\begin{equation}
\begin{split}
\Var\left(\mathcal{B}_n^\omega\right) &= \sum_{i=2}^{n} \E\left[(b_i^\omega)^2\right] - \sum_{i=2}^{n} \E[{b_1^\omega}]^2 = \sum_{i=2}^{n} \frac{\omega_1}{\sum_{k=1}^{i-1}\omega_k} - \sum_{i=2}^{n} \frac{\omega_1^2}{(\sum_{k=1}^{i-1}\omega_k)^2} \\
&= \sum_{i=2}^{n} \frac{ \omega_1 (\sum_{k=1}^{i-1}\omega_k) - \omega_1^2}{(\sum_{k=1}^{i-1}\omega_k)^2} = \omega_1 \sum_{i=3}^{n} \frac{ \sum_{k=2}^{i-1}\omega_k}{(\sum_{k=1}^{i-1}\omega_k)^2} = \omega_1 \sum_{i=2}^{n-1} \frac{ \sum_{k=2}^{i}\omega_k}{(\sum_{k=1}^{i}\omega_k)^2}.
\end{split}
\end{equation}
\end{proof}

\subsection{Some Examples of Weight Sequences}
We now give the exact values of the expected number of branches and their variance for some examples of weight sequences.
For the Hoppe tree, we write $\mathcal{B}_n^\theta$ for the number of branches and get 
\begin{equation}
\begin{split}
\E \left [\mathcal{B}_n^\theta\right ] &= \theta \sum_{i=2}^{n}\frac{1}{\theta + i-2} \\
&= \theta \sum_{i=0}^{n-2}\frac{1}{\theta + i} = \theta \sum_{i=0}^{n-2}\frac{1}{i+1} + \frac{1}{\theta + i} - \frac{1}{i+1}  \\ &= \theta H_{n-1} + \theta \sum_{i=0}^{n-2} \frac{1-\theta}{(\theta +i)(i+1)} \\
&= \theta H_{n-1} + \theta(1-\theta) \sum_{i=0}^{n-2} \frac{1}{i^2+(\theta+1)i+\theta}.
\end{split}
\end{equation}

For the variance we similarly get 
\begin{equation}
\begin{split}
\Var\left (\mathcal{B}_n^\theta\right) &= \theta \sum_{i=3}^n \frac{i-2}{(\theta + i-2)^2} = \theta \sum_{i=3}^n \frac{1}{i-2} + \frac{i-2}{(\theta + i-2)^2} - \frac{1}{i-2} \\ 
&= \theta \sum_{i=1}^{n-2} \frac{1}{i} + \theta \sum_{i=1}^{n-2} \frac{i^2 - (\theta + i)^2}{(\theta + i )^2 i} \\
&= \theta H_{n-2} - \theta^3 \sum_{i=1}^{n-2} \frac{1}{i (\theta +i)^2} - 2\theta^2 \sum_{i=1}^{n-2} \frac{1}{(\theta +i)^2}.
\end{split}
\end{equation}

Thus asymptotically as $n \to \infty$, we get in consistence with Theorem \ref{thm:HoppeBranches}
\begin{equation}
\E\left [\mathcal{B}_n^\theta \right ]  = \theta \ln(n) + \mathcal{O}(1) \text{ and } \Var \left (\mathcal{B}_n^\theta \right) = \theta \ln(n) + \mathcal{O}(1).
\end{equation}
\newpage
For the weighted recursive tree with $\omega_i= \theta$ for $i=1, \dots, k$ and $\omega_i= 1$ for $i>k$, we write  $T_n^{\theta^k}$ for the tree and $\mathcal{B}_n^{\theta^k}$ for the number of branches. We get
\begin{equation}
\begin{split}
 \E \left[\mathcal{B}_n^{\theta^k} \right]  &= \theta \sum_{i=1}^k \frac{1}{i\theta} + \theta \sum_{i=k+1}^{n-1} \frac{1}{k\theta + i-k} \\
  &=\theta\sum_{i=1}^{n-1} \frac{1}{i} + \theta\sum_{i=k+1}^{n-1} \frac{1}{\theta( k-1) + i} - \frac{1}{i} + \sum_{i=1}^k \frac{1}{i} -\frac{\theta}{i} \\
  &=\theta\sum_{i=1}^{n-1} \frac{1}{i} + k \theta(\theta-1)\sum_{i=k+1}^{n-1} \frac{1}{(k(\theta-1) + i)i} + (1-\theta) \sum_{i=1}^k \frac{1}{i}. \\
 \end{split}
\end{equation}
Also
\begin{equation}
\begin{split}
\Var & \left (\mathcal{B}_n^{\theta^k}\right) \\
&=  \theta \sum_{i=2}^{k} \frac{(i-1)\theta}{(i\theta)^2} + \theta \sum_{i=k+1}^{n-1} \frac{ (k-1)\theta+i-k}{(k\theta +i-k)^2} \\
&= \theta \sum_{i=2}^{n-1} \frac{(i-1)}{i^2} + \theta \sum_{i=2}^{k} \frac{(i-1)\theta}{(i\theta)^2} - \theta \sum_{i=2}^{k} \frac{(i-1)}{i^2} \\
& \hspace{2em} + \theta \sum_{i=k+1}^{n-1} \frac{ k(\theta-1)+i-\theta}{(k(\theta-1) +i)^2} -  \theta \sum_{i=k+1}^{n-1} \frac{(i-1)}{i^2} \\
&= \theta \sum_{i=2}^{n-1} \frac{(i-1)}{i^2} + (1-\theta) \sum_{i=2}^{k} \frac{(i-1)}{i^2} \\
& \hspace{2em} + \theta \sum_{i=k+1}^{n-1}\Bigg [ \frac{ k(\theta-1) i^2+i^3-\theta i^2 -k^2(\theta-1)^2i + k^2(\theta-1)^2}{(k(\theta-1) +i)^2 i^2}\\
& \hspace{7em} + \frac{ -i^3 + i^2 -2i^2k(\theta-1) + 2k(\theta-1) i}{(k(\theta-1) +i)^2 i^2}\Bigg ] \\
&= \theta \sum_{i=2}^{n-1} \frac{(i-1)}{i^2} + (1-\theta) \sum_{i=2}^{k} \frac{(i-1)}{i^2} \\
& \hspace{2em} + \theta \sum_{i=k+1}^{n-1} \frac{ - i^2 (k+1)(\theta-1) + i k(\theta-1)(k(\theta-1)+2) + k^2(\theta-1)^2}{(k(\theta-1) +i)^2 i^2}\\
& = \theta \left (H_{n-1} - H_{n-1}^{(2)}\right) + \mathcal{O}(1).
\end{split}
\end{equation}
Asymptotically this gives, as for the Hoppe tree,
\begin{equation}
\E\left [\mathcal{B}_n^{\theta^k}\right]  = \theta \ln(n) + \mathcal{O}(1)   \text{ and } \Var \left (\mathcal{B}_n^{\theta^k} \right) = \theta \ln(n) + \mathcal{O}(1).
\end{equation}

Concerning more general models the following can be said. If the weights are bounded from below and above, the expectation and variance of the number of branches will still be equal to $\mathcal{O}(\ln(n))$ asymptotically.

Let for all $i$, $0 < m < \omega_i < M$, then
\begin{equation}
\begin{split}
\omega_1 \sum_{i=1}^{n-1} \frac{1}{i M} &< \omega_1 \sum_{i=1}^{n-1} \frac{1}{\sum_{k=1}^{i}\omega_k} < \omega_1 \sum_{i=1}^{n-1} \frac{1}{i m} \\
\frac{\omega_1}{M} H_{n-1} &< \E\left [\mathcal{B}_n^\omega\right] < \frac{\omega_1}{m} H_{n-1}
\end{split}
\end{equation}
and 
\begin{equation}
\begin{split}
& \omega_1 \sum_{i=2}^{n-1} \frac{(i-1)m}{(iM)^2} < \omega_1 \sum_{i=2}^{n-1} \frac{ \sum_{k=2}^{i}\omega_k}{(\sum_{k=1}^{i}\omega_k)^2} <  \omega_1 \sum_{i=2}^{n-1} \frac{(i-1)M}{(im)^2} \\
& \frac{\omega_1 m}{M^2} \sum_{i=2}^{n-1} \frac{i-1}{i^2} < \Var\left(\mathcal{B}_n^\omega\right) <  \frac{\omega_1M}{m^2} \sum_{i=2}^{n-1} \frac{i-1}{i^2} \\
& \frac{\omega_1 m}{M^2} H_{n-1} + \mathcal{O}(1) < \Var\left(\mathcal{B}_n^\omega\right) <  \frac{\omega_1M}{m^2} H_{n-1} + \mathcal{O}(1).
\end{split}
\end{equation}

The situation is very different when the weights are not bounded. We now give some examples of weight sequences that give a different asymptotic behaviour. 
\begin{enumerate}
\item Let $(\omega_i)_{i\in \mathbb{N}} = (i)_{i \in \mathbb{N}}$. In this case we get from the above formulas
\begin{equation}
\begin{split}
\E\left [\mathcal{B}_n^\omega\right] &= 
 \sum_{i=1}^{n-1} \frac{1}{\sum_{k=1}^{i} k} =  \sum_{i=1}^{n-1} \frac{2}{(i+1)i} \\
 &= 2 \sum_{i=1}^{n-1}\left( \frac{1}{i}-\frac{1}{i+1} \right)= 2 \left (1- \frac{1}{n} \right ) \xrightarrow[n\to \infty]{} 2
\end{split}
\end{equation}
and 
\begin{equation}
\begin{split}
\Var\left(\mathcal{B}_n^\omega\right) 
&=  \sum_{i=2}^{n-1} \frac{ \frac{i(i+1)}{2}-1}{\left (\frac{i(i+1)}{2}\right )^2} \\
&=  \sum_{i=1}^{n-1} \frac{ \frac{i(i+1)}{2}-1}{\left (\frac{i(i+1)}{2}\right )^2} \\
&= \sum_{i=1}^{n-1} \frac{2i(i+1)}{(i(i+1))^2} - \frac{4}{(i(i+1))^2}\\
&= \sum_{i=1}^{n-1} \frac{2}{i(i+1)} - 4 \left (\frac{1}{i(i+1)}\right )^2\\
&= \sum_{i=1}^{n-1} \frac{2}{i(i+1)} - 4 \left (\frac{1}{i}- \frac{1}{i+1}\right )^2 \\
&= \sum_{i=1}^{n-1} \frac{2}{i(i+1)} - 4 \left (- \frac{2}{i(i+1)}+ \frac{1}{i^2} + \frac{1}{(i+1)^2} \right ) \\
&= \sum_{i=1}^{n-1} \frac{10}{i(i+1)} - 4 \left (\frac{1}{i^2} + \frac{1}{(i+1)^2} \right ) \\
&= 10 \sum_{i=1}^{n-1}\left ( \frac{1}{i}- \frac{1}{i+1}\right ) - 4 \sum_{i=1}^{n-1} \left (\frac{1}{i^2} + \frac{1}{(i+1)^2} \right ) \\
&= 10 \left ( 1 - \frac{1}{n} \right ) - 4 \sum_{i=1}^{n-1} \left (\frac{1}{i^2} + \frac{1}{(i+1)^2} \right ) \\
& \xrightarrow {n \to \infty} 10 - 4 \left ( 2\frac{\pi^2}{6} - 1 \right ) = 14 - \frac{4 \pi^2}{3} < 1.\\
\end{split}
\end{equation}

\item Let  $(\omega_i)_{i\in \mathbb{N}} = \left (\frac{1}{i} \right )_{i \in \mathbb{N}}$. In this case we get 
\begin{equation}
\E\left [\mathcal{B}_n^\omega\right] = \omega_1 \sum_{i=1}^{n-1} \frac{1}{\sum_{k=1}^{i}\omega_k} = \sum_{i=1}^{n-1} \frac{1}{\sum_{k=1}^{i}\frac{1}{k}} = \sum_{i=1}^{n-1} \frac{1}{H_i}
\end{equation}
and 
\begin{equation}\Var\left(\mathcal{B}_n^\omega\right) 
= \sum_{i=2}^{n-1} \frac{1}{H_i}  - \sum_{i=2}^{n-1} \frac{1}{H_i^2}.
\end{equation}

\item Let $(\omega_i)_{i\in \mathbb{N}} = \left (\frac{1}{i^2} \right )_{i \in \mathbb{N}}.$ In this case \begin{equation}
\E\left [\mathcal{B}_n^\omega\right] = \omega_1 \sum_{i=1}^{n-1} \frac{1}{\sum_{k=1}^{i}\omega_k} = \sum_{i=1}^{n-1} \frac{1}{\sum_{k=1}^{i}\frac{1}{k^2}} 
\end{equation}
hence 
\begin{equation}\frac{6}{\pi^2} (n-1) \leq \E\left [\mathcal{B}_n^\omega\right]  \leq n-1. 
\end{equation}
Also
\begin{equation}
\begin{split}
\Var\left(\mathcal{B}_n^\omega\right) &= \sum_{i=1}^{n-1} \frac{1}{\sum_{k=1}^{i}\frac{1}{k^2}}  - \sum_{i=1}^{n-1} \frac{1}{\left (\sum_{k=1}^{i}\frac{1}{k^2}\right )^2} \\
&= \sum_{i=2}^{n-1} \frac{1}{\sum_{k=1}^{i}\frac{1}{k^2}} \left ( 1 - \frac{1}{\sum_{k=1}^{i}\frac{1}{k^2}} \right ). \\
\end{split}
\end{equation}

Now we have on the one hand, for $0\leq a \leq 1$, that $ 0 \leq a(1-a) \leq \frac{1}{4}$. On the other hand $1 - \frac{1}{\sum_{k=1}^{i}\frac{1}{k^2}}$ is increasing in $i$ for $2 \leq i$ and $1 - \frac{1}{\sum_{k=1}^{2}\frac{1}{k^2}} = \frac{1}{5}$. Thus we get
\begin{equation}\frac{6}{5\pi^2} (n-2) \leq \Var\left(\mathcal{B}_n^\omega\right) \leq \frac{1}{4} (n-2).
\end{equation}
\end{enumerate}

\subsection{Central Limit Theorem}

As the number of branches can be written as a sum of independent random variables, we can apply Theorem \ref{thm:LiapounovBernoulli}, the Liapounov central limit theorem for Bernoulli random variables if $\Var(\mathcal{B}_n^{\omega})$ diverges.

\begin{thm} \label{CLT:branchesWRT}
If $\Var(\mathcal{B}_n^{\omega})$ diverges, $\mathcal{B}_n^\omega$ the number of branches of a weighted random recursive tree converges in distribution to a normal random variable:
\begin{equation}\frac{\mathcal{B}_n^{\omega} - \E[\mathcal{B}_n^{\omega}]}{\sqrt{\Var(\mathcal{B}_n^{\omega})}} \xrightarrow[d] {n \to \infty} \mathcal{G}.\end{equation}
In particular this is the case
\begin{enumerate}
\item if there is an $i>1$ s.t. $\omega_i>0$ and $\left ( \sum_{k=1}^{i} \omega_k \right )_{i \in \mathbb{N}}$ converges
\item and if  $\left ( \sum_{k=1}^{i} \omega_k \right )_{i \in \mathbb{N}}$ and $\left ( \sum_{i=1}^{n-1} \frac{1}{ \sum_{k=1}^{i} \omega_k} \right )_{n \in \mathbb{N}}$ diverge but  $\left ( \sum_{i=1}^{n-1} \frac{1}{\left ( \sum_{k=1}^{i} \omega_k \right )^2 } \right )_{n \in \mathbb{N}}$ converges.
\end{enumerate}
\end{thm}

\begin{proof}
In order to apply Liapounov's central limit theorem for sums of Bernoulli random variables we need to show that $\Var(\mathcal{B}_{n}^{\omega})$ diverges.
Now we have 
\begin{equation}\Var(\mathcal{B}_{n}^{\omega}) =  \omega_1 \sum_{i=2}^{n-1} \frac{ \sum_{k=2}^{i}\omega_k}{(\sum_{k=1}^{i}\omega_k)^2} =  \omega_1 \sum_{i=2}^{n-1} \frac{1}{\sum_{k=1}^{i}\omega_k} \frac{ \sum_{k=2}^{i}\omega_k}{\sum_{k=1}^{i}\omega_k}.
\end{equation}

If $\omega_1 > 0$ and for all $k>1$, $\omega_k=0$ this sum is zero so we do not have convergence. That the distribution of the branches is not normal in that case is obvious because all nodes will attach to 1, so the number of branches is just $n-1$.

Let us assume this is not the case, so there is an $i>1$ such that $\omega_i >0$. We now differentiate two cases. First let us assume that $\sum_{k=1}^{i-1}\omega_k$ converges to a number $a \in \mathbb{R}$. Then 
\begin{equation}\frac{  \sum_{k=2}^{i-1}\omega_k}{(\sum_{k=1}^{i-1}\omega_k)^2} \xrightarrow[i \to \infty]{}   \frac{a-\omega_1}{a^2} >0
\end{equation}
thus $\sum_{i=2}^{n} \Var(b_i^\omega)$ diverges.

 Let now $\sum_{k=1}^{i-1}\omega_k$ diverge. 
 Then \begin{equation} \Var(\mathcal{B}_n^{\omega}) = \sum_{i=2}^{n}\frac{  \sum_{k=2}^{i-1}\omega_k}{(\sum_{k=1}^{i-1}\omega_k)^2} = \sum_{i=2}^{n}\frac{1}{\sum_{k=1}^{i-1}\omega_k}- \frac{\omega_1}{(\sum_{k=1}^{i-1}\omega_k)^2}.
\end{equation}
Thus in that case if $ \sum_{i=2}^{n} \frac{1}{\sum_{k=1}^{i-1} \omega_k}$ diverges and $ \sum_{i=2}^{n} \frac{1}{\left ( \sum_{k=1}^{i-1} \omega_k\right )^2}$ converges, $\Var(\mathcal{B}_n^{\omega})$ diverges.
\end{proof}

\begin{rem}
Additionally to the statements in Theorem \ref{CLT:branchesWRT} we know that
if the series $\sum_{i=2}^{n} \frac{1}{\sum_{k=1}^{i-1} \omega_k}$ converges, i.e. if the weights increase too fast, the variance is finite and we cannot apply Theorem \ref{CLT:branchesWRT}. This is for example the case for $(\omega_i)_{i \in \mathbb{N}} = (i)_{i \in \mathbb{N}}$.
The other case where we cannot say anything general is when $\sum_{k=1}^{i-1}\omega_k$, $\sum_{i=2}^{n} \frac{1}{\sum_{k=1}^{i-1} \omega_k}$ and $\sum_{i=2}^{n} \frac{1}{\left ( \sum_{k=1}^{i-1} \omega_k\right )^2}$ all diverge. This is for example the case if $\omega_{i} = \frac{1}{i}$, since then $\sum_{k=1}^{i-1}\omega_k= \mathcal{O}(\ln(i))$. 
\end{rem}

\begin{ex}
For the other examples above, considering the order of divergence of the sum mostly allows to decide whether the central limit theorem holds. For example, if $\sum_{k=2}^{i-1}\omega_k = \mathcal{O}(i)$, then $ \frac{1}{\sum_{k=1}^{i-1}\omega_k} = \mathcal{O}(\frac{1}{i})$ and $ \frac{1}{\left ( \sum_{k=1}^{i-1}\omega_k \right )^2} = \mathcal{O}(\frac{1}{i^2})$ and thus $\Var(\mathcal{B}_n^{\omega})$ diverges. This is in particular the case when $\omega_i= \theta$ for $i=1, \dots, k$ and $\omega_i= 1$ for $i>k$. On the other hand, if $\sum_{k=1}^{i-1}\omega_k = \mathcal{O}(i^2)$, then $ \frac{  \sum_{k=2}^{i-1}\omega_k}{\left (\sum_{k=1}^{i-1}\omega_k \right )^2} = \mathcal{O}(\frac{1}{i^2})$ and thus $\sum_{i=2}^{n} \Var(b_i^\omega)$ converges. This is for example the case for $\omega_i=i$.
\end{ex} 
We have another possibility of proving convergence to a normal random variable.
\begin{thm} \label{thm:WRTBranchesCLTPoisson}
If  \begin{equation}
\frac{\sum_{i=1}^{n-1} \left ( \frac{1}{\sum_{k=1}^{i}\omega_k}\right )^2}{\sum_{i=1}^{n-1} \frac{1 }{\sum_{k=1}^{i}\omega_k}} \xrightarrow {n \to \infty} 0
\end{equation} and $\E\left [\mathcal{B}_n^\omega\right]$ diverges, then
\begin{equation} \frac{\mathcal{B}_n^\omega - \E\left [\mathcal{B}_n^\omega\right]}{\sqrt{\E\left [\mathcal{B}_n^\omega\right]}} \xrightarrow {n \to \infty} \mathcal{G}.\end{equation} 
\end{thm}
\begin{proof}
Let $\mu_n:= \E\left [\mathcal{B}_n^\omega\right] $. By Theorem \ref{thm:smallnumbers}
\begin{equation}
\begin{split}
d_{TV}(\mathcal{B}_{n}^{\omega}, \Po(\mu_n) ) &\leq \min\left \{1,\left ( \sum_{i=1}^{n-1} \frac{\omega_1 }{\sum_{k=1}^{i}\omega_k}\right ) ^{-1} \right \} \sum_{i=1}^{n-1} \left ( \frac{\omega_1}{\sum_{k=1}^{i}\omega_k}\right )^2 \\
& = \frac{\sum_{i=1}^{n-1} \left ( \frac{\omega_1}{\sum_{k=1}^{i}\omega_k}\right )^2}{\sum_{i=1}^{n-1} \frac{\omega_1 }{\sum_{k=1}^{i}\omega_k}} \\
& = \omega_1 \frac{\sum_{i=1}^{n-1} \left ( \frac{1}{\sum_{k=1}^{i}\omega_k}\right )^2}{\sum_{i=1}^{n-1} \frac{1}{\sum_{k=1}^{i}\omega_k}}.
\end{split}
\end{equation}
Thus, if $\frac{\sum_{i=1}^{n-1} \left ( \frac{1}{\sum_{k=1}^{i}\omega_k}\right )^2}{\sum_{i=1}^{n-1} \frac{1 }{\sum_{k=1}^{i}\omega_k}} \xrightarrow {n \to \infty} 0$ and $\E\left [\mathcal{B}_n^\omega\right]$ diverges, we can apply Theorem \ref{thm:PoisNorm} and get the result.
\end{proof}

\begin{rem}
Theorem \ref{thm:WRTBranchesCLTPoisson} implies that if for a given WRT the weight sequence satisfies $\omega_n \xrightarrow{n \to \infty} 0$, $\sum_{i=1}^{\infty} \omega_i = \infty$ and $\E\left [\mathcal{B}_n^\omega\right]$ diverges, then the number of branches of the corresponding WRT is asymptotically normal. This is implied by the fact that if $\omega_n \xrightarrow{n \to \infty} 0$, we have $\frac{\sum_{i=1}^n \omega_i^2}{\sum_{i=1}^n \omega_i} \xrightarrow{n \to \infty}  0$, which can be proved as follows:
Let $\varepsilon >0$ and $N$ be such that for all $n>N$, $\omega_n < \varepsilon$. Then we have for $n>N$,
\begin{equation}
\begin{split}
\frac{\sum_{i=1}^{n} \omega_i^2}{\sum_{i=1}^{n} \omega_i}& = \frac{\sum_{i=1}^{N} \omega_i^2}{\sum_{i=1}^{n} \omega_i} + \frac{\sum_{i=N+1}^{n} \omega_i^2}{\sum_{i=1}^{n} \omega_i} \\ & \leq \frac{\sum_{i=1}^{N} \omega_i}{\sum_{i=1}^{n} \omega_i^2} +  \varepsilon \frac{\sum_{i=N+1}^{n} \omega_i}{\sum_{i=1}^{n} \omega_i}  \leq \frac{\sum_{i=1}^{N} \omega_i^2}{\sum_{i=1}^{n} \omega_i} +  \varepsilon \xrightarrow{n \to \infty} \varepsilon. 
\end{split}
\end{equation}
Hence $\frac{\sum_{i=1}^n \omega_i^2}{\sum_{i=1}^n \omega_i} \xrightarrow{n \to \infty}  0$.
\end{rem}

\begin{rem}All of the above theorems only hold if $\Var\left(\mathcal{B}_n^\omega\right)$ diverges, which implies that none covers cases where $\E\left [\mathcal{B}_n^\omega\right]$ is finite. This follows from  \eqref{ExBranchesWRT} and \eqref{VarBranchesWRT} which show that if the variance diverges, the expectation diverges to.
\end{rem}

\subsection{Rate of Convergence}

We can use some theorems obtained by Stein's method to obtain bounds on the rate of convergence. We have
\begin{thm} \label{thm:BoundBranchesWRT}
Let $\mathcal{B}_n^{\omega}$ denote the number of branches of a weighted recursive tree. Then
\begin{equation}d_W\left ( \frac{\mathcal{B}_n^{\omega} - E[\mathcal{B}_n^{\omega}]}{\sqrt{\Var(\mathcal{B}_n^{\omega})}}, \mathcal{G} \right ) \leq \frac{1}{\sqrt{\Var(\mathcal{B}_n^{\omega})}}\frac{\sqrt{28}+\sqrt{\pi}}{\sqrt{\pi}}. \end{equation}
This bound decreases to 0 if and only if $\Var(\mathcal{B}_n^{\omega})$ diverges, which is the same criterion we had for the CLT.
\end{thm}
\begin{proof}
In order to apply Theorem \ref{thm:normalrate} we will consider $Y_i = b_{i+1}^{\omega} - \E[b_{i+1}^{\omega}]$ and hence $Y= \frac{\sum_{i=2}^n b_i^{\omega} - \E[b_i^{\omega}]}{ \sqrt{\Var(\mathcal{B}_n^{\omega})}} = \frac{\mathcal{B}_n^{\omega} - \E[\mathcal{B}_n^{\omega}]}{\sqrt{\Var(\mathcal{B}_n^{\omega})}}$. First of all the $Y_i$ are mutually independent thus giving $D=1.$ 
Now for $i=1, \dots, n-1$,
\begin{equation}
\begin{split}
\E\left [|Y_i|^3\right] &= \E\left [\left |b_i^{\omega} - \frac{\omega_1}{\omega_1 + \cdots + \omega_{i-1}}\right |^3 \right ] \\
&= \left ( 1 - \frac{\omega_1}{\omega_1 + \cdots + \omega_{i-1}} \right )^3 \frac{\omega_1}{\omega_1 + \cdots + \omega_{i-1}} \\
& \hspace{2em} + \left ( \frac{\omega_1}{\omega_1 + \cdots + \omega_{i-1}} \right )^3 \left ( 1- \frac{\omega_1}{\omega_1 + \cdots + \omega_{i-1}}\right) \\
&= \left ( 1 - \frac{\omega_1}{\omega_1 + \cdots + \omega_{i-1}} \right )\frac{\omega_1}{\omega_1 + \cdots + \omega_{i-1}} \\
& \hspace{2em} \cdot \left [\left ( 1 - \frac{\omega_1}{\omega_1 + \cdots + \omega_{i-1}} \right )^2 + \left (\frac{\omega_1}{\omega_1 + \cdots + \omega_{i-1}}\right)^2 \right ] \\
&= \left ( 1 - \frac{\omega_1}{\omega_1 + \cdots + \omega_{i-1}} \right )\frac{\omega_1}{\omega_1 + \cdots + \omega_{i-1}} \\
& \hspace{2em} \cdot \left [ 1 - 2\frac{\omega_1}{\omega_1 + \cdots + \omega_{i-1}} \left ( 1 -  \frac{\omega_1}{\omega_1 + \cdots + \omega_{i-1}}\right ) \right ] \\
&< \left ( 1 - \frac{\omega_1}{\omega_1 + \cdots + \omega_{i-1}} \right )\frac{\omega_1}{\omega_1 + \cdots + \omega_{i-1}} \\
&= \Var(b_i^{\omega}).
\end{split}
\end{equation}
We used that for $0<a<1$, $0<a(1-a)\leq\frac{1}{4}$, which implies $1> 1-2a(1-a)>\frac{1}{2}.$
Hence we get 
\begin{equation}\frac{D^2}{\sigma^3} \sum_{i=1}^{n-1} \E\left[|Y_i|^3\right] < \frac{1}{\sigma^3} \sum_{i=2}^{n}\Var(b_i^{\omega}) =\frac{1}{\sigma}.\end{equation}

Similarly 
\begin{equation}
\begin{split}
\E\left [Y_i^4\right] &= \E\left [\left (b_i^{\omega} - \frac{\omega_1}{\omega_1 + \cdots + \omega_{i-1}}\right )^4 \right ] \\
&= \left ( 1 - \frac{\omega_1}{\omega_1 + \cdots + \omega_{i-1}} \right )^4 \frac{\omega_1}{\omega_1 + \cdots + \omega_{i-1}} \\
& \hspace{2em} + \left ( \frac{\omega_1}{\omega_1 + \cdots + \omega_{i-1}} \right )^4 \left ( 1- \frac{\omega_1}{\omega_1 + \cdots + \omega_{i-1}}\right) \\
&= \left ( 1 - \frac{\omega_1}{\omega_1 + \cdots + \omega_{i-1}} \right )\frac{\omega_1}{\omega_1 + \cdots + \omega_{i-1}} \\
& \hspace{2em} \cdot \left [\left ( 1 - \frac{\omega_1}{\omega_1 + \cdots + \omega_{i-1}} \right )^3 + \left (\frac{\omega_1}{\omega_1 + \cdots + \omega_{i-1}}\right)^3 \right ] \\
&= \left ( 1 - \frac{\omega_1}{\omega_1 + \cdots + \omega_{i-1}} \right )\frac{\omega_1}{\omega_1 + \cdots + \omega_{i-1}} \\
& \hspace{2em} \cdot \Bigg [1 -3 \frac{\omega_1}{\omega_1 + \cdots + \omega_{i-1}} \left ( 1-  \frac{\omega_1}{\omega_1 + \cdots + \omega_{i-1}} \right ) \Bigg] \\
&< \left ( 1 - \frac{\omega_1}{\omega_1 + \cdots + \omega_{i-1}} \right )\frac{\omega_1}{\omega_1 + \cdots + \omega_{i-1}} \\
& = \Var(b_i^{\omega}).
\end{split}
\end{equation}
We again used that $0<a(1-a)<\frac{1}{4}$ for $0<a<1$, which implies $1>1-3a(1-a) > \frac{1}{4}$.
So we get
\begin{equation}\frac{\sqrt{28}D^{\frac{3}{2}}}{\sqrt{\pi}\sigma^2} \sqrt{\sum_{i=1}^{n-1} \E\left[Y_i^4\right]} < \frac{\sqrt{28}}{\sqrt{\pi}\sigma^2} \sqrt{\sum_{i=2}^n \Var(b_i^{\omega})} =  \frac{\sqrt{28}}{\sqrt{\pi}\sigma}.
\end{equation}
Hence in total we have
\begin{equation}d_{W} \left ( \frac{\mathcal{B}_n^{\omega} - \E[\mathcal{B}_n^{\omega}]}{\sigma}, \mathcal{G} \right ) \leq \frac{1}{\sigma} + \frac{\sqrt{28}}{\sqrt{\pi}\sigma} .\end{equation}
\end{proof}

Finally we will now give a result concerning the size of branches in Hoppe trees. We hope to generalize it to more general WRT models in the future.

\subsection{Largest Branch in a Hoppe Tree}

We denote by $\mathcal{B}_{n,i}$ the number of branches of size $i$ in $\mathcal{T}_n$, a recursive tree of size $n$. Now we define 
\begin{equation}\nu_n(\mathcal{T}_n) := \max\{i \in [n-1]: \mathcal{B}_{n, i} \geq 1\}\end{equation} to be the number of nodes  in the largest branch of that tree. In \cite{Feng05}, it was shown that  \begin{equation}\label{eqn:largestbranch}
\lim_{n \rightarrow \infty} \mathbb{P} \left(\nu_n(\mathcal{T}_n) \geq \frac{n}{2} \right) = \ln 2
\end{equation} when $\mathcal{T}_n$ is a URT on $n$ vertices.    The purpose of this section and the next theorem  is to  extend the  result of \cite{Feng05} to Hoppe trees, and to  provide more details about the asymptotic distribution, via exploiting the relation between Hoppe trees and Hoppe permutations, which were introduced in \ref{HoppePerm}.  Further, the result in \eqref{eqn:largestbranch} is now extended to an explicit expression for the limit $\lim_{n \rightarrow \infty} \mathbb{P} \left(\nu_n(\mathcal{T}_n) \geq cn \right)$ for $c \in [1/2,1]$.

\begin{thm} 
(i.) Let $\mathcal{T}_n^{\theta}$ be a Hoppe tree. Then $\frac{\nu_n (\mathcal{T}_n^{\theta})}{n}$ converges weakly to a random variable $\nu$ whose cumulative distribution function is given by
\begin{equation}
F_{\theta}(x)  = e^{\gamma \theta} x^{\theta - 1} \Gamma(\theta) p_{\theta}(1 / x) \text{ for } x > 0
\end{equation}
where   $\gamma$ is Euler's constant and
\begin{equation}p_{\theta}(x) = \frac{e^{- \gamma \theta} x^{\theta - 1}}{\Gamma(\theta)} \left(1 + \sum_{k=1}^{\infty} \frac{(- \theta)^k}{k!} \int \cdots \int_{\mathcal{S}_k(x)} \left(1 - \sum_{j=1}^k y_j \right)^{\theta - 1} \right) \frac{d y_1 \cdots d y_k}{y_1 \cdots y_k}
\end{equation}
with 
\begin{equation}\mathcal{S}_k(x)= \left\{y_1 > \frac{1}{x},\ldots, y_k > \frac{1}{x}, \sum_{j=1}^k y_j < 1\right\}.
\end{equation}

(ii) When $\theta = 1$, we obtain the following for the largest branch in a URT:
 $\frac{\nu_n(\mathcal{T}_n)}{n}$ converges weakly to a random variable $\nu$ whose cumulative distribution function is given by
\begin{equation}
F_1(x)  =
\begin{cases}
0 & \text{if } x<0 \\
1 + \sum_{k=1}^{\infty} \frac{(-1)^k}{k!} \int \cdots \int_{\mathcal{S}_k(x)} \frac{dy_1\ldots dy_k}{y_1\ldots y_k} & \text{if } x \in [0,1] \\ 
1 & \text{if } x >1
\end{cases}
\end{equation}
where 
  $\mathcal{S}_k(x)$ is as before. 

Also, for any $c \ in \left[\frac{1}{2},  1\right]$, we have \begin{equation}\lim_{n \rightarrow \infty} \mathbb{P}(\nu_n(\mathcal{T}_n) \leq c n) = 1 - \ln (c^{-1}).\end{equation} 
In particular, we have \begin{equation}\mathbb{E}[\nu] \geq \frac{n}{2}.\end{equation}
\end{thm}

\begin{proof}
(i.) First, we translate the problem into a random permutation setting. We have 
\begin{equation}\nu_n(\mathcal{T}_n^{\theta}) =_d \max\{i \in [n-1] : C_{n-1,i}(\theta) \geq 1\} =: \alpha_n(\theta)
\end{equation} 
where $C_{n-1,i}(\theta)$ is the number of cycles of length $i$ in a $\theta$-biased Hoppe permutation. In this setting, Kingman \cite{Kingman77} shows that $\frac{\alpha_n(\theta)}{n}$  converges in distribution to a random variable $\alpha$ with cumulative distribution function 
\begin{equation}F_{\theta} (x) = e^{\gamma \theta} x^{\theta - 1} \Gamma (\theta) p_{\theta} \left(\frac{1}{x} \right) \text{ for } x >0
\end{equation} 
where $\gamma$ is Euler's constant, 
\begin{equation}p_{\theta}(x) = \frac{e^{- \gamma \theta}}{\Gamma (\theta)} \left(1 + \sum_{k=1}^{\infty} \frac{(-\theta)^k}{k!} \right) \int \cdots \int_{\mathcal{S}_k(x)} \left(1 - \sum_{j=1}^k y_i \right)^{\theta - 1}\frac{dy_1\ldots dy_k}{y_1\ldots y_k}
 \end{equation} and 
  \begin{equation}\mathcal{S}_k(x)= \left\{y_1 > \frac{1}{x},\ldots, y_k > \frac{1}{x}, \sum_{j=1}^k y_j < 1\right\}.
  \end{equation}
This proves the first part. 

(ii) Setting $\theta =1$ in the argument of (i) and recalling that the random permutation in this case reduces to a uniformly random permutation   immediately  reveals the result. 

For the second claim, we first note Watterson \cite{Watterson76} shows that the derivative of $F_1(x)$ over $[1/2, 1]$ simplifies to \begin{equation}f_1(x) = \frac{1}{x}.\end{equation} Hence, for any $c \in \left[\frac{1}{2},  1\right]$
 \begin{equation} \lim_{n \rightarrow \infty} \mathbb{P}(\nu_n(\mathcal{T}_n) \leq c n) = \mathbb{P}(\nu \leq c n) = \int_{c}^1 \frac{1}{x} dx = \ln (1/c).
 \end{equation}
 We have \begin{equation}\mathbb{E}[\nu] \geq \int_{1/2}^1 x \frac{1}{x} dx = \frac{1}{2}.\end{equation}
\end{proof}

\begin{rem}
The value $\lim_{n \rightarrow \infty} \frac{\nu_n(\mathcal{T}_n)}{n}$ is known to be the 
  Golomb-Dickman constant in the literature. Its exact value is known to be $0.62432998854...$.
\end{rem}

\section{Depth of Node $n$}

\begin{thm} \label{thm:DepthWRT}

Let $\mathcal{D}_n^\omega$ denote the depth of node $n$ in a WRT $\mathcal{T}_n^{\omega}$ and let $Z_n^\omega$ denote the set of ancestors of $n$. Let moreover $A_{i,n}^\omega := \1(i \in Z_n^\omega)$.
Then 
\begin{equation}\mathcal{D}_n^\omega = 1 + \sum_{i=2}^{n-1} A_{i,n}^\omega.
\end{equation}
The $A_{i,n}^\omega$ are mutually independent Bernoulli random variables with 
\begin{equation} \Pro( A_{i,n}^\omega =1) = \frac{\omega_i}{\sum_{j=1}^{i}\omega_j}.\end{equation}
This directly yields the expectation and the variance:
\begin{equation}
\E\left[\mathcal{D}_n^\omega\right] = 1 + \sum_{i=2}^{n-1} \frac{\omega_i}{\sum_{j=1}^{i}\omega_j} = \sum_{i=1}^{n-1} \frac{\omega_i}{\sum_{j=1}^{i}\omega_j}
\end{equation}
and
\begin{equation} \Var\left (\mathcal{D}_n^\omega \right ) = \sum_{i=2}^{n-1} \frac{\omega_i}{\sum_{j=1}^{i}\omega_j} \left (1 - \frac{\omega_i}{\sum_{j=1}^{i}\omega_j} \right ) = \sum_{i=2}^{n-1} \frac{\omega_i\sum_{j=1}^{i-1}\omega_j}{(\sum_{j=1}^{i}\omega_j)^2}.
\end{equation}
\end{thm}
\begin{proof}
In a rooted tree, the depth of a node is equal to its number of ancestors, since these determine the path from the root to the node. Using that $1$ definitely is an ancestor of $n$, in the notation of the theorem we thus get
\begin{equation}\mathcal{D}_n^\omega = 1 + \sum_{i=2}^{n-1} A_{i,n}^\omega.
\end{equation}
We will first find the distribution law of the $A_{i,n}^\omega$ and then show mutual independence. For the distribution law we will use the method used in \cite{Feng05}: we will first find the values for $n-1$ and $n-2$ and then proceed by induction.

The node $n-1$ can only be an ancestor of $n$ if it is the parent of $n$, so we get 
\begin{equation}
\Pro(n-1 \in Z_n^\omega) = \frac{\omega_{n-1}}{\sum_{i=1}^{n-1} \omega_i}.
\end{equation}
Similarly $n-2$ can only be an ancestor of $n$ if it is the parent of $n$ or if it is the grandparent of $n$, in which case $n-2$ needs to be the parent of $n-1$ who needs to be the parent of $n$. This gives
\begin{equation}
\begin{split}
\Pro(n-1 \in Z_n^\omega) &= \underbrace{\frac{\omega_{n-2}}{\sum_{i=1}^{n-1} \omega_i}}_{n-2 \text{ is parent of } n} + \underbrace{\frac{\omega_{n-2}}{\sum_{i=1}^{n-2} \omega_i}\frac{\omega_{n-1}}{\sum_{i=1}^{n-1} \omega_i}}_{n-2 \text{ is grandparent } of n} \\
&= \frac{\omega_{n-2} \sum_{i=1}^{n-2} \omega_i +\omega_{n-2} \omega_{n-1}}{\sum_{i=1}^{n-1} \omega_i\sum_{i=1}^{n-2} \omega_i} \\
&= \frac{\omega_{n-2} \sum_{i=1}^{n-1}\omega_i}{\sum_{i=1}^{n-1}\omega_i\sum_{i=1}^{n-2} \omega_i} \\
&= \frac{\omega_{n-2}}{\sum_{i=1}^{n-2} \omega_i}.
\end{split}
\end{equation}
We will now show by induction that  for all $j = 2, \dots, n-1$,
\begin{equation}\Pro( j \in Z_n^\omega) = \frac{\omega_j}{\sum_{i=1}^{j} \omega_j}.
\end{equation}
Let the above be true for all $j \geq i+1$ and let $C_{i,j}^\omega$ denote the event that $j$ is a child of $i$. Then 
\begin{equation}\Pro( i \in Z_n^\omega) = \sum_{j = i+1}^{n-1} \Pro( j \in Z_n^\omega, C_{i,j}^\omega) + \Pro(C_{i,n}^\omega).
\end{equation}
Since $C_{i,j}^\omega$ only relates to the $j$-th step of the construction process and $j \in Z_n^\omega$ only depends on the $j+1$-th, \dots, $n$-th steps, these two events are independent. We thus get
\begin{equation}
\begin{split}
\Pro(i \in Z_n^\omega) & =  \sum_{j = i+1}^{n-1} \Pro( j \in Z_n^\omega)\Pro(C_{i,j}^\omega)+\Pro(C_{i,n}^\omega) \\
&= \sum_{j = i+1}^{n-1} \left (\frac{\omega_j}{\sum_{k=1}^{j}\omega_k} \frac{\omega_i}{\sum_{k=1}^{j-1}\omega_k} \right )+ \frac{\omega_i}{\sum_{j=1}^{n-1}\omega{j}}.
\end{split}
\end{equation}
To simplify this expression we first note that we can factor out $\omega_i$ and that the following holds:
\begin{equation}
\begin{split} 
& \frac{\omega_{i+1}}{\sum_{k=1}^{i+1}\omega_k\sum_{k=1}^{i}\omega_k} + \frac{\omega_{i+2}}{\sum_{k=1}^{i+2}\omega_k\sum_{k=1}^{i+1}\omega_k} \\
&\hspace{2em}= \frac{\omega_{i+1}\sum_{k=1}^{i+2}\omega_k+ \omega_{i+2}\sum_{k=1}^{i}\omega_k}{\sum_{k=1}^{i+2}\omega_k\sum_{k=1}^{i+1}\omega_k\sum_{k=1}^{i}\omega_k} \\
&\hspace{2em} = \frac{(\omega_{i+1}+ \omega_{i+2})\sum_{k=1}^{i}\omega_k+ \omega_{i+1}(\omega_{i+1}+ \omega_{i+2})}{\sum_{k=1}^{i+2}\omega_k\sum_{k=1}^{i+1}\omega_k\sum_{k=1}^{i}\omega_k} \\
&\hspace{2em}= \frac{(\omega_{i+1}+ \omega_{i+2})\sum_{k=1}^{i+1}\omega_k}{\sum_{k=1}^{i+2}\omega_k\sum_{k=1}^{i+1}\omega_k\sum_{k=1}^{i}\omega_k}\\
&\hspace{2em} = \frac{\omega_{i+1}+ \omega_{i+2}}{\sum_{k=1}^{i+2}\omega_k\sum_{k=1}^{i}\omega_k}.
\end{split}
\end{equation}
In general the following holds for $l \in \mathbb{N}$:
\begin{equation}
\begin{split}
&\frac{\omega_{i+1}+ \omega_{i+2} + \cdots + \omega_{i+l}}{\sum_{k=1}^{i} \omega_k \sum_{k=1}^{i+l} \omega_k} + \frac{\omega_{i+l+1}}{\sum_{k=1}^{i+l+1} \omega_k \sum_{k=1}^{i+l} \omega_k} \\
&\hspace{2em}= \frac{(\omega_{i+1}+  \cdots + \omega_{i+l})\sum_{k=1}^{i+l+1}\omega_k + \omega_{i+l+1}\sum_{k=1}^{i}\omega_k}{\sum_{k=1}^{i+l+1} \omega_k \sum_{k=1}^{i+l} \omega_k \sum_{k=1}^{i} \omega_k} \\
& \hspace{2em}= \frac{(\omega_{i+1}+  \cdots + \omega_{i+l} + \omega_{i+l+1})\sum_{k=1}^{i}\omega_k + (\omega_{i+1}+ \cdots +\omega_{i+l})\sum_{k=i+1}^{i+l+1}\omega_k}{\sum_{k=1}^{i+l+1} \omega_k \sum_{k=1}^{i+l} \omega_k \sum_{k=1}^{i} \omega_k} \\
&\hspace{2em} = \frac{(\omega_{i+1}+  \cdots + \omega_{i+l+1})\sum_{k=1}^{i+l}\omega_k} {\sum_{k=1}^{i+l+1} \omega_k \sum_{k=1}^{i+l} \omega_k \sum_{k=1}^{i} \omega_k} \\
&\hspace{2em}= \frac{\omega_{i+1}+  \cdots + \omega_{i+l+1}} {\sum_{k=1}^{i+l+1} \omega_k  \sum_{k=1}^{i} \omega_k}.
\end{split}
\end{equation}
By using this equality $n-i-2$ times, we thus get 
\begin{equation}
\begin{split}
\Pro(i \in Z_n^\omega) &= \omega_i \left ( \frac{\omega_{i+1} + \cdots + \omega_{n-1}}{\sum_{k=1}^{i}\omega_k \sum_{k=1}^{n-1}\omega_k} + \frac{1}{\sum_{k=1}^{n-1}\omega_k} \right ) \\
&= \omega_i \left ( \frac{\omega_{i+1} + \cdots + \omega_{n-1} + \sum_{k=1}^{i}\omega_k}{\sum_{k=1}^{i}\omega_k \sum_{k=1}^{n-1}\omega_k} \right ) \\
&= \frac{\omega_i}{\sum_{k=1}^{i} \omega_k}.
\end{split}
\end{equation}
Now we will show that the events $A_{i,n}^\omega$ are mutually independent for $j = 2, \dots, n-1$. For this we will use the method used in \cite{Hoppe}: for any $2 \leq k \leq n-2$ and $2 \leq j_k < \dots < j_2 < j_1 \leq n-1$ consider the event that all $j_i$'s and only the $j_i$'s are ancestors of $n$. We will denote this event by $E$. Then
\begin{equation}E := ( \1(j_i \in Z_n^\omega) =1, \1(j \in Z_n^\omega) = 0, \text{ for } j \neq j_i, i=1, \dots, k).
\end{equation}
By the structure of the recursive tree, to realize this event, $n$ must be a child of $j_1$, $j_1$ a child of $j_2, \dots , j_{k-1}$ a child of $j_k$ and $j_k$ a child of 1. In general for $i=1, \dots, k-1$, $j_i$ must be a child of $j_{i+1}.$  It does not matter what nodes  $j \neq j_i$ attach to. Hence, by the attachment probabilities we get:
\begin{equation}
\begin{split}
\Pro(E) &= \Pro(j_i \in Z_n^\omega, j \notin Z_n^\omega, \text{ for } j \neq j_i, i=1, \dots, k) \\
&= \underbrace{\frac{\omega_{j_1}}{\sum_{\ell=1}^{n-1} \omega_{\ell}}}_{n \text{ child of } j_1} \prod_{i=1}^{k-1} \underbrace{\frac{\omega_{j_{i+1}}}{\sum_{\ell=1}^{j_i-1}\omega_{\ell}}}_{j_i \text{ child of } j_{i+1}} \underbrace{\frac{\omega_1}{\sum_{\ell=1}^{j_k-1}\omega_{\ell}}}_{j_k \text{ child of } 1} \\
&= \omega_1 \omega_{j_1} \cdots \omega_{j_k} \prod_{i=1}^{n-1} \frac{1}{\sum_{\ell=1}^{i}\omega_{\ell}}  \prod_{\substack{1<j<n \\ j \neq j_i, i=1, \dots, k}} \left (\sum_{\ell=1}^{j-1}\omega_{\ell} \right ) \\
& = \prod_{i=1}^{k}\frac{\omega_{j_i}}{\sum_{\ell=1}^{j_i}\omega_{\ell}} \prod_{\substack{1<j<n \\ j \neq j_i, i=1, \dots, k}} \left ( \frac{\sum_{\ell=1}^{j-1}\omega_{\ell}}{\sum_{\ell=1}^{j}\omega_{\ell}}\right )  \frac{\omega_1}{\sum_{\ell=1}^{1}{\omega_{\ell}}} \\
&= \prod_{i=1}^{k}\frac{\omega_{j_i}}{\sum_{\ell=1}^{j_i}\omega_{\ell}} \prod_{\substack{1<j<n \\ j \neq j_i, i=1, \dots, k}} \left ( 1 - \frac{\omega_j}{\sum_{\ell=1}^{j}\omega_{\ell}}\right )  \\
&= \prod_{i=1}^{k} \Pro(j_i \in Z_n^\omega) \prod_{\substack{1<j<n \\ j \neq j_i, i=1, \dots, k}} \Pro(j \notin Z_n^\omega).
\end{split}
\end{equation}
This implies that the events $\1(i \in Z_n^\omega)$ are mutually independent. Hence the $A_{i,n}^\omega$ are mutually independent Bernoulli random variables, which immediately gives expectation and variance of $\mathcal{D}_{n}^\omega$ as stated in the theorem.
\end{proof}

\begin{rem} While the number of branches and the depth of node $n$ are identically distributed in URTs this is not generally the case for WRTs. Both of these statistics can be written as sums of Bernoulli random variables but these are not identically distributed in general.
\end{rem}

\subsection{Some Examples of Weight Sequences}
We will now consider the examples of weight sequences we also considered in Section \ref{sec:BranchesWRT} to illustrate the above remark.
For the Hoppe tree, we write $\mathcal{D}_n^\theta$ for the depth of node $n$ and have by Theorem \ref{thm:HoppeDepth} 
\begin{equation}
\begin{split}
&\E\left[\mathcal{D}_n^\theta\right] = 1+ \sum_{i=1}^{n-2} \frac{1}{\theta+1} = \ln(n) + \mathcal{O}(1) \text{ and }\\
&\Var \left (\mathcal{D}_n^\theta \right ) = \sum_{i=1}^{n-2} \frac{\theta+i-1}{(\theta+i)^2} = \ln(n) + \mathcal{O}(1).
\end{split}
\end{equation}
This implies that in a Hoppe tree the distribution of the depth of node $n$ is asymptotically close to the distribution of the depth in URTs, which is not the case for the number of branches of a Hoppe tree.

For the WRT with $\omega_i= \theta$ for $i=1, \dots, k$ and $\omega_i= 1$ for $i>k$, we call the tree $T_n^{\theta^k}$ and write $\mathcal{D}_n^{\theta^k}$ for the depth of node $n$. We get
\begin{equation}
\begin{split}
\E\left [\mathcal{D}_n^{\theta^k}\right] &= \sum_{i=1}^{k} \frac{\theta}{i \theta} + \sum_{i=k+1}^{n-1} \frac{1}{k \theta + i-k} \\
&= \sum_{i=1}^{k} \frac{1}{i} +\sum_{i=k+1}^{n-1} \frac{1}{i} + \sum_{i=k+1}^{n-1} \frac{1}{k \theta + i-k} - \sum_{i=k+1}^{n-1} \frac{1}{i} \\
& =  \sum_{i=1}^{n-1} \frac{1}{i} + \sum_{i=k+1}^{n-1} \frac{-k\theta +k }{(k \theta + i-k)i} \\
& = \sum_{i=1}^{n-1} \frac{1}{i} + \sum_{i=k+1}^{n-1} \frac{k( 1- \theta)}{(k (\theta-1) + i)i}. \\
\end{split}
\end{equation}
This shows that the expectation of the depth of node $n$ in a $\theta^k$-RT is smaller than the expectation in a URT by a constant term if $\theta>1$ and bigger by a constant term if $\theta<1$.

Similarly
\begin{equation} 
\begin{split}
\Var\left (\mathcal{D}_n^{\theta^k}\right) &= \sum_{i=2}^{k} \frac{\theta^2 (i-1)}{(i \theta)^2} + \sum_{i=k+1}^{n-1} \frac{k \theta + i-1-k }{(k\theta + i-k)^2} \\
&= \sum_{i=2}^{k} \frac{i-1}{i^2} + \sum_{i=k+1}^{n-1} \frac{ i-1}{i^2}  +  \sum_{i=k+1}^{n-1} \frac{k \theta + i-1-k }{(k\theta + i-k)^2} - \sum_{i=k+1}^{n-1} \frac{i-1}{i^2}  \\
&= \sum_{i=2}^{n-1} \frac{i-1}{i^2} +  \sum_{i=k+1}^{n-1} \frac{(k( \theta-1) + i-1 ) i^2 - ((k(\theta-1) + i)^2(i-1) }{(k(\theta-1) + i)^2i^2} \\
&= \sum_{i=2}^{n-1} \frac{i-1}{i^2} +  \sum_{i=k+1}^{n-1} \frac{k(\theta-1)(i^2+i+k(\theta-1))}{(i(k(\theta-1) + i))^2}. \\
\end{split}
\end{equation}
Asymptotically this gives, as for URTs, 
\begin{equation} \E\left[\mathcal{D}_n^{\theta^k}\right]  =  \ln(n) + \mathcal{O}(1)   \text{ and } \Var\left(\mathcal{D}_n^{\theta^k}\right) = \ln(n) + \mathcal{O}(1).
\end{equation}

More generally, as for the number of branches, if the weights are bounded from below and above, the expectation and variance of the depth of node $n$ will still be equal to $\mathcal{O}(\ln(n))$ asymptotically. Let  $0 < m < \omega_i < M$ for all $i$, then
\begin{equation}
\begin{split}
& \sum_{i=1}^{n-1} \frac{m}{i M} <  \sum_{i=1}^{n-1} \frac{\omega_i}{\sum_{j=1}^{i}\omega_j} <  \sum_{i=1}^{n-1} \frac{M}{i m} \\
& \Leftrightarrow \frac{m}{M} H_{n-1} < \E\left[\mathcal{D}_n^\omega\right] < \frac{M}{m} H_{n-1}.
\end{split}
\end{equation}
And similarly
\begin{equation}
\begin{split}
&\sum_{i=2}^{n-1} \frac{(i-1)m^2}{(iM)^2} < \sum_{i=2}^{n-1} \frac{\omega_i\sum_{j=1}^{i-1}\omega_j}{(\sum_{j=1}^{i}\omega_j)^2}  <  \sum_{i=2}^{n-1} \frac{(i-1)M^2}{(im)^2} \\
&\Leftrightarrow \frac{m^2}{M^2} \sum_{i=2}^{n-1} \frac{i-1}{i^2} < \Var\left (\mathcal{D}_n^\omega \right ) < \frac{M^2}{m^2} \sum_{i=2}^{n-1} \frac{i-1}{i^2} \\
&\Leftrightarrow \frac{m^2}{M^2} H_{n-1} + \mathcal{O}(1) < \Var\left (\mathcal{D}_n^\omega \right ) <  \frac{M^2}{m^2} H_{n-1} + \mathcal{O}(1).
\end{split}
\end{equation}

Thus in the examples above the qualitative behaviour of the depth of node $n$ is similar to the uniform case, since it is of order $\ln(n)$. For the branches the situtation was very different when the weights were not bounded. We will now consider these same examples and some others in order to compare the behaviours and get an idea of the range of values we can get for the expectation and variance of the depth of node $n$ in WRTs.

\begin{enumerate}
\item Let $(\omega_i)_{i\in \mathbb{N}} = (i)_{i \in \mathbb{N}}$. In this case we get from Theorem \ref{thm:DepthWRT}
\begin{equation}
\begin{split}
\E\left[\mathcal{D}_n^\omega\right]
&= \sum_{i=1}^{n-1} \frac{i}{\sum_{j=1}^{i} j} =  \sum_{i=1}^{n-1} \frac{2i}{(i+1)i} \\
&= 2 \sum_{i=1}^{n-1} \frac{1}{i+1} = 2 \ln(n) + \mathcal{O}(1)
\end{split}
\end{equation}
and 
\begin{equation}
\begin{split}
\Var\left (\mathcal{D}_n^\omega \right ) 
&=  \sum_{i=2}^{n-1} \frac{i \sum_{j=1}^{i-1} j}{\left ( \sum_{j=1}^{i} j \right )^2} 
=  \sum_{i=2}^{n-1} \frac{2 i^2 (i-1)}{(i(i+1) )^2} 
= 2 \sum_{i=2}^{n-1} \frac{ (i-1)}{(i+1)^2} \\
&= 2 \sum_{i=2}^{n-1} \frac{1}{i+1} - 2 \sum_{i=2}^{n-1} \frac{2}{(i+1)^2} =2\ln(n) + \mathcal{O}(1).
\end{split}
\end{equation}

\item We can get a similar result for a more general case. Let $(\omega_i)_{i \in \mathbb{N}} = (i^k)_{i \in \mathbb{N}}$ for $k \in \mathbb{N}$. In this case 
\begin{equation}
\begin{split}
\E\left[\mathcal{D}_n^\omega\right]
& = \sum_{i=1}^{n-1} \frac{i^k}{\sum_{j=1}^{i} j^k} =  \sum_{i=1}^{n-1} \frac{1}{\frac{1}{i^k} \frac{i^{k+1}}{k+1} + \mathcal{O}(i^k)} \\
&= \sum_{i=1}^{n-1} \frac{1}{\frac{i}{k+1} + \mathcal{O}(1)} = (k+1) \ln(n) + \mathcal{O}(1)
\end{split}
\end{equation}
and 
\begin{equation}
\begin{split}
\Var\left (\mathcal{D}_n^\omega \right ) 
&=  \sum_{i=2}^{n-1} \frac{i^k \left ( \frac{i^{k+1}}{k+1} +\mathcal{O}(i^k) \right )}{\left (  \frac{i^{k+1}}{k+1} +\mathcal{O}(i^k) \right )^2} \\
&= \sum_{i=2}^{n-1}  (\mathcal{O}(1)+ \mathcal{O} \left (\frac{1}{i} \right ) \frac{i^k}{ \frac{i^{k+1}}{k+1} +\mathcal{O}(i^k) } \\
&= \sum_{i=2}^{n-1}  \left ( \mathcal{O}(1)+ \mathcal{O} \left (\frac{1}{i} \right )\right) \frac{1}{\frac{i}{k+1} + \mathcal{O}(1)} \\
&= (k+1) \ln(n) +\mathcal{O}(1).
\end{split}
\end{equation}

\item Let $(\omega_i)_{i\in \mathbb{N}} = \left (\frac{1}{i} \right )_{i \in \mathbb{N}}$. In this case we get 
\begin{equation}
\E\left[\mathcal{D}_n^\omega\right] = \sum_{i=1}^{n-1} \frac{\frac{1}{i}}{\sum_{j=1}^{i}\frac{1}{j}} = \sum_{i=1}^{n-1} \frac{1}{iH_i} = \sum_{i=1}^{n-1} \frac{1}{i(\ln(i)+\mathcal{O}(1))} \xrightarrow {n \to \infty} \infty
\end{equation}
and 
\begin{equation}\Var\left (\mathcal{D}_n^\omega \right ) 
= \sum_{i=2}^{n-1} \frac{ \frac{1}{i} H_{i-1} }{iH_i^2} = \sum_{i=2}^{n-1} \frac{1 }{iH_i} -  \sum_{i=2}^{n-1} \frac{1 }{i^2H_i^2}  \xrightarrow {n \to \infty} \infty.
\end{equation}

\item Let $(\omega_i)_{i\in \mathbb{N}} = \left (\frac{1}{i^2} \right )_{i \in \mathbb{N}}$. Then \begin{equation}
\E\left[\mathcal{D}_n^\omega\right] =  \sum_{i=1}^{n-1} \frac{\frac{1}{i^2}}{\sum_{j=1}^{i}\frac{1}{j^2}}
\end{equation}
and since for all $i \in \mathbb{N}$ we have $1<\sum_{j=1}^{i}\frac{1}{j^2}<\frac{\pi^2}{6}$ we get
\begin{equation}
\begin{split}
& \frac{6}{\pi^2} \sum_{i=1}^{n-1} \frac{1}{i^2} \leq \E\left[\mathcal{D}_n^\omega\right] \leq \sum_{i=1}^{n-1} \frac{1}{i^2}\\
& \Rightarrow \frac{6}{\pi^2} \frac{\pi^2}{6} \leq \E\left[\mathcal{D}_n^\omega\right] \leq \frac{\pi^2}{6} \\
& \Rightarrow 1 \leq \E\left[\mathcal{D}_n^\omega\right] \leq \frac{\pi^2}{6}.
\end{split}
\end{equation} And since we have
\begin{equation}
\begin{split}
\Var\left (\mathcal{D}_n^\omega \right ) &= \sum_{i=1}^{n-1} \frac{\frac{1}{i^2} \sum_{j=1}^{i-1} \frac{1}{j^2}}{\left (\sum_{j=1}^{i}\frac{1}{j^2}\right )^2}
\end{split}
\end{equation}
and  for all $i\geq 2$, it holds that $\frac{4}{5} \leq \frac{\sum_{j=1}^{i-1} \frac{1}{j^2}}{\left (\sum_{j=1}^{i}\frac{1}{j^2}\right )^2}\leq 1$, we get
\begin{equation}
\begin{split}
\frac{4}{5} \sum_{i=2}^{n-1} \frac{1}{i^2} &\leq \Var\left (\mathcal{D}_n^\omega \right ) \leq \sum_{i=2}^{n-1} \frac{1}{i^2} \\
\Rightarrow \frac{4}{5} \frac{1}{4} &\leq \Var\left (\mathcal{D}_n^\omega \right ) \leq \frac{\pi^2}{6} -1 \\
\Rightarrow 0,2 &\leq \Var\left (\mathcal{D}_n^\omega \right ) \leq 0,65.
\end{split}
\end{equation}
We will also get a finite expectation and variance for $\left ( \frac{1}{i^k}\right )_{i \in \mathbb{N}}$ where $k \in \mathbb{N}$ by similar computations.

\item Let $(\omega_i)_{i \in \mathbb{N}} = \left (\ln(i) \right )_{i \in \mathbb{N}}$.
Then 
\begin{equation}
\begin{split}
\E\left[\mathcal{D}_n^\omega\right] & = \sum_{i=1}^{n-1} \frac{\ln(i)}{\sum_{j=1}^{i} \ln(j)} = \sum_{i=1}^{n-1} \frac{\ln(i)}{\ln(i!)} \\
& > \sum_{i=1}^{n-1} \frac{\ln(i)}{\ln(i^{i})} = \sum_{i=1}^{n-1} \frac{1}{i} = \ln(n) + \mathcal{O}(1).
\end{split}
\end{equation}
Also 
\begin{equation}
\begin{split}
\Var\left (\mathcal{D}_n^\omega \right ) &= \sum_{i=1}^{n-1} \frac{\ln(i)}{ \sum_{j=1}^{i} \ln(j)} - \sum_{i=1}^{n-1} \frac{(\ln(i))^2}{\left ( \sum_{j=1}^{i} \ln(j)\right )^2}  \\
&=  \sum_{i=1}^{n-1} \frac{\ln(i)}{ \ln(i!)} - \sum_{i=1}^{n-1} \frac{(\ln(i))^2}{\left ( \ln(i!) \right )^2} \\
&> \sum_{i=1}^{n-1} \frac{\ln(i)}{ \ln(i^{i})} - \sum_{i=1}^{n-1} \frac{(\ln(i))^2}{\left ( \ln(i^{\frac{i}{2}}) \right )^2} \\
&= \sum_{i=1}^{n-1} \frac{1}{i} - \sum_{i=1}^{n-1} \frac{4}{i^2} = \ln(n) + \mathcal{O}(1).
\end{split}
\end{equation}

\item Let $(\omega_i)_{i \in \mathbb{N}} = (a^{i-1})_{i \in \mathbb{N}}$.
First let us consider the case $a<1$. Then
\begin{equation}
\E\left[\mathcal{D}_n^\omega\right] = \sum_{i=1}^{n-1} \frac{a^{i-1}}{\sum_{j=1}^{i}a^{j-1}}  < \sum_{i=1}^{n-1} a^{i-1} < \frac{1}{1-a}
\end{equation}
and similarly
\begin{equation}
\Var\left (\mathcal{D}_n^\omega \right ) = \sum_{i=1}^{n-1} \frac{a^{i-1}}{\sum_{j=1}^{i} a^{j-1}} - \sum_{i=1}^{n-1} \frac{(a^{i-1})^2}{\left ( \sum_{j=1}^{i} a^{j-1}\right )^2} < \frac{1}{1-a}.
\end{equation}
Thus we have a finite expectation and variance.
On the other hand, if $a>1$, 
\begin{equation}
\begin{split}
\E\left[\mathcal{D}_n^\omega\right] &= \sum_{i=1}^{n-1} \frac{a^{i-1}}{\sum_{j=1}^{i}a^{j-1}}  = \sum_{i=0}^{n-2} \frac{(a-1)a^{i}}{a^{i+1}-1} =  \sum_{i=0}^{n-2} \frac{a-1}{a-\frac{1}{a^{i}}} \\
&> \sum_{i=0}^{n-2} \frac{a-1}{a} = (n-2)\frac{a-1}{a}
\end{split}
\end{equation}
and 
\begin{equation}
\begin{split}
& \Var\left (\mathcal{D}_n^\omega \right )  = \sum_{i=1}^{n-1} \frac{a^{i-1}}{\sum_{j=1}^{i} a^{j-1}}\left ( 1 - \frac{a^{i-1}}{ \sum_{j=1}^{i} a^{j-1}} \right ) = \sum_{i=0}^{n-2} \frac{(a-1)a^{i}}{a^{i+1}-1} \frac{\sum_{j=0}^{i-1}a^j}{ \sum_{j=0}^{i} a^{j}}  \\
 & = \sum_{i=0}^{n-2} \frac{(a-1)a^{i}}{a^{i+1}-1} \frac{a^{i}-1}{ a^{i+1} -1} > (a-1) \sum_{i=0}^{n-2} \frac{a^{2i}-a^{i}}{(a^{i+1}) ^2} > (a-1) \sum_{i=0}^{n-2} \frac{1-\frac{1}{a^{i}}}{a^2}  \\
 & > (a-1) \sum_{i=0}^{n-2} \frac{1}{a^2} -\frac{1}{a^{i+2}}  = \frac{a-1}{a^2} n + \mathcal{O}(1) . \\
\end{split}
\end{equation}
\end{enumerate}

\subsection{Central Limit Theorem}

As the depth of node $n$ can be written as a sum of independent Bernoulli random variables, we can apply Theorem \ref{thm:LiapounovBernoulli}, as we did for the number of branches.

\begin{thm}
Let $\mathcal{D}_n^\omega$ denote the depth of node $n$ in a weighted random recursive tree $\mathcal{T}_n^\omega$ with $n$ nodes.
If $\Var(\mathcal{D}_n^\omega)=\sum_{i=2}^{n-1} \frac{\omega_i}{\sum_{j=1}^{i}\omega_j} - \left ( \frac{\omega_i}{\sum_{j=1}^{i}\omega_j  } \right)^2$ diverges,  then
\begin{equation}\frac{\mathcal{D}_n^\omega - \E\left[\mathcal{D}_n^\omega\right]}{\sqrt{\Var(\mathcal{D}_n^\omega)}} \xrightarrow[d]{n \to \infty} \mathcal{G}.\end{equation}
\end{thm}

\begin{proof}
The result follows directly from Theorem \ref{thm:LiapounovBernoulli}, i.e. Liapounov's central limit theorem for sums of independent Bernoulli random variables. 
\end{proof}

\begin{rem}
We could not find a general condition for the weight sequence that implies divergence of the variance. We can though say the following:
\begin{enumerate}
\item When $\sum_{i=1}^{n} \omega_i$ converges, the variance does not diverge. This can be seen by bounding $\frac{\omega_i}{\sum_{j=1}^{i} \omega_j}$ from above by $\frac{\omega_i}{\omega_1}$. 
\item If the variance diverges the expectation diverges too.
\end{enumerate}
\end{rem}

\subsection{Rate of Convergence}
We can derive a rate of convergence for the depth of node $n$ that is similar to the rate of convergence for the number of branches.

\begin{thm}
Let $\mathcal{D}_n^{\omega}$ denote the number of branches of a weighted recursive tree. Then
\begin{equation}d_W\left ( \frac{\mathcal{D}_n^{\omega} - E[\mathcal{D}_n^{\omega}]}{\sqrt{\Var(\mathcal{D}_n^{\omega})}}, \mathcal{G} \right ) \leq \frac{1}{\sqrt{\Var(\mathcal{D}_n^{\omega})}}\frac{\sqrt{28}+\sqrt{\pi}}{\sqrt{\pi}}. \end{equation}
This bound decreases to 0 if and only if $ \Var(\mathcal{D}_n^{\omega})$ diverges, which is the same criterion as we had for the CLT.
\end{thm}
\begin{proof}
The proof follows steps similar to the proof of Theorem \ref{thm:BoundBranchesWRT}.
\end{proof}

This concludes the results about the depth of node $n$, we will continue with another statistic: the number of leaves.

\section{Number of Leaves}

We will compute the number of leaves for a specific kind of weighted recursive tree only. In this section we assume that the weight sequence $(\omega_i)_{i \in \mathbb{N}}$ is such that for some $\theta>0$,
\begin{equation}
\omega_i =  
\begin{cases} 
\theta \text{ for }1 \leq i \leq k \\ 1 \text{ otherwise} .
\end{cases}
\end{equation}
 We will call such WRTs $\theta^k$-RTs and denote them by $\mathcal{T}_n^{\theta^k}$. The techniques used here should also apply for more general tree models, but with some accompanying cumbersome notation.

\begin{thm} \label{thm:LeavesWRT}
Let $\mathcal{L}_n^{\theta^k}$ denote the number of leaves of a WRT $\mathcal{T}_n^{\theta^k}$. Then
\begin{enumerate}
\item for some constant $C'$  such that $ |C'|< | k (\theta-1) | + \frac{k (\theta +1)}{2}$, \begin{equation}\E\left [\mathcal{L}_n^{\theta^k}\right ] = \frac{n}{2} + C' + \mathcal{O}\left ( \frac{1}{n} \right )\end{equation}  
\item \begin{equation}\Var\left (\mathcal{L}_n^{\theta^k}\right ) =  \frac{n}{12} + \mathcal{O}(1) \text{ and }\end{equation} 
\item for all $t \geq 0$,
 \begin{equation}\Pro\left (\left |\mathcal{L}_n^{\theta^k} - \E\left [L_n^{\theta^k}\right ]\right | \geq t \right ) \leq 2e^{-6 \frac{t^2 }{k(\theta-1) +n+2}  } e^{\frac{6t (k\theta(k-1)) }{(k(\theta-1) +n-1) ( k(\theta-1) +n +2)}  } .\end{equation} 
\end{enumerate}
\end{thm}
\begin{proof}
The proof will make use of a martingale argument. First, note that for $n \geq 2$ we have
\begin{equation} \mathcal{L}_n^{\theta^k} = \mathcal{L}_{n-1}^{\theta^k} + Y_n^{\theta^k}\end{equation}
where 
\begin{equation}Y_n^{\theta^k} = \begin{cases} 1 \text{ if the parent of $n$ was not a leaf at time $n-1$} \\ 0 \text{ otherwise } \end{cases}
\end{equation}
and $\mathcal{L}_1^{\theta^k} = 0$, implying in particular $\mathcal{L}_2^{\theta^k} = 1$.

We will now construct the martingale that we will use throughout the proof.
We want to find an expression for  $\E\left[\mathcal{L}_{n}^{\theta^k} | \mathcal{L}_{n-1}^{\theta^k} \right]$, so that we can use a martingale to study the properties of $\mathcal{L}_{n}^{\theta^k}$. 
First of all let $N_n^{\theta^k}$ denote the number of leaves among the first $k$ nodes at time $n$. We have for $n \geq  k+1$,
\begin{equation}
\begin{split}
\E\left [Y_n^{\theta^k} \Big | \mathcal{L}_{n-1}^{\theta^k}\right ] &= \sum_{j=0}^k \E\left [Y_n^{\theta^k} \Big | \mathcal{L}_{n-1}^{\theta^k}, N_{n-1}^{\theta^k} = j\right ] \Pro\left (N_{n-1}^{\theta^k} = j\right ) \\
&= \sum_{j=0}^k \Pro\left (Y_n^{\theta^k}= 1  \Big | \mathcal{L}_{n-1}^{\theta^k}, N_{n-1}^{\theta^k} = j\right ) \Pro\left (N_{n-1}^{\theta^k} = j\right ) \\
&= \sum_{j=0}^k \left (1 - \Pro\left (Y_n^{\theta^k}= 0  \Big | \mathcal{L}_{n-1}^{\theta^k}, N_{n-1}^{\theta^k} = j\right ) \right ) \Pro\left (N_{n-1}^{\theta^k} = j\right ) \\
&= \sum_{j=0}^k \left (1 - \frac{\mathcal{L}_{n-1}^{\theta^k} -j}{k\theta +n-1 -k} - \frac{\theta j}{k\theta +n-1 -k} \right ) \Pro\left (N_{n-1}^{\theta^k} = j\right ) \\
&= \sum_{j=0}^k \Pro\left (N_{n-1}^{\theta^k} = j\right ) - \sum_{j=0}^k \frac{\mathcal{L}_{n-1}^{\theta^k}}{k\theta +n-1 -k}\Pro\left (N_{n-1}^{\theta^k} = j\right ) \\
& + \sum_{j=0}^k \frac{(1-\theta) j}{k\theta +n-1 -k} \Pro\left (N_{n-1}^{\theta^k} = j\right ) \\
&= 1 - \frac{\mathcal{L}_{n-1}^{\theta^k}}{k\theta +n-1 -k} + \frac{(1-\theta) \E\left [N_{n-1}^{\theta^k}\right ]}{k\theta +n-1 -k}. \\
\end{split}
\end{equation}
We thus get 
\begin{equation}
\begin{split}
\E\left [\mathcal{L}_n^{\theta^k} \Big | \mathcal{L}_{n-1}^{\theta^k}\right ] &= \E\left [\mathcal{L}_{n-1}^{\theta^k} + Y_n^{\theta^k} \Big | \mathcal{L}_{n-1}^{\theta^k}\right ] \\
&= \mathcal{L}_{n-1}^{\theta^k} + \E\left [ Y_n^{\theta^k} \Big| \mathcal{L}_{n-1}^{\theta^k}\right ] \\
&= \mathcal{L}_{n-1}^{\theta^k} + 1 - \frac{\mathcal{L}_{n-1}^{\theta^k}}{k\theta +n-1 -k} + \frac{(1-\theta) \E\left [N_{n-1}^{\theta^k}\right]}{k\theta +n-1 -k} \\
&= \frac{k(\theta-1) +n-2}{k(\theta-1)+n-1}\mathcal{L}_{n-1}^{\theta^k} + \frac{(1-\theta) \left (\E\left [N_{n-1}^{\theta^k}\right]-k\right ) +n-1}{k(\theta-1) +n-1} . \\
\end{split}
\end{equation}

\newpage
We now set $X_n^{\theta^k} := a(n) \mathcal{L}_n^{\theta^k} + b(n)$ and try to find $a(n)$ and $b(n)$ such that \linebreak $\E\left [X_n^{\theta^k} \Big | \mathcal{L}_{n-1}^{\theta^k} \right ] = X_{n-1}^{\theta^k}$. 
Now we have that
\begin{equation}
\begin{split}
& \E\left [X_n^{\theta^k} \Big | \mathcal{L}_{n-1}^{\theta^k} \right ] = X_{n-1}^{\theta^k} \\&\Leftrightarrow a(n) \E\left [\mathcal{L}_n^{\theta^k} \Big | \mathcal{L}_{n-1}^{\theta^k}\right ] +b(n) = a(n-1)\mathcal{L}_{n-1}^{\theta^k} + b(n-1) \\
&\Leftrightarrow a(n)\left ( \frac{k(\theta-1) +n-2}{k(\theta-1)+n-1}\mathcal{L}_{n-1}^{\theta^k} + \frac{(1-\theta) \left (\E\left [N_{n-1}^{\theta^k}\right]-k\right) +n-1}{k(\theta-1) +n-1} \right ) +b(n) \\
&\hspace{4em}= a(n-1)\mathcal{L}_{n-1}^{\theta^k} + b(n-1) \\
&\Leftrightarrow \left (a(n) \frac{k(\theta-1) +n-2}{k(\theta-1)+n-1}- a(n-1) \right ) \mathcal{L}_{n-1}^{\theta^k}  \\ 
&\hspace{4em}+ a(n) \frac{(1-\theta) \left (\E\left [N_{n-1}^{\theta^k}\right ]-k\right) +n-1}{k(\theta-1) +n-1} +b(n) - b(n-1) = 0.
\end{split}
\end{equation}
To simplify this equation we first set the factor of $\mathcal{L}_n^{\theta^k}$ equal to 0, which gives
\begin{equation}
\begin{split}
a(n) \frac{k(\theta-1) +n-2}{k(\theta-1)+n-1}- a(n-1) = 0 \Leftrightarrow \frac{k(\theta-1) +n-2}{k(\theta-1)+n-1} = \frac{a(n-1)}{a(n)}.
\end{split}
\end{equation}
Since we do not have boundary conditions we choose $a(n) = k(\theta-1)+n-1$.

Now we similarly need 
\begin{equation}
\begin{split}
&a(n) \frac{(1-\theta) \left (\E\left [N_{n-1}^{\theta^k}\right]-k\right) +n-1}{k(\theta-1) +n-1} +b(n) - b(n-1) = 0 \\ 
&\Leftrightarrow  b(n) - b(n-1) = (\theta-1) \left  (\E\left [N_{n-1}^{\theta^k}\right ]-k\right ) -(n-1).
\end{split}
\end{equation}
This is true for
\begin{equation} b(n) = (\theta -1)\left  (\sum_{i=k+1}^{n} \E\left [N_{i-1}^{\theta^k}\right] - k \right ) - \sum_{i=k+1}^{n} i-1.
\end{equation}
We thus define for $n\geq k+1$, \begin{equation}
X_n^{\theta^k} := (k(\theta-1)+n-1) \mathcal{L}_{n}^{\theta^k} + (\theta -1)\left  (\sum_{i=k+1}^{n} \E\left [N_{i-1}^{\theta^k}\right] - k \right ) - \sum_{i=k+1}^{n} i-1
\end{equation}
and thus have $\E\left [X_n^{\theta^k} | \mathcal{L}_{n-1}^{\theta^k}\right ]=X_{n-1}^{\theta^k}$.

We can now use this martingale to get the expectation.
Since $X_n^{\theta^k}$ is a martingale we get for all $n\geq k+1$,
\begin{equation} \E\left [X_n^{\theta^k}\right] = \E\left [X_{k+1}^{\theta^k}\right].
\end{equation}
Also we have
\begin{equation}
\begin{split}
\E\left [X_{k+1}^{\theta^k}\right] &= (k(\theta -1) +k+1-1) \E\left [\mathcal{L}_{k+1}^{\theta^k}\right] + (1-\theta)\left ( \E\left [N_k^{\theta^k}\right]-k\right) +k \\
&= k\theta \frac{k+1}{2} + (1-\theta) \left ( \frac{k}{2}-k \right ) +k \\
&= k\theta \frac{k+1}{2} + \frac{k}{2}(1+\theta) \\
&= \frac{k}{2} (k\theta +1 + 2\theta)  \\
\end{split}
\end{equation}
where the second equation comes from the fact that up to time $k+1$ the tree has the same attachment probabilities as a uniform recursive tree.

This gives us an expression for $\E\left [\mathcal{L}_n^{\theta^k}\right ] $:
\begin{equation}
\begin{split}
&\E\left [X_n^{\theta^k}\right ] = \frac{k}{2} (k\theta +1 + 2\theta)\\
&\Leftrightarrow a(n) \E\left [\mathcal{L}_n^{\theta^k}\right ]  +b(n) = \frac{k}{2} (k\theta +1 + 2\theta) \\
&\Leftrightarrow \E\left [\mathcal{L}_n^{\theta^k}\right ]  =\frac{1}{a(n)} \left ( \frac{k}{2} (k\theta +1 + 2\theta) -b(n) \right ).
\end{split}
\end{equation}
So we have
\begin{equation}
\begin{split}
&\E\left [\mathcal{L}_n^{\theta^k}\right ]  \\
&\hspace{1ex} =\frac{1}{k(\theta-1)+n-1} \\
& \hspace{2em} \cdot \left ( \frac{k}{2} (k\theta +1 + 2\theta) - (\theta -1)\left  (\sum_{i=k+1}^{n} \E\left [N_{i-1}^{\theta^k}\right ] - k \right ) + \sum_{i=k+1}^{n} i-1 \right )\\
&\hspace{1ex} = \frac{k(k\theta +1 + 2\theta)}{2(k(\theta-1)+n-1)}  + \frac{1-\theta}{k(\theta-1)+n-1} \left  (\sum_{i=k+1}^{n} \E\left [N_{i-1}^{\theta^k}\right] - k \right ) \\
&\hspace{2em} + \frac{1}{k(\theta-1)+n-1} \sum_{i=k+1}^{n} i-1 \\
&\hspace{1ex} = \frac{1}{k(\theta-1)+n-1} \frac{(n-k)(n-k-1)}{2} + C + \mathcal{O}\left ( \frac{1}{n} \right ) \\
&\hspace{1ex} = \frac{n}{2} + \frac{(n-k)(n-k-1)- n(k(\theta-1)+n-1) }{2(k(\theta-1)+n-1)} + C + \mathcal{O}\left ( \frac{1}{n} \right ) \\
&\hspace{1ex} = \frac{n}{2} + \frac{n(-\theta-1)k +k+k^2 }{2(k(\theta-1)+n-1)} + C  + \mathcal{O}\left ( \frac{1}{n} \right ) \\
&\hspace{1ex} = \frac{n}{2} + C' + \mathcal{O}\left ( \frac{1}{n} \right )
\end{split}
\end{equation}
where $|C| < \left | k (\theta-1) \right |$ and $|C'|< | k (\theta-1) | + \frac{k (\theta +1)}{2}$. In the fifth line  we used that for all $i=k+1, \dots, n$,  we have $|\E[N_{i-1}^{\theta^k}] - k|<k $ since at any time there can be between 0 and $k-1$ leaves among the first $k$ nodes. 

We will now derive the concentration equality.
Because it is easier to manipulate zero-martingales we now set 
\begin{equation}
\begin{split}
Z_n^{\theta^k} :&= X_n^{\theta^k} - \E\left [X_{k+1}^{\theta^k}\right] \\
&= (k(\theta-1)+n-1) \mathcal{L}_{n}^{\theta^k} + (\theta -1)\left  (\sum_{i=k+1}^{n} \E\left [N_{i-1}^{\theta^k}\right] - k \right ) \\
& \hspace{4em}- \sum_{i=k+1}^{n} i-1 - \frac{k}{2}(k(\theta+2) +1).
\end{split}
\end{equation}
Then in particular
\begin{equation}Z_{k+1}^{\theta^k} = X_{k+1}^{\theta^k} - \E\left[X_{k+1}^{\theta^k}\right] = k\theta \left ( \mathcal{L}_{k+1}^{\theta^k} - \E\left [\mathcal{L}_{k+1}^{\theta^k}\right]\right).
\end{equation}
We have 
\begin{equation}
\begin{split}
&Z_i^{\theta^k} -Z_{i-1}^{\theta^k} \\
&\hspace{1ex} = X_i^{\theta^k}-X_{i-1}^{\theta^k} \\
&\hspace{1ex} = (k(\theta-1) +i-1)\mathcal{L}_i^{\theta^k} - (1 - \theta) \sum_{j=k}^{i-1}\left (\E\left [N_{j}^{\theta^k}\right] -k \right) - \sum_{j=k}^{i-1} j\\
&\hspace{2em} - \left [(k(\theta-1) +i-2) \mathcal{L}_{i-1}^{\theta^k} - (1-\theta) \sum_{j=k}^{i-2} \left (\E\left  [N_j^{\theta^k}\right]-k\right) -\sum_{j=k}^{i-2} j \right ] \\
&\hspace{1ex} = (k(\theta-1) +i-1)\mathcal{L}_i^{\theta^k} - (k(\theta-1) +i-2) \mathcal{L}_{i-1}^{\theta^k}\\
& \hspace{2em} -(1-\theta) \E\left [N_{i-1}^{\theta^k}\right] +(1-\theta)k -i+1. \\
\end{split}
\end{equation}
Now we use that
\begin{equation}\mathcal{L}_i^{\theta^k} = \mathcal{L}_{i-1}^{\theta^k} +Y_i^{\theta^k} \end{equation}
and
\begin{equation}\E\left [Y_i^{\theta^k}\right] = 1 - \frac{E\left [\mathcal{L}_{i-1}^{\theta^k}\right ] - \E\left [N_{i-1}^{\theta^k}\right]}{k(\theta-1) +i-1} - \frac{\theta E\left [N_{i-1}^{\theta^k}\right]}{k(\theta-1)+i-1}
\end{equation}
to get 
\begin{equation}
\begin{split}
&Z_i^{\theta^k}-Z_{i-1}^{\theta^k}  \\
&= Z_i^{\theta^k}-Z_{i-1}^{\theta^k} - (k(\theta-1)+i-1) \\
& \hspace{3em} \cdot \left [ E\left [Y_i^{\theta^k}\right] - \left ( 1 - \frac{E\left [\mathcal{L}_{i-1}^{\theta^k}\right] - \E\left [N_{i-1}^{\theta^k}\right]}{k(\theta-1) +i-1} - \frac{\theta E\left [N_{i-1}^{\theta^k}\right]}{k(\theta-1)+i-1} \right ) \right ]\\
&= (k(\theta-1) +i-1)\left (\mathcal{L}_{i-1}^{\theta^k}+Y_{i}^{\theta^k}\right) -  (k(\theta-1) +i-2) \mathcal{L}_{i-1}^{\theta^k} \\
& \hspace{2em} -(1-\theta) \E\left [N_{i-1}^{\theta^k}\right ] +(1-\theta)k -i+1 -(k(\theta-1)+i-1) E\left [Y_i^{\theta^k}\right] \\
& \hspace{2em} + k(\theta-1)+i-1 - \E\left [\mathcal{L}_{i-1}^{\theta^k}\right ] + \E\left [N_{i-1}^{\theta^k}\right] - \theta\E\left [N_{i-1}^{\theta^k}\right] \\
&= \mathcal{L}_{i-1}^{\theta^k} - \E\left [\mathcal{L}_{i-1}^{\theta^k}\right] + \left [k(\theta-1) +i-1\right ]\left [Y_{i}^{\theta^k} - E\left [Y_i^{\theta^k}\right] \right ].
\end{split}
\end{equation}
Since $Z_i^{\theta^k}$ is only defined for $k+1 \leq i \leq n$, the factor $k(\theta-1) +i-1$, which actually is the sum of all weights at time $i-1$, is positive for all such $i$.
Setting $R_i^{\theta^k} := \mathcal{L}_{i-1}^{\theta^k} - \E\left [\mathcal{L}_{i-1}^{\theta^k}\right ] - [k(\theta-1) +i-1]E\left [Y_i^{\theta^k}\right ]$, we get 
\begin{equation} R_i^{\theta^k} \leq Z_i^{\theta^k}-Z_{i-1}^{\theta^k} = R_i^{\theta^k} + [k(\theta-1) +i-1]Y_{i}^{\theta^k} \leq R_i^{\theta^k} + k(\theta-1) +i-1. \end{equation}
This allows us to derive a concentration inequality via Theorem \ref{thm:McDiarmid}. In our case, the martingale starts with $k+1$ and we thus get
\begin{equation}\Pro\left  ( \left | \sum_{i=k+1}^{n-1} Z_{i+1}^{\theta^k} -Z_{i}^{\theta^k} \right | \geq t \right )  \leq 2 e^{-\frac{2t^2}{\sum_{i=k+1}^n (k(\theta-1) +i-1)^2}}
\end{equation} 
which is equivalent to 
\begin{equation}\Pro\left  ( \left |Z_{n}^{\theta^k} -Z_{k+1}^{\theta^k} \right | \geq t \right )  \leq 2 e^{-\frac{2t^2}{\sum_{i=k+1}^n (k(\theta-1) +i-1)^2}}.
\end{equation} 
\newpage
We can now bound $\sum_{i=k+1}^n (k(\theta-1) +i-1)^2 $ by an integral bound. Since $(k(\theta-1) +i-1)^2$ is an increasing function, we have
\begin{equation}
\begin{split}
\sum_{i=k+1}^n  (k(\theta-1) +i-1)^2 &\leq \int_{k+1}^{n} (k(\theta-1) +x-1)^2 + (k(\theta-1) +n-1)^2 \\
&= \frac{(k(\theta-1) +x-1)^3}{3}\Bigg |_{x=k+1}^{n} + (k(\theta-1) +n-1)^2 \\
&= \frac{(k(\theta-1) +n-1)^3}{3} - \frac{(k\theta)^3}{3} + (k(\theta-1) +n-1)^2\\
&\leq \frac{(k(\theta-1) +n-1)^3}{3} + (k(\theta-1) +n-1)^2.
\end{split}
\end{equation}
Thus, \begin{equation}
e^{-\frac{2t^2}{\sum_{i=k+1}^n (k(\theta-1) +i-1)^2}} \leq e^{-\frac{2t^2}{\frac{(k(\theta-1) +n-1)^3}{3} + (k(\theta-1) +n-1)^2}} = e^{-\frac{6t^2}{(k(\theta-1) +n-1)^3 + 3 (k(\theta-1) +n-1)^2}}.
\end{equation}
In order to derive a concentration equality for $\mathcal{L}_n^{\theta^k}$ we notice first of all that
\begin{equation}
\begin{split}
Z_n^{\theta^k} &= (k(\theta-1)+n-1) \mathcal{L}_{n}^{\theta^k} + (\theta -1)\left  (\sum_{i=k+1}^{n} \E\left[N_{i-1}^{\theta^k}\right] - k \right ) \\
& \hspace{2em} - \sum_{i=k+1}^{n} i-1 - \frac{k}{2}- \E\left[X_{k+1}^{\theta^k}\right] \\
&=  (k(\theta-1) +n -1 ) \left (\mathcal{L}_n^{\theta^k} - \E\left[\mathcal{L}_n^{\theta^k}\right]\right ). \\
\end{split}
\end{equation}
Moreover \begin{equation}\left |Z_{k+1}^{\theta^k}\right| = \left | \theta k \left (\mathcal{L}_{k+1}^{\theta^k}-\E\left [\mathcal{L}_{k+1}^{\theta^k}\right]\right )\right | \leq \theta k\left (k-\frac{k+1}{2} \right ) = \frac{\theta k(k-1)}{2}
\end{equation}
where we use that the expected number of leaves is that of a URT until time $k+1$.
Thus for all $t \geq 0$,
\begin{equation}
\begin{split}
\Pro& \left ( \left |\mathcal{L}_n^{\theta^k} - \E\left [\mathcal{L}_n^{\theta^k}\right ]\right | \geq t \right ) \\
&= \Pro \left  (k(\theta-1) +n -1 ) \left |\mathcal{L}_n^{\theta^k} - \E\left [\mathcal{L}_n^{\theta^k}\right ]\right |  \geq (k(\theta-1) +n -1 ) t \right )\\
& = \Pro\left ( \left |Z_n^{\theta^k}\right| \geq (k(\theta-1) +n -1 ) t \right ) \\
& = \Pro\left ( \left  |Z_n^{\theta^k} \right | - \left|Z_{k+1}^{\theta^k}\right|  \geq  (k(\theta-1) +n -1 ) t - \left |Z_{k+1}^{\theta^k}\right |   \right ) \\
& \leq \Pro\left (\left |Z_n^{\theta^k}\right | - \left |Z_{k+1}^{\theta^k}\right|  \geq (k(\theta-1) +n -1 ) t -  \frac{k\theta(k-1)}{2} \right ) \\
& \leq \Pro\left (\left |  \left |Z_n^{\theta^k}\right| - \left |Z_{k+1}^{\theta^k}\right| \right | \geq (k(\theta-1) +n -1 ) t -  \frac{k\theta(k-1)}{2} \right ) \\
& \leq \Pro\left (\left |  Z_n^{\theta^k} - Z_{k+1}^{\theta^k} \right | \geq (k(\theta-1) +n -1 ) t -  \frac{k\theta(k-1)}{2} \right ) \\
& \leq 2e^{-\frac{6\left ((k(\theta-1) +n -1 ) t -  \frac{k\theta(k-1)}{2} \right )^2}{(k(\theta-1) +n-1)^3 + 3 (k(\theta-1) +n-1)^2}} \\
& = 2e^{-6 \frac{(k(\theta-1) +n -1 )^2 t^2 + \left ( \frac{k\theta(k-1)}{2}\right )^2 - 2(k(\theta-1) +n -1 ) t \frac{k\theta(k-1)}{2} }{(k(\theta-1) +n-1)^3 + 3 (k(\theta-1) +n-1)^2}  } \\
& = 2e^{-6 \frac{t^2 }{(k(\theta-1) +n-1) + 3}  } e^{-3 \frac{ \left ( k\theta(k-1)\right )^2 - 4(k(\theta-1) +n -1 ) t (k\theta(k-1)) }{2(k(\theta-1) +n-1)^3 + 6 (k(\theta-1) +n-1)^2}  } \\
& \leq 2e^{-6 \frac{t^2 }{k(\theta-1) +n+2}  } e^{\frac{12 t (k\theta(k-1)) }{2(k(\theta-1) +n-1)^2 + 6 (k(\theta-1) +n-1)}  } \\
& \leq 2e^{-6 \frac{t^2 }{k(\theta-1) +n+2}  } e^{\frac{6t (k\theta(k-1)) }{(k(\theta-1) +n-1) ( k(\theta-1) +n +2)}  }.  \\
\end{split}
\end{equation}
We will now calculate the variance, $\Var\left (\mathcal{L}_n^{\theta^k}\right )$. We will use that 
\begin{equation} \label{eq:Zn}
Z_n^{\theta^k} = \frac{k(\theta-1)+n-1}{k(\theta-1)+n-2} Z_{n-1}^{\theta^k}+ [k(\theta-1) +n-1] \left (Y_n^{\theta^k} - \E\left [Y_n^{\theta^k}\right ] \right )
\end{equation}
which we obtain by observing the following. First of all \begin{equation}Z_n^{\theta^k} = Z_{n-1}^{\theta^k} +\mathcal{L}_{n-1}^{\theta^k}- \E\left [\mathcal{L}_{n-1}^{\theta^k} \right ] + [k(\theta-1) +n-1] \left (Y_n^{\theta^k}-\E\left [Y_{n}^{\theta^k}\right ]\right).
\end{equation}
Also
\begin{equation}
\begin{split}\mathcal{L}_{n-1}^{\theta^k} = \frac{1}{k(\theta-1)+n-2}\Bigg ( Z_{n-1}^{\theta^k}& + E\left [X_{k+1}^{\theta^k}\right] \\
& - (\theta -1)\left  (\sum_{i=k}^{n-2} \E\left [N_{i}^{\theta^k}\right] - k \right ) + \sum_{i=k}^{n-2} i\Bigg)
\end{split}
\end{equation} and similarly
\begin{equation}
\begin{split}
\E \left [\mathcal{L}_{n-1}^{\theta^k} \right ] = \frac{1}{k(\theta-1)+n-2}\Bigg (\E & \left [Z_{n-1}^{\theta^k}\right]  + E\left [X_{k+1}^{\theta^k}\right] \\
& - (\theta -1)\left  (\sum_{i=k}^{n-2} \E \left  [N_{i}^{\theta^k}\right ] - k \right ) + \sum_{i=k}^{n-2} i\Bigg).
\end{split}
\end{equation}
Moreover we use that for all $n$
\begin{equation}\E\left [Z_{n}^{\theta^k}\right]= 0.
\end{equation}
Using \eqref{eq:Zn} we get,
\begin{equation}
\begin{split}
\E\left [{Z_n^{\theta^k}}^2\right] &= \left (\frac{k(\theta-1)+n-1}{k(\theta-1)+n-2} \right )^2 \E\left [{Z_{n-1}^{\theta^k}}^2\right ] \\
& \hspace{1em} + 2 \frac{(k(\theta-1)+n-1)^2}{k(\theta-1)+n-2} \E \left [ Z_{n-1}^{\theta^k}\left (Y_n^{\theta^k} - \E\left [Y_n^{\theta^k}\right ] \right )\right ]  \\
&\hspace{1em} + \left ( k(\theta-1) +n-1\right )^2 \E\left[\left (Y_n^{\theta^k} - \E \left [Y_n^{\theta^k}\right] \right )^2 \right].
\end{split}
\end{equation}
\newpage
Using again $\E\left [Z_{n}^{\theta^k}\right] = 0$ for all $n$ and the fact that $Z_{n-1}^{\theta^k}$ is $\mathcal{L}_1^{\theta^k}, \dots, \mathcal{L}_{n-1}^{\theta^k}$ measurable, we get
\begin{equation}
\begin{split}
\E &\left  [ Z_{n-1}^{\theta^k}\left (Y_n^{\theta^k} - \E\left[Y_n^{\theta^k}\right ]\right)\right] \\
&= \E\left [ Z_{n-1}^{\theta^k} Y_n^{\theta^k}\right] - \E\left[Z_{n-1}^{\theta^k}\E\left[Y_n^{\theta^k}\right]\right]  \\
&= \E\left [ \E \left [Z_{n-1}^{\theta^k} Y_n^{\theta^k} \Big | \mathcal{L}_1^{\theta^k}, \dots, \mathcal{L}_{n-1}^{\theta^k}\right ] \right ] \\
&= \E\left [ Z_{n-1}^{\theta^k} \E \left[Y_n^{\theta^k} \Big | \mathcal{L}_1^{\theta^k}, \dots, \mathcal{L}_{n-1}^{\theta^k}\right ] \right ]\\
&= \E \left [ Z_{n-1}^{\theta^k}\left ( 1 - \frac{\mathcal{L}_{n-1}^{\theta^k}}{k\theta +n-1 -k} + \frac{(1-\theta) \E\left [N_{n-1}^{\theta^k}\right]}{k\theta +n-1 -k} \right ) \right ]\\
&= \E \left [ Z_{n-1}^{\theta^k}\right ] - \frac{\E\left [Z_{n-1}^{\theta^k} \mathcal{L}_{n-1}^{\theta^k}\right ]}{k\theta +n-1 -k} + \E\left [Z_{n-1}^{\theta^k}\right ]\frac{(1-\theta) \E\left [N_{n-1}^{\theta^k}\right]}{k\theta +n-1 -k} \\
&= - \E \Bigg [\frac{ Z_{n-1}^{\theta^k}}{k\theta +n-1 -k}\\
& \hspace{4em} \cdot \left (\frac{Z_{n-1}^{\theta^k} + \E\left [X_{k+1}^{\theta^k}\right ] + (\theta-1)\sum_{i=k}^{n-2} \left (\E\left [N_i^{\theta^k}\right ]-k\right ) - \sum_{i=k}^{n-2} i}{k(\theta-1)+n-2)}\right ) \Bigg ]\\
&= - \frac{ \E \left [{Z_{n-1}^{\theta^k}}^2 \right ]}{(k(\theta-1) +n-1)(k(\theta-1)+n-2)}. \\
\end{split}
\end{equation}
Moreover 
\begin{equation}
\begin{split}
\E \left [Y_n^{\theta^k}\right] &= 1 - \frac{\E\left [\mathcal{L}_{n-1}^{\theta^k}\right ]}{k(\theta-1) +n-1} + \frac{(1-\theta) \E\left [N_{n-1}^{\theta^k}\right ]}{k(\theta-1) +n-1}\\
&= 1 - \frac{ \mathcal{O}\left (\frac{1}{n} \right) + \mathcal{O}(1)  +\frac{n}{2}}{k(\theta-1) +n-1} + \frac{(1-\theta) \E\left [N_{n-1}^{\theta^k}\right]}{k(\theta-1) +n-1}\\
&= \frac{1}{2} + \mathcal{O}\left (\frac{1}{n} \right)
\end{split}
\end{equation}
and 
\begin{equation}
\begin{split}
\Var \left (Y_n^{\theta^k}\right ) &=  \E\left [Y_n^{\theta^k}\right ] - \E\left [Y_n^{\theta^k}\right ]^2 \\
&= \frac{1}{2} + \mathcal{O}\left (\frac{1}{n} \right) - \left ( \frac{1}{2} + \mathcal{O}\left (\frac{1}{n} \right) \right  )^2 \\
&= \frac{1}{4} + \mathcal{O}\left (\frac{1}{n} \right).\\
\end{split}
\end{equation}
We thus get the recursion 
\begin{equation}
\begin{split}
\E \left [{Z_n^{\theta^k}}^2\right ]  &=  \left (\frac{k(\theta-1)+n-1}{k(\theta-1)+n-2} \right )^2 \E\left [{Z_{n-1}^{\theta^k}}^2\right ] \\
& \hspace{1em} - 2 \frac{(k(\theta-1)+n-1)^2}{k(\theta-1)+n-2} \frac{ \E \left  [{Z_{n-1}^{\theta^k}}^2\right ]}{(k(\theta-1) +n-1)(k(\theta-1)+n-2)} \\
&\hspace{1em} + \left ( k(\theta-1) +n-1\right )^2 \left ( \frac{1}{4} + \mathcal{O}\left (\frac{1}{n}\right ) \right) \\
&= \frac{(k(\theta-1) +n-1)^2 -2(k(\theta-1)+n-1)}{(k(\theta-1)+n-2)^2} \E\left [{Z_{n-1}^{\theta^k}}^2\right ] \\
& \hspace{1em}+  \left ( k(\theta-1) +n-1\right )^2 \left ( \frac{1}{4} + \mathcal{O}\left (\frac{1}{n}\right ) \right) \\
&= \frac{(k(\theta-1) +n-3)(k(\theta-1)+n-1)}{(k(\theta-1)+n-2)^2} \E\left [{Z_{n-1}^{\theta^k}}^2\right] \\
& \hspace{1em}+  \left ( k(\theta-1) +n-1\right )^2 \left ( \frac{1}{4} + \mathcal{O}\left (\frac{1}{n}\right ) \right). \\
\end{split}
\end{equation}
This is equivalent to 
\begin{equation}
\begin{split}
&\frac{k(\theta-1) +n-2}{k(\theta-1)+n-1} \E \left [{Z_n^{\theta^k}}^2\right] \\
&\hspace{3em} = \frac{k(\theta -1) +n-3}{k(\theta-1)+n-2} \E\left [{Z_{n-1}^{\theta^k}}^2 \right ] \\
& \hspace{6em} + (k(\theta-1) +n-2) \left ( k(\theta-1) +n-1\right ) \left ( \frac{1}{4} + \mathcal{O}\left (\frac{1}{n}\right ) \right).
\end{split}
\end{equation}
Now set \begin{equation}
W_n^{\theta^k} := \frac{k(\theta-1) +n-2}{k(\theta-1)+n-1} \E\left [{Z_n^{\theta^k}}^2\right ]
\end{equation} then we have 
\begin{equation}W_n^{\theta^k} = W_{n-1}^{\theta^k} + (k(\theta-1) +n-2) \left ( k(\theta-1) +n-1\right ) \left ( \frac{1}{4} + \mathcal{O}\left (\frac{1}{n}\right ) \right).
\end{equation}
This is satisfied for $n \geq 2$ by
\begin{equation}
W_n^{\theta^k} = \sum_{i=2}^{n} (k(\theta-1) +i-2) \left ( k(\theta-1) +i-1\right ) \left ( \frac{1}{4} + \mathcal{O}\left (\frac{1}{n}\right ) \right). \end{equation}
Hence 
\begin{equation}
\begin{split}
\E \left [{Z_n^{\theta^k}}^2\right ] = & \frac{k(\theta-1) +n-1}{k(\theta-1)+n-2} \\
&\cdot \sum_{i=2}^{n} (k(\theta-1) +i-2) \left ( k(\theta-1) +i-1\right ) \left ( \frac{1}{4} + \mathcal{O}\left (\frac{1}{n}\right ) \right).
\end{split}
\end{equation}
And thus
\begin{equation}
\begin{split}
& \Var  \left (\mathcal{L}_n^{\theta^k}\right ) \\
&\hspace{1ex}= \Var \Bigg ( \frac{1}{k(\theta-1)+n-1} \Bigg [ Z_n^{\theta^k} - (\theta -1) \sum_{i=k} \left (\E\left [N_i^{\theta^k}\right]-k\right ) \\
& \hspace{15em} + \sum_{i=k}^{n-1} i + \E\left [X_{k+1}^{\theta^k}\right]\Bigg] \Bigg )\\
&\hspace{1ex}= \frac{1}{(k(\theta-1)+n-1)^2} \left (\E \left [{Z_n^{\theta^k}}^2\right] - \E\left [{Z_n^{\theta^k}}\right ]^2\right ) \\
&\hspace{1ex}= \frac{1}{(k(\theta-1)+n-1)^2} \frac{k(\theta-1) +n-1}{k(\theta-1)+n-2} \\
&\hspace{2em} \sum_{i=2}^{n} (k(\theta-1) +i-2) \left ( k(\theta-1) +i-1\right ) \left ( \frac{1}{4} + \mathcal{O}\left (\frac{1}{n}\right ) \right) \\
&\hspace{1ex}= \frac{1}{(k(\theta-1)+n-1)(k(\theta-1)+n-2)} \\
&\hspace{2em} \sum_{i=2}^{n} (k(\theta-1) +i-2) \left ( k(\theta-1) +i-1\right ) \left ( \frac{1}{4} + \mathcal{O}\left (\frac{1}{n}\right ) \right) \\
&\hspace{1ex}=  \frac{1}{(k(\theta-1)+n-1)(k(\theta-1)+n-2)} \frac{1}{4} \sum_{i=2}^{n} i^2  + \mathcal{O}(1) \\
&\hspace{1ex}= \frac{1}{(k(\theta-1)+n-1)(k(\theta-1)+n-2)} \frac{1}{4} \left (\frac{(2n+1)(n+1)n}{6} -1\right )  + \mathcal{O}(1) \\
&\hspace{1ex}= \frac{n}{12} + \mathcal{O}(1).
\end{split}
\end{equation}
\end{proof}

\begin{rem} \label{rem:leavesProb} We can also get some results by writing $\mathcal{L}_{n}^{\omega}$ as the sum of $\1(\ell_i^{\omega})$ where $\ell_i^{\omega}$ denotes the event that $i$ is a leaf in a WRT. It follows from the construction principle that
\begin{equation}
\Pro \left (\ell_i^{\omega}\right ) = \prod_{j=i+1}^{n} \left ( 1- \frac{\omega_i}{\omega_1 + \dots + \omega_{j-1}} \right ).
\end{equation}
After some manipulation this expression becomes
\begin{equation}
\Pro \left (l_i^{\omega}\right )  = \frac{ \omega_1 + \dots + \omega_{i-1}}{\omega_1 + \dots + \omega_{n-1}} \prod_{j=i+1}^{n-1} \left ( 1 + \frac{\omega_{j} - \omega_{i}}{\omega_1 + \dots + \omega_{j-1}} \right ).
\end{equation}
Moreover we get the exact expression for the expectation of the number of leaves of a $\theta^k$-tree by writing $\E\left [\mathcal{L}_n^{\theta^k}\right ]  = \sum_{i=2}^{n} \E\left [\1\left (\ell_i^{\theta^k}\right )\right ]$ and using the above expression. After some computations we get by this method
\begin{equation}\E\left [\mathcal{L}_n^{\theta^k} \right ] = \frac{n}{2} + \frac{k(\theta-1)}{2} + \frac{k \theta(1-k\theta)}{2(k (\theta-1) + n-1)}+ \frac{k-1}{2} \prod_{i=1}^{n-1-k} \frac{\theta (k-1)+i}{\theta k +i}.
\end{equation}
\end{rem}

\begin{rem}
It is possible to derive a CLT by using this martingale, but we will now introduce another coupling construction, that allows much easier inferences concerning the distribution of the number of leaves.
\end{rem}

\section{A Coupling View of Some Special Kinds of WRTs}

We will now introduce three coupling constructions for special kinds of WRTs. They all have in common that they can only be applied to WRTs for which there is a $k \in \mathbb{N}$ such that the weights are constant for all nodes $i$ with $i>k$. This is a quite restrictive assumption, but we will see that the first and the third coupling allow very easy conclusions concerning several statistics.

\subsection{Construction of a $\theta^k$-RT from a URT in the Case $\theta \in \mathbb{N}^+$} \label{subsec:URT2mkRT}

In the case where $\theta \in \mathbb{N}$, we can use the following reconstruction to get a $\theta^k$-RT from a URT. To emphasize the assumption, we write $\theta = m$ from now on. First construct a URT on $mk + n-k$ nodes. We write $\mathcal{T}_{mk+n-k}$ for this URT. Since we want the weight of the first $k$ nodes to be $m$ we then join several nodes into one in the following way. To avoid confusion let us denote the nodes in the URT by $i$ and the nodes in the reconstructed tree by $i^T$. Then we have the following:
\begin{itemize}
\item The nodes $1, \dots, m$ will become the node $1^T$,
\item $m+1, \dots, 2m$ will become the node $2^T$ \dots 
\item $(k-1)m +1, \dots, km$ become $k^T$.
\end{itemize}

The new node $i^T$ gets all the children of $(i-1)m+1, \dots, im$. But since we joined several nodes into one and the nodes $(j-1)m +1, \dots, jm$ might have different parents, for $1<j\leq k$, we set the parent of $j^T$ as the parent of $(j-1)m+1$, i.e. of the node with the smallest label among those that become $j^T$. If in the URT the parent of $(j-1)m+1$ is any of the nodes $(i-1)m+1, \dots, im$, the parent of $j^T$ is $i^T$.

For $j >k$, we set $j^T = j+k(m-1)$, so all nodes after $k$ only correspond to a single node, we just need to "translate" the names of the nodes to take into account that we used $mk$ nodes instead of $k$ for the first $k$ nodes in the reconstructed tree. If the parent of node $j+k(m-1)$ for $j>k$ is among the first $km$ nodes of the URT, we check into which range this node falls and the parent of $j^T$ is chosen as above. In other words if, for $1 \leq i \leq k$, the parent of node $j+k(m-1)$ is one of $(i-1)m+1, \dots, im$, the parent of node $j^T$ is node $i^T$. If the parent of node $j+k(m-1)$ is equal to $h+k(m-1)$ with $h>k$, the parent of node $j^T$ is node $h^T$.
 The tree we thus obtain is called $\mathcal{T}_n^{m^k}$.
We can easily verify that the probabilities are right.
\begin{itemize}
\item  For $i<j\leq k$, \begin{equation}
\begin{split}
\Pro&(j^T \text{ attaches to } i^T) \\
& = \Pro((j-1)m+1 \text{ attaches to any of the nodes } (i-1)m+1, \dots, im) \\
&= \frac{m}{(j-1)m} \\
&= \frac{1}{j-1}.
\end{split}
\end{equation}
\item For $i \leq k < j$, 
\begin{equation}
\begin{split} 
\Pro&(j^T \text{ attaches to } i^T) \\
&= \Pro(j+k(m-1) \text{ attaches to any of the nodes } (i-1)m+1, \dots, im) \\
& = \frac{m}{j+k(m-1)-1} \\
& = \frac{m}{j-1-k+km}.
\end{split}
\end{equation}
\item For $ k<i < j$, 
\begin{equation}
\begin{split}
\Pro& (j^T \text{ attaches to } i^T) \\
& = \Pro(j+k(m-1) \text{ attaches to } i+k(m-1)) \\
&= \frac{1}{j+k(m-1)-1} \\
&= \frac{1}{j-1-k+km}.
\end{split}
\end{equation}
\end{itemize}

Let now $\mathcal{L}_{mk+n-k}$ denote the number of leaves of $\mathcal{T}_{mk+n-k}$ and $\mathcal{L}_n^{m^k}$ denote the number of leaves of $\mathcal{T}_{n}^{m^k}.$
Then $\mathcal{L}_{mk+n-k}$ can be used to bound $\mathcal{L}_n^{m^k}$.
First of all if a node $i > km $ is a leaf in $\mathcal{T}_{mk+n-k}$, the corresponding node in $\mathcal{T}_n^{m^k}$, which is ${i-k(m-1)}^T$, is also a leaf.  The reconstruction process thus only affects the children of the nodes i^T with $1 \leq i \leq k$, so we can have at most $k$ additional leaves.

 For $2 \leq i \leq k$, we note that $2m+1, 3m+1, \dots, (k-1)m+1$, determine the parent of $3^T, 4^T, \dots k^T$ respectively. Similarly the nodes $km+1, km+2, \dots, km+n-k$  determine the parent of $(k+1)^T, (k+2)^T, \dots n^T$ respectively. Hence if any of the above are children of one of $(i-1)m+1, (i-1)m+2, \dots, im$, the node $i^T$ will not be a leaf. Thus for $i\leq k$, the node $i^T$ is a leaf if and only if there is no node with the label $2m+1, 3m+1, \dots, (k-1)m+1$ and no node with label $km+1, km+2, \dots, km+n-k$ that attaches to any of the nodes $(i-1)m+1, (i-1)m+2, \dots, im$. This can theoretically be used to calculate the exact probability that $i^T$ is leaf for $1\leq i \leq k$.

 But without calculating these exact probabilities we can conclude the following. For each $2 \leq i \leq k$ we can at most "loose" $m-1$ leaves since if all $(i-1)m+1, \dots, im$ are leaves in $\mathcal{T}_{mk+n-k}$, $i^T$ will be a leaf in $\mathcal{T}_n^{m^k}$.  Hence we can conclude that 
\begin{equation}
\mathcal{L}_{mk+n-k}+ k > \mathcal{L}_n^{m^k} > \mathcal{L}_{mk+n-k}-k(m-1).
\end{equation}
Let moreover $\mathcal{L}_n$ denote the number of leaves in a URT on $n$ nodes, then
\begin{equation}
\begin{split}
 &\E[\mathcal{L}_{mk+n-k}] +k > \E\left[\mathcal{L}_n^{m^k}\right] > \E[\mathcal{L}_{mk+n-k}]-k(m-1) \\
 & \Leftrightarrow \frac{n+k(m-1)}{2} +k > \E\left[\mathcal{L}_n^{m^k}\right] > \frac{n+k(m-1)}{2}-k(m-1) \\
 &   \Leftrightarrow \frac{n+k(m+1)}{2} > \E\left[\mathcal{L}_n^{m^k}\right] > \frac{n-k(m-1)}{2} \\
  &   \Leftrightarrow \E[\mathcal{L}_n] + \frac{k(m+1)}{2} > \E\left[\mathcal{L}_n^{m^k}\right] > \E[\mathcal{L}_n] - \frac{k(m-1)}{2}  \\
    &   \Rightarrow \left |\E[\mathcal{L}_n^{m^k}]-\E[\mathcal{L}_n] \right | \leq \frac{k(m+1)}{2}. \\
\end{split}
\end{equation}

\subsection{Construction of a $\theta^k$-RT from a URT in the Case $\frac{1}{\theta} \in \mathbb{N}^+$}

In this case the construction turns out to be a little bit more complicated. When using the idea introduced above, in order to reduce the weight of the first $k$ nodes given a URT on an appropriate number of nodes we will now specify, we need to increase the weight of the remaining $n-k$ nodes. This means that the reconstruction affects a much bigger number of nodes. This in turn implies that the coupling is not as useful as the previous one but we nevertheless  introduce it, to show the same principle can be applied.

We now assume that the first $k$ nodes have weight $\theta$  with $\theta = \frac{1}{m}$ for some $m \in \mathbb{N}$ and that the remaining ones weight 1. Such a recursive tree on $n$ nodes has a total weight of $k \frac{1}{m} + (n-k)$. Since the actual value of the weights does not matter, but only the weights relative to each other, it is equivalent to construct a tree with the first $k$ nodes having weight 1 and the remaining ones $m$. This would result in a total weight of $k + (n-k)m$. So we first construct a URT with $nm - k(m-1)$ nodes. 

We leave the first $k$ nodes as they are, so $i^T = i$ for $i \leq k$. Then, for $i \geq k+1$ we need to join several nodes together. In particular:
\begin{itemize}
\item Nodes $k+1, \dots, k+m$ are joined together to form node $(k+1)^T$, 
\item nodes $k+m+1, \dots, k+2m$ are joined together to form node $(k+2)^T$ \dots 
\item nodes $k+ (i-k-1)m +1, \dots, k+ (i-k)m$ form node $i^T$ \dots
\item nodes $k+(n-k-1)m + 1, \dots, k+ (n-k)m$ form node $n^T$.
\end{itemize}

Again, as before, the parent of $i^T$ in the reconstructed tree is the parent of \linebreak $k+ (i-k-1)m +1$ in the URT and all children of the nodes we join together become children of the new node. Except for the first $k$ nodes, we need to remember that only if a node has a child with label $k+\ell m+1$ for some $0 \leq \ell \leq n-k-1$, the corresponding node will also have a child in the reconstructed tree.

We now check that this rearrangement gives us the right attachment probabilities.
\begin{itemize}
\item For $i < j \leq k+1$ both the attachment probabilities in the URT and the weighted recursive tree correspond to a URT, so there is nothing to check.
\item For $i\leq k < j $, 
\begin{equation}\begin{split}
&\Pro\left (j^T \text{ attaches to } i^T\right )\\
& = \Pro(k+(j-k-1)m+1 \text{ attaches to } i) \\
&= \frac{1}{k+(j-k-1)m} \\
&= \frac{\frac{1}{m}}{\frac{k}{m}+ j-k-1}.
\end{split}
\end{equation}
\item For $k<i<j$, 
\begin{equation}
\begin{split}
& \Pro(j^T \text{ attaches to }i^T) \\
& = \Pro(k+(j-k-1)m+1\text{ attaches to }\\
& \hspace{6em} \text{one of } k+ (i-k-1)m +1, \dots, k+ (i-k)m) \\
&= \frac{m}{k+(j-k-1)m} \\
&= \frac{1}{\frac{k}{m}+j-k-1}.
\end{split}
\end{equation}
\end{itemize}
While this reconstruction gives the tree we wanted, i.e. $\mathcal{T}_n^{\frac{1}{m}^k}$, it does not allow easy conclusions concerning the number of leaves since the coupling affects the children of all but a finite number of nodes.

\subsection{Constructing a Special Kind of WRT from a Hoppe Tree}
Finally we will introduce a third coupling construction that can be applied for a special case of WRTs that includes the two models presented above. We will now rearrange a Hoppe tree instead of a URT in order to get a WRT. Let $(\omega_i)_{i \in \mathbb{N}}$ be such that there is a $k\in \mathbb{N}$ such that $\omega_i = 1$ for $i>k$. We can then construct the corresponding WRT from a Hoppe tree using an inverse process than above. Instead of joining nodes into one we will now split the root into $k$ nodes. 

First construct a Hoppe tree with $n-k+1$ nodes and with $\theta$, the weight of the root, equal to $\sum_{i=1}^{k} \omega_i$. Then construct a weighted recursive tree of size $k$ corresponding to $(\omega_i)_{i \in \mathbb{N}}$. Now we replace the root of the Hoppe tree by this weighted recursive tree of size $k$. For $i \geq 2$, node $i$ in the Hoppe tree becomes node $i+k-1^T$ in the reconstructed tree. Then for all $i \geq 2$, if $i$ is a child of 1 in the Hoppe tree, $i+k-1^T$ becomes a child of one of the nodes $1^T, \dots, k^T$ in the reconstructed tree proportional to their weights. This means that if $i$ is a child of 1 in the Hoppe tree, $i+k-1^T$ will become a child of node $j^T$ in the reconstructed tree with probability $\frac{\omega_{j}}{\sum_{\ell=1}^{k} \omega_{\ell}}$. Let us check that this gives the right probabilities. 
\begin{itemize}
\item For $1\leq i \leq k<j$,
\begin{equation}
\begin{split}
\Pro & (j^T \text{ is child of } i^T \text{ in reconstructed tree})\\
& = \Pro(j-k+1 \text{ is child of } 1 \text{ in Hoppe tree}) \frac{\omega_i}{\sum_{\ell=1}^{k} \omega_{\ell}} \\
&= \frac{\sum_{\ell=1}^{k} \omega_{\ell}}{j-k+1 -2 +\sum_{\ell=1}^{k} \omega_{\ell}} \frac{\omega_i}{\sum_{\ell=1}^{k} \omega_{\ell}} \\
&= \frac{\omega_i}{j-1 -k+\sum_{\ell=1}^{k} \omega_{\ell} }. \\
\end{split}
\end{equation}
\item For $k<i<j$, 
\begin{equation}
\begin{split}
&\Pro(j^T \text{ is child of } i^T \text{ in reconstructed tree})\\
& = \Pro(j-k+1 \text{ is child of } i-k+1 \text{ in Hoppe tree})\\
&= \frac{1}{j-k+1 -2 +\sum_{\ell=1}^{k} \omega_{\ell}} \\
&= \frac{1}{j-1+\sum_{\ell=1}^{k} \omega_{\ell} -k}. \\
\end{split}
\end{equation}
\end{itemize}
\subsection{Using the Coupling for WRT Statistics Analysis} 

The coupling constructions given in previous subsections are useful in understanding various WRT statistics with the aid of well known results on URTs and Hoppe trees. We already gave an example in subsection \ref{subsec:URT2mkRT}. We demonstrate this briefly here for the number of leaves and the height of a WRT $\mathcal{T}_n^{\omega}$ with $(\omega_i)_{i \in \mathbb{N}}$ such that $\omega_i=1$ for $i>k$.

First focusing on the number of leaves, the reconstruction process implies that all the leaves of the Hoppe tree are still leaves in the reconstructed tree, since we do not change any relation among the nodes $2, \dots, n-k+1$ of the Hoppe tree or respectively $k+1^T, \dots n^T$ of the reconstructed tree. There can be at most $k-1$ additional leaves among the first $k$ nodes. Thus,
\begin{equation} \label{eq:CouplingLeaves}
\begin{split}
&\mathcal{L}_{n-k+1}^{\theta} \leq \mathcal{L}_n^{\omega} \leq  \mathcal{L}_{n-k+1}^{\theta} +k-1 \\
& \Rightarrow \frac{n-k+1}{2} + \frac{\sum_{i=1}^{k} \omega_i-1}{2} + \mathcal{O}\left ( \frac{1}{n} \right ) \leq \E\left [\mathcal{L}_n^{\omega}\right ] \\
& \hspace{7em} \leq \frac{n-k+1}{2} + \frac{\sum_{i=1}^{k} \omega_i-1}{2} + k -1 + \mathcal{O}\left ( \frac{1}{n} \right ) \\
& \Rightarrow \frac{n}{2} + \frac{\sum_{i=1}^{k} \omega_i-k}{2} + \mathcal{O}\left ( \frac{1}{n} \right ) \leq \E \left [\mathcal{L}_n^{\omega}\right ] \leq \frac{n}{2} + \frac{\sum_{i=1}^{k} \omega_i+k}{2} -1 + \mathcal{O}\left ( \frac{1}{n} \right ) \\
& \Rightarrow \E\left [\mathcal{L}_n^{\omega}\right ] = \frac{n}{2} + \mathcal{O}(1).
\end{split}
\end{equation}
In a similar way, one can make conclusions about the variance, concentration and asymptotic distribution of the number of leaves in a WRT with $\omega_i=1$ for $i>k$ for some $k \in \mathbb{N}$.

\begin{thm}
Let $\mathcal{L}^{\omega}_{n}$ denote the number of leaves of a WRT of size $n$ with $(\omega_i)_{i \in \mathbb{R}}$ such that there is a $k \in \mathbb{N}$ such that for all $i>k$ we have $\omega_i = 1$. Then $\mathcal{L}^{\omega}_{n}$ is asymptotically normal.
\end{thm}
\begin{proof}
In order to derive a central limit theorem for $\mathcal{L}^{\omega}_{n}$, we write 
\begin{equation}\frac{\mathcal{L}^{\omega}_{n}- \E[\mathcal{L}_{n-k+1}^\theta]}{\sqrt{\Var(\mathcal{L}_{n-k+1}^\theta)}} = \frac{\mathcal{L}^{\omega}_{n-k+1}- \E[\mathcal{L}_{n-k+1}^\theta]}{\sqrt{\Var(\mathcal{L}_{n-k+1}^\theta)}} - \frac{\mathcal{L}^{\omega}_{n-k+1}- \mathcal{L}_{n}^\theta}{\sqrt{\Var(\mathcal{L}_{n-k+1}^\theta)}}.
\end{equation}
Now we have by Theorem \ref{thm:HoppeLeaves} that 
\begin{equation}\frac{\mathcal{L}^{\omega}_{n-k+1}- \E[\mathcal{L}_{n-k+1}^\theta]}{\sqrt{\Var(\mathcal{L}_{n-k+1}^\theta)}} \xrightarrow[d] {n \to \infty} \mathcal{G}\end{equation}
and by \eqref{eq:CouplingLeaves} that  
\begin{equation}\left | \frac{\mathcal{L}^{\omega}_{n-k+1}- \mathcal{L}_{n}^\theta}{\sqrt{\Var(\mathcal{L}_{n-k+1}^\theta)}} \right | \leq \frac{k}{\sqrt{\Var(\mathcal{L}_{n-k+1}^\theta)}} \xrightarrow{a.s.} 0
\end{equation}
since $\Var(\mathcal{L}_{n-k+1}^\theta) \xrightarrow{n \to \infty} \infty$.
Now we can apply Slutsky's theorem, which we stated as Theorem \ref{thm:Slutsky} and conclude that
\begin{equation} \frac{\mathcal{L}^{\omega}_{n}- \E[\mathcal{L}_{n-k+1}^\theta]}{\sqrt{\Var(\mathcal{L}_{n-k+1}^\theta)}} \xrightarrow[d] {n\to \infty} \mathcal{G}.
\end{equation}
\end{proof}
As a second example, we discuss the height. Let $\mathcal{H}_n^\omega$ denote the height of a WRT with $(\omega_i)_{i \in \mathbb{N}}$ such that $\omega_i =1$ for $i>k$.  Let moreover  $\mathcal{D}_{1,i}^\omega$ and $\mathcal{D}_{1,i}^\theta$ denote the distance between the root and node $i$ in the reconstructed tree and the original tree respectively. Then for any node $i^T$, the path from the root to $i^T$ corresponds to the path from the root to the corresponding node in the original tree,  $i-k+1$, except that we might have an additional path among the first $k$ nodes instead of the first edge. Thus $\mathcal{D}^\omega_{1,i}$ is at least as big as $\mathcal{D}_{i-k+1}^\theta$.

Also $\mathcal{D}^\omega_{1,i}$ is at most  $k-1$ bigger than the distance between the root and the corresponding node in the original tree: For $i \leq k$, $\mathcal{D}_{1,i}^\omega$ is at most $k-1$ anyway. For $i>k$, node $i-k+1$ in the original tree becomes node $i^T$ in the reconstructed tree. Let $j-k+1$ be the first node on the path from 1 to $i-k+1$ in the original tree. Then in the reconstructed tree $j^T$ will be attached to some $h$, where $1\leq h \leq k$. Thus  for all $k+1 \leq i \leq n$, there is some $h \leq k$ such that, 
\begin{equation}
{D}^\omega_{1,i} = \mathcal{D}_{1,h}^\omega + 1 + \mathcal{D}_{j-k+1,i-k+1}^\theta = \mathcal{D}_{1,h}^\omega + \mathcal{D}_{1,i-k+1}^\theta.
\end{equation}
Also $\mathcal{D}_{1,h}^\omega \leq k-1$, so we have
\begin{equation}
\mathcal{D}_{1,i-k+1}^\theta \leq {D}^\omega_{1,i} \leq \mathcal{D}_{1,i-k+1}^\theta +k-1
\end{equation}
which implies that
\begin{equation}
\begin{split}
\max_{i =1, \dots, n-k+1} \{\mathcal{D}_{1,i}\}  \leq \max_{i =1, \dots, n} \{\mathcal{D}_{1,i}\} \leq \max_{i =1, \dots, n-k+1} \{\mathcal{D}_{1,i}\} + k-1.
\end{split}
\end{equation}
Thus from $\mathcal{H}_n = \max_{i =1, \dots, n} \{\mathcal{D}_{1,i}\},$ we can derive that
\begin{equation}
\begin{split}
\mathcal{H}_{n-k+1}^\theta \leq \mathcal{H}_n^{\omega} \leq \mathcal{H}_{n-k+1}^\theta + k-1.
\end{split}
\end{equation}
According to Theorem \ref{thm:HoppeHeight} we have 
$\E[\mathcal{H}_n^{\theta}] = e\ln(n) - \frac{3}{2}\ln \ln (n) + \mathcal{O}(1).$
Now we have 
\begin{equation}
\begin{split}
& e\ln(n-k+1) - \frac{3}{2}\ln \ln (n-k+1) \\
& = e\left (  \ln(n) + \ln\left (1-\frac{k-1}{n} \right) \right ) \\
& \hspace{3em} - \frac{3}{2} \left ( \ln \ln(n) + \ln \left ( 1+ \frac{\ln\left (1-\frac{k-1}{n} \right )}{\ln(n)} \right ) \right ) \\
&= e \ln(n) - \frac{3}{2} \ln \ln (n) + \mathcal{O}(1).
\end{split}
\end{equation}
Thus we get for $k>n$,
\begin{equation}
e \ln(n) - \frac{3}{2} \ln \ln (n) + \mathcal{O}(1) \leq \E\left [\mathcal{H}_n^{\omega} \right ] \leq e \ln(n) - \frac{3}{2} \ln \ln (n) + \mathcal{O}(1)+ k-1
\end{equation}
which  implies 
\begin{equation}
\E\left [\mathcal{H}_n^{\omega} \right ] = e \ln(n) - \frac{3}{2} \ln \ln (n) + \mathcal{O}(1).
\end{equation}
The CLT for $\mathcal{H}_n^{\omega}$ can be derived similarly as above by using Slutsky's Theorem.

\chapter{BIASED RECURSIVE TREES} \label{chapter:brt}

In Chapter \ref{chapter:review} we introduced the representation of recursive trees as permutations and showed that there is a bijection between uniform random permutations and uniform random recursive trees. In \cite{AltokIslak} Altok and I\c{s}lak raise the question how the recursive trees change when a different distribution on $S_n$ is chosen. In particular, the properties of the random recursive trees that are obtained from a riffle shuffle distribution on $S_n$ are studied.

In this chapter after introducing riffle-shuffle permutations, we will  be analyzing the associated $p$-biased recursive trees. Our study of this tree model will begin by reviewing some results from \cite{AltokIslak} on the number of leaves. Then we will present our further investigations  on $p$-biased recursive trees. In particular, we will be working on  the number of branches, the number of nodes with at least $k$ descendants and the depth of node $n$.

\section{Definitions and Basics} \label{sec:BRTDefinition}

Riffle shuffle permutations are based on a common method to shuffle cards: a deck of cards is first cut into two piles of approximately equal size and these piles are then riffled together, so that the cards of the two piles interleave. A mathematical model for riffle shuffles is given in \cite{BayerDiaconis}.

\begin{mydef}[\cite{BayerDiaconis}] \label{2-shuffle}
In order to \emph{riffle shuffle} a deck of $n$ cards, the deck is first cut into two piles according to a binomial distribution, so that the probability that the first pile has $k$ cards is $\frac{\binom{n}{k}}{2^n}$. Intuitively this means that the deck is approximately cut in half. The two piles are then riffled together by dropping the cards face down one by one proportionally to the size of the remaining piles. More precisely, let $A_1$ and $A_2$ denote the sizes of the remaining piles. Then the probability that the next card comes from pile $i$ is equal to $\frac{A_i}{A_1+A_2}$,  $i=1, 2$.
\end{mydef}

This model is the simplest version of a riffle shuffle and was later generalized into biased riffle shuffles, where the deck is cut into $a$ piles and the sizes of the piles are chosen according to a multinomial distribution, see \cite{Fulman98}.

\begin{mydef}[\cite{Fulman98}]
A \emph{$p$-biased riffle shuffle permutation} is obtained by first cutting a deck of $n$ cards into $a$ piles by determining the pile sizes according to $\mult(a;\vec{p})$ with $\vec{p}=(p_1,p_2, \dots, p_a)$. This means that pile sizes $b_1, \dots, b_a$ are chosen with probability $\binom{n}{b_1,\dots,b_a} \prod_{i=1}^{a}{p_i}^{b_i}$. Given the piles, they are riffled such that any of the $\binom{n}{b_1,\dots,b_a}$ ways of interleaving is equally likely. The order of the cards in each pile is not changed in this process.
We call the resulting distribution on $S_n$ the \emph{$p$-biased riffle shuffle distribution}.
When $\vec{p}$ is the uniform distribution over $[a]$, i.e. $\vec{p}= (\frac{1}{a}, \dots, \frac{1}{a})$, the resulting permutation is called an \emph{$a$-shuffle}.
\end{mydef}

The first definition of a riffle shuffle we gave above corresponds to a 2-shuffle. The shuffling of the cards after determining the sizes of the piles was described in different ways in these first two definitions. It is standard in the literature that these ways are equivalent. The following are a few more equivalent descriptions of $p$-biased riffle shuffles.

\begin{thm}[\cite{BayerDiaconis, Fulman98}] \label{thm:EquivalentRiffleShuffle}
The following are equivalent. 
\begin{enumerate}
\item $\gamma$ is a $p$-biased riffle shuffle.
\item A deck of $n$ cards is cut according to $mult(a,\vec{p})$ and then the cards are dropped face down one by one proportional to the size of the remaining piles. More precisely, let $A_i$ denote the remaining size of each pile, then the probability that the next card comes from pile $i$ is equal to $\frac{A_i}{A_1 + \dots + A_a}$, $i=1, \dots a$.
\item A deck of $n$ cards is cut according to $mult(a,\vec{p})$. The piles are then sequentially riffled together: first riffle pile 1 and 2 according to the rule in Definition \ref{2-shuffle}. Then riffle this combined pile with pile 3 according to that rule. Continue until all piles are riffled together.
\item Partition the interval $[0,1]$ into $a$ subintervals of length $\frac{1}{a}$. Then drop $n$ points into this interval according to the following rule: choose subinterval $i$ with probability $p_i$ and then drop the point uniformly in this interval. Subsequently label the points from 1 to $n$ according to their order from smallest to largest and apply the function $f(x)= ax \hspace{1ex}(\hspace{-1ex}\mod 1)$ to rearrange their order. 
\item $\gamma$ is the inverse of a permutation constructed in the following way: First assign a digit from 1 to $a$ to each card according to $\vec{p}$. Then reorder the cards by first taking all cards with digit 1, then all cards with digit 2, and so on, without changing the order of cards having the same digit. 
\end{enumerate}
\end{thm}
Given a $p$-biased riffle shuffle permutation $\gamma$ the corresponding recursive tree is constructed in the usual way: take 1 as the root, attach 2 to 1, then attach every node $i$ to the node with the label that is the rightmost number to the left of $i$ that is smaller than $i$ in $\gamma$. If there is no such element, attach $i$ to 1. 
\begin{mydef}[\cite{AltokIslak}]
A tree constructed from a permutation over $\{2, \dots, n\}$ having the $p$-biased riffle shuffle distribution is called a \emph{$p$-biased recursive tree}, short BRT, and denoted by $\mathcal{T}_{n}^{p}$. When $\vec{p}$ is the uniform distribution over $[a]$ we call the corresponding tree an \emph{$a$-recursive tree}, short $a$-RT, and denote it by $\mathcal{T}_{n}^{a}$.
\end{mydef}
 
In Figure \ref{fig:BRT4vert} all 2-RTs on 4 vertices are depicted. We can see that not all permutations of $n$ numbers can be obtained as $2$-shuffles, for example the permutation (432) is not a 2-shuffle and the recursive  tree on $4$ vertices where all nodes are children of 1 is thus not a 2-RT. This can be seen by considering that either 2 and 3 or 3 and 4 must be in the same pile and thus not all number can be in reversed order. Moreover we can see that the probability of obtaining the permutation (234) is very high compared to the others. The reason for this is that no matter how the deck is cut, just putting the piles on top of each other in the original order is one way of interleaving them. 

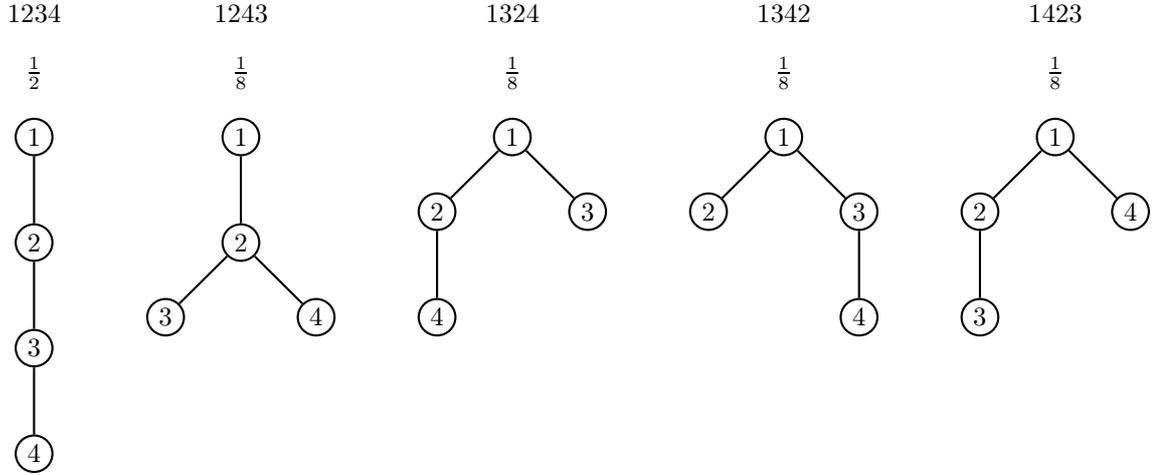
\begin{figure} \label{fig:BRT4vert}
  \tikzstyle{every node}=[circle,inner sep=2pt, draw]
\tikzset{ font={\fontsize{10pt}{12}\selectfont}}
\begin{tikzpicture}[scale=0.85, node distance=1.4cm,
  thick]
  
\node[draw=none,rectangle]  (p) {1234};
  \node [draw=none,rectangle] (t) [below=3mm of p] {$\frac{1}{2}$};
  
  \node (1) [below=3mm  of t]  {1};
  \node (2)  [below of=1] {2};
  \node (3) [below of =2] {3};
  \node (4) [below of =3] {4};

\path[every node/.style={font=\sffamily\small}]
    (1) edge node {} (2)
    (2) edge node {} (3)
    (3) edge node {} (4);

\begin{scope}[xshift=3.2cm]
    
\node[draw=none,rectangle]  (p) {1243};
  \node [draw=none,rectangle] (t) [below=3mm of p] {$\frac{1}{8}$};
  
  \node (1) [below=3mm  of t]  {1};
  \node (2)  [below of=1] {2};
  \node (3) [below left of =2] {3};
  \node (4) [below right of =2] {4};
  
\path[every node/.style={font=\sffamily\small}]
    (1) edge node {} (2)
    (2) edge node {} (3)
    (2) edge node {} (4);

\end{scope}

\begin{scope}[xshift=7.4cm]     
    
\node[draw=none,rectangle]  (p) {1324};
  \node [draw=none,rectangle] (t) [below=3mm of p] {$\frac{1}{8}$};
  
  \node (1) [below=3mm  of t]  {1};
  \node (2)  [below left of=1] {2};
  \node (3) [below right of =1] {3};
  \node (4) [below of =2] {4};
  
\path[every node/.style={font=\sffamily\small}]
    (1) edge node {} (2)
    (1) edge node {} (3)
    (2) edge node {} (4);

\end{scope}

\begin{scope}[xshift = 11.6cm]
\node[draw=none,rectangle]  (p) {1342};
  \node [draw=none,rectangle] (t) [below=3mm of p] {$\frac{1}{8}$};
  
  \node (1) [below=3mm  of t]  {1};
  \node (2)  [below left of=1] {2};
  \node (3) [below right of =1] {3};
  \node (4) [below of =3] {4};
  
\path[every node/.style={font=\sffamily\small}]
    (1) edge node {} (2)
    (1) edge node {} (3)
    (3) edge node {} (4);
    \end{scope}

    \begin{scope}[xshift = 15.8cm]
\node[draw=none,rectangle]  (p) {1423};
  \node [draw=none,rectangle] (t) [below=3mm of p] {$\frac{1}{8}$};
  
  \node (1) [below=3mm  of t]  {1};
  \node (2)  [below left of=1] {2};
  \node (4) [below right of =1] {4};
  \node (3) [below of =2] {3};
  
\path[every node/.style={font=\sffamily\small}]
    (1) edge node {} (2)
    (2) edge node {} (3)
    (1) edge node {} (4);
    \end{scope}
   
\end{tikzpicture}
\caption{All 2-recursive trees on $4$ vertices with corresponding permutation and probability.} \label{fig:BRT4}
\end{figure}

What makes this model very different from the weighted model and the other generalizations mentioned in the literature is that we do not have a way of picturing its dynamic growth. We mentioned in Chapter \ref{chapter:review} that for many kinds of increasing trees we can obtain a tree of size $n$ by attaching $n$ according to some rule to a node of a tree of size $n-1$. BRTs are neither defined in that way nor can such a rule be easily derived from the definition. In \cite{AltokIslak}, given the pile sizes, a way of constructing a BRT step by step is given but it cannot be considered as a dynamic growth because we cannot construct a BRT on $n$ vertices given a BRT on $n-1$ vertices using this method. We thus need to base our results about BRTs on the permutation representation. The description of $p$-biased riffle shuffles via their inverse will be crucial in many arguments.

\section{Number of Leaves} 
We now review some results from \cite{AltokIslak} on the distribution of the number of leaves in $p$-biased recursive trees. 
\begin{thm}
Let $\mathcal{T}_n^p$ be a BRT and $\mathcal{L}_n^p$ the number of leaves of $\mathcal{T}_n^p$. Then 
\begin{equation}\E[\mathcal{L}_n^p] = (n-2) \frac{1- \sum_{s=1}^{a} p_s^2}{2}+1\end{equation}
and
\begin{equation}
\begin{split}
\Var(\mathcal{L}_n^p) = &(n-2) \left ( \frac{1- \sum_{s=1}^{a} p_s^2}{2} - \left (\frac{1- \sum_{s=1}^{a} p_s^2}{2}\right )^2\right ) \\
&+ 2(n-3) \left (\sum_{1\leq s_1 <s_2 < s_3 \leq a} p_{s_1} p_{s_2} p_{s_3} - \left ( \frac{1-\sum_{s=1}^{a} p_s^2}{2} \right )^2 \right ).
\end{split}
\end{equation}
When $\vec{p}$ is the uniform distribution over $[a]$, these expressions simplify to
\begin{equation}\E[\mathcal{L}_n^{a}] = 1 + \frac{(n-2)(a-1)}{2a} \text{ and } \Var(\mathcal{L}_n^{a}) = (n-2) \frac{a^2 -1}{4a^2} - (n-3) \frac{a^2 -1}{6a^2}.
\end{equation}
In particular for fixed $n$ asymptotically we get the same values as for URTs
\begin{equation}
\lim_{a \to \infty} \E[\mathcal{L}_n^{a}] = \frac{n}{2} \text{ and } \lim_{a \to \infty} \Var(\mathcal{L}_n^{a}) = \frac{n}{12}.
\end{equation}
\end{thm}

\begin{proof}
As already mentioned in the review of uniform recursive trees, $i$ is a leaf if and only if it is greater than its right neighbour in the permutation representation. So given a tree constructed from a permutation $\gamma$ with $\gamma(j) = i$, $i$ is a leaf if and only if $\gamma(j)> \gamma(j+1)$. Moreover the last element of the permutation is always a leaf. This allows us to write the number of leaves of a $p$-biased recursive tree as a sum of indicator random variables. So we have
\begin{equation}\mathcal{L}_n^p =_d \sum_{i=2}^{n-1} \1(\gamma(i)>\gamma(i+1)) +1.
\end{equation}

In the description of the inverse of a riffle shuffle each card is assigned a digit from 1 to $a$ with probability $p_1, \dots, p_a$ and the cards are then rearranged according to the digits. This implies that when $i<j$, $i$ will come before $j$ in the inverse permutation if and only if card $i$ gets a digit that is smaller or equal to the digit of $j$. Since $\gamma(i) > \gamma(i+1)$ if and only if $i+1$ comes before $i$ in the inverse permutation,  $\gamma(i) > \gamma(i+1)$ is equivalent to $i$ getting a bigger digit than $i+1$, i.e. to $X_i > X_{i+1}$.

Let $X_i$, for $2\leq i \leq n$ be independent random variables with distribution $\vec{p}$, over $\{1, \dots, a\}$, then $\1(\gamma(i) > \gamma(i+1)) = \1(X_i > X_{i+1})$ and we can thus write
\begin{equation}\mathcal{L}_n^p =_d \sum_{i=2}^{n-1} \1(X_i > X_{i+1}) +1.
\end{equation}

Since all $X_i$ are identically distributed we have $\Pro(X_i>X_{i+1}) = \Pro(X_{i+1}>X_{i})$. Moreover $\Pro(X_i> X_{i+1}) + \Pro(X_i=X_{i+1}) + \Pro(X_i<X_{i+1}) =1$ and $\Pro(X_i=X_{i+1}) = \sum_{s=1}^{a} p_s^2$, thus
\begin{equation}\Pro(X_i>X_{i+1}) = \frac{1- \sum_{s=1}^{a} p_s^2}{2}.\end{equation}
So we get 
\begin{equation}\E[\mathcal{L}_n^p] = \sum_{i=2}^{n-1} \E[\1(X_i > X_{i+1})] +1 = (n-2) \frac{1- \sum_{s=1}^{a} p_s^2}{2}+1.\end{equation}
Let now $\ell_i = \1(X_i > X_{i+1})$, then $\mathcal{L}_n^p = \sum_{i=2}^{n-1} \ell_i +1$ and by writing $\Var(\mathcal{L}_n^p) = \sum_{i=2}^{n-1} \Var(\ell_i) +2 \sum_{2\leq i<j \leq n-1}\Cov(\ell_i,\ell_j)$ Altok and I\c{s}lak get
\begin{equation}
\begin{split}
\Var(\mathcal{L}_n^p) = &(n-2) \left ( \frac{1- \sum_{s=1}^{a} p_s^2}{2} - \left (\frac{1- \sum_{s=1}^{a} p_s^2}{2}\right )^2\right ) \\
&+ 2(n-3) \left (\sum_{1\leq s_1 <s_2 < s_3 \leq a} p_{s_1} p_{s_2} p_{s_3} - \left ( \frac{1-\sum_{s=1}^{a} p_s^2}{2} \right )^2 \right ).
\end{split}
\end{equation}

The results for the uniform case can be derived by substituting $\frac{1}{a}$ for all $p_s$ in these expressions.
\end{proof}
As in the uniform case, the distribution of the number of leaves tends to a normal distribution when $\vec{p}$ is non-degenerate, i.e. when there is no $i=1, \dots, a$ such that $p_i=1$.
\newpage
\begin{thm}[\cite{AltokIslak}] \label{thm:BRTleavesCLT}
Let $\mathcal{L}_n^p$ denote the number of leaves in a $p$-biased recursive tree and assume that $\vec{p}$ is non degenerate. Then for any $n\in \mathbb{N}$ there is a $C \in \mathbb{R}$ such that
\begin{equation}d_K\left (\frac{ \mathcal{L}_n^p - \E[\mathcal{L}_n^p]}{\sqrt{\Var(\mathcal{L}_n^p)}}, \mathcal{G} \right ) \leq \frac{C}{\sqrt{n}}.\end{equation}
\end{thm}

Moreover a result about the distance between the number of leaves of URTs and of $p$-biased recursive trees is given.
\begin{thm}[\cite{AltokIslak}]\label{thm:compare}
Let us denote the number of leaves in a URT, $p$-BRT and $a$-RT of size $n$ as $\mathcal{L}_n$, $\mathcal{L}_n^p$ and $\mathcal{L}_n^{a}$ respectively. Then we have the following:
\begin{enumerate}
\item For $n \geq 3$, \begin{equation}d_{TV}(\mathcal{L}_n,\mathcal{L}_n^p) \leq \binom{n-1}{2} \sum_{s=1}^{a} p_s^2.
\end{equation}
\item For $a \geq n\geq 3$  and $\vec{p}$ the uniform distribution this bound can be improved and we get 
\begin{equation}d_{TV}(\mathcal{L}_n,\mathcal{L}_n^{a}) \leq 1 - \frac{a!}{(a-n)!a^n}.\end{equation}
\item These two bounds imply that $\mathcal{L}_n^{a}$ converges in distribution to $\mathcal{L}_n$ as $a \to \infty$ and that $\mathcal{L}_n^{p}$ converges in distribution to $\mathcal{L}_n$ as $a \to \infty$ if $\vec{p}$ is non-degenerate. 
\item For a given $a$, among all distributions on $[a]$, the uniform distribution maximizes the expected number of leaves of $\mathcal{T}_n^{p}$. Thus if $\vec{p}$ is any distribution on $[a]$, then $\E[\mathcal{L}_n^{p} ] \leq \E[\mathcal{L}_n^{a}]$.
\end{enumerate}
\end{thm}
The proofs of Theorems \ref{thm:BRTleavesCLT} and \ref{thm:compare} can be found in \cite{AltokIslak}. 

We now move on to giving new results about BRTs, namely, on the number of branches, the number of nodes with at least $k$ descendants and the depth of node $n$.

\section{Number of Branches}

We will present two different methods to calculate the number of branches in a BRT. While we use the first one to get explicit expressions we still mention the second approach because it might be helpful for other problems.

\subsection{Anti-records}

As we have seen in Chapter \ref{chapter:review}, in URTs there are several possibilities to obtain results on the number of branches. For WRTs we used the attachment probabilities, which allowed us to write the number of branches as a sum of independent random variables. For BRTs we will use the observation that the number of branches of a recursive tree is equal to the number of anti-records in its permutation representation. First of all let us observe that in a riffle shuffle permutation, for $\gamma(i)$ to be an anti-record, it must be the first card of one of the $a$ piles. There can thus be at most $a$ anti-records. Moreover, as we consider riffle shuffles of $\{2, \dots, n\}$, there definitely is an anti-record at $\gamma(2)$. Using this approach we will prove the following theorem.

\begin{thm}
Let $\mathcal{B}_{n}^{p}$ denote the number of branches in a $p$-biased recursive tree $\mathcal{T}_{n}^{p}$. Then 
\begin{equation} \label{thm:BranchesExpBRT}
\begin{split}
\E[\mathcal{B}_{n}^{p}] = \sum_{s=1}^{a-1} \frac{p_s}{\sum_{\ell=1}^{s} p_{\ell}} \left ( 1- \left (\sum_{\ell=s+1}^{a} p_{\ell} \right)^{n-1} \right ) + p_a.
\end{split}
\end{equation}
Moreover we have the following asymptotic result
\begin{equation}
\E[\mathcal{B}_{n}^{p}]  \xrightarrow{n \to \infty} \sum_{s=1}^{a} \frac{p_s}{\sum_{\ell=1}^{s} p_{\ell}}.
\end{equation}
When $\vec{p}$ is the uniform distribution over $[a]$, letting $\mathcal{B}_n^{a}$ denote the number of branches in an $a$-RT $\mathcal{T}_n^{a}$,
\begin{equation}
\E[\mathcal{B}_{n}^{a}]= \sum_{s=1}^{a-1} \frac{1}{s} \left ( 1- \left ( 1- \frac{s}{a}\right )^{n-1} \right )+ \frac{1}{a} 
\end{equation}
and asymptotically
\begin{equation}
\E[\mathcal{B}_{n}^{a}] \xrightarrow{n \to \infty} H_a.
\end{equation}
Moreover, when a tends to infinity, the expectation tends to the expectation for URTs, namely
\begin{equation}
\E[\mathcal{B}_{n}^{a}] \xrightarrow{a \to \infty} H_{n-1}.
\end{equation}
\end{thm}

\begin{proof}
We will use the inverse formulation for biased riffle shuffles in order to study the number of anti-records. When $X_i < \min\{X_2, \dots, X_{i-1}\}$, we get at $i$ the first card from a pile with cards that are smaller than all the previous ones, hence there is an anti-record at $i$. Thus for $3\leq i < n$, and $\gamma$ a $p$-biased riffle shuffle, $\gamma(i)$ is an anti-record if and only if $X_i$ is strictly less than $X_2, \dots X_{i-1}$ where $X_j$, $j=2, \dots, n$, are independent random variables such that $X_j = s$, $s \in [a]$, with probability $p_s$. 
By independence of  the $X_i$'s we get for $3\leq i < n$,
\begin{equation}
\begin{split}
& \Pro (X_i < \min\{X_2, \dots, X_{i-1}\})  \\
&= \sum_{s=1}^{a} \Pro(X_i < \min\{X_2, \dots, X_{i-1}\} | X_i = s) \Pro(X_i=s) \\
&= \sum_{s=1}^{a} \Pro(s < \min\{X_2, \dots, X_{i-1}\}) \Pro(X_i=s) \\
&= \sum_{s=1}^{a-1} \left ( \sum_{\ell=s+1}^{a} p_{\ell}\right )^{i-2} p_s.
\end{split}
\end{equation}
We only sum from 1 to $a-1$ because for $i\geq3$, $\min\{X_2, \dots, X_{i-1}\} \leq a$ and thus $\Pro(a < \min\{X_2, \dots, X_{i-1}\}) = 0$.

Let $A_i =  \1(\gamma(i) \text{ is an anti-record})$, then $\mathcal{B}_n^p = \sum_{i=2}^{n} A_i$ and $\gamma(2)$ definitely is an anti-record. We need to assume $p_1 \neq 0$ from now on in order  to avoid division by 0. This is not an important restriction though, since given $\vec{p}$, we can define  $\vec{q}$ as $\vec{p}$ without the $p_i$'s that are zero. Then the distributions obtained from $\vec{p}$ and $\vec{q}$ are the same. Now,
\begin{equation}
\begin{split}
\E[\mathcal{B}_n^p] &= 1+  \sum_{i=3}^n  \E[A_i] \\
&= 1+  \sum_{i=3}^n  \Pro (X_i < \min\{X_2, \dots, X_{i-1}\}) \\
&= 1 + \sum_{i=3}^n \sum_{s=1}^{a-1} \left ( \sum_{\ell=s+1}^{a} p_{\ell}\right )^{i-2} p_s \\
&= 1 + \sum_{s=1}^{a-1} p_s \sum_{i=1}^{n-2} \left ( \sum_{\ell=s+1}^{a} p_{\ell} \right )^{i} \\
&= 1 + \sum_{s=1}^{a-1} p_s \left ( \frac{1- (\sum_{\ell=s+1}^{a} p_{\ell})^{n-1}}{1- \sum_{\ell=s+1}^{a} p_{\ell}} -1 \right ) \\
&= \sum_{s=1}^{a-1} \frac{p_s}{\sum_{\ell=1}^{s} p_{\ell}} \left ( 1- \left (\sum_{\ell=s+1}^{a} p_{\ell} \right)^{n-1} \right ) + p_a.
\end{split}
\end{equation}
Since  $p_1>0$ implies $\left (\sum_{\ell=s+1}^{a} p_{\ell} \right)^{n-1} \xrightarrow{n \to \infty} 0$ for all $s$, \begin{equation}\E[\mathcal{B}_n^p]  \xrightarrow{n \to \infty}{} \sum_{s=1}^{a-1} \frac{p_s}{\sum_{\ell=1}^{s} p_{\ell}} + p_a = \sum_{s=1}^{a} \frac{p_s}{\sum_{\ell=1}^{s} p_{\ell}}.\end{equation}
In particular, for $\vec{p}$ the uniform distribution over $[a]$ we have $p_i= \frac{1}{a}$, so we get
\begin{equation} \label{ExpectBranchA}
\begin{split}
\E[\mathcal{B}_n^a] &= \sum_{s=1}^{a-1} \frac{\frac{1}{a}}{\sum_{\ell=1}^{s} \frac{1}{a}} \left ( 1- \left (\sum_{\ell=s+1}^{a} \frac{1}{a} \right)^{n-1} \right ) + \frac{1}{a}\\
&= \sum_{s=1}^{a-1} \frac{1}{s} \left ( 1- \left ( \frac{a-s}{a}\right )^{n-1} \right )+ \frac{1}{a} \\
&= \sum_{s=1}^{a-1} \frac{1}{s} \left ( 1- \left ( 1- \frac{s}{a}\right )^{n-1} \right )+ \frac{1}{a} \\
\end{split}
\end{equation}
and 
\begin{equation}\E[\mathcal{B}_n^a]  \xrightarrow{n \to \infty}  \sum_{s=1}^{a} \frac{1}{s} = H_a.\end{equation}

Moreover, for a fixed $n$, if $a \to \infty$, for fixed $n$ the expectation tends to the expectation of the number of branches for URTs. To show this we will manipulate the second line of \ref{ExpectBranchA} and then use the following identity:
\begin{equation}
 \sum_{s=1}^{a} s^k = \frac{a^{k+1}}{k+1} + \mathcal{O}(a^k) \text{ for all } k \in \mathbb{N}, k \geq 0.
\end{equation}
We have
\begin{equation}
\begin{split}
\E[\mathcal{B}_n^a] &= \sum_{s=1}^{a-1} \frac{1}{s} \left ( 1- \left ( \frac{a-s}{a}\right )^{n-1} \right )+ \frac{1}{a} \\
&=  \sum_{s=1}^{a-1} \frac{1}{a} \frac{ 1- \left ( \frac{a-s}{a}\right )^{n-1} }{1-\frac{a-s}{a}}+ \frac{1}{a} \\
&= \sum_{s=1}^{a-1} \frac{1}{a} \sum_{\ell=0}^{n-2}\left ( \frac{a-s}{a} \right )^{\ell}+ \frac{1}{a} \\
&= \sum_{\ell=0}^{n-2} \frac{1}{a^{\ell+1}}  \sum_{s=1}^{a-1} s^{\ell}+ \frac{1}{a} \\
&= \sum_{\ell=0}^{n-2} \frac{1}{a^{\ell+1}}  \left [ \frac{a^{\ell+1}}{\ell+1} + \mathcal{O}(a^{\ell}) \right ]+ \frac{1}{a} \\
&= \sum_{\ell=0}^{n-2} \frac{1}{\ell+1} + \mathcal{O}\left( \frac{1}{a} \right)\\
& \xrightarrow{a \to \infty} H_{n-1}.
\end{split}
\end{equation}
\end{proof}
\newpage
\begin{thm} \label{thm:BRTVarBranches}
Let $\mathcal{B}_{n}^{p}$ denote the number of branches of a $p$-BRT. Then for $\vec{p}$ such that $p_1>0$,
\begin{equation}  \label{equ:VarBranches}
\begin{split}
&\Var(\mathcal{B}_n^p) = \\
& \sum_{s=1}^{a-1} \frac{p_s}{\sum_{\ell=1}^{s}p_{\ell}} \left ( 1- \left ( \sum_{\ell=s+1}^{a} p_{\ell}\right ) ^{n-1}\right )  - \sum_{s=1}^{a-1} p_s \\
& \hspace{1ex} - \sum_{s=1}^{a-1} p_s^2  \left (\sum_{\ell=s+1}^{a} p_{\ell} \right )^2  \frac{1- \left (\sum_{\ell=s+1}^{a} p_{\ell} \right )^{2(n-2)}}{1 -\left (\sum_{\ell=s+1}^{a} p_{\ell} \right )^{2}} \\
&\hspace{1ex} - 2  \sum_{s=2}^{a-1} \sum_{r=1}^{s-1} p_s p_r \left (\sum_{\ell=s+1}^{a} p_{\ell} \sum_{\ell=r+1}^{a} p_{\ell} \right ) \frac{1- \left ( \sum_{\ell=s+1}^{a} p_{\ell} \sum_{\ell=r+1}^{a} p_{\ell} \right )^{n-2}}{1 - \sum_{\ell=s+1}^{a} p_{\ell} \sum_{\ell=r+1}^{a} p_{\ell} }\\
& \hspace{1ex} +2  \sum_{s=2}^{a-1} \frac{p_s \sum_{\ell=s+1}^{a}p_{\ell}}{\sum_{\ell=1}^{s}p_{\ell}} \sum_{r=1}^{s-1} \frac{p_r}{\sum_{q = 1}^{r} p_q}  \left ( 1 -  \left ( \sum_{\ell=s+1}^{a}p_{\ell}\right )^{n-3}\right )\\
& \hspace{1ex} - 2 \sum_{s=2}^{a-1}p_s \sum_{\ell=s+1}^{a}p_{\ell} \sum_{r=1}^{s-1} \frac{p_r \sum_{q = r+1}^{a} p_q}{\sum_{q = 1}^{r} p_q}\frac{1}{\sum_{q = r+1}^{s} p_q} \\
&\hspace{7em} \cdot \left ( \left (\sum_{q = r+1}^{a} p_q\right)^{n-3} - \left( \sum_{\ell=s+1}^{a}p_{\ell}\right)^{n-3} \right ) \\
&\hspace{1ex} - 2 \sum_{s=1}^{a-1}p_s  \sum_{\ell=s+1}^{a}p_{\ell} \sum_{r=1}^{a-1}\frac{p_r}{\sum_{\ell=1}^{r}p_{\ell} }  \left ( \sum_{\ell=r+1}^{a}p_{\ell} \right )^2 \frac{1- \left ( \sum_{\ell=s+1}^{a}p_{\ell} \sum_{\ell=r+1}^{a}p_{\ell}\right )^{n-3}}{1- \sum_{\ell=s+1}^{a}p_{\ell} \sum_{\ell=r+1}^{a}p_{\ell}} \\
&\hspace{1ex} + 2 \sum_{s=1}^{a-1}\frac{p_s}{\sum_{\ell=1}^{s}p_{\ell}}  \sum_{\ell=s+1}^{a} p_{\ell} \sum_{r=1}^{a-1}\frac{p_r}{\sum_{\ell=1}^{r}p_{\ell} } \left ( \sum_{\ell=r+1}^{a}p_{\ell}\right )^{n-1} \left ( 1 - \left ( \sum_{\ell=s+1}^{a}p_{\ell}\right )^{n-3}\right )  \\
\end{split}
\end{equation}
and
\begin{equation}
\begin{split}
\Var&(\mathcal{B}_n^p) \xrightarrow{n \to \infty} \\
&  \sum_{s=1}^{a-1} \frac{p_s}{\sum_{\ell=1}^{s}p_{\ell}}  - \sum_{s=1}^{a-1} p_s \\
& \hspace{1ex} - \sum_{s=1}^{a-1} p_s^2  \left (\sum_{\ell=s+1}^{a} p_{\ell} \right )^2  \frac{1}{1 -\left (\sum_{\ell=s+1}^{a} p_{\ell} \right )^{2}} \\
&\hspace{1ex} - 2  \sum_{s=2}^{a-1} \sum_{r=1}^{s-1} p_s p_r \left (\sum_{\ell=s+1}^{a} p_{\ell} \sum_{\ell=r+1}^{a} p_{\ell} \right ) \frac{1}{1 - \sum_{\ell=s+1}^{a} p_{\ell} \sum_{\ell=r+1}^{a} p_{\ell} }\\
& \hspace{1ex} +2  \sum_{s=2}^{a-1} \frac{p_s \sum_{\ell=s+1}^{a}p_{\ell}}{\sum_{\ell=1}^{s}p_{\ell}} \sum_{r=1}^{s-1} \frac{p_r}{\sum_{q = 1}^{r} p_q}  \\
&\hspace{1ex} - 2 \sum_{s=1}^{a-1}p_s  \sum_{\ell=s+1}^{a}p_{\ell} \sum_{r=1}^{a-1}\frac{p_r}{\sum_{\ell=1}^{r}p_{\ell} }  \left ( \sum_{\ell=r+1}^{a}p_{\ell} \right )^2 \frac{1}{1- \sum_{\ell=s+1}^{a}p_{\ell} \sum_{\ell=r+1}^{a}p_{\ell}}. \\
\end{split}
\end{equation}

\end{thm}
\begin{proof}
See Appendix \ref{app:Variances}.
\end{proof}

\begin{thm} \label{thm:aRTBranchesVariance}
Let $\mathcal{B}_n^{a}$ denote the number of branches in an $a$-RT. Then
\begin{equation} \label{VarBranchesA}
\begin{split}
\Var(\mathcal{B}_n^a) =& \hspace{1em} \sum_{s=1}^{a-1} \frac{1}{s} \left ( 1- \left ( \frac{a-s}{a}\right ) ^{n-1}\right )  - \frac{a-1}{a} \\
& \hspace{1em} - \sum_{s=1}^{a-1} \frac{1}{a^2}  \left ( \frac{a-s}{a} \right )^2  \frac{1- \left (\frac{a-s}{a} \right )^{2(n-2)}}{1 -\left (\frac{a-s}{a} \right )^{2}} \\
&\hspace{1em} - 2  \sum_{s=2}^{a-1} \sum_{r=1}^{s-1} \frac{1}{a^2} \frac{a-s}{a} \frac{a-r}{a} \frac{1- \left ( \frac{a-s}{a} \frac{a-r}{a} \right )^{n-2}}{1 - \frac{a-s}{a}  \frac{a-r}{a} }\\
& \hspace{1em} +2  \sum_{s=2}^{a-1}\frac{a-s}{as} \sum_{r=1}^{s-1} \frac{1}{r}  \left ( 1 -  \left (\frac{a-s}{a}\right )^{n-3}\right )\\
& \hspace{1em} - 2 \sum_{s=2}^{a-1} \frac{a-s}{a^2} \sum_{r=1}^{s-1} \frac{a-r}{r}\frac{1}{ s-r} \left ( \left (\frac{a-r}{a}\right)^{n-3} - \left(\frac{a-s}{a}\right)^{n-3} \right ) \\
&\hspace{1em} - 2 \sum_{s=1}^{a-1}\frac{a-s}{a^2} \sum_{r=1}^{a-1}\frac{1}{r}  \left (\frac{a-r}{a} \right )^2 \frac{1- \left ( \frac{a-s}{a} \frac{a-r}{a}\right )^{n-3}}{1- \frac{a-s}{a} \frac{a-r}{a}} \\
&\hspace{1em} + 2 \sum_{s=1}^{a-1}  \frac{a-s}{sa} \sum_{r=1}^{a-1}\frac{1}{r} \left (\frac{a-r}{a}\right )^{n-1} \left ( 1 - \left (\frac{a-s}{a}\right )^{n-3}\right )  \\
\end{split}
\end{equation}
and 
\begin{equation} \label{VarBranchesANinfty}
\begin{split}
\Var(\mathcal{B}_n^{a}) \xrightarrow {n \to \infty} &  H_a  - \frac{a-1}{a}  + \frac{2}{a} \sum_{s=2}^{a-1}\frac{a-s}{s} \sum_{r=1}^{s-1} \frac{1}{r}\\
& -  \frac{1}{a^2} \sum_{s=1}^{a-1}  \frac{s^2}{a^2 - s^{2}} \\
& - \frac{2}{a^2}   \sum_{s=1}^{a-2} \sum_{r=1}^{a-s+1} \frac{sr}{a^2 - sr }\\
& - \frac{2}{a^2} \sum_{s=1}^{a-1} \sum_{r=1}^{a-1}\frac{1}{r} \frac{s (a-r)^2}{a^2- s (a-r)}. \\
\end{split}
\end{equation}
Moreover, for fixed $n$,  when $a$ increases the variance approaches the variance from the uniform case:
\begin{equation}\Var(\mathcal{B}_n^{a}) \xrightarrow {a \to \infty} H_{n-1}- H_{n-1}^{(2)}. \end{equation}
\end{thm}

\begin{proof}
See Appendix \ref{app:Variances}.
\end{proof}

\begin{rem}
There are various results in the literature on central limit theorems for the number of records in random words, for example \cite{Gouet05, Bai98}. We hope to adapt these for obtaining asymptotic results on the number of branches in a subsequent work. 
\end{rem}

\subsection{Sequential Shuffling}

We now present another possibility to calculate the expectation and variance of the number of branches of a BRT.
As we described in Section \ref{sec:BRTDefinition}, in a riffle shuffle, after choosing the pile sizes, we might first shuffle pile 1 and 2 such that if there are $A_1$ cards remaining in pile 1 and $A_2$ cards in pile 2, the probability that the next card is from pile i is $\frac{A_i}{A_1+A_2}$ for $i=1,2$. More generally, for $i=2, \dots, n$, after shuffling the first $i-1$ piles, we shuffle the obtained shuffled pile with the $i$-th pile such that  if $A_{i}$ are the cards remaining in the $i$-th pile and $B_{i-1}$ the cards remaining in the already shuffled pile, the probability that the next card comes from the $i$-th pile is $\frac{A_{i}}{B_{i-1} + A_{i}}$.

Only the first card of each pile can be the start of a branch, since all other cards will definitely have a smaller card of the same pile to their left. Hence we only need to consider where these are. Let us call the first cards of each pile $F_1, F_2, \dots, F_a$ and the piles $S_1, \dots, S_a$. If pile $i$ is empty $F_i$ refers to no card. We know that all cards of $S_1, \dots, S_{i}$ are smaller than $F_{i+1}$. Hence, if $S_{i+1}$ is non-empty, $F_{i+1}$ attaches to 1 if and only if it comes before any card from piles with lower index. Since each time we riffle a new pile this only depends on the first step, the remaining pile sizes actually are the initial pile sizes in our case.

Set $I_i= \1 (F_i \text{ attaches to } 1)$, then we get
\begin{equation}
\Pro (I_i = 1) = 
\begin{cases}
0, &\text{ if $b_1, \dots, b_i=0$} \\
\frac{b_i}{b_1+ \cdots + b_i}, &\text{otherwise}.
\end{cases}
\end{equation}
Or alternatively \begin{equation}\E[I_i] = \Pro(I_i = 1) = \frac{b_i}{\max\{1, b_1 + \dots + b_i\}}.\end{equation}

Now  $\mathcal{B}_n^p= \sum_{i=1}^{a} I_i$, and we can write, conditioned on the pile sizes,
\begin{equation}
\E[\mathcal{B}_n^p | b_1, \dots, b_a] = \sum_{i=1}^{a} \E[I_i | b_1, \dots, b_a] = \sum_{i=1}^{a} \frac{b_i}{\max\{1, b_1 + \dots + b_i\}}.
\end{equation}

Since the pile sizes are chosen according to the multinomial distribution this gives 
\begin{equation}
\E[\mathcal{B}_n^p] = \sum_{\substack{b_1, \dots, b_a\\ b_1 + \cdots + b_a=n}} \binom{n}{b_1, \dots, b_a} \prod_{i = 1}^{a} p_i^{b_i} \sum_{j=1}^{a} \frac{b_i}{\max\{1,b_1 + \cdots + b_i\}}
\end{equation}
where we used $\E[X] = \E [\E[X |Y]]$ with $X=\mathcal{B}_n^p$ and $Y= (b_1, \dots, b_a)$.
\newpage
We can give this a nicer form by considering for all $j=1, \dots, a$ the expression $\sum_{\substack{b_1, \dots, b_a \\ b_1 + \cdots + b_a=n}} \binom{n}{b_1, \dots, b_a} \prod_{i = 1}^{a} p_i^{b_i} \frac{b_j}{\max\{1,b_1 + \cdots + b_j\}}$ separately.

For $j=1$ this gives 
\begin{equation}\sum_{\substack{b_1, \dots, b_a \\ b_1 + \cdots + b_a=n}} \binom{n}{b_1, \dots, b_a} \prod_{i = 1}^{a} p_i^{b_i} \frac{b_1}{\max\{1,b_1\}} = 1-\Pro(b_1=0) = 1 - \left (\sum_{i=2}^{a} p_i \right)^n.\end{equation}

For $j=2$, we get 
\begin{equation} \label{eq:seqBran}
\begin{split}
&\sum_{\substack{b_1, \dots, b_a\\ b_1 + \cdots + b_a=n}} \binom{n}{b_1, \dots, b_a} \prod_{i = 1}^{a} p_i^{b_i} \frac{b_2}{\max\{1,b_1 + b_2\}} \\
&\hspace{1ex}= \sum_{k=0}^{n} \sum_{\substack{b_3, \dots, b_a\\ b_3 + \cdots + b_a=k}} \sum_{b_2=0}^{n-k} \binom{n}{b_1, \dots, b_a} \prod_{i = 1}^{a} p_i^{b_i} \frac{b_2}{\max\{1,n-k\}}\\
&\hspace{1ex}= \sum_{k=0}^{n-1} \sum_{\substack{b_3, \dots, b_a\\ b_3 + \cdots + b_a=k}} \frac{n!}{b_3! \dots b_a!} \prod_{i = 3}^{a} p_i^{b_i}  \sum_{b_2=1}^{n-k} \frac{b_2 p_1^{n-k-b_2} p_2^{b_2}}{\max\{1,n-k\}(n-k-b_2)! b_2! } \\
&\hspace{1ex}= \sum_{k=0}^{n-1} \sum_{\substack{b_3, \dots, b_a\\ b_3 + \cdots + b_a=k}} \frac{n \cdots (n-k+1)}{b_3! \dots b_a!} \prod_{i = 3}^{a} p_i^{b_i}  \sum_{b_2=1}^{n-k} \frac{(n-k)! b_2}{(n-k-b_2)!b_2!} \frac{ p_1^{n-k-b_2}p_2^{b_2} }{\max\{1,n-k\}}. \\
\end{split}
\end{equation}

Now we have for $1 \leq k <n$,
\begin{equation}
\begin{split}
\sum_{b_2=1}^{n-k} & \frac{(n-k)! b_2}{(n-k-b_2)!b_2!} \frac{ p_1^{n-k-b_2}p_2^{b_2} }{\max\{1,n-k\}} \\
 &\hspace{1ex} =  \sum_{b_2=1}^{n-k} \frac{(n-k)!}{(n-k-b_2)!b_2!} \frac{b_2}{n-k}p_1^{n-k-b_2}p_2^{b_2} \\
&\hspace{1ex}=  \sum_{b_2=1}^{n-k} \frac{(n-k-1) \cdots (n-k-b_2+1)}{(b_2-1)!}p_1^{n-k-b_2}p_2^{b_2} \\
&\hspace{1ex}= p_2  \sum_{b_2=0}^{n-k-1} \frac{(n-k-1) \cdots (n-k-b_2)}{b_2!}p_1^{n-k-1-b_2}p_2^{b_2} \\
& =p_2  (p_1 + p_2)^{n-k-1}.
\end{split}
\end{equation}
By inserting this into \ref{eq:seqBran} we thus get
\begin{equation}
\begin{split}
 p_2 & \sum_{k=0}^{n-1} \sum_{\substack{b_3, \dots, b_a\\ b_3 + \cdots + b_a=k}} \frac{n \cdots (n-k+1)}{b_3! \cdots b_a!} \prod_{i = 3}^{a} p_i^{b_i} (p_1 + p_2)^{n-k-1} \\
&\hspace{1ex}= \frac{p_2}{p_1+p_2} \sum_{k=0}^{n-1} \sum_{\substack{b_3, \dots, b_a\\ b_3 + \cdots + b_a=k}} \frac{n \cdots (n-k+1)}{b_3! \cdots b_a!} \prod_{i = 3}^{a} p_i^{b_i} (p_1 + p_2)^{n-k} \\
&\hspace{1ex}= \frac{p_2}{p_1+p_2} \left [ \sum_{\substack{B,b_3, \dots, b_a\\ B+b_3 + \cdots + b_a=n}} \frac{n!}{B!b_3! \cdots b_a!} \prod_{i = 3}^{a} p_i^{b_i} (p_1 + p_2)^B - \Pro(b_1 = b_2 = 0) \right ] \\
&\hspace{1ex}= \frac{p_2}{p_1+p_2} \left [(p_1+p_2 + \dots + p_a)^n - (1-(p_1+p_2))^n\right ]\\
&\hspace{1ex}= \frac{p_2}{p_1+p_2}\left (1-\left( \sum_{i=3}^{n} p_i \right )^n \right).\\
\end{split}
\end{equation}
The cases $j=3, \dots, a$ can be calculated similarly, thus we get in total
\begin{equation}\E[\mathcal{B}_n^p] = \sum_{i=1}^{a} \frac{p_i}{p_1 + \dots + p_i} \left (1-\left (\sum_{j=i+1}^{a} p_i \right )^{n} \right )
\end{equation}
which is the same expression as in Theorem \ref{thm:BranchesExpBRT}.

For the variance we will use that $\Var(X) = \E[\Var(X |Y)] + \Var(\E[X | Y])$. We have
\begin{equation}
\begin{split}
\Var&(\mathcal{B}_n^p | b_1, \dots, b_a) \\
&= \E \left [ \left (\sum_{i=1}^{a} I_i \right )^2 \bigg | b_1, \dots, b_a \right] - \E \left [\sum_{i=1}^{a} I_i  \bigg | b_1, \dots, b_a \right ]^2 \\
&= \sum_{i=1}^{a}\E[I_i | b_1, \dots, b_a] + 2 \sum_{i=1}^{a} \sum_{j=i}^{a} \E[I_i I_j | b_1, \dots, b_a ] \\
& \hspace{2em} - \sum_{i=1}^{a}\E[I_i| b_1, \dots, b_a]^2 - 2 \sum_{i=1}^{a} \sum_{j=i}^{a}\E[I_i| b_1, \dots, b_a]\E[I_j| b_1, \dots, b_a] \\
& = \sum_{i=1}^{a} \E[I_i| b_1, \dots, b_a] - \E[I_i| b_1, \dots, b_a]^2.
\end{split}
\end{equation}
Here we were able to get rid of the mixed terms because $I_i$ and $I_j$ are independent for $ i<j$, given the pile sizes $b_1, \dots, b_a$: $F_j$ attaches to 1 if and only if it is the first card in the shuffle of $S_1, \dots, S_{j-1}, S_j$, where $S_1, \dots, S_{j-1}$ are already mixed together. This only depends on the sum of the first $j$ pile sizes, i.e. $S_1+ \dots + S_{j-1}$, and $S_j$, not on the order of the previously shuffled cards,  particularly not on the position of $F_i$. Hence 
\begin{equation}
\Var(\mathcal{B}_n^p | b_1, \dots, b_a) = \sum_{i=1}^{a} \frac{b_i}{\max\{1,b_1 + \dots + b_i\}} - \frac{b_i^2}{(\max\{1, b_1 + \dots + b_i\})^2}.
\end{equation}
Also \begin{equation}
\Var( \E[\mathcal{B}_n^p | b_1, \dots, b_a]) = \E [ \E [\mathcal{B}_n^p | b_1, \dots, b_a]^2] - \E[\E[\mathcal{B}_n^p | b_1, \dots, b_a]]^2
\end{equation} 
and
\begin{equation}
\E[\E[\mathcal{B}_n^p | b_1, \dots, b_a]^2] = \sum_{b_1, \dots, b_a} \binom{n}{b_1, \dots, b_i} \prod_{i=1}^{a} p_i^{b_i} \left (\sum_{i=1}^{a} \frac{b_i}{\max\{1, b_1 + \dots + b_i\}}\right )^2
\end{equation}
and
\begin{equation}\E[\E[\mathcal{B}_n^p | b_1, \dots, b_a]]^2 = \left (\sum_{b_1, \dots, b_i} \binom{n}{b_1, \dots, b_a} \prod_{i=1}^{a} p_i^{b_i} \sum_{i=1}^{a} \frac{b_i}{\max\{1, b_1 + \dots + b_i\}} \right )^2.
\end{equation}
So we get in total
\begin{equation} 
\begin{split}
\Var\left (\mathcal{B}_n^p\right ) =& \sum_{b_1, \dots, b_a} \binom{n}{b_1, \dots, b_a} \prod_{i=1}^{a} p_i^{b_i}\\
&  \cdot \Bigg[ \left ( \sum_{i=1}^{a} \frac{b_i}{\max\{1, b_1 + \dots + b_i\}} - \frac{b_i^2}{(\max\{1, b_1 + \dots + b_i\})^2}\right )\\
&+ \left ( \sum_{i=1}^{a} \frac{b_i}{\max\{1, b_1 + \dots + b_i\}} \right )^2  \\
& - \binom{n}{b_1, \dots, b_a} \prod_{i=1}^{a} p_i^{b_i} \left ( \sum_{i=1}^{a} \frac{b_i}{\max\{1, b_1 + \dots + b_i\}} \right )^2  \Bigg ].
\end{split}
\end{equation}
\newpage
\begin{rem}
We argued at the beginning that the number of branches is limited by the number of piles we split the deck into, i.e. $a$. In fact that argument  also works for the degree of any node: if two nodes are children of $i$, the smaller one must come later in the permutation, thus the nodes attached to $i$ form a decreasing sequence. But in a $p$-biased riffle shuffle permutation any decreasing sequence must consist of cards from different piles since cards from the same pile remain in the same order. Since there are $a$ piles, this implies that the degree of any node is limited by $a$. Since $a$-ary trees also have this property, it would be interesting to see whether $a$-ary recursive trees also share other properties with trees constructed from $a$-shuffles.
\end{rem}

\section{Number of Nodes with at least $k$ Descendants}

In a BRT the number of nodes with at least $k$ descendants can also be calculated using the construction of an inverse riffle shuffle. In a recursive tree $\mathcal{T}$, node $i$ has at least $k$ descendants if in the permutation representation of the tree at least the $k$ entries following $i$ are bigger than $i$. This is the case if, given $\gamma(j)=i$, we have $\gamma(j+1), \dots, \gamma(j+k) > i$. This is equivalent to $X_j \leq \{ X_{j+1}, \dots, X_{j+k} \}$ in the inverse riffle shuffle construction. We thus get a node with a least $k$ descendants for every $X_j$ satisfying the above, more precisely, if $X_j \leq \{X_{j+1} \dots X_{j+k}\}$, node $i=\gamma(j)$ is a node with at least $k$ descendants.

\begin{ex}
Let our deck consist of 8 cards. Assume that we cut it into 3 piles. We first construct the inverse riffle shuffle by assigning digits from 1 to 3 to every card as can be seen in Figure \ref{fig:RiffleShuffleInverseConstruction}. 
\begin{figure}[H]
\begin{center}
\begin{tabular}{ c c c c c c c }
$ X_2$ & $X_3$ & $X_4$ & $X_5$ & $X_6$ & $X_7$ & $X_8  $\\ 
1 & 2 & 3 & 1 & 3 & 1  & 1 
\end{tabular}
\end{center}
\caption{Example of the construction of an inverse riffle shuffle permutation.} \label{fig:RiffleShuffleInverseConstruction}
\end{figure}

\begin{figure}
\tikzstyle{every node}=[circle, draw]

\begin{center}
\begin{tikzpicture}[node distance=1.6cm,
  thick]
  \node (1) {1};
  \node (2) [below of=1] {2};
  \node (3) [below left of=2] {3};
  \node (4) [below left of=3] {4};
  \node (5) [below  of = 4] {5};
  \node (6) [below right of=2] {6};
  \node (7) [below  of=6] {7};
  \node (8) [below right of=3] {8};
 
\path[every node/.style={font=\sffamily\small}]
    (1) edge node {} (2)
    (2) edge node {} (3)
    (2) edge node {} (6)
    (3) edge node {} (4)
    (3) edge node {} (8)
    (4) edge node {} (5)
    (6) edge node {} (7);

\end{tikzpicture}
\end{center}
\caption{The biased recursive tree corresponding to $\gamma= 2673845$.} \label{RifSufExDes}
\end{figure}
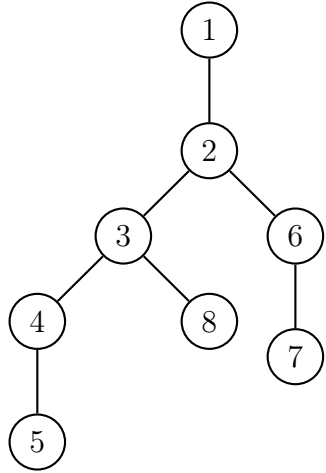

This gives the inverse permutation $\gamma^{-1} = 2578346$ and thus $\gamma=2673845$ with the corresponding recursive tree in Figure \ref{RifSufExDes}. As can be seen the tree has two nodes with at least two descendants, 2 and 3. This corresponds to what we can derive from the permutation: $\pi(5) = 3$ and we have $X_5 \leq X_6, X_7$. Also $\pi(2)= 2$ and we have $X_2 \leq X_3, X_4$.
\end{ex}

We will now use this observation to calculate the expectation and variance of the number of nodes with at least $k$ descendants, and then will prove a central limit theorem by using the fact that the dependence is local.

\begin{thm} \label{thm:BRTExpkDesce}
Let $k \in \mathbb{N}$ and let $Y_{\geq k, n}^{p}$ denote the number of nodes with at least $k$ descendants of $\mathcal{T}_n^p$, a $p$-biased recursive tree of size $n$. Then
\begin{equation}
\E\left [Y_{\geq k, n}^{p}\right ] = (n-k-1)\sum_{s=1}^{a} p_s \left (\sum_{r=s}^{a} p_r  \right )^k +1.
\end{equation}
Moreover if $\vec{p}$ is the uniform distribution over $[a]$, 
\begin{equation}
\E\left [Y_{\geq k, n}^{a}\right ] = (n-k-1)  \frac{1}{a^{k+1}} \sum_{s=1}^{a} s^k +1 
\end{equation}
and in particular, as $a \to \infty$, we get 
\begin{equation}
\E\left [Y_{\geq k, n}^{a}\right ] \xrightarrow{a \to \infty} \frac{n}{k+1}.
\end{equation}
\end{thm}

\begin{proof}
Let $C_i^k = \1(X_i \leq \{X_{i+1}, \dots, X_{i+k}\})$. Then
\begin{equation}
Y_{\geq k, n}^{p} =_d \sum_{i=1}^{n-k} C_i^k = \sum_{i=2}^{n-k} C_i^k +1.
\end{equation}
Also for $2 \leq i \leq n-k$, 
\begin{equation}
\begin{split}
\E\left [C_i^k\right ] &=\Pro(X_i \leq \{X_{i+1}, \dots, X_{i+k}\}) \\
&= \sum_{s=1}^{a} \Pro(X_i \leq \{X_{i+1}, \dots, X_{i+k}\} | X_i = s) \Pro(X_i = s)\\
&= \sum_{s=1}^{a} \Pro(s \leq \{X_{i+1}, \dots, X_{i+k}\}) \Pro(X_i = s)\\
&= \sum_{s=1}^{a} p_s \left (\sum_{r=s}^{a} p_r  \right )^k. \\
\end{split}
\end{equation} 
Thus 
\begin{equation}
\begin{split}
\E\left [Y_{\geq k, n}^{p}\right ] &= \sum_{i=1}^{n-k} \E\left [C_i^k\right ]\\
&= \sum_{i=2}^{n-k} \sum_{s=1}^{a} p_s \left (\sum_{r=s}^{a} p_r  \right )^k +1 \\
&= (n-k-1)\sum_{s=1}^{a} p_s \left (\sum_{r=s}^{a} p_r  \right )^k +1.
\end{split}
\end{equation} 
In particular, if $p_s = \frac{1}{a}$ for all $s$, we get 
\begin{equation}
\begin{split}
\E\left [Y_{\geq k, n}^{a}\right ] &= (n-k-1) \sum_{s=1}^{a} \frac{1}{a} \left ( \sum_{r=s}^{a} \frac{1}{a} \right )^k +1 \\
&= (n-k-1)  \frac{1}{a^{k+1}} \sum_{s=1}^{a} \left ( \sum_{r=s}^{a} 1 \right )^k +1 \\
&= (n-k-1)  \frac{1}{a^{k+1}} \sum_{s=1}^{a} s^k +1. \\
\end{split}
\end{equation}
This gives asymptotically \begin{equation}
\begin{split}
\lim_{a \to \infty} \E\left [Y_{\geq k, n}^{a}\right ] &= \lim_{a \to \infty} (n-k-1)  \frac{1}{a^{k+1}} \sum_{s=1}^{a} s^k +1 \\
 &= \lim_{a \to \infty} (n-k-1)  \frac{1}{a^{k+1}} \left [ \frac{a^{k+1}}{k+1} + \mathcal{O}(a^k) \right ] +1 \\
 &=  \frac{n}{k+1}.
\end{split}
\end{equation}
\end{proof}
\begin{rem}
We know from Theorem \ref{thm:URTkDes} that in a URT, for $A_{k,n}$, the number of nodes with exactly $k$ descendants, $\frac{A_{k,n}}{n} \xrightarrow{n \to \infty} \frac{1}{(k+1)(k+2)}$ in probability holds. This is consistent with the expectation of the number of node with at least $k$ descendants we found for $a \to \infty$. 
\end{rem}

\begin{rem}
For the case $k=1$ and $\vec{p}$ the uniform distribution over $a$ this gives the expected number of internal nodes
\begin{equation}\E\left [Y_{\geq1, n}^{a}\right ] = (n-2)  \frac{1}{a^{2}} \sum_{s=1}^{a} s +1 = (n-2) \frac{1}{a^2} \frac{a(a+1)}{2} +1 = \frac{n-2}{2}\frac{a+1}{a} +1.
\end{equation} 
This also follows from the expected number of leaves in an $a$-RT.
\end{rem}

\begin{thm} \label{thm:kDesVarGeneral}
Let $k \in \mathbb{N}$ and let $Y_{\geq k, n}^{p}$ denote the number of nodes with at least $k$ descendants in a $p$-BRT. Then
\begin{equation}
\begin{split}
\Var \left (Y_{\geq k, n}^{p}\right ) =& \sum_{s=1}^{a} p_s \left (\sum_{r=s}^{a} p_r  \right )^k \left [(n-k-1)+ p_1\left(2nk -3k(k+1)\right) \right ] \\
&+2 \sum_{s=2}^{a} p_s  \frac{1}{\sum_{u=1}^{s-1} p_u} \sum_{r=s}^{a} p_r \left ( \sum_{t=r}^{a} p_t \right )^k \\
& \hspace{1em} \cdot \left [n-k-1  -(n-2k-1)\left (\sum_{u=s}^{a} p_u\right )^k - \frac{1- \left ( \sum_{u=s}^{a} p_u\right )^{k}}{1-\sum_{u=s}^{a} p_u}  \right ]\\
& -  \left ( \sum_{s=1}^{a} p_s \left (\sum_{r=s}^{a} p_r  \right )^k \right )^2 \left [ n(2k+1) - (3k+1)(k+1)\right ].
\end{split}
\end{equation}
We moreover have
\begin{equation}
\begin{split}
\lim_{n \to \infty} \frac{\Var(Y_{\geq k, n}^{p})}{n} =&  \sum_{s=1}^{a} p_s \left (\sum_{r=s}^{a} p_r  \right )^k \left ( 2kp_1 + 1 -  (2k+1) \sum_{s=1}^{a} p_s \left (\sum_{r=s}^{a} p_r  \right )^k \right )  \\
&+2 \sum_{s=2}^{a} \frac{p_s}{\sum_{u=1}^{s-1} p_u} \sum_{r=s}^{a} p_r \left ( \sum_{t=r}^{a} p_t \right )^k \left [1 -\left (\sum_{u=s}^{a} p_u\right )^k \right ].
\end{split}
\end{equation}
\end{thm}
\begin{proof}
See Appendix \ref{app:Variances}.
\end{proof}

\begin{cor} \label{cor:aRTkDescendantsVar}
If we choose the uniform distribution over $[a]$ we get
\begin{equation}  \label{eq:VarkDesART}
\begin{split}
\Var \left (Y_{\geq k , n}^{a}\right )= &\frac{1}{a^{k+1}}\sum_{s=1}^{a} s^k\left [(n-k-1)+ \frac{1}{a}\left(2nk -3k(k+1)\right) \right ] \\
&+2 \frac{1}{a^{k+1}} \sum_{s=2}^{a} \frac{1}{s-1} \sum_{r=1}^{a-s+1} r^k \\
& \hspace{1em} \cdot \left [n-k-1  -(n-2k-1)\left ( \frac{a-s+1}{a}\right )^k - \frac{1- \left ( \frac{a-s+1}{a}\right )^{k}}{1-\frac{a-s+1}{a}}  \right ]\\
& -  \left (\frac{1}{a^{k+1}}\sum_{s=1}^{a} s^k \right )^2 \left [ n(2k+1) - (3k+1)(k+1)\right ]
\end{split}
\end{equation}
and for fixed $a$ we moreover have
\begin{equation}
\begin{split}
\lim_{n \to \infty} \frac{\Var\left (Y_{\geq k , n}^{a}\right )}{n} = &
\frac{1}{a^{k+1}}\sum_{s=1}^{a} s^k\left [1 + \frac{2k}{a} - \frac{2k+1}{a^{k+1}}\sum_{s=1}^{a} s^k \right ] \\
& + \frac{2}{a^{k+1}} \sum_{s=1}^{a-1} \frac{1}{s} \left [1  - \left ( \frac{a-s}{a}\right )^k  \right ] \sum_{r=1}^{a-s+1} r^k .
\end{split}
\end{equation}
For  fixed $n$ we moreover get asymptotically
\begin{equation}
\begin{split}
\Var \left (Y_{\geq k , n}^{a} \right ) \xrightarrow{a \to \infty}&  \frac{n-k-1}{k+1}-  2 \frac{(n-k-1)}{k+1} H_{k+1}  + 2 \frac{n-2k-1}{k+1} H_{2k+1} \\
& +  \frac{2}{k+1} \sum_{\ell=0}^{k-1} H_{k+\ell+1} - \frac{ n(2k+1) - (3k+1)(k+1)}{(k+1)^2} .
\end{split}
\end{equation}
If we then take the limit as $n \to \infty$ we get
\begin{equation}
\lim_{n\to \infty} \lim_{a \to \infty} \frac{\Var \left (Y_{\geq k , n}^{a} \right )}{n} =   \frac{1}{k+1} + \frac{2H_{2k+1}-2H_{k+1}}{k+1} - \frac{2k+1}{(k+1)^2}.
\end{equation}
\end{cor}
\begin{proof}
See Appendix \ref{app:Variances}.
\end{proof}
\begin{thm} \label{thm:BRTkdescendantsCLT}
Let $k \in \mathbb{N}$ and let $Y_{\geq k, n}^{p}$ denote the number of nodes with at least $k$ descendants in a $p$-BRT. Then 
\begin{equation}
d_W\left  (Y_{\geq k, n}^{p}, \mathcal{G} \right )  \leq  \frac{2k+1}{\sqrt{\Var\left (Y_{\geq k, n}^{p}\right ) }} \left ( (2k+1) + \frac{\sqrt{28}(2k+1)^{\frac{1}{2}}}{\sqrt{\pi}} \right ).
\end{equation}
\end{thm}
\begin{proof}
We will again use Theorem \ref{thm:normalrate}. For this we set $Y_i := C_i^{k}-\E[C_i^{k}]$ for $i=1, \dots, n-k$. The dependency neighbourhoods are $N_i = \{C_j^k: i-k \leq j \leq i+k\}$ for $i\leq n-2k$ and $N_i = \{C_j^k: i-k \leq j \leq n-k\}$ for $i= n-2k+1, \dots, n-k$. Hence we have $D= \max_{1\leq i \leq n} \{|N_i|\} = 2k+1$. Let $\sigma^2 = \Var \left ( \sum_{i=1}^n C_i^{k} \right )$ and define \begin{equation}W:= \sum_{i=1}^n \frac{C_i^{k}-\E[C_i^{k}]}{\sqrt{\Var(\sum_{i=1}^n C_i^{k})}}.\end{equation}
We now need to estimate $\sum_{i=1}^{n} \E[|Y_i|^3]$ and $\sum_{i=1}^n \E[Y_i^4]$.
Since the $Y_i$'s can take values $1-p_i$ or $p_i$, where $p_i = \E[C_i^{k}]$, we have as in the proof of Theorem \ref{thm:BoundBranchesWRT} that 
\begin{equation}
\begin{split}
\E[|Y_i|^3] &= |1-p_i|^3 p_i + |-p_i|^3(1-p_i) \\
&= p_i(1-p_i) ((1-p_i)^2 + p_i^2) \\
&= p_i(1-p_i) (1-2p_i(1- p_i)) \\
& \leq \Var[C_i^{k}]
\end{split}
\end{equation}
since for $0<a<1$, we have $0<a(1-a)<\frac{1}{4}$ and thus $1>1-2p_i(1- p_i)>\frac{1}{2}$.

Similarly 
\begin{equation}
\begin{split}
\E[|Y_i|^3] &= (1-p_i)^4 p_i + p_i^4(1-p_i) \\
&= p_i(1-p_i) ((1-p_i)^3 + p_i^3) \\
&= p_i(1-p_i) (1-3(pi(1-p_i)) \\
&\leq \Var(C_i^{k})
\end{split}
\end{equation}
since for $0<a<1$, we have $1>1-3p_i(1- p_i) > \frac{1}{4}$.

Thus we get 
\begin{equation}
\begin{split}
d_W(W,\mathcal{G}) &\leq \frac{(2k+1)^2}{\sigma^3} \sum_{i=1}^{n-k} \Var(C_i^k)+ \frac{\sqrt{28}(2k+1)^{\frac{3}{2}}}{\sqrt{\pi}\sigma^2} \sqrt{\sum_{i=1}^n \Var(C_i^k)} \\
& = \frac{2k+1}{\sigma} \left ( (2k+1) + \frac{\sqrt{28}(2k+1)^{\frac{1}{2}}}{\sqrt{\pi}} \right ).
\end{split}
\end{equation} 
\end{proof}
\begin{rem}
As we know by Theorem \ref{thm:kDesVarGeneral} that the variance is of order $n$, the bound in Theorem \ref{thm:BRTkdescendantsCLT} decreases with order $\frac{1}{\sqrt{n}}$.
\end{rem}
\begin{rem}
Since $Y_{\leq k, n}^{p}$, the number of nodes with at most $k$ descendants is equal to $n- Y_{\geq k+1, n}^{p}$  the expectation, variance and CLT of $Y_{\leq k, n}^{p}$ follows directly from the above results.
\end{rem}

Moreover the expectation and variance of the number of nodes with exactly $k$ descendants can be calculated from the results of the previous section by some additional calculations.
\newpage
\begin{cor} \label{cor:ExactKDesBRTExpect}
Let $X_{k,n}^{p}$ denote the number of nodes with exactly $k$ descendants in a $p$-BRT. Then
\begin{equation}
\begin{split}
\E[X_{k,n}^{p}] &= (n-k-1)\sum_{s=1}^{a} p_s \left ( \left (\sum_{r=s}^{a} p_r  \right )^k - \left (\sum_{r=s}^{a} p_r  \right )^{k+1} \right ) \\
& \hspace{2em} + \sum_{s=1}^{a} p_s \left (\sum_{r=s}^{a} p_r  \right )^{k+1}.
\end{split}
\end{equation}
Moreover if $X_{k,n}^{a}$ is the number of nodes with exactly $k$ descendants in a $a$-RT, then
\begin{equation}
\begin{split}
& \E[X_{k,n}^{a}] = (n-k-1) \left (  \frac{1}{a^{k+1}} \sum_{s=1}^{a} s^k -  \frac{1}{a^{k+2}} \sum_{s=1}^{a} s^{k+1}\right )  + \frac{1}{a^{k+2}} \sum_{s=1}^{a} s^{k+1} \\
\end{split}
\end{equation}
and asymptotically
\begin{equation} 
\E[X_{k,n}^{a}] \xrightarrow{a \to \infty} \frac{n}{(k+1)(k+2)}.
\end{equation}
\end{cor}
\begin{proof}
We have $X_{k,n}^{p} = Y_{\geq k, n}^{p} - Y_{\geq k+1, n}^{p}$. Thus we get from Theorem \ref{thm:BRTExpkDesce},
\begin{equation}
\begin{split}
& \E \left [X_{k,n}^{p} \right ] = \E\left [Y_{\geq k, n}^{p} \right]- \E \left [Y_{\geq k+1, n}^{p} \right ] \\
&= (n-k-1)\sum_{s=1}^{a} p_s \left (\sum_{r=s}^{a} p_r  \right )^k +1 - (n-k-2)\sum_{s=1}^{a} p_s \left (\sum_{r=s}^{a} p_r  \right )^{k+1} - 1 \\
&= (n-k-1)\sum_{s=1}^{a} p_s \left ( \left (\sum_{r=s}^{a} p_r  \right )^k - \left (\sum_{r=s}^{a} p_r  \right )^{k+1} \right ) + \sum_{s=1}^{a} p_s \left (\sum_{r=s}^{a} p_r  \right )^{k+1}.  \\
\end{split}
\end{equation}
Similarly we get from Theorem \ref{thm:BRTExpkDesce} for $\vec{p}$ uniform,
\begin{equation}
\begin{split}
\E[X_{k,n}^{a}] &= (n-k-1) \left (  \frac{1}{a^{k+1}} \sum_{s=1}^{a} s^k -  \frac{1}{a^{k+2}} \sum_{s=1}^{a} s^{k+1}\right )  + \frac{1}{a^{k+2}} \sum_{s=1}^{a} s^{k+1}. \\
\end{split}
\end{equation}
This implies in particular that
\begin{equation}
\begin{split}
&\E[X_{k,n}^{a}]\\
& = (n-k-1) \left (  \frac{1}{a^{k+1}} \left [ \frac{a^{k+1}}{k+1} + \mathcal{O}(a^k) \right ] -  \frac{1}{a^{k+2}} \left [ \frac{a^{k+2}}{k+2} + \mathcal{O}(a^{k+1}) \right ] \right ) \\
& \hspace{2em} + \frac{1}{a^{k+2}} \left [ \frac{a^{k+2}}{k+2} + \mathcal{O}(a^{k+1}) \right ]\\
&\xrightarrow{a \to \infty}  (n-k-1) \left ( \frac{1}{k+1} -  \frac{1}{k+2} \right ) +  \frac{1}{k+2} \\
&= \frac{n}{(k+1)(k+2)} - \frac{k+1}{(k+1)(k+2)} + \frac{1}{k+2} \\
&= \frac{n}{(k+1)(k+2)}.
\end{split}
\end{equation}
\end{proof}

\begin{rem}
For the variance of the number of nodes with exactly $k$ descendants we have that
\begin{equation} 
\begin{split}
\Var\left (X_{k,n}^{p}\right ) & = \Var\left (Y_{\geq k, n}^{p} - Y_{\geq k+1,n}^{p} \right) \\
& =\Var \left (Y_{\geq k, n}^{p}\right ) + \Var\left (Y_{\geq k+1,n}^{p}\right ) - 2 \Cov\left (Y_{\geq k, n}^{p}, Y_{\geq k+1,n}^{p}\right).
\end{split}
\end{equation}
The covariance can be calculated by the same method as before, so we can obtain an expression for $\Var\left (X_{k,n}^{p}\right )$. Since the calculations as well as the expressions one gets in the end are quite long, we did not include them here. The asymptotic value of $\Var\left (X_{k,n}^{a}\right)$ as  $a \to \infty$ can also be calculated using the same methods as in the proof of Theorem \ref{thm:aRTBranchesVariance}. As expected the value one gets in the end for $\lim_{n \to \infty} \lim_{a \to \infty} \frac{\Var\left (X_{k,n}^{a}\right )}{n}$ corresponds to the value of the same expression for the number of nodes with exactly $k$ descendants of URTs, which is stated in Theorem \ref{thm:URTkDes}.
\end{rem}

\section{Depth of Node $n$}

First we observe that given a permutation representation for a recursive tree we can determine $\mathcal{D}_n^{p}$, the depth of node $n$ by counting the steps we go down from that position when we go to the left until 1. It is important to note that we must take any step down if we can, and we don't go up again once we went down. This implies that we are not looking for the longest increasing subsequence from 1 to n. Rather we are looking for all anti-records when we start at $n$ and then go to the left until we reach 1. The following examples will clarify this difference:

\begin{ex}
If $\gamma = 13456728$ the depth of node 8 is 2, because it is attached to 2 which is attached to 1. This corresponds to the number of steps down we take from 8 since we first go down to 2 and then down to 1. The longest increasing subsequence is $1345678$, which is much longer. 

For $\rho = 12574863$ the depth of node $8$ is 3 by the construction principle. We can also see this because from 8 we go down to 4 then there is 7 and 5, which is higher so we ignore it. Then we go down to 2, and finally to 1. So we take 3 steps down. We don't count the step down from 7 to 5 since we already went down to 4 before that, so we are not interested in anything concerning nodes bigger than 4 anymore.
\end{ex}

Thus given $\mathcal{P}_n = \gamma^{-1}(n)$, the position of $n$ in the permutation, we get the depth of node $n$ by calculating for all $i< \mathcal{P}_n$, the probability that $\gamma(i) < \min_{i<j<\mathcal{P}_n} \{\gamma(j)\}$. We will now use these observations to prove the following theorem:

\begin{thm} \label{ThmDepthBRT}
Let $\mathcal{D}_n^{p}$ denote the depth of node $n$ in a $p$-BRT. Then 
\begin{equation} \label{eq:DepthBRT}
\begin{split}
&\E[\mathcal{D}_n^{p}] = \\
& \sum_{s=2}^{a} \frac{p_s}{\sum_{r=1}^{s-1} p_r} \sum_{s'=2}^{a} \left [\left (\sum_{r=1}^{s'} p_r \right )^{n-1}  - \left ( \sum_{r=1}^{s'-1} p_r \right )^{n-1} \right ] \\
&\hspace{1ex} -\sum_{s=2}^{a} \frac{p_s}{\sum_{r=1}^{s-1} p_r} \sum_{s'=2}^{a} p_{s'} \frac{\left (\sum_{r=1}^{s'-1} p_r \right ) ^{n-1}  - \left ( \sum_{r=s}^{a} p_r  \sum_{r=1}^{s'} p_r \right )^{n-1}}{\sum_{r=1}^{s'-1} p_r - \sum_{r=s}^{a} p_r  \sum_{r=1}^{s'} p_r  }\\
& \hspace{1ex} + p_1  \sum_{s=2}^{a} \frac{1}{p_s} \\
& \hspace{3ex} \cdot \left [(n-2)\left (\sum_{r=1}^{s} p_r \right)^{n}- (n-1)\left (\sum_{r=1}^{s} p_r\right)^{n-1}\sum_{r=1}^{s-1} p_r  + \sum_{r=1}^{s} p_r \left (\sum_{r=1}^{s-1} p_r \right )^{n-1} \right ] \\
& \hspace{1ex} + \sum_{s=2}^{a} \left( \left (\sum_{r=1}^{s} p_r \right ) ^{n-1} -  \left (\sum_{r=1}^{s-1} p_r \right ) ^{n-1}  \right )\\
& \hspace{1ex} + \left [\sum_{s=2}^{a} p_s \frac{1 - \left ( \sum_{r=s}^{a} p_r \right )^{n-2}}{\sum_{r=1}^{s-1} p_r}+(n-2)p_1 +1\right]p_1^{n-1}. \\
\end{split}
\end{equation}
Moreover asymptotically we have
\begin{equation} \label{equ:DepthBRTninf}
\lim_{n\to \infty} \frac{\E[\mathcal{D}_n^{p}]}{n} = p_1.
\end{equation}
\end{thm}
 
As we will use the position of node $n$ in order to derive its depth, we will need the following lemma.

\begin{lem}
Let $\mathcal{P}_{n}^{p}$ denote the position of $n$ in a $p$-biased random permutation. Then for $2 \leq k < n$,
\begin{equation}
\Pro \left ( \mathcal{P}_{n}^{p} = k\right )  =
 \sum_{s=2}^{a} p_s\left (\sum_{r=1}^{s} p_r \right ) ^{k-2} \left  (\sum_{r=1}^{s-1} p_r \right ) ^{n-k}
\end{equation}
and
\begin{equation}
\Pro \left( \mathcal{P}_{n}^{p}= n \right) =  \sum_{s=1}^{a} p_s\left (\sum_{r=1}^{s} p_r \right ) ^{n-2}.
\end{equation}
\end{lem}
\begin{proof}
As before we will use the construction of an inverse biased riffle shuffle permutation as introduced in Theorem \ref{thm:EquivalentRiffleShuffle}. We know that $n$ will be the last card in the last non-empty pile. Also the digits we assign each index when constructing an inverse riffle shuffle permutation define the position of the cards in the pile with that digit. In other words, the indices that get digit $s$, are the positions of the cards in the $s$-th pile. This means that $n$ gets the position of the last index that gets the highest digit. We get
\begin{equation} \mathcal{P}_{n}^{p} = \max \left \{2 \leq i \leq n: X_i \geq \{X_2, \dots, X_n\} \right \}
\end{equation}
so 
\begin{equation}\Pro(  \mathcal{P}_{n}^{p} = k) = \Pro \left ( X_k \geq \{X_2, \dots, X_{k-1} \}, X_k > \{X_{k+1}, \dots, X_{n}\} \right ).
\end{equation}

By conditioning on $X_k$ we get independent events and can thus calculate this probability for $2\leq k<n$. We get
\begin{equation}
\begin{split}
\Pro& (  \mathcal{P}_{n}^{p} = k) \\
&= \Pro \left ( X_k \geq \{X_2, \dots, X_{k-1} \}, X_k > \{X_{k+1}, \dots, X_{n}\} \right ) \\
&= \sum_{s=1}^{a} \Pro \left ( X_k \geq \{X_2, \dots, X_{k-1} \}, X_k > \{X_{k+1}, \dots, X_{n}\} | X_k =s \right ) \Pro(X_k =s)\\
&= \sum_{s=1}^{a} \Pro \left ( s \geq \{X_2, \dots, X_{k-1} \}, s > \{X_{k+1}, \dots, X_{n}\} \right ) \Pro(X_k =s)\\
&= \sum_{s=1}^{a} \Pro \left ( s \geq \{X_2, \dots, X_{k-1} \}\right ) \Pro \left ( s > \{X_{k+1}, \dots, X_{n}\} \right ) \Pro(X_k =s)\\
&= \sum_{s=2}^{a} p_s\left (\sum_{r=1}^{s} p_r \right ) ^{k-2} \left (\sum_{r=1}^{s-1} p_r \right ) ^{n-k}. \\
\end{split}
\end{equation}
We only sum from $s=2$ since if $X_k=1$ it cannot be strictly greater than $X_{k+1}, \dots, X_n$. If $k=n$ we have the additional possibility that all $X_i$ are equal to 1 and thus get 
\begin{equation}\Pro(  \mathcal{P}_{n}^{p} = n) =  \sum_{s=1}^{a} p_s\left (\sum_{r=1}^{s} p_r \right ) ^{n-2}.
\end{equation}
\end{proof}
We can now start with the proof of Theorem \ref{ThmDepthBRT}.
\begin{proof}[Proof of Theorem \ref{ThmDepthBRT}]
First of all we will find an expression for $\E[\mathcal{D}_n^{p} |  \mathcal{P}_{n}^{p} =k]$. As we said above, given $ \mathcal{P}_{n}^{p}$ we need to calculate the number of anti-records when we go from the position of $n$ to the left until we reach 1. Hence, given $ \mathcal{P}_{n}^{p}$, we define for $1 \leq i <  \mathcal{P}_{n}^{p}$, 
\begin{equation}E_i := \1 \left (\gamma(i) = \min_{i\leq k\leq  \mathcal{P}_{n}^{p}} \{\gamma(k)\}\right ).\end{equation} 
Then, given $ \mathcal{P}_{n}^{p}$, we get $\mathcal{D}_n^{p} = \sum_{i=1}^{ \mathcal{P}_{n}^{p}-1} E_i$, and so
\begin{equation} \E[\mathcal{D}_n^{p} |  \mathcal{P}_{n}^{p}] = \sum_{i=1}^{ \mathcal{P}_{n}^{p}-1} \E[E_i |  \mathcal{P}_{n}^{p}].\end{equation}
     We can simplify this sum by first observing that 
\begin{equation}\E(E_{ \mathcal{P}_{n}^{p}-1}) = \Pro( \gamma({ \mathcal{P}_{n}^{p}-1}) < \gamma({ \mathcal{P}_{n}^{p}})) = \Pro(\gamma({ \mathcal{P}_{n}^{p}-1}) < n) =1
\end{equation} and in general for all $i$, we have $\gamma(i) < \gamma( \mathcal{P}_{n}^{p}) = n$, so we can rewrite $E_i$ as \begin{equation}E_i = \1 \left (\gamma(i) = \min_{i\leq k<  \mathcal{P}_{n}^{p}} \{\gamma(k)\}\right ).\end{equation} 
Moreover $\Pro(E_1) = \Pro( \gamma(1) = \min_{1\leq k \leq  \mathcal{P}_{n}^{p}}\{\gamma(k)\} ) =1$ since $\gamma(1)=1$.
Now we will again use the inverse riffle shuffle construction to calculate the rest of these probabilities. We know that $\gamma(i) < \gamma(j)$ for all $i<j$ if and only if $X_i \leq X_j$, since this means that in the spot $j$ will come a higher card from the same pile or from a pile corresponding to a higher digit, thus with higher labeled cards. 

Let $2\leq i \leq  \mathcal{P}_{n}^{p}$, then 
\begin{equation}
\begin{split}
\E&[E_i |  \mathcal{P}_{n}^{p}]\\
& = \Pro \left ( \gamma(i) = \min_{i \leq k<  \mathcal{P}_{n}^{p}}  \{\gamma(k)\} | \gamma( \mathcal{P}_{n}^{p}) = n \right)  \\
&= \Pro \left ( X_i = \min_{i \leq k<  \mathcal{P}_{n}^{p}}  \{X_k\} | X_{ \mathcal{P}_{n}^{p}} \geq \{X_2, \dots, X_{ \mathcal{P}_{n}^{p}-1} \}, X_{ \mathcal{P}_{n}^{p}} > \{X_{ \mathcal{P}_{n}^{p}+1}, \dots, X_{n} \} \right ) \\
&= \Pro \left ( X_i = \min_{i \leq k<  \mathcal{P}_{n}^{p}}  \{X_k\}\right ) \\
&= \sum_{s=1}^{a}\Pro \left ( X_i = \min_{i \leq k<  \mathcal{P}_{n}^{p}}  \{X_k\} | X_i = s  \right ) \Pro ( X_i =s) \\
&= \sum_{s=1}^{a}\Pro \left ( s \leq \{X_{i+1}, \dots,  X_{ \mathcal{P}_{n}^{p}-1} \} \right ) \Pro ( X_i =s) \\
&= \sum_{s=1}^{a} p_s \left ( \sum_{r=s}^{a} p_r \right )^{ \mathcal{P}_{n}^{p}-1-(i+1)+1} \\
&= \sum_{s=1}^{a} p_s \left ( \sum_{r=s}^{a} p_r \right )^{ \mathcal{P}_{n}^{p}-i-1}. \\
\end{split}
\end{equation}
We could get rid of the conditional in the third line because the $X_k$ are mutually independent and thus the fact that $X_{ \mathcal{P}_{n}^{p}} \geq \{X_2, \dots, X_{ \mathcal{P}_{n}^{p}-1} \}$ does not tell us anything about the ordering of $ \{X_2, \dots, X_{ \mathcal{P}_{n}^{p}-1} \}$. Moreover $X_{ \mathcal{P}_{n}^{p}} > \{X_{ \mathcal{P}_{n}^{p}+1}, \dots, X_{n} \}$ concerns different $X_i$s, so is also independent of $E_i$.
We again assume $p_i>0$ for all $i$ in order to avoid division by 0.
In total we thus get, for $2 < k \leq n$,
\begin{equation}
\begin{split}
\E[\mathcal{D}_n^{p} | \mathcal{P}_{n}^{p}=k] &= \sum_{i=1}^{k-1} \E[E_i | \mathcal{P}_{n}^{p} = k] \\
&= \sum_{i=2}^{k-1} \sum_{s=1}^{a} p_s \left ( \sum_{r=s}^{a} p_r \right )^{k-i-1} + 1 \\
&= \sum_{s=1}^{a} p_s \sum_{i=2}^{k-1}  \left ( \sum_{r=s}^{a} p_r \right )^{k-i-1} +1\\
&= \sum_{s=1}^{a} p_s \sum_{i=0}^{k-3}  \left ( \sum_{r=s}^{a} p_r \right )^{i} +1\\
&= \sum_{s=2}^{a} p_s \frac{ 1 - \left ( \sum_{r=s}^{a} p_r \right )^{k-2} }{1 - \sum_{r=s}^{a} p_r}+(k-2)p_1 +1\\
&= \sum_{s=2}^{a} p_s \frac{1 - \left ( \sum_{r=s}^{a} p_r \right )^{k-2}}{\sum_{r=1}^{s-1} p_r}+(k-2)p_1 +1\\
\end{split}
\end{equation}
and \begin{equation}\E[\mathcal{D}_n^{p}| \mathcal{P}_{n}^{p} = 2] = 1.
\end{equation}
To get $\E[\mathcal{D}_n^{p}]$ we now need to use the tower rule, this gives
\begin{equation}
\begin{split}
& \E[\mathcal{D}_n^{p}]\\
&= \E \left [ \E[\mathcal{D}_n^{p} | \mathcal{P}_{n}^{p}=k] \right ] \\
&= \sum_{k=2}^n \E[\mathcal{D}_n^{p} |  \mathcal{P}_{n}^{p} =k] \Pro( \mathcal{P}_{n}^{p} =k) \\
&= \sum_{k=3}^{n-1} \E[\mathcal{D}_n^{p} |  \mathcal{P}_{n}^{p} =k] \Pro( \mathcal{P}_{n}^{p} =k) \\
& \hspace{1em} +\E[\mathcal{D}_n^{p} |  \mathcal{P}_{n}^{p} =n] \Pro( \mathcal{P}_{n}^{p} =n) \\
& \hspace{1em}+ \E[\mathcal{D}_n^{p} |  \mathcal{P}_{n}^{p} =2] \Pro( \mathcal{P}_{n}^{p} =2) \\
&= \sum_{k=3}^{n-1} \left [ \sum_{s=2}^{a} p_s \frac{1 - \left ( \sum_{r=s}^{a} p_r \right )^{k-2}}{\sum_{r=1}^{s-1} p_r} + (k-2)p_1 +1 \right ] \\
& \hspace{3em} \cdot \left [ \sum_{s=2}^{a} p_s\left (\sum_{r=1}^{s} p_r \right ) ^{k-2} \left (\sum_{r=1}^{s-1} p_r \right ) ^{n-k} \right ]\\
& \hspace{1em} + \left [\sum_{s=2}^{a} p_s \frac{1 - \left ( \sum_{r=s}^{a} p_r \right )^{n-2}}{\sum_{r=1}^{s-1} p_r}+(n-2)p_1 +1\right] \left [\sum_{s=1}^{a} p_s\left (\sum_{r=1}^{s} p_r \right ) ^{n-2}\right ] \\
&\hspace{1em} + \sum_{s=2}^{a} p_s \left (\sum_{r=1}^{s-1} p_r \right ) ^{n-2}  \\
&= \sum_{k=3}^{n} \left [ \sum_{s=2}^{a} p_s \frac{1 - \left ( \sum_{r=s}^{a} p_r \right )^{k-2}}{\sum_{r=1}^{s-1} p_r}\right ] \left [ \sum_{s=2}^{a} p_s\left (\sum_{r=1}^{s} p_r \right ) ^{k-2} \left (\sum_{r=1}^{s-1} p_r \right ) ^{n-k} \right ]\\
& \hspace{1em} + \sum_{k=3}^{n} \left [ (k-2)p_1 +1 \right ] \left [ \sum_{s=2}^{a} p_s\left (\sum_{r=1}^{s} p_r \right ) ^{k-2} \left (\sum_{r=1}^{s-1} p_r \right ) ^{n-k} \right ]\\
& \hspace{1em} + \left [\sum_{s=2}^{a} p_s \frac{1 - \left ( \sum_{r=s}^{a} p_r \right )^{n-2}}{\sum_{r=1}^{s-1} p_r}+(n-2)p_1 +1\right]p_1^{n-1} \\
&\hspace{1em} + \sum_{s=2}^{a} p_s \left (\sum_{r=1}^{s-1} p_r \right ) ^{n-2} . \\
\end{split}
\end{equation}
In order to simplify this expression we first  look at the first line
\begin{dmath} \label{eq:Depth1}
\sum_{k=3}^{n} \left [ \sum_{s=2}^{a} p_s \frac{1 - \left ( \sum_{r=s}^{a} p_r \right )^{k-2}}{\sum_{r=1}^{s-1} p_r}\right ] \left [ \sum_{s=2}^{a} p_s\left (\sum_{r=1}^{s} p_r \right ) ^{k-2} \left (\sum_{r=1}^{s-1} p_r \right ) ^{n-k} \right ] 
= \sum_{k=3}^n \sum_{s=2}^{a} \frac{p_s}{\sum_{r=1}^{s-1} p_r} \sum_{s'=2}^{a} p_{s'} \left ( 1 - \left ( \sum_{r=s}^{a} p_r \right )^{k-2} \right ) \left (\sum_{r=1}^{s'} p_r \right ) ^{k-2} \left (\sum_{r=1}^{s'-1} p_r \right ) ^{n-k} 
= \sum_{s=2}^{a} \frac{p_s}{\sum_{r=1}^{s-1} p_r} \sum_{s'=2}^{a} p_{s'} \left [ \sum_{k=3}^n  \left (\sum_{r=1}^{s'} p_r \right ) ^{k-2} \left (\sum_{r=1}^{s'-1} p_r \right ) ^{n-k} -  \sum_{k=3}^n \left ( \sum_{r=s}^{a} p_r \right )^{k-2} \left (\sum_{r=1}^{s'} p_r \right ) ^{k-2} \left (\sum_{r=1}^{s'-1} p_r \right ) ^{n-k} \right ].
\end{dmath}
Now we have 
\begin{equation}
\begin{split}
&\sum_{k=3}^n  \left (\sum_{r=1}^{s'} p_r \right ) ^{k-2} \left (\sum_{r=1}^{s'-1} p_r \right ) ^{n-k} \\
&= \left (\sum_{r=1}^{s'-1} p_r \right ) ^{n-2} \sum_{k=3}^n  \left (\sum_{r=1}^{s'} p_r \right ) ^{k-2} \left (\sum_{r=1}^{s'-1} p_r \right ) ^{-k+2} \\
&= \left (\sum_{r=1}^{s'-1} p_r \right )^{n-2}  \left ( \frac{1- \left (\left ( \sum_{r=1}^{s'} p_r \right ) \left (\sum_{r=1}^{s'-1} p_r \right ) ^{-1}\right )^{n-1}}{1- \left (\sum_{r=1}^{s'} p_r \right ) \left (\sum_{r=1}^{s'-1} p_r \right )^{-1}} -1 \right ) \\
&=  \frac{\left (\sum_{r=1}^{s'-1} p_r \right )^{n-1}  - \left ( \sum_{r=1}^{s'} p_r \right )^{n-1}}{\sum_{r=1}^{s'-1} p_r   - \sum_{r=1}^{s'} p_r } - \left (\sum_{r=1}^{s'-1} p_r \right )^{n-2}. \\
\end{split}
\end{equation}
Inserting this into the first part of the last line of \eqref{eq:Depth1} we get
\begin{equation}
\begin{split}
&\sum_{s=2}^{a} \frac{p_s}{\sum_{r=1}^{s-1} p_r} \sum_{s'=2}^{a} p_{s'} \sum_{k=3}^n  \left (\sum_{r=1}^{s'} p_r \right ) ^{k-2} \left (\sum_{r=1}^{s'-1} p_r \right ) ^{n-k} \\
&= \sum_{s=2}^{a} \frac{p_s}{\sum_{r=1}^{s-1} p_r} \sum_{s'=2}^{a} p_{s'}  \left [ \frac{\left (\sum_{r=1}^{s'-1} p_r \right )^{n-1}  - \left ( \sum_{r=1}^{s'} p_r \right )^{n-1}}{\sum_{r=1}^{s'-1} p_r   - \sum_{r=1}^{s'} p_r } - \left (\sum_{r=1}^{s'-1} p_r \right )^{n-2}\right ] \\
&= \sum_{s=2}^{a} \frac{p_s}{\sum_{r=1}^{s-1} p_r} \sum_{s'=2}^{a} p_{s'}\frac{\left (\sum_{r=1}^{s'-1} p_r \right )^{n-1}  - \left ( \sum_{r=1}^{s'} p_r \right )^{n-1}}{-p_{s'}} \\
&\hspace{3em} - \sum_{s=2}^{a} \frac{p_s}{\sum_{r=1}^{s-1} p_r} \sum_{s'=2}^{a} p_{s'}\left (\sum_{r=1}^{s'-1} p_r \right )^{n-2} \\
&= \sum_{s=2}^{a} \frac{p_s}{\sum_{r=1}^{s-1} p_r} \sum_{s'=2}^{a} \left [\left (\sum_{r=1}^{s'} p_r \right )^{n-1}  - \left ( \sum_{r=1}^{s'-1} p_r \right )^{n-1} \right ] \\
&\hspace{3em} - \sum_{s=2}^{a} \frac{p_s}{\sum_{r=1}^{s-1} p_r} \sum_{s'=2}^{a} p_{s'}\left (\sum_{r=1}^{s'-1} p_r \right )^{n-2}.
\end{split} 
\end{equation}
Similarly we have that
\begin{equation}
\begin{split}
& \sum_{k=3}^n \left ( \sum_{r=s}^{a} p_r \right )^{k-2} \left (\sum_{r=1}^{s'} p_r \right ) ^{k-2} \left (\sum_{r=1}^{s'-1} p_r \right ) ^{n-k} \\
&=  \left (\sum_{r=1}^{s'-1} p_r \right ) ^{n-2} \sum_{k=3}^n \left ( \sum_{r=s}^{a} p_r \right )^{k-2} \left (\sum_{r=1}^{s'} p_r \right ) ^{k-2} \left (\sum_{r=1}^{s'-1} p_r \right ) ^{-(k+2)} \\
&= \left (\sum_{r=1}^{s'-1} p_r \right ) ^{n-2} \left ( \frac{1 - \left ( \sum_{r=s}^{a} p_r  \sum_{r=1}^{s'} p_r \left ( \sum_{r=1}^{s'-1} p_r \right )^{-1}\right )^{n-1}}{1 - \sum_{r=s}^{a} p_r  \sum_{r=1}^{s'} p_r \left ( \sum_{r=1}^{s'-1} p_r \right )^{-1} }-1 \right )\\
&= \left ( \frac{\left (\sum_{r=1}^{s'-1} p_r \right ) ^{n-1}  - \left ( \sum_{r=s}^{a} p_r  \sum_{r=1}^{s'} p_r \right )^{n-1}}{\sum_{r=1}^{s'-1} p_r - \sum_{r=s}^{a} p_r  \sum_{r=1}^{s'} p_r  }- \left (\sum_{r=1}^{s'-1} p_r \right ) ^{n-2} \right ).\\
\end{split}
\end{equation}
Inserting this expression into the second part of the last line of \eqref{eq:Depth1} we get
\begin{equation}
\begin{split}
&\sum_{s=2}^{a} \frac{p_s}{\sum_{r=1}^{s-1} p_r} \sum_{s'=2}^{a} p_{s'} \sum_{k=3}^n \left ( \sum_{r=s}^{a} p_r \right )^{k-2} \left (\sum_{r=1}^{s'} p_r \right ) ^{k-2} \left (\sum_{r=1}^{s'-1} p_r \right ) ^{n-k}\\
&= \sum_{s=2}^{a} \frac{p_s}{\sum_{r=1}^{s-1} p_r} \sum_{s'=2}^{a} p_{s'} \frac{\left (\sum_{r=1}^{s'-1} p_r \right ) ^{n-1}  - \left ( \sum_{r=s}^{a} p_r  \sum_{r=1}^{s'} p_r \right )^{n-1}}{\sum_{r=1}^{s'-1} p_r - \sum_{r=s}^{a} p_r  \sum_{r=1}^{s'} p_r  }\\
& \hspace{2em}- \sum_{s=2}^{a} \frac{p_s}{\sum_{r=1}^{s-1} p_r} \sum_{s'=2}^{a} p_{s'}\left (\sum_{r=1}^{s'-1} p_r \right ) ^{n-2}.
\end{split}
\end{equation}
Now we look at the second line. First we separate it into two terms, 
\begin{equation}
\begin{split}
 \sum_{k=3}^{n}&  \left [ (k-2)p_1 +1 \right ] \left [ \sum_{s=2}^{a} p_s\left (\sum_{r=1}^{s} p_r \right ) ^{k-2} \left (\sum_{r=1}^{s-1} p_r \right ) ^{n-k} \right ] \\
& = \sum_{k=3}^{n} (k-2)p_1 \left [ \sum_{s=2}^{a} p_s\left (\sum_{r=1}^{s} p_r \right ) ^{k-2} \left (\sum_{r=1}^{s-1} p_r \right ) ^{n-k} \right ] \\
& \hspace{3em} + \sum_{k=3}^{n} \sum_{s=2}^{a} p_s\left (\sum_{r=1}^{s} p_r \right ) ^{k-2} \left (\sum_{r=1}^{s-1} p_r \right ) ^{n-k} \\
& {:=} A_1 + A_2.
\end{split}
\end{equation}
Now we easily get
\begin{equation}
\begin{split}
A_2 & = \sum_{k=3}^{n} \sum_{s=2}^{a} p_s\left (\sum_{r=1}^{s} p_r \right ) ^{k-2} \left (\sum_{r=1}^{s-1} p_r \right ) ^{n-k} \\
& = \sum_{s=2}^{a} p_s  \left (\sum_{r=1}^{s-1} p_r \right ) ^{n-2}  \sum_{k=3}^{n}  \left (\sum_{r=1}^{s} p_r \right ) ^{k-2} \left (\sum_{r=1}^{s-1} p_r \right ) ^{-k+2} \\
& = \sum_{s=2}^{a} p_s  \left (\sum_{r=1}^{s-1} p_r \right ) ^{n-2}  \left( \frac{1-  \left (\sum_{r=1}^{s} p_r \right ) ^{n-1} \left (\sum_{r=1}^{s-1} p_r \right ) ^{-(n-1)} }{1- \sum_{r=1}^{s} p_r  \left (\sum_{r=1}^{s-1} p_r \right )^{-1}} -1 \right ) \\
& = \sum_{s=2}^{a} p_s \left( \frac{\left (\sum_{r=1}^{s-1} p_r \right ) ^{n-1} -  \left (\sum_{r=1}^{s} p_r \right ) ^{n-1}}{\sum_{r=1}^{s-1} p_r  - \sum_{r=1}^{s} p_r } - \left (\sum_{r=1}^{s-1} p_r \right ) ^{n-2}  \right )  \\
& = \sum_{s=2}^{a} p_s \left( \frac{\left (\sum_{r=1}^{s} p_r \right ) ^{n-1} -  \left (\sum_{r=1}^{s-1} p_r \right ) ^{n-1}}{p_s } - \left (\sum_{r=1}^{s-1} p_r \right ) ^{n-2}  \right ) \\
&= \sum_{s=2}^{a} \left( \left (\sum_{r=1}^{s} p_r \right ) ^{n-1} -  \left (\sum_{r=1}^{s-1} p_r \right ) ^{n-1}  \right ) - \sum_{s=2}^{a} p_s \left (\sum_{r=1}^{s-1} p_r \right ) ^{n-2}.  \\
\end{split}
\end{equation}
For the first term we need a little bit more work. First of all
\begin{equation}
\begin{split}
A_1 & = \sum_{k=3}^{n} (k-2)p_1 \left [ \sum_{s=2}^{a} p_s\left (\sum_{r=1}^{s} p_r \right ) ^{k-2} \left (\sum_{r=1}^{s-1} p_r \right ) ^{n-k} \right ] \\
& =  p_1  \sum_{s=2}^{a} p_s \left (\sum_{r=1}^{s-1} p_r \right ) ^{n-2}  \sum_{k=3}^{n}  (k-2) \left (\sum_{r=1}^{s} p_r \right ) ^{k-2} \left (\sum_{r=1}^{s-1} p_r \right ) ^{-k+2}.
\end{split}
\end{equation}
\newpage
We now set $x:= \left (\sum_{r=1}^{s} p_r \left (\sum_{r=1}^{s-1} p_r \right ) ^{-1}\right)$ and only look at the sum indexed with $k$.
Then we have a sum of the form
\begin{equation}
\begin{split}
\sum_{k=3}^n (k-2) x^{k-2} &= x \sum_{k=1}^{n-2} k x^{k-1} \\
&= x \frac{d}{dx} \sum_{k=0}^{n-2} x^k \\
&= x \frac{d}{dx} \frac{1-x^{n-1}}{1-x} \\
&= x \frac{-(n-1)x^{n-2}(1-x)-(1-x^{n-1})(-1)}{(1-x)^2} \\
&= x \frac{(n-1)x^{n-1} -(n-1)x^{n-2} +1 -x^{n-1}}{(1-x)^2}\\
&= x\frac{(n-2)x^{n-1} - (n-1)x^{n-2} +1}{(1-x)^2}.
\end{split}
\end{equation}
We now substitute $\left (\sum_{r=1}^{s} p_r \left (\sum_{r=1}^{s-1} p_r \right ) ^{-1}\right)$ for $x$ and insert the resulting expression into $A_1$. This gives
\begin{equation}
\begin{split}
&A_1 = p_1  \sum_{s=2}^{a} p_s \left (\sum_{r=1}^{s-1} p_r \right ) ^{n-2} \left (\sum_{r=1}^{s} p_r \left (\sum_{r=1}^{s-1} p_r \right ) ^{-1} \right ) \\
& \hspace{4em} \cdot \Bigg [ \frac{(n-2)\left (\sum_{r=1}^{s} p_r \left (\sum_{r=1}^{s-1} p_r \right ) ^{-1}\right)^{n-1}}{\left (1-\left (\sum_{r=1}^{s} p_r \left (\sum_{r=1}^{s-1} p_r \right ) ^{-1}\right)\right )^2} \\
& \hspace{7em} - \frac{ (n-1)\left (\sum_{r=1}^{s} p_r \left (\sum_{r=1}^{s-1} p_r \right ) ^{-1}\right)^{n-2} - 1}{\left (1-\left (\sum_{r=1}^{s} p_r \left (\sum_{r=1}^{s-1} p_r \right ) ^{-1}\right)\right )^2} \Bigg] \\
&= p_1  \sum_{s=2}^{a} p_s  \cdot \Bigg [ \frac{(n-2)\left (\sum_{r=1}^{s} p_r \right)^{n}}{\left (\sum_{r=1}^{s-1} p_r - \sum_{r=1}^{s} p_r \right ) ^2} \\
& \hspace{7em} - \frac{ (n-1)\left (\sum_{r=1}^{s} p_r\right)^{n-1}\sum_{r=1}^{s-1} p_r  - \sum_{r=1}^{s} p_r \left (\sum_{r=1}^{s-1} p_r \right )^{n-1}}{\left (\sum_{r=1}^{s-1} p_r - \sum_{r=1}^{s} p_r \right ) ^2}  \Bigg ]\\
&= p_1 \sum_{s=2}^{a} \frac{1}{p_s} \Bigg [ (n-2)\left (\sum_{r=1}^{s} p_r \right)^{n}  \\
& \hspace{7em} - (n-1)\left (\sum_{r=1}^{s} p_r\right)^{n-1}\sum_{r=1}^{s-1} p_r  + \sum_{r=1}^{s} p_r \left (\sum_{r=1}^{s-1} p_r \right )^{n-1} \Bigg]. \\
\end{split}
\end{equation}
In total now we have as an expression for the expectation
\begin{equation}
\begin{split}
&\E[\mathcal{D}_n^{p}] = \\
&  \sum_{s=2}^{a} \frac{p_s}{\sum_{r=1}^{s-1} p_r} \sum_{s'=2}^{a} \left [\left (\sum_{r=1}^{s'} p_r \right )^{n-1}  - \left ( \sum_{r=1}^{s'-1} p_r \right )^{n-1} \right ] \\
&\hspace{1ex} - \sum_{s=2}^{a} \frac{p_s}{\sum_{r=1}^{s-1} p_r} \sum_{s'=2}^{a} p_{s'}\left (\sum_{r=1}^{s'-1} p_r \right )^{n-2}\\
&\hspace{1ex} -\sum_{s=2}^{a} \frac{p_s}{\sum_{r=1}^{s-1} p_r} \sum_{s'=2}^{a} p_{s'} \frac{\left (\sum_{r=1}^{s'-1} p_r \right ) ^{n-1}  - \left ( \sum_{r=s}^{a} p_r  \sum_{r=1}^{s'} p_r \right )^{n-1}}{\sum_{r=1}^{s'-1} p_r - \sum_{r=s}^{a} p_r  \sum_{r=1}^{s'} p_r  }\\
& \hspace{1ex} + \sum_{s=2}^{a} \frac{p_s}{\sum_{r=1}^{s-1} p_r} \sum_{s'=2}^{a} p_{s'}\left (\sum_{r=1}^{s'-1} p_r \right ) ^{n-2} \\
& \hspace{1ex} + p_1  \sum_{s=2}^{a} \frac{1}{p_s}
\Bigg  [ (n-2)\left (\sum_{r=1}^{s} p_r \right)^{n} \\
& \hspace{7em} - (n-1)\left (\sum_{r=1}^{s} p_r\right)^{n-1}\sum_{r=1}^{s-1} p_r  + \sum_{r=1}^{s} p_r \left (\sum_{r=1}^{s-1} p_r \right )^{n-1}\Bigg ] \\
& \hspace{1ex} + \sum_{s=2}^{a} \left( \left (\sum_{r=1}^{s} p_r \right ) ^{n-1} -  \left (\sum_{r=1}^{s-1} p_r \right ) ^{n-1}  \right ) - \sum_{s=2}^{a} p_s \left (\sum_{r=1}^{s-1} p_r \right ) ^{n-2}  \\
& \hspace{1ex} + \left [\sum_{s=2}^{a} p_s \frac{1 - \left ( \sum_{r=s}^{a} p_r \right )^{n-2}}{\sum_{r=1}^{s-1} p_r}+(n-2)p_1 +1\right]p_1^{n-1} \\
& \hspace{1ex} + \sum_{s=2}^{a} p_s \left (\sum_{r=1}^{s-1} p_r \right ) ^{n-2}. \\
\end{split}
\end{equation}
After some simplifications we get 
\begin{equation}
\begin{split}
&\E[\mathcal{D}_n^{p}] \\
& = \sum_{s=2}^{a} \frac{p_s}{\sum_{r=1}^{s-1} p_r} \sum_{s'=2}^{a} \left [\left (\sum_{r=1}^{s'} p_r \right )^{n-1}  - \left ( \sum_{r=1}^{s'-1} p_r \right )^{n-1} \right ] \\
&\hspace{1ex} -\sum_{s=2}^{a} \frac{p_s}{\sum_{r=1}^{s-1} p_r} \sum_{s'=2}^{a} p_{s'} \frac{\left (\sum_{r=1}^{s'-1} p_r \right ) ^{n-1}  - \left ( \sum_{r=s}^{a} p_r  \sum_{r=1}^{s'} p_r \right )^{n-1}}{\sum_{r=1}^{s'-1} p_r - \sum_{r=s}^{a} p_r  \sum_{r=1}^{s'} p_r  }\\
& \hspace{1ex} + p_1  \sum_{s=2}^{a} \frac{1}{p_s}  \Bigg[(n-2)\left (\sum_{r=1}^{s} p_r \right)^{n} \\
& \hspace{7em} - (n-1)\left (\sum_{r=1}^{s} p_r\right)^{n-1}\sum_{r=1}^{s-1} p_r  + \sum_{r=1}^{s} p_r \left (\sum_{r=1}^{s-1} p_r \right )^{n-1} \Bigg]\\
& \hspace{1ex} + \sum_{s=2}^{a} \left( \left (\sum_{r=1}^{s} p_r \right ) ^{n-1} -  \left (\sum_{r=1}^{s-1} p_r \right ) ^{n-1}  \right )\\
& \hspace{1ex} + \left [\sum_{s=2}^{a} p_s \frac{1 - \left ( \sum_{r=s}^{a} p_r \right )^{n-2}}{\sum_{r=1}^{s-1} p_r}+(n-2)p_1 +1\right]p_1^{n-1}. \\
\end{split}
\end{equation}

Asymptotically this gives, for $\vec{p}$ such that $p_i \neq 0$ for all $i \in [a]$ and $a>2$, for $n \to \infty$,
\begin{equation}
\begin{split}
\lim_{n\to \infty} \frac{\E[\mathcal{D}_n^{p}]}{n}=& \lim_{n\to \infty} p_1\frac{(n-2)- (n-1)\sum_{r=1}^{a-1} p_r }{n p_a} = p_1\\
\end{split}
\end{equation}
since 
\begin{equation}
 p_1\frac{(n-2)- (n-1)\sum_{r=1}^{a-1} p_r }{p_a} =  p_1\frac{(n-1)\left (1- \sum_{r=1}^{a-1} p_r \right) }{p_a}- \frac{p_1}{p_a} = (n-1)p_1 - \frac{p_1}{p_a}.
\end{equation}
\end{proof}
\newpage
As a corollary we get for  $\vec{p}$ uniformly distributed over $a$ the following results.

\begin{cor} \label{cor:aRTExptDepth}
Let $\mathcal{D}_n^{a}$ denote the depth of node $n$ in an $a$-RT. Then 
\begin{equation}
\begin{split}
\E\left [\mathcal{D}_n^{a}\right ] = &  H_{a-1} +1  + \frac{n-2}{a^{n}} - \frac{1}{a^{2n-3}}\sum_{s=1}^{a-1} \frac{(a-s)^{n-2}}{ s}  \\
& - \frac{1}{a^{n-2}} \sum_{s=1}^{a-1}  \sum_{s'=1}^{a-1} \frac{1}{s} \frac{{s' }^{n-1}  - \left( (a-s)(s'+1)\right )^{n-1}}{ss' +s -a }\\
& + \frac{1}{a^n} \sum_{s=1}^{a-1}
(n-2)\left (s+1\right)^{n}- (n-1)\left ( s+1\right)^{n-1}s  + (s+1) s^{n-1}.\\
\end{split}
\end{equation}
Moreover asymptotically
\begin{equation}
\frac{\E\left [\mathcal{D}_n^{a} \right ]}{n} \xrightarrow {n \to \infty} \frac{1}{a}
\end{equation}
and as $a$ approaches infinity, we get the same expectation as for URTs:
\begin{equation}
\E\left [\mathcal{D}_n^{a} \right ] \xrightarrow{a \to \infty} H_{n-1} .
\end{equation}
\end{cor}

\begin{proof}
See Appendix \ref{app:Variances}.
\end{proof}

\chapter{CONCLUSION} \label{chapter:conclusion}

We end the thesis with some  concluding remarks which contain some discussions regarding the results  above, and some questions which we are planning to   pursue in  subsequent work. 

The main motivation for this thesis was to understand the `most general' random tree model, which we call an inhomogeneous random tree, where each node has a specific probability of attaching to an already existing node. Namely, once we have $i$ nodes present, node $i + 1$ attaches to node $j$, $1 \leq j \leq i$ with probability $p_{i+1,j}$ where $\sum_{k=1}^i p_{i+1,j} = 1$, and where the attachment is  independent of the previous evolution of the tree. Both of the tree models discussed in this thesis provide approximations for inhomogeneous random trees, and our next  step will be to see if our results can be translated into this more general framework  by employing certain limiting arguments.  We were not able to follow up asymptotic equivalence of various random tree models described in this thesis. In particular, results towards asymptotic  equivalence of weighted recursive trees for varying weight sequences, and for certain classes of statistics should be of interest.

Among the problems we could not solve are the number of leaves of a general weighted recursive tree. The martingale and coupling argument we used only work when finitely many nodes are assigned a weight different from 1. There might be the possibility to use a generalization of Friedman's urn with a variable matrix. This would allow us to add at each step balls according to the weight of the newly attached node. Another way might be to use the exact probabilities we introduced, at least for specific kinds of weight sequences. This might give as an idea of what a result on asymptotic behaviour could look like.

Concerning weight sequences it would in general be of interest to find precise conditions on when central limit theorems apply and when the distribution is asymptotically close to the one of the uniform case. Especially concerning the expectation of the depth we could not establish general conditions for a central limit theorem to hold, although many of the examples we checked exhibited a similar asymptotic behaviour. Additionally there are many statistics of interest we did not study or only for special cases, like the height, the maximum degree or the number of nodes with a certain number of descendants. We believe that the coupling constructions given in Chapter \ref{chapter:weight} will help us to analyze various other statistics of weighted recursive trees by making use of corresponding results for uniform recursive trees. This coupling can be used to obtain concentration inequalities for underlying statistics besides its use in understanding asymptotic distributions. We hope to follow this in an upcoming work. 
 
Concerning biased recursive trees an open problem is the asymptotic distribution of the number of branches. Because of the global dependence of the anti-records, we could not apply any of the central limit theorems mentioned in the preliminaries. A different approach to the number of branches of biased recursive trees would be to analyze the cycle structure and more specifically the number of cycles of riffle shuffle permutations.

Further results on the depth are also necessary for biased recursive trees. While we could compute the expectation of the depth our method might be too complicated for computing the variance. Since the expectation only depends on $p_1$ when $n$ is large, it would be interesting to see if and how the other parameters affect the variance. 

We saw that in a biased recursive tree constructed from a riffle shuffle permutation based on the cutting of the deck into $a$ piles, the maximum degree is $a$. Thus an interesting question would be to compare $a$-ary recursive tree with biased recursive trees and especially $a$-recursive trees. It is not clear if this common restriction on the degree of the nodes is the only property these trees have in common or if they are more similar that one might expect at first sight. In order to get insights concerning this relation a dynamic construction of biased recursive trees would be very useful. If such a construction exists it would probably be very different from the construction principles we know. Of course such a dynamic growth rule would be useful for many other questions as well.

Apart from the problems we could not solve for the models we investigated, there are moreover other non-uniform models that would be worth investigation in order to understand the behaviour of inhomogeneous recursive trees. We might gain further insight by considering weighted tree models where the weights change with time, or are also randomly distributed. Another possibility would be to use the bijection between binary recursive trees and permutations, see for example \cite{Arratia03}. 
Extending uniform binary trees via other random permutation distributions in order to understand all sorts of random binary trees also seems a promising direction to follow. 

Throughout the study of biased recursive trees, the use of biased riffle shuffles instead of uniformly random permutations stemmed from the fact that certain statistics could be expressed in terms of independent random variables. There is another random permutation framework allowing such use of independence, namely unfair permutations. See \cite{Prodinger11, Arslan16} for the definition and analysis of various statistics of unfair permutations. We are not sure one would gain more insight for inhomogeneous random trees by replacing riffle shuffles with unfair permutations, and we believe that this should be checked in subsequent work. Since in that model the rank of $i$ is determined by the maximum of $i$ identically distributed independent uniform random variables, it probably has similarities to models where each node can choose among $k$ potential parents.

\bibliographystyle{fbe_tez_v11}
\bibliography{tez}

\appendix
\chapter{PROOFS FROM CHAPTER \ref{chapter:brt}} \label{app:Variances}
Since many of the calculations in the proofs of Chapter \ref{chapter:brt} are long, we give them in this appendix.
\section{Proof of Theorem \ref{thm:BRTVarBranches}}
\begin{proof} Let $\mathcal{T}_n^{p}$ be a BRT of size $n$ and $\gamma$ be its permutation representation. Let moreover $\mathcal{B}_n^p$ denote the number of branches of $\mathcal{T}_n^{p}$.
For the calculation of the variance difficulty arises from the dependence of the events $A_i=\1(\gamma(i) \text{ is an anti-record})$. We will again assume $p_1 \neq 0$. 

We can write the variance as follows:
\begin{equation} \begin{split}
\Var(\mathcal{B}_n^p) & = \Var\left ( 1+ \sum_{i=3}^n A_i \right ) = \Var\left (\sum_{i=3}^n A_i \right) \\
& = \sum_{i=3}^{n} \Var(A_i) + 2 \sum_{3 \leq i < j \leq n} \Cov (A_i,A_j).
\end{split}
\end{equation}
Moreover we have 
\begin{equation}
\sum_{i=3}^{n} \Var(A_i) = \sum_{i=3}^{n} \E[A_i^2] - \sum_{i=3}^{n} \E[A_i]^2 = \E[\mathcal{B}_n^p] -1 - \sum_{i=3}^{n} \E[A_i]^2
\end{equation}
and 
\begin{equation}
\sum_{3 \leq i < j \leq n} \Cov (A_i,A_j) = \sum_{i=3}^{n-1} \sum_{j=i+1}^{n} \E[A_iA_j] - \sum_{i=3}^{n-1} \sum_{j=i+1}^{n} \E[A_i] \E[A_j].
\end{equation}
Since the $A_i$ are Bernoulli random variables, we have 
\begin{equation}
\begin{split}
\sum_{i=3}^{n}& \E[A_i]^2 \\
&= \sum_{i=3}^{n} \left ( \sum_{s=1}^{a-1}  p_s \left ( \sum_{\ell=s+1}^{a} p_{\ell} \right )^{i-2} \right )^2 \\
&= \sum_{i=3}^{n} \sum_{s=1}^{a-1} p_s^2 \left (\sum_{\ell=s+1}^{a} p_{\ell} \right )^{2(i-2)}  \\
& \hspace{2em} + 2 \sum_{i=3}^{n} \sum_{s=2}^{a-1} \sum_{r=1}^{s-1} p_s p_r \left (\sum_{\ell=s+1}^{a} p_{\ell} \right )^{i-2}  \left (\sum_{\ell=r+1}^{a} p_{\ell} \right )^{i-2}\\
&= \sum_{s=1}^{a-1} p_s^2  \sum_{i=3}^{n} \left ( \left (\sum_{\ell=s+1}^{a} p_{\ell} \right )^2 \right )^{i-2} \\
& \hspace{2em} + 2  \sum_{s=2}^{a-1} \sum_{r=1}^{s-1} p_s p_r  \sum_{i=3}^{n} \left (\sum_{\ell=s+1}^{a} p_{\ell} \sum_{\ell=r+1}^{a} p_{\ell} \right )^{i-2}\\
&= \sum_{s=1}^{a-1} p_s^2  \left (\sum_{\ell=s+1}^{a} p_{\ell} \right )^2 \sum_{i=0}^{n-3} \left ( \left (\sum_{\ell=s+1}^{a} p_{\ell} \right )^2 \right )^{i} \\
&\hspace{2em} + 2  \sum_{s=2}^{a-1} \sum_{r=1}^{s-1} p_s p_r \left (\sum_{\ell=s+1}^{a} p_{\ell} \sum_{\ell=r+1}^{a} p_{\ell} \right ) \sum_{i=0}^{n-3} \left (\sum_{\ell=s+1}^{a} p_{\ell} \sum_{\ell=r+1}^{a} p_{\ell} \right )^{i}\\
&= \sum_{s=1}^{a-1} p_s^2  \left (\sum_{\ell=s+1}^{a} p_{\ell} \right )^2  \frac{1- \left (\sum_{\ell=s+1}^{a} p_{\ell} \right )^{2(n-2)}}{1 -\left (\sum_{\ell=s+1}^{a} p_{\ell} \right )^{2}} \\
&\hspace{2em} + 2  \sum_{s=2}^{a-1} \sum_{r=1}^{s-1} p_s p_r \left (\sum_{\ell=s+1}^{a} p_{\ell} \sum_{\ell=r+1}^{a} p_{\ell} \right ) \frac{1- \left ( \sum_{\ell=s+1}^{a} p_{\ell} \sum_{\ell=r+1}^{a} p_{\ell} \right )^{n-2}}{1 - \sum_{\ell=s+1}^{a} p_{\ell} \sum_{\ell=r+1}^{a} p_{\ell} }.\\
\end{split}
\end{equation}

Now let $3 \leq i < j \leq n$, then
\begin{equation}
\begin{split}
& \E[A_i A_j]  = \Pro \left (X_i < \min\{X_2, \dots, X_{i-1}\}, X_j < \min\{X_2, \dots, X_{j-1}\} \right ) \\
&= \Pro \left (X_j < \min\{X_2, \dots, X_{j-1}\} | X_i < \min\{X_2, \dots, X_{i-1}\} \right ) \\
& \hspace{3em}\cdot \Pro \left (X_i < \min\{X_2, \dots, X_{i-1}\}\right )\\
&= \sum_{s=1}^{a-1} \Pro \left (X_j < \min_{k=2, \dots, j-1}\{X_k\} | X_i < \min_{k=2, \dots, i-1}\{X_k\}, X_i=s \right ) \\
& \hspace{4em}\cdot  \Pro \left (X_i < \min_{k=2, \dots, i-1}\{X_k\} | X_i =s \right ) \Pro \left (X_i =s \right ).\\
\end{split}
\end{equation}
If $X_i=1$, $X_j < X_i$ is not possible, so the above expression is only positive for $2\leq s \leq a-1$. For these $s$ we have:
\begin{equation}
\begin{split}
& \Pro \left (X_j < \min_{k=2, \dots, j-1}\{X_k\} | X_i < \min_{k=2, \dots, i-1}\{X_k\}, X_i=s \right ) \\
&= \sum_{r=1}^{s-1} \underbrace{\Pro \left ( X_j < \min_{k=2,\dots, j-1}\{X_k\} | X_j = r, X_i =s , X_i< \min_{k=2, \dots, i-1}\{X_k\} \right )}_{= \Pro(r < X_{i+1}, \dots, X_{j-1})} \Pro(X_j=r)\\
&= \sum_{r=1}^{s-1}\left (\sum_{q = r+1}^{a} p_q\right)^{j-i-1}p_r.
\end{split}
\end{equation}
So in total we have for $3 \leq i < j \leq n$,
\begin{equation}\E[A_iA_j] = \sum_{s=2}^{a-1}p_s \left ( \sum_{\ell=s+1}^{a}p_{\ell}\right )^{i-2}  \sum_{r=1}^{s-1} p_r \left (\sum_{q = r+1}^{a} p_q\right)^{j-i-1} .
\end{equation}

Thus when summing over $i$ and $j$ we get
\begin{equation}
\begin{split}
&\sum_{i=3}^{n-1}\sum_{j=i+1}^{n}\E[ A_i A_j]\\
&= \sum_{i=3}^{n-1}\sum_{j=i+1}^{n} \sum_{s=2}^{a-1}p_s \left ( \sum_{\ell=s+1}^{a}p_{\ell}\right )^{i-2}  \sum_{r=1}^{s-1} p_r \left (\sum_{q = r+1}^{a} p_q\right)^{j-i-1} \\
&= \sum_{s=2}^{a-1}p_s \sum_{r=1}^{s-1} p_r  \sum_{i=3}^{n-1}  \left ( \sum_{\ell=s+1}^{a}p_{\ell}\right )^{i-2} \sum_{j=i+1}^{n}\left (\sum_{q = r+1}^{a} p_q\right)^{j-i-1} \\
&= \sum_{s=2}^{a-1}p_s \sum_{r=1}^{s-1} p_r  \sum_{i=3}^{n-1}  \left ( \sum_{\ell=s+1}^{a}p_{\ell}\right )^{i-2} \sum_{j=0}^{n-i-1}\left (\sum_{q = r+1}^{a} p_q\right)^{j} \\
&= \sum_{s=2}^{a-1}p_s \sum_{r=1}^{s-1} p_r  \sum_{i=1}^{n-3}  \left ( \sum_{\ell=s+1}^{a}p_{\ell}\right )^{i} \sum_{j=0}^{n-i-3}\left (\sum_{q = r+1}^{a} p_q\right)^{j} \\
&= \sum_{s=2}^{a-1}p_s \sum_{r=1}^{s-1} p_r  \sum_{i=1}^{n-3}  \left ( \sum_{\ell=s+1}^{a}p_{\ell}\right )^{i} \frac{1 - \left (\sum_{q = r+1}^{a} p_q\right)^{n-i-2}}{1-\sum_{q = r+1}^{a} p_q} \\
&= \sum_{s=2}^{a-1}p_s \sum_{r=1}^{s-1} \frac{p_r}{\sum_{q =1}^{r} p_q}  \sum_{i=1}^{n-3}  \left ( \sum_{\ell=s+1}^{a}p_{\ell}\right )^{i} \\
& \hspace{2em} - \sum_{s=2}^{a-1}p_s \sum_{r=1}^{s-1} \frac{p_r}{\sum_{q = 1}^{r} p_q}  \sum_{i=1}^{n-3}  \left ( \sum_{\ell=s+1}^{a}p_{\ell}\right )^{i} \left (\sum_{q = r+1}^{a} p_q\right)^{n-i-2} \\
&= \sum_{s=2}^{a-1}p_s \sum_{r=1}^{s-1} \frac{p_r}{\sum_{q = 1}^{r} p_q}   \frac{\sum_{\ell=s+1}^{a}p_{\ell} -  \left ( \sum_{\ell=s+1}^{a}p_{\ell}\right )^{n-2}}{1- \sum_{\ell=s+1}^{a}p_{\ell}}  \\
& \hspace{2em} - \sum_{s=2}^{a-1}p_s \sum_{r=1}^{s-1} \frac{p_r}{\sum_{q = 1}^{r} p_q} \left (\sum_{q = r+1}^{a} p_q\right)^{n-2}  \sum_{i=1}^{n-3}  \left ( \sum_{\ell=s+1}^{a}p_{\ell}\right )^{i} \left (\sum_{q = r+1}^{a} p_q\right)^{-i} \\
&= \sum_{s=2}^{a-1} \frac{p_s \sum_{\ell=s+1}^{a}p_{\ell}}{\sum_{\ell=1}^{s}p_{\ell}} \sum_{r=1}^{s-1} \frac{p_r}{\sum_{q = 1}^{r} p_q}  \left ( 1 -  \left ( \sum_{\ell=s+1}^{a}p_{\ell}\right )^{n-3}\right )\\
& \hspace{2em} - \sum_{s=2}^{a-1}p_s \sum_{r=1}^{s-1} \frac{p_r}{\sum_{q = 1}^{r} p_q} \left (\sum_{q = r+1}^{a} p_q\right)^{n-2}  \left ( \frac{\frac{ \sum_{\ell=s+1}^{a}p_{\ell} }{\sum_{q = r+1}^{a} p_q}- \left( \frac{\sum_{\ell=s+1}^{a}p_{\ell}}{\sum_{q = r+1}^{a} p_q}\right )^{n-2}}{1- \frac{ \sum_{\ell=s+1}^{a}p_{\ell} }{\sum_{q = r+1}^{a} p_q}} \right ).\\
\end{split}
\end{equation}
The last term can be simplified a little bit more, such that we get
\begin{equation}
\begin{split}
& \sum_{s=2}^{a-1}p_s \sum_{r=1}^{s-1} \frac{p_r}{\sum_{q = 1}^{r} p_q} \left (\sum_{q = r+1}^{a} p_q\right)^{n-2}  \left ( \frac{\frac{ \sum_{\ell=s+1}^{a}p_{\ell} }{\sum_{q = r+1}^{a} p_q}- \left( \frac{\sum_{\ell=s+1}^{a}p_{\ell}}{\sum_{q = r+1}^{a} p_q}\right )^{n-2}}{1- \frac{ \sum_{\ell=s+1}^{a}p_{\ell} }{\sum_{q = r+1}^{a} p_q}} \right )\\
& \hspace{1ex} = \sum_{s=2}^{a-1}p_s \sum_{\ell=s+1}^{a}p_{\ell} \sum_{r=1}^{s-1} \frac{p_r \sum_{q = r+1}^{a} p_q}{\sum_{q = 1}^{r} p_q}  \left ( \frac{\left (\sum_{q = r+1}^{a} p_q\right)^{n-3} - \left( \sum_{\ell=s+1}^{a}p_{\ell}\right)^{n-3}}{\sum_{q = r+1}^{a} p_q- \sum_{\ell=s+1}^{a}p_{\ell}} \right )\\
& \hspace{1ex} = \sum_{s=2}^{a-1}p_s \sum_{\ell=s+1}^{a}p_{\ell} \sum_{r=1}^{s-1}
\frac{p_r \sum_{q = r+1}^{a} p_q}{\sum_{q = 1}^{r} p_q \sum_{q = r+1}^{s} p_q} \left ( \left (\sum_{q = r+1}^{a} p_q\right)^{n-3} - \left( \sum_{\ell=s+1}^{a}p_{\ell}\right)^{n-3} \right ).\\
\end{split}
\end{equation}
So 
\begin{equation}
\begin{split}
\sum_{i=3}^{n-1}\sum_{j=i+1}^{n}\E[ A_i A_j]&=  \sum_{s=2}^{a-1} \frac{p_s \sum_{\ell=s+1}^{a}p_{\ell}}{\sum_{\ell=1}^{s}p_{\ell}} \sum_{r=1}^{s-1} \frac{p_r}{\sum_{q = 1}^{r} p_q}  \left ( 1 -  \left ( \sum_{\ell=s+1}^{a}p_{\ell}\right )^{n-3}\right )\\
& \hspace{1em} - \sum_{s=2}^{a-1}p_s \sum_{\ell=s+1}^{a}p_{\ell} \sum_{r=1}^{s-1}
\frac{p_r \sum_{q = r+1}^{a} p_q}{\sum_{q = 1}^{r} p_q}\frac{1}{\sum_{q = r+1}^{s} p_q} \\
& \hspace{4em} \cdot \left ( \left (\sum_{q = r+1}^{a} p_q\right)^{n-3} - \left( \sum_{\ell=s+1}^{a}p_{\ell}\right)^{n-3} \right ).
\end{split}
\end{equation}
Moreover we have
\begin{equation}
\begin{split}
&\sum_{3 \leq i < j \leq n} \E[A_i]\E[A_j]\\
&= \sum_{i=3}^{n-1} \sum_{j=i+1}^{n} \sum_{s=1}^{a-1}p_s \left ( \sum_{\ell=s+1}^{a}p_{\ell}\right )^{i-2}\sum_{r=1}^{a-1}p_r \left ( \sum_{\ell=r+1}^{a}p_{\ell}\right )^{j-2} \\
&= \sum_{s=1}^{a-1}p_s  \sum_{i=3}^{n-1} \left ( \sum_{\ell=s+1}^{a}p_{\ell}\right )^{i-2} \sum_{r=1}^{a-1}p_r \sum_{j=i+1}^{n} \left ( \sum_{\ell=r+1}^{a}p_{\ell}\right )^{j-2} \\
&= \sum_{s=1}^{a-1}p_s  \sum_{\ell=s+1}^{a}p_{\ell}  \sum_{i=3}^{n-1} \left ( \sum_{\ell=s+1}^{a}p_{\ell}\right )^{i-3} \sum_{r=1}^{a-1}p_r \left ( \sum_{\ell=r+1}^{a}p_{\ell}\right )^{i-1} \sum_{j=0}^{n-i-1} \left ( \sum_{\ell=r+1}^{a}p_{\ell}\right )^{j} \\
&= \sum_{s=1}^{a-1}p_s  \sum_{\ell=s+1}^{a}p_{\ell} \sum_{r=1}^{a-1}p_r  \sum_{i=0}^{n-4} \left ( \sum_{\ell=s+1}^{a}p_{\ell}\right )^{i} \left ( \sum_{\ell=r+1}^{a}p_{\ell}\right )^{i+2} \sum_{j=0}^{n-i-4} \left ( \sum_{\ell=r+1}^{a}p_{\ell}\right )^{j} \\
&= \sum_{s=1}^{a-1}p_s  \sum_{\ell=s+1}^{a}p_{\ell}  \sum_{r=1}^{a-1}p_r   \sum_{i=0}^{n-4} \left ( \sum_{\ell=s+1}^{a}p_{\ell}\right )^{i}\left ( \sum_{\ell=r+1}^{a}p_{\ell}\right )^{i+2} \frac{1- \left ( \sum_{\ell=r+1}^{a}p_{\ell}\right )^{n-i-3} }{1- \sum_{\ell=r+1}^{a}p_{\ell}}\\
&= \sum_{s=1}^{a-1}p_s  \sum_{\ell=s+1}^{a}p_{\ell} \sum_{r=1}^{a-1}\frac{p_r}{\sum_{\ell=1}^{r}p_{\ell} }  \sum_{i=0}^{n-4} \left ( \sum_{\ell=s+1}^{a}p_{\ell}\right )^{i} \left ( \sum_{\ell=r+1}^{a}p_{\ell}\right )^{i+2} \\
&\hspace{1em} - \sum_{s=1}^{a-1}p_s  \sum_{\ell=s+1}^{a}p_{\ell}  \sum_{r=1}^{a-1}\frac{p_r}{\sum_{\ell=1}^{r}p_{\ell} } \sum_{i=0}^{n-4} \left ( \sum_{\ell=s+1}^{a}p_{\ell}\right )^{i} \left ( \sum_{\ell=r+1}^{a}p_{\ell}\right )^{i+2} \left ( \sum_{\ell=r+1}^{a}p_{\ell}\right )^{n-i-3}\\
&= \sum_{s=1}^{a-1}p_s  \sum_{\ell=s+1}^{a}p_{\ell} \sum_{r=1}^{a-1}\frac{p_r}{\sum_{\ell=1}^{r}p_{\ell} } \left ( \sum_{\ell=r+1}^{a}p_{\ell} \right )^2 \sum_{i=0}^{n-4} \left ( \sum_{\ell=s+1}^{a}p_{\ell}\right )^{i}  \left ( \sum_{\ell=r+1}^{a}p_{\ell}\right )^{i} \\
&\hspace{1em} - \sum_{s=1}^{a-1}p_s  \sum_{\ell=s+1}^{a}p_{\ell}   \sum_{r=1}^{a-1}\frac{p_r}{\sum_{\ell=1}^{r}p_{\ell} }  \left ( \sum_{\ell=r+1}^{a}p_{\ell}\right )^{n-1} \sum_{i=0}^{n-4} \left ( \sum_{\ell=s+1}^{a}p_{\ell}\right )^{i}\\
&= \sum_{s=1}^{a-1}p_s  \sum_{\ell=s+1}^{a}p_{\ell} \sum_{r=1}^{a-1}\frac{p_r}{\sum_{\ell=1}^{r}p_{\ell} } \left ( \sum_{\ell=r+1}^{a}p_{\ell} \right )^2  \frac{1- \left ( \sum_{\ell=s+1}^{a}p_{\ell} \sum_{\ell=r+1}^{a}p_{\ell}\right )^{n-3}}{1- \sum_{\ell=s+1}^{a}p_{\ell} \sum_{\ell=r+1}^{a}p_{\ell}} \\
&\hspace{1em} - \sum_{s=1}^{a-1}p_s  \sum_{\ell=s+1}^{a}p_{\ell}   \sum_{r=1}^{a-1}\frac{p_r}{\sum_{\ell=1}^{r}p_{\ell} } \left ( \sum_{\ell=r+1}^{a}p_{\ell}\right )^{n-1} \frac{1- \left ( \sum_{\ell=s+1}^{a}p_{\ell}\right )^{n-3}}{1-\sum_{\ell=s+1}^{a}p_{\ell}}\\
&= \sum_{s=1}^{a-1}p_s  \sum_{\ell=s+1}^{a}p_{\ell} \sum_{r=1}^{a-1}\frac{p_r}{\sum_{\ell=1}^{r}p_{\ell} }  \left ( \sum_{\ell=r+1}^{a}p_{\ell} \right )^2 \frac{1- \left ( \sum_{\ell=s+1}^{a}p_{\ell} \sum_{\ell=r+1}^{a}p_{\ell}\right )^{n-3}}{1- \sum_{\ell=s+1}^{a}p_{\ell} \sum_{\ell=r+1}^{a}p_{\ell}} \\
&\hspace{1em} - \sum_{s=1}^{a-1}\frac{p_s}{\sum_{\ell=1}^{s}p_{\ell}}  \sum_{\ell=s+1}^{a} p_{\ell} \sum_{r=1}^{a-1}\frac{p_r}{\sum_{\ell=1}^{r}p_{\ell} } \left ( \sum_{\ell=r+1}^{a}p_{\ell}\right )^{n-1} \left ( 1 - \left ( \sum_{\ell=s+1}^{a}p_{\ell}\right )^{n-3}\right ). \\
\end{split}
\end{equation}
This gives in total
\begin{equation}
\begin{split}
&\Var(\mathcal{B}_n^p) = \\
& \sum_{s=1}^{a-1} \frac{p_s}{\sum_{\ell=1}^{s}p_{\ell}} \left ( 1- \left ( \sum_{\ell=s+1}^{a} p_{\ell}\right ) ^{n-1}\right )  - \sum_{s=1}^{a-1} p_s \\
& \hspace{1ex} - \sum_{s=1}^{a-1} p_s^2  \left (\sum_{\ell=s+1}^{a} p_{\ell} \right )^2  \frac{1- \left (\sum_{\ell=s+1}^{a} p_{\ell} \right )^{2(n-2)}}{1 -\left (\sum_{\ell=s+1}^{a} p_{\ell} \right )^{2}} \\
&\hspace{1ex} - 2  \sum_{s=2}^{a-1} \sum_{r=1}^{s-1} p_s p_r \left (\sum_{\ell=s+1}^{a} p_{\ell} \sum_{\ell=r+1}^{a} p_{\ell} \right ) \frac{1- \left ( \sum_{\ell=s+1}^{a} p_{\ell} \sum_{\ell=r+1}^{a} p_{\ell} \right )^{n-2}}{1 - \sum_{\ell=s+1}^{a} p_{\ell} \sum_{\ell=r+1}^{a} p_{\ell} }\\
&\hspace{1ex}+2  \sum_{s=2}^{a-1} \frac{p_s \sum_{\ell=s+1}^{a}p_{\ell}}{\sum_{\ell=1}^{s}p_{\ell}} \sum_{r=1}^{s-1} \frac{p_r}{\sum_{q = 1}^{r} p_q}  \left ( 1 -  \left ( \sum_{\ell=s+1}^{a}p_{\ell}\right )^{n-3}\right )\\
& \hspace{1ex} - 2 \sum_{s=2}^{a-1}p_s \sum_{\ell=s+1}^{a}p_{\ell} \sum_{r=1}^{s-1} \frac{p_r \sum_{q = r+1}^{a} p_q}{\sum_{q = 1}^{r} p_q}\frac{1}{\sum_{q = r+1}^{s} p_q} \\
& \hspace{7em} \cdot \left ( \left (\sum_{q = r+1}^{a} p_q\right)^{n-3} - \left( \sum_{\ell=s+1}^{a}p_{\ell}\right)^{n-3} \right ) \\
&\hspace{1ex} - 2 \sum_{s=1}^{a-1}p_s  \sum_{\ell=s+1}^{a}p_{\ell} \sum_{r=1}^{a-1}\frac{p_r}{\sum_{\ell=1}^{r}p_{\ell} }  \left ( \sum_{\ell=r+1}^{a}p_{\ell} \right )^2 \frac{1- \left ( \sum_{\ell=s+1}^{a}p_{\ell} \sum_{\ell=r+1}^{a}p_{\ell}\right )^{n-3}}{1- \sum_{\ell=s+1}^{a}p_{\ell} \sum_{\ell=r+1}^{a}p_{\ell}} \\
&\hspace{1ex} + 2 \sum_{s=1}^{a-1}\frac{p_s}{\sum_{\ell=1}^{s}p_{\ell}}  \sum_{\ell=s+1}^{a} p_{\ell} \sum_{r=1}^{a-1}\frac{p_r}{\sum_{\ell=1}^{r}p_{\ell} } \left ( \sum_{\ell=r+1}^{a}p_{\ell}\right )^{n-1} \left ( 1 - \left ( \sum_{\ell=s+1}^{a}p_{\ell}\right )^{n-3}\right ).  \\
\end{split}
\end{equation}
Hence we get asymptotically, for all $\vec{p}$ with $p_1>0$,
\begin{equation}
\begin{split}
&\Var(\mathcal{B}_n^p)  \xrightarrow{n \to \infty} \\
& \hspace{1em} \sum_{s=1}^{a-1} \frac{p_s}{\sum_{\ell=1}^{s}p_{\ell}}  - \sum_{s=1}^{a-1} p_s \\
& \hspace{2em} - \sum_{s=1}^{a-1} p_s^2  \left (\sum_{\ell=s+1}^{a} p_{\ell} \right )^2  \frac{1}{1 -\left (\sum_{\ell=s+1}^{a} p_{\ell} \right )^{2}} \\
&\hspace{2em} - 2  \sum_{s=2}^{a-1} \sum_{r=1}^{s-1} p_s p_r \left (\sum_{\ell=s+1}^{a} p_{\ell} \sum_{\ell=r+1}^{a} p_{\ell} \right ) \frac{1}{1 - \sum_{\ell=s+1}^{a} p_{\ell} \sum_{\ell=r+1}^{a} p_{\ell} }\\
& \hspace{2em} +2  \sum_{s=2}^{a-1} \frac{p_s \sum_{\ell=s+1}^{a}p_{\ell}}{\sum_{\ell=1}^{s}p_{\ell}} \sum_{r=1}^{s-1} \frac{p_r}{\sum_{q = 1}^{r} p_q}  \\
&\hspace{2em} - 2 \sum_{s=1}^{a-1}p_s  \sum_{\ell=s+1}^{a}p_{\ell} \sum_{r=1}^{a-1}\frac{p_r}{\sum_{\ell=1}^{r}p_{\ell} }  \left ( \sum_{\ell=r+1}^{a}p_{\ell} \right )^2 \frac{1}{1- \sum_{\ell=s+1}^{a}p_{\ell} \sum_{\ell=r+1}^{a}p_{\ell}} .\\
\end{split}
\end{equation}
\end{proof}

\section{Proof of Theorem \ref{thm:aRTBranchesVariance}}

\begin{proof}
If $\vec{p}$ is the uniform distribution over $[a]$,  \eqref{equ:VarBranches} gives
\begin{equation}
\begin{split}
&\Var(\mathcal{B}_n^a) =\\
&\sum_{s=1}^{a-1} \frac{\frac{1}{a}}{\sum_{\ell=1}^{s}\frac{1}{a}} \left ( 1- \left ( \sum_{\ell=s+1}^{a} \frac{1}{a}\right ) ^{n-1}\right )  - \sum_{s=1}^{a-1} \frac{1}{a} \\
& \hspace{1ex} - \sum_{s=1}^{a-1} \frac{1}{a^2}  \left (\sum_{\ell=s+1}^{a} \frac{1}{a} \right )^2  \frac{1- \left (\sum_{\ell=s+1}^{a} \frac{1}{a} \right )^{2(n-2)}}{1 -\left (\sum_{\ell=s+1}^{a} \frac{1}{a} \right )^{2}} \\
&\hspace{1ex} - 2  \sum_{s=2}^{a-1} \sum_{r=1}^{s-1} \frac{1}{a^2} \left (\sum_{\ell=s+1}^{a} \frac{1}{a} \sum_{\ell=r+1}^{a} \frac{1}{a} \right ) \frac{1- \left ( \sum_{\ell=s+1}^{a} \frac{1}{a} \sum_{\ell=r+1}^{a} \frac{1}{a} \right )^{n-2}}{1 - \sum_{\ell=s+1}^{a} \frac{1}{a} \sum_{\ell=r+1}^{a} \frac{1}{a} }\\
& \hspace{1ex} +2  \sum_{s=2}^{a-1} \frac{\frac{1}{a} \sum_{\ell=s+1}^{a}\frac{1}{a}}{\sum_{\ell=1}^{s}\frac{1}{a}} \sum_{r=1}^{s-1} \frac{\frac{1}{a}}{\sum_{q = 1}^{r} \frac{1}{a}}  \left ( 1 -  \left ( \sum_{\ell=s+1}^{a}\frac{1}{a}\right )^{n-3}\right )\\
& \hspace{1ex} - 2 \sum_{s=2}^{a-1}\frac{1}{a^2} \sum_{\ell=s+1}^{a} \sum_{r=1}^{s-1} \frac{\frac{1}{a} \sum_{q = r+1}^{a} \frac{1}{a}}{\sum_{q = 1}^{r} \frac{1}{a}}\frac{1}{\sum_{q = r+1}^{s} \frac{1}{a}} \\
&\hspace{10em} \cdot \left ( \left (\sum_{q = r+1}^{a} \frac{1}{a}\right)^{n-3} - \left( \sum_{\ell=s+1}^{a}\frac{1}{a}\right)^{n-3} \right ) \\
&\hspace{1ex} - 2 \sum_{s=1}^{a-1}\frac{1}{a^2}  \sum_{\ell=s+1}^{a}  \sum_{r=1}^{a-1}\frac{\frac{1}{a}}{\sum_{\ell=1}^{r}\frac{1}{a} }  \left ( \sum_{\ell=r+1}^{a}\frac{1}{a} \right )^2 \frac{1- \left ( \sum_{\ell=s+1}^{a}\frac{1}{a} \sum_{\ell=r+1}^{a}\frac{1}{a}\right )^{n-3}}{1- \sum_{\ell=s+1}^{a}\frac{1}{a} \sum_{\ell=r+1}^{a}\frac{1}{a}} \\
&\hspace{1ex} + 2 \sum_{s=1}^{a-1}\frac{\frac{1}{a}}{\sum_{\ell=1}^{s}\frac{1}{a}}  \sum_{\ell=s+1}^{a} \frac{1}{a} \sum_{r=1}^{a-1}\frac{\frac{1}{a}}{\sum_{\ell=1}^{r}\frac{1}{a} } \left ( \sum_{\ell=r+1}^{a}\frac{1}{a}\right )^{n-1} \left ( 1 - \left ( \sum_{\ell=s+1}^{a}\frac{1}{a}\right )^{n-3}\right ) . \\
\end{split}
\end{equation}
This gives after some simplifications
\begin{equation} \label{eq:VarBranchesaRT}
\begin{split}
\Var(\mathcal{B}_n^a) =& \hspace{1em} \sum_{s=1}^{a-1} \frac{1}{s} \left ( 1- \left ( \frac{a-s}{a}\right ) ^{n-1}\right )  - \frac{a-1}{a} \\
& \hspace{2em} - \sum_{s=1}^{a-1} \frac{1}{a^2}  \left ( \frac{a-s}{a} \right )^2  \frac{1- \left (\frac{a-s}{a} \right )^{2(n-2)}}{1 -\left (\frac{a-s}{a} \right )^{2}} \\
&\hspace{2em} - 2  \sum_{s=2}^{a-1} \sum_{r=1}^{s-1} \frac{1}{a^2} \frac{a-s}{a} \frac{a-r}{a} \frac{1- \left ( \frac{a-s}{a} \frac{a-r}{a} \right )^{n-2}}{1 - \frac{a-s}{a}  \frac{a-r}{a} }\\
& \hspace{2em} +2  \sum_{s=2}^{a-1}\frac{a-s}{as} \sum_{r=1}^{s-1} \frac{1}{r}  \left ( 1 -  \left (\frac{a-s}{a}\right )^{n-3}\right )\\
& \hspace{2em} - 2 \sum_{s=2}^{a-1} \frac{a-s}{a^2} \sum_{r=1}^{s-1} \frac{a-r}{r}\frac{1}{ s-r} \left ( \left (\frac{a-r}{a}\right)^{n-3} - \left(\frac{a-s}{a}\right)^{n-3} \right ) \\
&\hspace{2em} - 2 \sum_{s=1}^{a-1}\frac{a-s}{a^2} \sum_{r=1}^{a-1}\frac{1}{r}  \left (\frac{a-r}{a} \right )^2 \frac{1- \left ( \frac{a-s}{a} \frac{a-r}{a}\right )^{n-3}}{1- \frac{a-s}{a} \frac{a-r}{a}} \\
&\hspace{2em} + 2 \sum_{s=1}^{a-1}  \frac{a-s}{sa} \sum_{r=1}^{a-1}\frac{1}{r} \left (\frac{a-r}{a}\right )^{n-1} \left ( 1 - \left (\frac{a-s}{a}\right )^{n-3}\right ).  \\
\end{split}
\end{equation}

We can easily derive the asymptotic term for $n \to \infty$:
\begin{equation} 
\begin{split}
\Var(\mathcal{B}_n^{a}) \xrightarrow {n \to \infty} & \sum_{s=1}^{a} \frac{1}{s}  - \frac{a-1}{a} - \sum_{s=1}^{a-1} \frac{1}{a^2}  \left ( \frac{a-s}{a} \right )^2  \frac{1}{1 -\left (\frac{a-s}{a} \right )^{2}} \\
& - 2  \sum_{s=2}^{a-1} \sum_{r=1}^{s-1} \frac{1}{a^2} \frac{a-s}{a} \frac{a-r}{a} \frac{1}{1 - \frac{a-s}{a}  \frac{a-r}{a} }\\
&  +2  \sum_{s=2}^{a-1}\frac{a-s}{as} \sum_{r=1}^{s-1} \frac{1}{r}\\
& - 2 \sum_{s=1}^{a-1}\frac{a-s}{a^2} \sum_{r=1}^{a-1}\frac{1}{r}  \left (\frac{a-r}{a} \right )^2 \frac{1}{1- \frac{a-s}{a} \frac{a-r}{a}} \\
=&  H_a  - \frac{a-1}{a}  + \frac{2}{a} \sum_{s=2}^{a-1}\frac{a-s}{s} \sum_{r=1}^{s-1} \frac{1}{r}\\
& -  \frac{1}{a^2} \sum_{s=1}^{a-1}  \frac{s^2}{a^2 - s^{2}} \\
& - \frac{2}{a^2}   \sum_{s=1}^{a-2} \sum_{r=1}^{a-s+1} \frac{sr}{a^2 - sr }\\
& - \frac{2}{a^2} \sum_{s=1}^{a-1} \sum_{r=1}^{a-1}\frac{1}{r} \frac{s (a-r)^2}{a^2- s (a-r)} .\\
\end{split}
\end{equation}

We now look at the terms in Equation \ref{eq:VarBranchesaRT} separately to get an asymptotic result for $a \to \infty$. First of all we know that the first line +1 is equal to the expectation of the number of branches, for which we already have an asymptotic result, namely that it is asymptotically equal to $H_{n-1}$. Again the following equation will repeatedly be used in our calculations:
\begin{equation}
 \sum_{s=1}^{a} s^k = \frac{a^{k+1}}{k+1} + \mathcal{O}(a^k) \text{ for all } k \in \mathbb{N}, k \geq 0.
\end{equation}

For the second term we will now show that it converges to 0 as $a$ goes to infinity:
\begin{equation}
\begin{split}
& \sum_{s=1}^{a-1} \frac{1}{a^2}  \left ( \frac{a-s}{a} \right )^2  \frac{1- \left (\frac{a-s}{a} \right )^{2(n-2)}}{1 -\left (\frac{a-s}{a} \right )^{2}} \\
& \hspace{2em}= \sum_{s=1}^{a-1} \frac{1}{a^2}  \left ( \frac{s}{a} \right )^2  \frac{1- \left (\frac{s}{a} \right )^{2(n-2)}}{1 -\left (\frac{s}{a} \right )^{2}} \\
&\hspace{2em} = \sum_{s=1}^{a-1} \frac{1}{a^2}  \left ( \frac{s}{a} \right )^2  \sum_{\ell=0}^{n-3} \left (\frac{s}{a} \right )^{2\ell} \\
&\hspace{2em} = \sum_{\ell=0}^{n-3} \frac{1}{a^{2\ell+4}}  \sum_{s=1}^{a-1} s^{2\ell+2} \\
&\hspace{2em} = \sum_{\ell=0}^{n-3} \frac{1}{a^{2\ell+4}} \mathcal{O}(a^{2\ell+3})\\
&\hspace{2em} \xrightarrow{a \to \infty} 0.
\end{split}
\end{equation}

\newpage
For the next term we will show that it converges to the second harmonic number plus a constant:
\begin{equation}
\begin{split}
& \sum_{s=2}^{a-1} \sum_{r=1}^{s-1} \frac{1}{a^2} \frac{a-s}{a} \frac{a-r}{a} \frac{1- \left ( \frac{a-s}{a} \frac{a-r}{a} \right )^{n-2}}{1 - \frac{a-s}{a}  \frac{a-r}{a} } \\
&= \sum_{s=1}^{a-2} \sum_{r=s+1}^{a-1} \frac{1}{a^2} \frac{s}{a} \frac{r}{a} \frac{1- \left ( \frac{s}{a} \frac{r}{a} \right )^{n-2}}{1 - \frac{s}{a}  \frac{r}{a} } \\
&= \sum_{s=1}^{a-2} \sum_{r=s+1}^{a-1} \frac{1}{a^2} \frac{s}{a} \frac{r}{a} \sum_{\ell=0}^{n-3} \left ( \frac{s}{a}\frac{r}{a} \right )^{\ell} \\
&=  \sum_{\ell=0}^{n-3} \frac{1}{a^{2\ell+4}} \sum_{s=1}^{a-2}  s^{\ell+1} \sum_{r=s+1}^{a-1} r^{\ell+1}\\
&=  \sum_{\ell=0}^{n-3} \frac{1}{a^{2\ell+4}} \sum_{s=1}^{a-2}  s^{\ell+1} \left [ \frac{a^{\ell+2}}{\ell+2} + \mathcal{O}(a^{\ell+1}) - \frac{s^{\ell+2}}{\ell+2} + \mathcal{O}(s^{\ell+1})\right ] \\
&=  \sum_{\ell=0}^{n-3} \frac{1}{a^{2\ell+4}} \left [ a^{\ell+2} \sum_{s=1}^{a-2} \frac{s^{\ell+1}}{\ell+2} - \sum_{s=1}^{a-2}  \frac{s^{2\ell+3}}{\ell+2} + \mathcal{O}(a^{\ell+1}) \sum_{s=1}^{a-2}  s^{\ell+1} + \sum_{s=1}^{a-2} \mathcal{O}(s^{2\ell+2}) \right ] \\
&=  \sum_{\ell=0}^{n-3} \frac{1}{a^{2\ell+4}} \left [\frac{a^{2\ell+4}}{(\ell+2)^2} -\frac{a^{2\ell+4}}{(2\ell+4)(\ell+2)} + \mathcal{O}(a^{2\ell+3}) \right ] \\
&=  \sum_{\ell=0}^{n-3} \left [\frac{1}{(\ell+2)^2} -\frac{1}{2(\ell+2)^2} \right ] + \mathcal{O}\left ( \frac{1}{a} \right ) \\
&= \frac{1}{2} \left [ \sum_{\ell=1}^{n-1} \frac{1}{\ell^2} - 1 \right ] + \mathcal{O}\left ( \frac{1}{a} \right )  \\
& \xrightarrow{a \to \infty} \frac{1}{2}H_{n-1}^{(2)}- \frac{1}{2}.
\end{split}
\end{equation}
We will now show that the 4th and the 6th term, and the 5th and the 7th term of Equation \ref{eq:VarBranchesaRT} are asymptotically equal. We will repeatedly need the following lemma in these calculations.
\begin{lem} \label{lem:IdentityH}
Let $k, n \in \mathbb{N}$. Then
\begin{equation} 
\sum_{k=1}^{n-1}  (-1)^{k+1}  \binom{n-1}{k}\frac{1}{k} = \sum_{k=1}^{n-1} \frac{1}{k}.
\end{equation}
\end{lem}
\newpage
\begin{proof}
We will prove this identity by the use of generating functions.
First we calculate the generating function for $\sum_{k=1}^{n-1}  (-1)^{k+1}  \binom{n-1}{k}\frac{1}{k}$:
\begin{equation}
\begin{split}
\sum_{n=1}^{\infty} \sum_{k=1}^{n-1}  (-1)^{k+1}  \binom{n-1}{k}\frac{1}{k} x^{n-1}
& = \sum_{n=0}^{\infty} \sum_{k=1}^{n}  (-1)^{k+1}  \binom{n}{k}\frac{1}{k} x^{n} \\
& =  \sum_{k=1} ^{\infty} (-1)^{k+1} \frac{1}{k} \sum_{n=k}^{\infty}   \binom{n}{k} x^{n} \\
& =  \sum_{k=1}^{\infty} (-1)^{k+1} \frac{1}{k}  \sum_{n=0}^{\infty}   \binom{n+k}{k} x^{n+k} \\
& =  \sum_{k=1} ^{\infty} (-1)^{k+1} \frac{1}{k} x^k \sum_{n=0}^{\infty}   \binom{n+k}{k} x^{n} \\
& =  \sum_{k=1}^{\infty} (-1)^{k+1} \frac{1}{k} x^k \frac{1}{(1-x)^{k+1}}\\
& = \frac{1}{1-x} \sum_{k=1}^{\infty} (-1)^{k+1} \frac{1}{k} \left ( \frac{x}{1-x}\right ) ^k\\
& = \frac{1}{1-x} \ln\left ( 1 + \frac{x}{1-x} \right )\\
& = \frac{1}{1-x} \ln\left ( \frac{1}{1-x} \right )\\
& = \frac{-\ln(1-x)}{1-x}.\\
\end{split}
\end{equation}
Where we used the following two equalities:
\begin{equation}\sum_{n=0}^{\infty}   \binom{n+k}{k} x^{n} = \frac{1}{(1-x)^{k+1}}\end{equation} from \cite{GenFun} 
and the Taylor series for the logarithm:
\begin{equation}\ln(1+x) = \sum_{k=1}^{\infty} (-1)^{k+1} \frac{x^k}{k} \text{ for } |x| <1.\end{equation}
The generating function for $\sum_{k=1}^{n-1} \frac{1}{k}$ is easier:
\begin{equation}
\begin{split}
\sum_{n=1}^{\infty} \sum_{k=1}^{n-1} \frac{1}{k} x^{n-1}  &= \sum_{n=0}^{\infty} \sum_{k=1}^{n} \frac{1}{k} x^{n} \\
&  = \sum_{k=1}^{\infty} \frac{1}{k} \sum_{n=k}^{\infty}  x^{n} \\
& = \sum_{k=1}^{\infty} \frac{x^k}{k} \sum_{n=0}^{\infty}  x^{n} \\
&  = \sum_{k=1}^{\infty} \frac{x^k}{k} \frac{1}{1-x} \\
&  = - \frac{ \ln(1-x) }{1-x}.\\ 
\end{split}
\end{equation}
\end{proof}
We will now first consider the 4th and the 6th expression and show that they are asymptotically equal. For the 4th expression we get
\begin{equation}
\begin{split}
& \sum_{s=2}^{a-1}\frac{a-s}{as} \sum_{r=1}^{s-1} \frac{1}{r}  \left ( 1 -  \left (\frac{a-s}{a}\right )^{n-3}\right )\\
&=  \sum_{s=2}^{a-1}\frac{a-s}{as} \frac{s}{a} \frac{ 1 -  \left (\frac{a-s}{a}\right )^{n-3}}{1- \frac{a-s}{a}} \sum_{r=1}^{s-1} \frac{1}{r}   \\
&=  \sum_{s=2}^{a-1}\frac{a-s}{a^2} \frac{ 1 -  \left (\frac{a-s}{a}\right )^{n-3}}{1- \frac{a-s}{a}} \sum_{r=1}^{s-1} \frac{1}{r}   \\
&=  \sum_{r=1}^{a-2} \frac{1}{r}  \sum_{s=r+1}^{a-1}\frac{a-s}{a^2} \frac{ 1 -  \left (\frac{a-s}{a}\right )^{n-3}}{1- \frac{a-s}{a}}   \\
&=  \sum_{r=1}^{a-2} \frac{1}{r}  \sum_{s=1}^{a-r-1}\frac{s}{a^2} \frac{ 1 -  \left (\frac{s}{a}\right )^{n-3}}{1- \frac{s}{a}}   \\
&=  \sum_{r=1}^{a-2} \frac{1}{r}  \sum_{s=1}^{a-r-1}\frac{s}{a^2} \sum_{\ell=0}^{n-4} \left ( \frac{s}{a}\right )^{\ell}  \\
&=  \sum_{r=1}^{a-2} \frac{1}{r} \sum_{\ell=0}^{n-4} \frac{1}{a^{\ell+2}}  \sum_{s=1}^{a-r-1}s^{\ell+1}  \\
&=  \sum_{r=1}^{a-2} \frac{1}{r} \sum_{\ell=0}^{n-4} \frac{1}{a^{\ell+2}} \left [ \frac{(a-r)^{\ell+2}}{l+2} +\mathcal{O}((a-r)^{\ell+1}) \right ] .\\
\end{split}
\end{equation}
Now we have
\begin{equation}
\begin{split}
\sum_{r=1}^{a-2} \frac{1}{r} \sum_{\ell=0}^{n-4} \frac{1}{a^{\ell+2}}\mathcal{O}((a-r)^{\ell+1}) = \mathcal{O}\left ( \frac{ln(a)}{a} \right ).
\end{split}
\end{equation}
We continue with the rest of the expression:
\begin{equation}
\begin{split}
&\sum_{r=1}^{a-2} \frac{1}{r} \sum_{\ell=0}^{n-4} \frac{1}{a^{\ell+2}} \frac{(a-r)^{\ell+2}}{l+2} \\
&=  \sum_{r=1}^{a-2} \frac{1}{r} \sum_{\ell=0}^{n-4} \frac{1}{a^{\ell+2}}\frac{1}{l+2} \sum_{h=0}^{l+2} \binom{l+2}{h} a^{l+2-h}(-r)^h  \\
&=  \sum_{\ell=0}^{n-4} \frac{1}{a^{\ell+2}}\frac{1}{l+2} \sum_{h=0}^{l+2} \binom{l+2}{h} a^{l+2-h} (-1)^h \sum_{r=1}^{a-2} r^{h-1} \\
&=  \sum_{\ell=0}^{n-4} \frac{1}{a^{\ell+2}}\frac{1}{l+2} \left [ \sum_{h=1}^{l+2} \binom{l+2}{h} a^{l+2-h} (-1)^h \left ( \frac{a^h}{h} + \mathcal{O}(a^{h-1}) \right ) + a^{\ell+2}H_{a-2} \right ]\\
&=  \sum_{\ell=0}^{n-4}\frac{1}{l+2} \sum_{h=1}^{l+2} \binom{l+2}{h} (-1)^h  \frac{1}{h}+ \sum_{\ell=0}^{n-4}\frac{1}{l+2}H_{a-2} + \mathcal{O}\left(\frac{1}{a}\right) \\
&= - \sum_{\ell=0}^{n-4}\frac{H_{\ell+2}}{l+2}+ \sum_{\ell=0}^{n-4}\frac{1}{l+2}H_{a-2} + \mathcal{O}\left(\frac{1}{a}\right)\\
\end{split}
\end{equation}
where we used Lemma \ref{lem:IdentityH} in the last line.

Now we look at the 6th term of Equation \ref{eq:VarBranchesaRT} 
\begin{equation}
\begin{split}
& \sum_{s=1}^{a-1}\frac{a-s}{a^2} \sum_{r=1}^{a-1}\frac{1}{r}  \left (\frac{a-r}{a} \right )^2 \frac{1- \left ( \frac{a-s}{a} \frac{a-r}{a}\right )^{n-3}}{1- \frac{a-s}{a} \frac{a-r}{a}} \\
& = \sum_{s=1}^{a-1}\frac{a-s}{a^2} \sum_{r=1}^{a-1}\frac{1}{r}  \left (\frac{a-r}{a} \right )^2  \sum_{\ell=0}^{n-4} \left ( \frac{a-s}{a}\frac{a-r}{a}\right )^{\ell}  \\
& = \sum_{\ell=0}^{n-4} \frac{1}{a^{2\ell+4}} \sum_{s=1}^{a-1} s^{\ell+1} \sum_{r=1}^{a-1}\frac{1}{r}  (a-r)^{\ell+2} \\
& = \sum_{\ell=0}^{n-4} \frac{1}{a^{2\ell+4}} \left [ \frac{a^{\ell+2}}{\ell+2} + \mathcal{O}(a^{\ell+1}) \right ] \sum_{r=1}^{a-1}\frac{1}{r}  \sum_{h=0}^{l+2} \binom{\ell+2}{h} a^{\ell+2-h} r^h (-1)^h \\
& = \sum_{\ell=0}^{n-4} \frac{1}{a^{2\ell+4}} \left [ \frac{a^{\ell+2}}{\ell+2} + \mathcal{O}(a^{\ell+1}) \right ]  \sum_{h=0}^{l+2} \binom{\ell+2}{h} (-1)^h a^{\ell+2-h} \sum_{r=1}^{a-1} r^{h-1}  \\
& = \sum_{\ell=0}^{n-4} \frac{1}{a^{2\ell+4}} \left [ \frac{a^{\ell+2}}{\ell+2} + \mathcal{O}(a^{\ell+1}) \right ] \\
& \hspace{2em} \cdot \left \{ \sum_{h=1}^{l+2} \binom{\ell+2}{h} (-1)^h a^{\ell+2-h} \left [\frac{a^{h}}{h} + \mathcal{O}(a^{h-1})\right ]+ a^{\ell+2} H_{a-1}  \right  \} \\
& = \sum_{\ell=0}^{n-4} \frac{1}{a^{2\ell+4}} \left [ \frac{a^{\ell+2}}{\ell+2} + \mathcal{O}(a^{\ell+1}) \right ] \\
& \hspace{2em} \cdot \left \{a^{l+2} \sum_{h=1}^{l+2} \binom{\ell+2}{h} (-1)^h \frac{1}{h} + \mathcal{O}\left (\frac{1}{a}\right) + a^{\ell+2} H_{a-1}  \right  \} \\
& = \sum_{\ell=0}^{n-4}  \frac{1}{\ell+2}  \sum_{h=1}^{l+2} \binom{\ell+2}{h} (-1)^h \frac{1}{h} + \mathcal{O}\left (\frac{1}{a}\right) + \sum_{\ell=0}^{n-4} \frac{1}{l+2} H_{a-1} \\
& = - \sum_{\ell=0}^{n-4}  \frac{H_{\ell+2}}{\ell+2} + \mathcal{O}\left (\frac{1}{a}\right) + \sum_{\ell=0}^{n-4} \frac{1}{l+2} H_{a-1} .\\
\end{split}
\end{equation}
Hence the sum of the 4th and the 6th term is asymptotically negligible:
\begin{equation}
\begin{split}
 &2\sum_{\ell=0}^{n-4}\frac{1}{l+2}H_{a-2}- 2\sum_{\ell=0}^{n-4} \frac{1}{l+2} H_{a-1} + \mathcal{O}\left (\frac{\ln(a)}{a} \right ) \\
 & \hspace{2em} = - 2 \sum_{\ell=0}^{n-4} \frac{1}{l+2} \frac{1}{a-1} + \mathcal{O}\left (\frac{\ln(a)}{a} \right )  \\
& \hspace{2em} \xrightarrow{a \to \infty} 0.
\end{split}
\end{equation}
Finally we show that the sum of the 5th and the 7th term is asymptotically negligible:
\begin{equation}
\begin{split}
&\sum_{s=2}^{a-1} \frac{a-s}{a^2} \sum_{r=1}^{s-1} \frac{a-r}{r}\frac{1}{ s-r} \left ( \left (\frac{a-r}{a}\right)^{n-3} - \left(\frac{a-s}{a}\right)^{n-3} \right ) \\
&= \sum_{s=2}^{a-1} \frac{a-s}{a^2} \sum_{r=1}^{s-1} \frac{a-r}{r}\frac{1}{a-r} \left (\frac{a-r}{a}\right)^{n-3} \frac{ 1 - \left(\frac{a-s}{a-r}\right)^{n-3}}{1-\frac{a-s}{a-r}} \\
&= \sum_{s=2}^{a-1} \frac{a-s}{a^2} \sum_{r=1}^{s-1} \frac{1}{r} \left (\frac{a-r}{a}\right)^{n-3} \sum_{\ell=0}^{n-4} \left (\frac{a-s}{a-r}\right )^{\ell} \\
&=\frac{1}{a^{n-1}} \sum_{\ell=0}^{n-4} \sum_{s=2}^{a-1}  (a-s)^{\ell+1} \sum_{r=1}^{s-1} \frac{(a-r)^{n-3-l}}{r} \\
&=\frac{1}{a^{n-1}} \sum_{\ell=0}^{n-4} \sum_{r=1}^{a-2} \frac{(a-r)^{n-3-l}}{r}  \sum_{s=r+1}^{a-1}  (a-s)^{\ell+1} \\
&=\frac{1}{a^{n-1}} \sum_{\ell=0}^{n-4} \sum_{r=1}^{a-2} \frac{(a-r)^{n-3-l}}{r}  \sum_{s=1}^{a-r-1}  s^{\ell+1} \\
&=\frac{1}{a^{n-1}} \sum_{\ell=0}^{n-4} \sum_{r=1}^{a-2} \frac{(a-r)^{n-3-l}}{r}\left [ \frac{(a-r)^{\ell+2}}{\ell+2} + \mathcal{O}((a-r)^{\ell+1}) \right ].\\
\end{split}
\end{equation}
Now we have
\begin{equation}\frac{1}{a^{n-1}} \sum_{\ell=0}^{n-4} \sum_{r=1}^{a-2} \frac{(a-r)^{n-3-l}}{r} \mathcal{O}((a-r)^{\ell+1}) = \mathcal{O}\left ( \frac{\ln(a)}{a}\right ).\end{equation}
Continuing with the rest we get
\begin{equation}
\begin{split}
&\frac{1}{a^{n-1}} \sum_{\ell=0}^{n-4} \sum_{r=1}^{a-2} \frac{(a-r)^{n-3-l}}{r}\frac{(a-r)^{\ell+2}}{\ell+2}\\
&=\frac{1}{a^{n-1}} \sum_{\ell=0}^{n-4} \frac{1}{\ell+2} \sum_{r=1}^{a-2} \frac{(a-r)^{n-1}}{r}\\
&=\frac{1}{a^{n-1}} \sum_{\ell=0}^{n-4} \frac{1}{\ell+2} \sum_{r=1}^{a-2} \sum_{h=0}^{n-1} \binom{n-1}{h} a^{n-1-h} (-1)^h r^{h-1} \\
&=\frac{1}{a^{n-1}} \sum_{\ell=0}^{n-4} \frac{1}{\ell+2}  \sum_{h=0}^{n-1} \binom{n-1}{h} a^{n-1-h} (-1)^h  \sum_{r=1}^{a-2}r^{h-1} \\
&=\frac{1}{a^{n-1}} \sum_{\ell=0}^{n-4} \frac{1}{\ell+2}  \Bigg \{ \sum_{h=1}^{n-1} \binom{n-1}{h} a^{n-1-h} (-1)^h \left [ \frac{a^h}{h} + \mathcal{O}(a^{h-1}) \right ] \\
& \hspace{10em} + a^{n-1}\sum_{r=1}^{a-2} r^{-1} \Bigg \} \\
&= \sum_{\ell=0}^{n-4} \frac{1}{\ell+2}  \left \{ \sum_{h=1}^{n-1} \binom{n-1}{h}  (-1)^h \left [ \frac{1}{h} + \mathcal{O}\left (\frac{1}{a} \right ) \right ] + H_{a-2}\right \} \\
&= \sum_{\ell=0}^{n-4} \frac{1}{\ell+2}  \left \{ \sum_{h=1}^{n-1} -\frac{1}{h}+ H_{a-2}\right \} + \mathcal{O}\left (\frac{1}{a} \right ). \\
\end{split}
\end{equation}
We again used Lemma \ref{lem:IdentityH} in the last line.
Finally we consider the last term of Equation \ref{eq:VarBranchesaRT}:
\begin{equation}
\begin{split}
&\sum_{s=1}^{a-1}  \frac{a-s}{sa} \sum_{r=1}^{a-1}\frac{1}{r} \left (\frac{a-r}{a}\right )^{n-1} \left ( 1 - \left (\frac{a-s}{a}\right )^{n-3}\right )  \\
&= \sum_{s=1}^{a-1}  \frac{a-s}{sa} \frac{s}{a} \frac{1 - \left (\frac{a-s}{a}\right )^{n-3}}{1- \frac{a-s}{a}}  \sum_{r=1}^{a-1}\frac{1}{r} \left (\frac{a-r}{a}\right )^{n-1} . \\
\end{split}
\end{equation}
\newpage
We now consider the term dependent on $s$ and the term dependent on $r$ separately. This gives
\begin{equation}
\begin{split}
 \sum_{s=1}^{a-1}  \frac{a-s}{sa} \frac{s}{a} \frac{1 - \left (\frac{a-s}{a}\right )^{n-3}}{1- \frac{a-s}{a}}  &= \sum_{s=1}^{a-1}  \frac{a-s}{a^2} \sum_{\ell=0}^{n-4} \left ( \frac{a-s}{a} \right )^{\ell}\\
&= \sum_{\ell=0}^{n-4} \frac{1}{a^{\ell+2}} \sum_{s=1}^{a-1} ( a-s )^{\ell+1}\\
&= \sum_{\ell=0}^{n-4} \frac{1}{a^{\ell+2}} \sum_{s=1}^{a-1}s ^{\ell+1}\\
&= \sum_{\ell=0}^{n-4} \frac{1}{a^{\ell+2}} \left [ \frac{a^{\ell+2}}{\ell+2} + \mathcal{O}(a^{\ell+1})\right ]\\
&= \sum_{\ell=0}^{n-4}\frac{1}{\ell+2} + \mathcal{O}\left (\frac{1}{a}\right )\\
\end{split}
\end{equation}
and
\begin{equation}
\begin{split}
 \sum_{r=1}^{a-1} & \frac{1}{r} \left (\frac{a-r}{a}\right )^{n-1}\\
&=\sum_{r=1}^{a-1} \frac{1}{a^{n-1}} \frac{1}{r} \sum_{h=0}^{n-1} \binom{n-1}{h} a^{n-1-h} (-1)^h r^h \\
&= \frac{1}{a^{n-1}}  \sum_{h=0}^{n-1} \binom{n-1}{h} a^{n-1-h} (-1)^h \sum_{r=1}^{a-1}  r^{h-1} \\
&= \frac{1}{a^{n-1}}  \sum_{h=1}^{n-1} \binom{n-1}{h} a^{n-1-h} (-1)^h \left [ \frac{a^h}{h} + \mathcal{O}(a^{h-1}) \right ] +\frac{1}{a^{n-1}} a^{n-1} H_{a-1}\\
&=  \sum_{h=1}^{n-1} \binom{n-1}{h}  (-1)^h \frac{1}{h} + \mathcal{O}\left (\frac{1}{a}\right ) +  H_{a-1}\\
&= \sum_{h=1}^{n-1} - \frac{1}{h} + \mathcal{O}\left (\frac{1}{a}\right ) +  H_{a-1}.\\
\end{split}
\end{equation}

Thus in total we have 
\begin{equation}
\begin{split}
&\sum_{s=1}^{a-1}  \frac{a-s}{sa} \sum_{r=1}^{a-1}\frac{1}{r} \left (\frac{a-r}{a}\right )^{n-1} \left ( 1 - \left (\frac{a-s}{a}\right )^{n-3}\right )  \\
&= \left [\sum_{\ell=0}^{n-4}\frac{1}{\ell+2} + \mathcal{O}\left (\frac{1}{a}\right )\right ]
\left [ \sum_{h=1}^{n-1} - \frac{1}{h} + \mathcal{O}\left (\frac{1}{a}\right ) +  H_{a-1}\right ] \\
&= \sum_{\ell=0}^{n-4}\frac{1}{\ell+2}
\left [ \sum_{h=1}^{n-1} - \frac{1}{h} +  H_{a-1}\right ] + \mathcal{O}\left (\frac{\ln(a)}{a}\right ).  \\
\end{split}
\end{equation}
Implying that the sum of the fifth and the seventh term of Equation \ref{eq:VarBranchesaRT} is asymptotically negligible:
\begin{equation}
\begin{split}
-2 \sum_{\ell=0}^{n-4} \frac{1}{\ell+2} &  \left \{ \sum_{h=1}^{n-1} -\frac{1}{h}+ H_{a-2}\right \} \\
& +2\sum_{\ell=0}^{n-4}\frac{1}{\ell+2}
\left [ \sum_{h=1}^{n-1} - \frac{1}{h} +  H_{a-1}\right ] + \mathcal{O}\left (\frac{\ln(a)}{a}\right ) \xrightarrow{a \to \infty} 0.
\end{split}
\end{equation}

Thus finally we get 
\begin{equation}
\Var(\mathcal{B}_n^{a}) \xrightarrow{a \to \infty} H_{n-1}-1 - 2 \frac{1}{2} \left ( H_{n-1}^{(2)} - 1 \right ) = H_{n-1}- H_{n-1}^{(2)}.
\end{equation}
So the variance of the number of branches in an $a$-biased tree converges in the case of $\vec{p}$ uniform to the variance of the number of branches of uniform recursive trees.
\end{proof}

\section{Proof of Theorem \ref{thm:kDesVarGeneral}}
\begin{proof}
We will now calculate the variance of $Y_{\geq k, n}^{p}$, the number of nodes with at least $k$ descendants in a BRT. We will use the following expression for the variance:
\begin{equation}
\begin{split}
\Var \left (Y_{\geq k, n}^{p}\right ) =& \Var \left ( \sum_{i=2}^{n-k} C_{i}^{k} \right) \\
=& \E\left[\left (\sum_{i=2}^{n-k} C_{i}^{k} \right )^2\right] - \E\left [\sum_{i=2}^{n-k} C_{i}^{k} \right]^2 \\
=&  \sum_{i=2}^{n-k}  \E\left[{C_{i}^{k}}^2 \right]- \sum_{i=2}^{n-k}  \E\left [C_{i}^{k} \right]^2\\
& + 2 \sum_{i=2}^{n-k-1} \sum_{j=i+1}^{n-k} \E \left [ C_{i}^{k} C_{j}^{k} \right ] - 2 \sum_{i=2}^{n-k-1} \sum_{j=i+1}^{n-k} \E \left [ C_{i}^{k} \right] \E \left [ C_{j}^{k} \right ] .  \\
\end{split}
\end{equation}
First of all we have
\begin{equation}
\begin{split}
\sum_{i=2}^{n-k} \E\left [{C_i^k}^2\right ] &= \sum_{i=2}^{n-k} \E\left [{C_i^k}\right ] = (n-k-1)\sum_{s=1}^{a} p_s \left (\sum_{r=s}^{a} p_r  \right )^k \\
\end{split}
\end{equation}
and
\begin{equation}
\begin{split}
\sum_{i=2}^{n-k} \E \left [C_i^k \right ]^2 &= (n-k-1) \left (\sum_{s=1}^{a} p_s \left (\sum_{r=s}^{a} p_r  \right )^k \right )^2. \\
\end{split}
\end{equation}

Since the events $C_i^k$ and $C_j^k$ are not mutually independent when $i< j \leq i+k$, we also need to express $\E[C_i^kC_j^k]$ for $i < j \leq i+k$. 

Let $2 \leq i < j \leq i+k$, then
\begin{equation}
\begin{split}
&\E[C_i^k C_j^k] \\
&= \Pro\left (X_i \leq \{X_{i+1}, \dots, X_{i+k}\},  X_j \leq \{X_{j+1}, \dots, X_{j+k}\}\right )\\
&= \Pro\left (X_i \leq \{X_{i+1}, \dots, X_{j-1}\},  X_i \leq X_j, X_j \leq \{X_{j+1}, \dots, X_{j+k}\}\right )\\
&= \sum_{s=1}^{a} \Pro\left (X_i \leq \{X_{i+1}, \dots, X_{j-1}\}, X_i \leq X_j, X_j \leq \{X_{j+1}, \dots, X_{j+k}\} | X_i =s\right )\\
& \hspace{4em} \cdot \Pro(X_i=s) \\
&= \sum_{s=1}^{a} \Pro\left (X_i \leq \{X_{i+1}, \dots, X_{j-1}\} | X_i \leq X_j, X_j \leq \{X_{j+1}, \dots, X_{j+k}\}, X_i =s\right )  \\
&\hspace{4em} \cdot \Pro\left (X_j \leq \{X_{j+1}, \dots, X_{j+k}\} | X_i \leq X_j, X_i =s\right )\\
& \hspace{4em} \cdot \Pro\left (X_i \leq X_j | X_i =s\right ) \Pro(X_i=s) \\
&= \sum_{s=1}^{a} \Pro(X_i=s) \Pro\left (s \leq \{X_{i+1}, \dots, X_{j-1}\} \right)\\
&\hspace{4em} \cdot \Pro\left (s \leq X_j\right ) \Pro\left (X_j \leq \{X_{j+1}, \dots, X_{j+k}\} | s \leq X_j \right)  \\
&= \sum_{s=1}^{a} \Pro(X_i=s) \Pro\left (s \leq \{X_{i+1}, \dots, X_{j-1}\} \right)\\
&\hspace{4em} \cdot \sum_{r=s}^{a} \Pro\left (X_j = r \right ) \Pro\left (X_j \leq \{X_{j+1}, \dots, X_{j+k}\} | X_j = r\right)  \\
&= \sum_{s=1}^{a} \Pro(X_i=s) \Pro\left (s \leq \{X_{i+1}, \dots, X_{j-1}\} \right)\\
&\hspace{4em}\cdot \sum_{r=s}^{a} \Pro\left (X_j = r \right ) \Pro\left (r \leq \{X_{j+1}, \dots, X_{j+k}\}\right )  \\
&= \sum_{s=1}^{a} p_s  \left ( \sum_{u=s}^{a} p_u\right )^{j-i-1}\sum_{r=s}^{a} p_r \left ( \sum_{t=r}^{a} p_t \right )^k. \\
\end{split}
\end{equation}

Now we have
\begin{equation}
\begin{split}
& \sum_{i=2}^{n-k-1} \sum_{j =  i+1}^{n-k} \E[C_i^k C_j^k] - \sum_{i=2}^{n-k-1} \sum_{j =  i+1}^{n-k} \E\left [C_i^k\right ]\E[C_j^k]  \\
&= \sum_{i=2}^{n-2k} \sum_{j =  i+1}^{i+k} \E[C_i^k C_j^k]  - \sum_{i=2}^{n-2k} \sum_{j =  i+1}^{i+k}\E\left [C_i^k\right ]\E[C_j^k]\\
& \hspace{1em}+ \sum_{i=n-2k+1}^{n-k-1} \sum_{j =  i+1}^{n-k} \E[C_i^k C_j^k]  - \sum_{i=n-2k+1}^{n-k-1} \sum_{j =  i+1}^{n-k} \E\left [C_i^k\right ]\E[C_j^k].\\
\end{split}
\end{equation}
Where we eliminated all terms where $i+k<j$ and thus $\E[C_i^kC_j^k] - \E\left [C_i^k\right ]\E[C_j^k]=0$. Moreover we separated the two sums because if $i+k > n-k$, only the $C_j^k$ until $n-k$, not until $i+k$ are relevant. We now consider all 4 expressions separately.
As in the expression for the branches, we assume $p_1>0$ in order to avoid division by $0$. We get
\begin{equation}
\begin{split}
\sum_{i=2}^{n-2k} & \sum_{j =  i+1}^{i+k} \E[C_i^k C_j^k] \\
&= \sum_{i=2}^{n-2k} \sum_{j =  i+1}^{i+k} \sum_{s=1}^{a} p_s  \left ( \sum_{u=s}^{a} p_u\right )^{j-i-1}\sum_{r=s}^{a} p_r \left ( \sum_{t=r}^{a} p_t \right )^k \\
&= \sum_{i=2}^{n-2k} \sum_{j =  0}^{k-1} \sum_{s=1}^{a} p_s  \left ( \sum_{u=s}^{a} p_u\right )^{j}\sum_{r=s}^{a} p_r \left ( \sum_{t=r}^{a} p_t \right )^k \\
&=  \sum_{i=2}^{n-2k} p_1  \sum_{j =  0}^{k-1}  \left ( \sum_{u=1}^{a} p_u\right )^{j}\sum_{r=1}^{a} p_r \left ( \sum_{t=r}^{a} p_t \right )^k \\
 &\hspace{2em} + \sum_{i=2}^{n-2k} \sum_{s=2}^{a} p_s  \sum_{j =  0}^{k-1}  \left ( \sum_{u=s}^{a} p_u\right )^{j}\sum_{r=s}^{a} p_r \left ( \sum_{t=r}^{a} p_t \right )^k \\
&= \sum_{i=2}^{n-2k} p_1  k \sum_{r=1}^{a} p_r \left ( \sum_{t=r}^{a} p_t \right )^k \\
 &\hspace{2em} +\sum_{i=2}^{n-2k} \sum_{s=2}^{a} p_s  \frac{1-\left (\sum_{u=s}^{a} p_u\right )^k}{1- \sum_{u=s}^{a} p_u} \sum_{r=s}^{a} p_r \left ( \sum_{t=r}^{a} p_t \right )^k \\
&=  (n-2k-1) p_1  k \sum_{r=1}^{a} p_r \left ( \sum_{t=r}^{a} p_t \right )^k \\
 &\hspace{2em} + (n-2k-1) \sum_{s=2}^{a} p_s  \frac{1-\left (\sum_{u=s}^{a} p_u\right )^k}{\sum_{u=1}^{s-1} p_u}
\sum_{r=s}^{a} p_r \left ( \sum_{t=r}^{a} p_t \right )^k .\\
\end{split}
\end{equation}
Also
\begin{equation}
\begin{split}
\sum_{i=2}^{n-2k} \sum_{j =  i+1}^{i+k} \E\left [C_i^k\right ]\E[C_j^k]
&= \sum_{i=2}^{n-2k} \sum_{j =  i+1}^{i+k}  \left (\sum_{s=1}^{a} p_s \left (\sum_{r=s}^{a} p_r  \right )^k \right ) ^2\\
&= \sum_{i=2}^{n-2k} k  \left (\sum_{s=1}^{a} p_s \left (\sum_{r=s}^{a} p_r  \right )^k \right ) ^2\\
&= (n-2k-1) k  \left (\sum_{s=1}^{a} p_s \left (\sum_{r=s}^{a} p_r  \right )^k \right ) ^2.\\
\end{split}
\end{equation}

For the second part of the sum we get
\begin{equation}
\begin{split}
 \sum_{i=n-2k+1}^{n-k-1} & \sum_{j = i+1}^{n-k} \E[C_i^k C_j^k] \\
&=  \sum_{i=n-2k+1}^{n-k-1} \sum_{j = i+1}^{n-k} \sum_{s=1}^{a} p_s  \left ( \sum_{u=s}^{a} p_u\right )^{j-i-1}\sum_{r=s}^{a} p_r \left ( \sum_{t=r}^{a} p_t \right )^k \\
&=  \sum_{i=n-2k+1}^{n-k-1} \sum_{j = 0}^{n-k-i-1} \sum_{s=1}^{a} p_s  \left ( \sum_{u=s}^{a} p_u\right )^{j}\sum_{r=s}^{a} p_r \left ( \sum_{t=r}^{a} p_t \right )^k \\
&= \sum_{i=n-2k+1}^{n-k-1} \sum_{j = 0}^{n-k-i-1} p_1  \left ( \sum_{u=1}^{a} p_u\right )^{j} \sum_{r=1}^{a} p_r \left ( \sum_{t=r}^{a} p_t \right )^k\\
&\hspace{2em} + \sum_{i=n-2k+1}^{n-k-1} \sum_{s=2}^{a} p_s \frac{ 1- \left ( \sum_{u=s}^{a} p_u\right )^{n-k-i}}{1- \sum_{u=s}^{a} p_u} \sum_{r=s}^{a} p_r \left ( \sum_{t=r}^{a} p_t \right )^k  \\
&= \sum_{i=n-2k+1}^{n-k-1} (n-k-i) p_1 \sum_{r=1}^{a} p_r \left ( \sum_{t=r}^{a} p_t \right )^k\\
&\hspace{2em} +  \sum_{i=n-2k+1}^{n-k-1} \sum_{s=2}^{a} p_s \frac{ 1}{1- \sum_{u=s}^{a} p_u} \sum_{r=s}^{a} p_r \left ( \sum_{t=r}^{a} p_t \right )^k  \\
&\hspace{2em} - \sum_{i=n-2k+1}^{n-k-1} \sum_{s=2}^{a} p_s \frac{\left ( \sum_{u=s}^{a} p_u\right )^{n-k-i}}{1- \sum_{u=s}^{a} p_u} \sum_{r=s}^{a} p_r \left ( \sum_{t=r}^{a} p_t \right )^k\\
&= \sum_{i=1}^{k-1} i p_1 \sum_{r=1}^{a} p_r \left ( \sum_{t=r}^{a} p_t \right )^k\\
&\hspace{2em} + (k-1)\sum_{s=2}^{a} p_s \frac{ 1}{ \sum_{u=1}^{s-1} p_u} \sum_{r=s}^{a} p_r \left ( \sum_{t=r}^{a} p_t \right )^k \\
&\hspace{2em}  -  \sum_{s=2}^{a} p_s \frac{ 1}{ \sum_{u=1}^{s-1} p_u} \sum_{i=n-2k+1}^{n-k-1} \left ( \sum_{u=s}^{a} p_u\right )^{n-k-i}  \sum_{r=s}^{a} p_r \left ( \sum_{t=r}^{a} p_t \right )^k.\\
\end{split}
\end{equation}
Continuing with the previous equation we get
\begin{equation}
\begin{split}
 \sum_{i=n-2k+1}^{n-k-1} & \sum_{j = i+1}^{n-k} \E[C_i^k C_j^k] \\
&= \frac{k(k-1)}{2} p_1 \sum_{r=1}^{a} p_r \left ( \sum_{t=r}^{a} p_t \right )^k\\ 
&\hspace{2em} + (k-1)\sum_{s=2}^{a} p_s \frac{ 1}{ \sum_{u=1}^{s-1} p_u} \sum_{r=s}^{a} p_r \left ( \sum_{t=r}^{a} p_t \right )^k \\
&\hspace{2em}  -  \sum_{s=2}^{a} p_s \frac{ 1}{ \sum_{u=1}^{s-1} p_u} \sum_{i=1}^{k-1} \left ( \sum_{u=s}^{a} p_u\right )^{i}  \sum_{r=s}^{a} p_r \left ( \sum_{t=r}^{a} p_t \right )^k\\
&=  \frac{k(k-1)}{2} p_1 \sum_{r=1}^{a} p_r \left ( \sum_{t=r}^{a} p_t \right )^k\\ 
&\hspace{2em} + (k-1)\sum_{s=2}^{a} p_s \frac{ 1}{ \sum_{u=1}^{s-1} p_u} \sum_{r=s}^{a} p_r \left ( \sum_{t=r}^{a} p_t \right )^k \\
&\hspace{2em}  -  \sum_{s=2}^{a} p_s \frac{ 1}{ \sum_{u=1}^{s-1} p_u}\left (  \frac{1- \left ( \sum_{u=s}^{a} p_u\right )^{k}}{1-\sum_{u=s}^{a} p_u} -1 \right ) \sum_{r=s}^{a} p_r \left ( \sum_{t=r}^{a} p_t \right )^k.\\
\end{split}
\end{equation} 

Moreover we have
\begin{equation}
\begin{split}
 \sum_{i=n-2k+1}^{n-k-1} & \sum_{j = i+1}^{n-k}  \E\left [C_i^k\right ]\E[C_j^k] \\
 &=  \sum_{i=n-2k+1}^{n-k-1} \sum_{j = i+1}^{n-k} \left ( \sum_{s=1}^{a} p_s \left (\sum_{r=s}^{a} p_r  \right )^k \right )^2 \\
 &=  \sum_{i=n-2k+1}^{n-k-1} (n-k-i) \left ( \sum_{s=1}^{a} p_s \left (\sum_{r=s}^{a} p_r  \right )^k \right )^2 \\
 &=  \sum_{i=1}^{k-1} i \left ( \sum_{s=1}^{a} p_s \left (\sum_{r=s}^{a} p_r  \right )^k \right )^2 \\
  &=  \frac{k(k-1)}{2}\left ( \sum_{s=1}^{a} p_s \left (\sum_{r=s}^{a} p_r  \right )^k \right )^2 . \\
\end{split}
\end{equation}

After adding up these terms, we get  
\begin{equation}
\begin{split}
\Var(Y_n^k) &= \sum_{i=2}^{n-k} \E[{C_i^k}^2] + 2 \sum_{i=2}^{n-2k} \sum_{j =  i+1}^{i+k} \E[C_i^k C_j^k] \\
& \hspace{2em} + 2 \sum_{i=n-2k+1}^{n-k-1} \sum_{j = i+1}^{n-k} \E[C_i^k C_j^k] - \sum_{i=2}^{n-k} \E\left [C_i^k\right ]^2 \\
& \hspace{2em} -  2 \sum_{i=2}^{n-2k} \sum_{j =  i+1}^{i+k} \E\left [C_i^k\right ]\E[C_j^k] - 2 \sum_{i=n-2k+1}^{n-k-1} \sum_{j = i+1}^{n-k}  \E\left [C_i^k\right ]\E[C_j^k]  \\
&=  (n-k-1)\sum_{s=1}^{a} p_s \left (\sum_{r=s}^{a} p_r  \right )^k \\
&\hspace{2em} + 2 (n-2k-1) p_1  k \sum_{r=1}^{a} p_r \left ( \sum_{t=r}^{a} p_t \right )^k \\
 &\hspace{2em} + 2 (n-2k-1) \sum_{s=2}^{a} p_s  \frac{1-\left (\sum_{u=s}^{a} p_u\right )^k}{\sum_{u=1}^{s-1} p_u}
\sum_{r=s}^{a} p_r \left ( \sum_{t=r}^{a} p_t \right )^k \\
&\hspace{2em} + 2 \frac{k(k-1)}{2} p_1 \sum_{r=1}^{a} p_r \left ( \sum_{t=r}^{a} p_t \right )^k\\ 
&\hspace{2em} +  2 (k-1)\sum_{s=2}^{a} p_s \frac{ 1}{ \sum_{u=1}^{s-1} p_u} \sum_{r=s}^{a} p_r \left ( \sum_{t=r}^{a} p_t \right )^k \\
&\hspace{2em}  -  2 \sum_{s=2}^{a} p_s \frac{ 1}{ \sum_{u=1}^{s-1} p_u}\left (  \frac{1- \left ( \sum_{u=s}^{a} p_u\right )^{k}}{1-\sum_{u=s}^{a} p_u} -1 \right ) \sum_{r=s}^{a} p_r \left ( \sum_{t=r}^{a} p_t \right )^k \\
& \hspace{2em} - (n-k-1) \left (\sum_{s=1}^{a} p_s \left (\sum_{r=s}^{a} p_r  \right )^k \right )^2 \\
&\hspace{2em} - 2 (n-2k-1) k  \left (\sum_{s=1}^{a} p_s \left (\sum_{r=s}^{a} p_r  \right )^k \right ) ^2 \\
&\hspace{2em} - 2 \frac{k(k-1)}{2}\left ( \sum_{s=1}^{a} p_s \left (\sum_{r=s}^{a} p_r  \right )^k \right )^2.
\end{split}
\end{equation}
This becomes after grouping some terms
\begin{equation}
\begin{split}
&\Var(Y_{\geq k, n}^p)=\\
 & \hspace{1ex} \sum_{s=1}^{a} p_s \left (\sum_{r=s}^{a} p_r  \right )^k \left [(n-k-1)+ 2 (n-2k-1) p_1  k + k(k-1) p_1 \right ] \\
&\hspace{2ex} +2 \sum_{s=2}^{a} p_s  \frac{1}{\sum_{u=1}^{s-1} p_u} \sum_{r=s}^{a} p_r \left ( \sum_{t=r}^{a} p_t \right )^k \\
& \hspace{3ex} \cdot \left [ (n-2k-1) \left (1-\left (\sum_{u=s}^{a} p_u\right )^k\right ) + (k-1) - \left (  \frac{1- \left ( \sum_{u=s}^{a} p_u\right )^{k}}{1-\sum_{u=s}^{a} p_u} -1 \right )  \right ]\\
& \hspace{2ex}- \left [ k(k-1) + (n-k-1)+ 2(n-2k-1) k  \right ] \left ( \sum_{s=1}^{a} p_s \left (\sum_{r=s}^{a} p_r  \right )^k \right )^2.
\end{split}
\end{equation}
And after some simplifications we finally get
\begin{equation}
\begin{split}
\Var(Y_{\geq k, n}^p)= &\sum_{s=1}^{a} p_s \left (\sum_{r=s}^{a} p_r  \right )^k \left [(n-k-1)+ p_1\left(2nk -3k(k+1)\right) \right ] \\
&+2 \sum_{s=2}^{a} p_s  \frac{1}{\sum_{u=1}^{s-1} p_u} \sum_{r=s}^{a} p_r \left ( \sum_{t=r}^{a} p_t \right )^k \\
& \hspace{1em} \cdot \left [n-k-1  -(n-2k-1)\left (\sum_{u=s}^{a} p_u\right )^k - \frac{1- \left ( \sum_{u=s}^{a} p_u\right )^{k}}{1-\sum_{u=s}^{a} p_u}  \right ]\\
& -  \left ( \sum_{s=1}^{a} p_s \left (\sum_{r=s}^{a} p_r  \right )^k \right )^2 \left [ n(2k+1) - (3k+1)(k+1)\right ].
\end{split}
\end{equation}
We will now look at the asymptotic for a fixed $k$ as $n \to \infty$:
\begin{equation}
\begin{split}
\lim_{n \to \infty}& \frac{\Var(Y_{\geq k, n}^{p})}{n} \\
&= \sum_{s=1}^{a} p_s \left (\sum_{r=s}^{a} p_r  \right )^k \left [\frac{n-k-1}{n}+ p_1 \frac{2nk -3k(k+1)}{n} \right ] \\
&\hspace{1em} +2 \sum_{s=2}^{a} p_s  \frac{1}{\sum_{u=1}^{s-1} p_u} \sum_{r=s}^{a} p_r \left ( \sum_{t=r}^{a} p_t \right )^k \\
& \hspace{2em} \cdot \left [\frac{n-k-1}{n}  -\frac{n-2k-1}{n}\left (\sum_{u=s}^{a} p_u\right )^k - \frac{1}{n}\frac{1- \left ( \sum_{u=s}^{a} p_u\right )^{k}}{1-\sum_{u=s}^{a} p_u}  \right ]\\
& \hspace{1em} -  \left ( \sum_{s=1}^{a} p_s \left (\sum_{r=s}^{a} p_r  \right )^k \right )^2 \left [ \frac{n(2k+1)}{n} - \frac{(3k+1)(k+1)}{n}\right ] \\
& =  \sum_{s=1}^{a} p_s \left (\sum_{r=s}^{a} p_r  \right )^k \left [2kp_1 + 1 \right ] \\
&\hspace{1em}+2 \sum_{s=2}^{a} p_s  \frac{1}{\sum_{u=1}^{s-1} p_u} \sum_{r=s}^{a} p_r \left ( \sum_{t=r}^{a} p_t \right )^k \left [1 -\left (\sum_{u=s}^{a} p_u\right )^k \right ]\\
&\hspace{1em} -  \left ( \sum_{s=1}^{a} p_s \left (\sum_{r=s}^{a} p_r  \right )^k \right )^2 \left [ 2k+1\right ] \\
& =  \sum_{s=1}^{a} p_s \left (\sum_{r=s}^{a} p_r  \right )^k \left ( 2kp_1 + 1 -  (2k+1) \sum_{s=1}^{a} p_s \left (\sum_{r=s}^{a} p_r  \right )^k \right )  \\
&\hspace{1em}+2 \sum_{s=2}^{a} \frac{p_s}{\sum_{u=1}^{s-1} p_u} \sum_{r=s}^{a} p_r \left ( \sum_{t=r}^{a} p_t \right )^k \left [1 -\left (\sum_{u=s}^{a} p_u\right )^k \right ].\\
\end{split}
\end{equation}
\end{proof}

\section{Proof of Corollary \ref{cor:aRTkDescendantsVar}}
\begin{proof}
If we choose the uniform distribution over $[a]$, Theorem \ref{thm:kDesVarGeneral} gives
\begin{equation}
\begin{split}
&\Var \left (Y_{\geq k , n}^{a} \right )=\\
 & \hspace{1em} \sum_{s=1}^{a} \frac{1}{a} \left (\sum_{r=s}^{a} \frac{1}{a}  \right )^k \left [(n-k-1)+ \frac{1}{a}\left(2nk -3k(k+1)\right) \right ] \\
&+2 \sum_{s=2}^{a} \frac{1}{a}  \frac{1}{\sum_{u=1}^{s-1} \frac{1}{a}} \sum_{r=s}^{a} \frac{1}{a} \left ( \sum_{t=r}^{a} \frac{1}{a} \right )^k \\
& \hspace{1em} \cdot \left [n-k-1  -(n-2k-1)\left (\sum_{u=s}^{a} \frac{1}{a}\right )^k - \frac{1- \left ( \sum_{u=s}^{a} \frac{1}{a}\right )^{k}}{1-\sum_{u=s}^{a} \frac{1}{a}}  \right ]\\
& -  \left ( \sum_{s=1}^{a} \frac{1}{a} \left (\sum_{r=s}^{a} \frac{1}{a}  \right )^k \right )^2 \left [ n(2k+1) - (3k+1)(k+1) \right ].
\end{split}
\end{equation}
Using
\begin{equation}
 \sum_{s=1}^{a}  \left (\sum_{r=s}^{a} \frac{1}{a}  \right )^k = \sum_{s=1}^{a} \left (\frac{a-s+1}{a}  \right )^k = \frac{1}{a^k}\sum_{s=1}^{a} s^k
 \end{equation}
and
\begin{equation}
\begin{split}
&\sum_{s=2}^{a} \frac{1}{a}  \frac{1}{\sum_{u=1}^{s-1} \frac{1}{a}} \sum_{r=s}^{a} \frac{1}{a} \left ( \sum_{t=r}^{a} \frac{1}{a} \right )^k \\
&\hspace{1em} = \sum_{s=2}^{a} \frac{1}{s-1} \sum_{r=s}^{a} \frac{1}{a} \left (\frac{a-r+1}{a} \right )^k \\
&\hspace{1em} = \frac{1}{a^{k+1}} \sum_{s=2}^{a} \frac{1}{s-1} \sum_{r=1}^{a-s+1} r^k
\end{split}
\end{equation}
we can simplify this expression and get
\begin{equation}
\begin{split}
\Var \left (Y_{\geq k , n}^{a} \right )= &\frac{1}{a^{k+1}}\sum_{s=1}^{a} s^k\left [(n-k-1)+ \frac{1}{a}\left(2nk -3k(k+1)\right) \right ] \\
&+2 \frac{1}{a^{k+1}} \sum_{s=2}^{a} \frac{1}{s-1} \sum_{r=1}^{a-s+1} r^k \\
& \hspace{1em} \cdot \left [n-k-1  -(n-2k-1)\left ( \frac{a-s+1}{a}\right )^k - \frac{1- \left ( \frac{a-s+1}{a}\right )^{k}}{1-\frac{a-s+1}{a}}  \right ]\\
& -  \left (\frac{1}{a^{k+1}}\sum_{s=1}^{a} s^k \right )^2 \left [ n(2k+1) - (3k+1)(k+1)\right ].
\end{split}
\end{equation}
For $n \to \infty$ this directly gives, as it also follows from Theorem \ref{thm:kDesVarGeneral},
\begin{equation}
\begin{split}
\lim_{n \to \infty} \frac{\Var \left (Y_{\geq k , n}^{a} \right )}{n} =& \frac{1}{a^{k+1}}\sum_{s=1}^{a} s^k\left [1 + \frac{2k}{a} - \frac{2k+1}{a^{k+1}}\sum_{s=1}^{a} s^k \right ] \\
&+ \frac{2}{a^{k+1}} \sum_{s=1}^{a-1} \frac{1}{s} \left [1  - \left ( \frac{a-s}{a}\right )^k  \right ] \sum_{r=1}^{a-s+1} r^k. \\
\end{split}
\end{equation}
In order to calculate the asymptotic value as $a$ tends to infinity, we will first of all use the following equation, that is easy to see
\begin{equation}
 \sum_{s=1}^{a} s^k = \frac{a^{k+1}}{k+1} + \mathcal{O}(a^k) \text{ for all } k \in \mathbb{N}, k \geq 0.
\end{equation}

This directly gives results for the first and third term:
\begin{equation}
\begin{split}
& \frac{1}{a^{k+1}}\sum_{s=1}^{a} s^k\left [(n-k-1)+ \frac{1}{a}\left(2nk -3k(k+1)\right) \right ] \\
& = \frac{1}{a^{k+1}} \left  [\frac{a^{k+1}}{k+1} + \mathcal{O}(a^k)\right ]\left [(n-k-1)+ \frac{1}{a}\left(2nk -3k(k+1)\right) \right ] \\
& = \frac{n-k-1}{k+1} + \mathcal{O}\left ( \frac{1}{a} \right )
\end{split}
\end{equation}
and
\begin{equation}
\begin{split}
& \left (\frac{1}{a^{k+1}} \sum_{s=1}^{a} s^k \right )^2 \left [  n(2k+1) - (3k+1)(k+1)\right ] \\
& \hspace{2em} = \left (\frac{1}{a^{k+1}}\left [ \frac{a^{k+1}}{k+1} + \mathcal{O}(a^k) \right ]\right )^2 \left [  n(2k+1) - (3k+1)(k+1)\right ] \\
& \hspace{2em}  = \frac{ n(2k+1) - (3k+1)(k+1)}{(k+1)^2} + \mathcal{O}\left ( \frac{1}{a} \right ).
\end{split}
\end{equation}
For the middle term we need a little bit more work.
First of all 
\begin{equation}
\begin{split}
& \frac{1}{a^{k+1}} \sum_{s=2}^{a} \frac{1}{s-1} \sum_{r=1}^{a-s+1} r^k \\
& \hspace{2em} = \frac{1}{a^{k+1}} \sum_{s=2}^{a} \frac{1}{s-1} \left [  \frac{(a-(s-1))^{k+1}}{k+1} + \mathcal{O}((a-(s-1))^k) \right ] \\
& \hspace{2em} = \frac{1}{a^{k+1}} \sum_{s=2}^{a} \frac{1}{s-1} \left [  \frac{(a-(s-1))^{k+1}}{k+1} + \mathcal{O}((a-(s-1))^k) \right ] .\\
\end{split}
\end{equation}
We moreover have 
\begin{equation}\frac{1}{a^{k+1}} \sum_{s=2}^{a} \frac{1}{s-1} \mathcal{O}((a-(s-1))^k)  = \mathcal{O}\left (\frac{\ln(a)}{a} \right ). \end{equation}
For the rest we develop the term and get
\begin{equation}
\begin{split}
& \frac{1}{a^{k+1}} \sum_{s=2}^{a} \frac{1}{s-1}  \frac{(a-(s-1))^{k+1}}{k+1} \\
& \hspace{2em} = \frac{1}{k+1} \frac{1}{a^{k+1}} \sum_{s=2}^{a} \frac{1}{s-1} \sum_{\ell=0}^{k+1} \binom{k+1}{\ell} a^{k+1-\ell} (-1)^{\ell} (s-1)^{\ell} \\
& \hspace{2em} = \frac{1}{k+1} \frac{1}{a^{k+1}} \sum_{\ell=0}^{k+1} \binom{k+1}{\ell} a^{k+1-\ell} (-1)^{\ell}  \sum_{s=2}^{a} (s-1)^{\ell-1} \\
& \hspace{2em} = \frac{1}{k+1} \frac{1}{a^{k+1}} \left \{ \sum_{\ell=1}^{k+1} \binom{k+1}{\ell} a^{k+1-\ell} (-1)^{\ell}  \left [\frac{a^{\ell}}{\ell} + \mathcal{O}(a^{\ell-1}) \right] + a^{k+1} H_{a-1} \right \} \\
&\hspace{2em} = \frac{1}{k+1} \left \{ \sum_{\ell=1}^{k+1}  \binom{k+1}{\ell} (-1)^{\ell} \frac{1}{\ell} +  H_{a-1} \right \}  \\
& \hspace{2em} =  \frac{1}{k+1} \left \{ H_{a-1} - H_{k+1}\right\}  + \mathcal{O}\left (\frac{1}{a} \right ) .
\end{split}
\end{equation}
Hence we get
\begin{equation}
(n-k-1) \frac{1}{a^{k+1}} \sum_{s=2}^{a} \frac{1}{s-1} \sum_{r=1}^{a-s+1} r^k\\
=  \frac{(n-k-1)}{k+1} \left \{ H_{a-1} - H_{k+1}\right \} + \mathcal{O}\left ( \frac{\ln(a)}{a} \right ) .
\end{equation}
Next we consider 
\begin{equation}
\begin{split}
& \frac{1}{a^{k+1}} \sum_{s=2}^{a} \frac{1}{s-1} \sum_{r=1}^{a-s+1} r^k \left ( \frac{a-s+1}{a}\right )^k \\
& \hspace{2em} = \frac{1}{a^{2k+1}} \sum_{s=2}^{a} \frac{1}{s-1} \left (a-s+1\right )^k  \left [ \frac{(a-(s-1))^{k+1}}{k+1} + \mathcal{O}((a-(s-1))^k) \right ]. \\
\end{split}
\end{equation}
Similarly to before we have 
\begin{equation} \frac{1}{a^{2k+1}} \sum_{s=2}^{a} \frac{1}{s-1} \left (a-s+1\right )^k   \mathcal{O}((a-(s-1))^k)  = \mathcal{O}\left ( \frac{\ln(a)}{a}\right ).\end{equation}
Continuing with the rest of the term we get
\begin{equation}
\begin{split}
&\frac{1}{a^{2k+1}} \sum_{s=2}^{a} \frac{1}{s-1} \left (a-s+1\right )^k  \frac{(a-(s-1))^{k+1}}{k+1} \\
&\hspace{1em} = \frac{1}{k+1} \frac{1}{a^{2k+1}} \sum_{s=2}^{a} \frac{1}{s-1} \left (a-(s-1)\right )^{2k+1}  \\
& \hspace{1em} = \frac{1}{k+1} \frac{1}{a^{2k+1}} \sum_{s=2}^{a} \frac{1}{s-1} \sum_{\ell=0}^{2k+1} \binom{2k+1}{\ell} a^{2k+1-\ell} (-1)^{\ell} (s-1)^{\ell}\\
& \hspace{1em} = \frac{1}{k+1} \frac{1}{a^{2k+1}}  \sum_{\ell=0}^{2k+1} \binom{2k+1}{\ell} a^{2k+1-\ell} (-1)^{\ell} \sum_{s=2}^{a} (s-1)^{\ell-1}\\
& \hspace{1em} = \frac{1}{k+1} \frac{1}{a^{2k+1}}\left \{  \sum_{\ell=1}^{2k+1} \binom{2k+1}{\ell} a^{2k+1-\ell} (-1)^{\ell} \left [ \frac{a^{\ell}}{\ell} + \mathcal{O}(a^{\ell-1}) \right ]+ a^{2k+1} H_{a-1} \right \}\\
& \hspace{1em} = \frac{1}{k+1}  \sum_{\ell=1}^{2k+1} \binom{2k+1}{\ell}  (-1)^{\ell}  \frac{1}{\ell}+ \frac{H_{a-1}}{k+1} + \mathcal{O}\left ( \frac{1}{a}\right ) \\
& \hspace{1em} = - \frac{ H_{2k+1}}{k+1}+ \frac{H_{a-1}}{k+1}  + \mathcal{O}\left ( \frac{1}{a}\right ). \\
\end{split}
\end{equation}
Thus 
\begin{equation}
\begin{split}
- (n-2k-1) & \frac{1}{a^{k+1}} \sum_{s=2}^{a} \frac{1}{s-1} \sum_{r=1}^{a-s+1} r^k \left ( \frac{a-s+1}{a}\right )^k
\\
&= \frac{n-2k-1}{k+1} H_{2k+1} -\frac{n-2k-1}{k+1}  H_{a-1} + \mathcal{O}\left ( \frac{\ln(a)}{a}\right ).
\end{split}
\end{equation}
Finally we consider
\begin{equation}
\begin{split}
& \frac{1}{a^{k+1}} \sum_{s=2}^{a} \frac{1}{s-1} \sum_{r=1}^{a-s+1} r^k \frac{1- \left ( \frac{a-s+1}{a}\right )^{k}}{1-\frac{a-s+1}{a}} \\
& \hspace{1em} = \frac{1}{a^{k+1}} \sum_{s=2}^{a} \frac{1}{s-1} 
\sum_{\ell=0}^{k-1} \left (\frac{a-(s-1)}{a} \right)^{\ell} \\
& \hspace{4em} \cdot \left [ \frac{(a-s+1)^{k+1}}{k+1} + \mathcal{O}((a-(s-1))^k)\right ].\\
\end{split}
\end{equation}
Now we have for all $\ell=0, \dots, k-1$,  
\begin{equation} \frac{1}{a^{k+1}} \sum_{s=2}^{a} \frac{1}{s-1} 
\left (\frac{a-(s-1)}{a} \right)^{\ell} \mathcal{O}((a-(s-1))^k) = \mathcal{O}\left ( \frac{\ln(a)}{a} \right ).\\ \end{equation}
Continuing with the rest we get 
\begin{equation}
\begin{split}
& \frac{1}{a^{k+1}} \sum_{s=2}^{a} \frac{1}{s-1} 
\sum_{\ell=0}^{k-1} \left (\frac{a-(s-1)}{a} \right)^{\ell} \frac{(a-s+1)^{k+1}}{k+1}\\
& \hspace{1em} = \frac{1}{k+1} \sum_{\ell=0}^{k-1}  \frac{1}{a^{k+\ell+1}} \sum_{s=2}^{a} \frac{1}{s-1} 
\left (a-(s-1) \right)^{k+\ell+1} \\
& \hspace{1em} = \frac{1}{k+1} \sum_{\ell=0}^{k-1}  \frac{1}{a^{k+\ell+1}} \sum_{s=2}^{a} \frac{1}{s-1} 
\sum_{h=0}^{k+\ell+1} \binom{k+\ell+1}{h} a^{k+\ell+1-h} (-1)^{h} (s-1)^h \\
& \hspace{1em} = \frac{1}{k+1} \sum_{\ell=0}^{k-1}  \frac{1}{a^{k+\ell+1}} \sum_{h=0}^{k+\ell+1} \binom{k+\ell+1}{h} a^{k+\ell+1-h} (-1)^{h} \sum_{s=2}^{a} (s-1)^{h-1} \\
& \hspace{1em} = \frac{1}{k+1} \sum_{\ell=0}^{k-1}  \frac{1}{a^{k+\ell+1}} \\
& \hspace{3em} \cdot \left \{ \sum_{h=1}^{k+\ell+1} \binom{k+\ell+1}{h} a^{k+\ell+1-h} (-1)^{h} \left [ \frac{a^{h}}{h} + \mathcal{O}(a^{h-1}) \right ] + a^{k+l+1}H_{a-1} \right \} \\
& \hspace{1em} = \frac{1}{k+1} \sum_{\ell=0}^{k-1} \sum_{h=1}^{k+\ell+1} \binom{k+\ell+1}{h}  (-1)^{h}  \frac{1}{h}  + \frac{k}{k+1}H_{a-1}  + \mathcal{O}\left ( \frac{1}{a} \right ) \\
& \hspace{1em} = - \frac{1}{k+1} \sum_{\ell=0}^{k-1} H_{k+\ell+1}  + \frac{k}{k+1}H_{a-1}  + \mathcal{O}\left ( \frac{1}{a} \right ) .\\
\end{split}
\end{equation}
This gives us the asymptotic value for $\Var(Y_{\leq k, n}^{a})$:
\begin{equation}
\begin{split}
  \Var & \left (Y_{\geq k, n}^{a}\right ) \\
 &= \frac{n-k-1}{k+1}+  2 \frac{(n-k-1)}{k+1} \left \{ H_{a-1} - H_{k+1}\right \}  + 2 \frac{n-2k-1}{k+1} H_{2k+1} \\
& \hspace{2em} - 2 \frac{n-2k-1}{k+1}  H_{a-1} +  2 \frac{1}{k+1} \sum_{\ell=0}^{k-1} H_{k+\ell+1}  - 2 \frac{k}{k+1}H_{a-1}  \\
& \hspace{2em} - \frac{ n(2k+1) - (3k+1)(k+1)}{2(k+1)} + \mathcal{O}\left(\frac{\ln(a)}{a}\right) \\ 
& \xrightarrow{a \to \infty}   \frac{n-k-1}{k+1}-  2 \frac{(n-k-1)}{k+1} H_{k+1}  + 2 \frac{n-2k-1}{k+1} H_{2k+1} \\
& \hspace{3em} +  \frac{2}{k+1} \sum_{\ell=0}^{k-1} H_{k+\ell+1} - \frac{ n(2k+1) - (3k+1)(k+1)}{(k+1)^2}.
\end{split}
\end{equation}
And in particular we get 
\begin{equation}
\lim_{n \to \infty} \lim_{a \to \infty} \frac{\Var \left (Y_{\geq k, n}^{a}\right ) }{n} =  \frac{1}{k+1} + \frac{2H_{2k+1}-2H_{k+1}}{k+1} - \frac{2k+1}{(k+1)^2}.
 \end{equation}
\end{proof}

\section{Proof of Corollary \ref{cor:aRTExptDepth}}
\begin{proof} We will first derive the exact value for the expected value of the depth of $n$ in an $a$-RT, which we denote by $\mathcal{D}_n^{a}$. First we exchange all $p_i$ for $\frac{1}{a}$ in Equation \ref{eq:DepthBRT}. This gives
\begin{equation}
\begin{split}
\E&\left [\mathcal{D}_n^{a}\right ] = \\
& \sum_{s=2}^{a} \frac{\frac{1}{a}}{\sum_{r=1}^{s-1} \frac{1}{a}} \sum_{s'=2}^{a} \left [\left (\sum_{r=1}^{s'} \frac{1}{a} \right )^{n-1}  - \left ( \sum_{r=1}^{s'-1} \frac{1}{a} \right )^{n-1} \right ] \\
&\hspace{1ex} -\sum_{s=2}^{a} \frac{\frac{1}{a}}{\sum_{r=1}^{s-1} \frac{1}{a}} \sum_{s'=2}^{a} \frac{1}{a} \frac{\left (\sum_{r=1}^{s'-1} \frac{1}{a} \right ) ^{n-1}  - \left ( \sum_{r=s}^{a} \frac{1}{a}  \sum_{r=1}^{s'} \frac{1}{a} \right )^{n-1}}{\sum_{r=1}^{s'-1} \frac{1}{a} - \sum_{r=s}^{a} \frac{1}{a}  \sum_{r=1}^{s'} \frac{1}{a}  }\\
& \hspace{1ex} + \frac{1}{a}  \sum_{s=2}^{a} \frac{1}{\frac{1}{a}}  \Bigg [(n-2)\left (\sum_{r=1}^{s} \frac{1}{a} \right)^{n} \\
& \hspace{6em} - (n-1)\left (\sum_{r=1}^{s} \frac{1}{a}\right)^{n-1}\sum_{r=1}^{s-1} \frac{1}{a}  + \sum_{r=1}^{s} \frac{1}{a} \left (\sum_{r=1}^{s-1} \frac{1}{a} \right )^{n-1} \Bigg ] \\
& \hspace{1ex} + \sum_{s=2}^{a} \left( \left (\sum_{r=1}^{s} \frac{1}{a} \right ) ^{n-1} -  \left (\sum_{r=1}^{s-1} \frac{1}{a} \right ) ^{n-1}  \right )\\
& \hspace{1ex} + \left [\sum_{s=2}^{a} \frac{1}{a}\frac{1 - \left ( \sum_{r=s}^{a} \frac{1}{a} \right )^{n-2}}{\sum_{r=1}^{s-1} \frac{1}{a}}+(n-2)\frac{1}{a} +1\right]\frac{1}{a}^{n-1} .\\
\end{split}
\end{equation}
This expression can be simplified to get
\begin{equation} \label{eq:ExpDepthARTLong}
\begin{split}
&\E\left [\mathcal{D}_n^{a}\right ] = \sum_{s=2}^{a} \frac{1}{s-1} \sum_{s'=2}^{a} \left [\left (\frac{s'}{a} \right )^{n-1}  - \left ( \frac{s'-1}{a} \right )^{n-1} \right ] \\
&\hspace{3em} -\sum_{s=2}^{a} \frac{1}{s-1} \sum_{s'=2}^{a} \frac{1}{a} \frac{\left (\frac{s'-1}{a} \right ) ^{n-1}  - \left (\frac{a-s+1}{a} \frac{s'}{a} \right )^{n-1}}{\frac{s'-1}{a} - \frac{a-s+1}{a}  \frac{s'}{a}  }\\
& \hspace{3em} +  \sum_{s=2}^{a}
(n-2)\left (\frac{s}{a} \right)^{n}- (n-1)\left ( \frac{s}{a}\right)^{n-1}\frac{s-1}{a}  + \frac{s}{a} \left ( \frac{s-1}{a} \right )^{n-1}\\
& \hspace{3em} + \sum_{s=2}^{a} \left( \left ( \frac{s}{a} \right ) ^{n-1} -  \left (\frac{s-1}{a} \right ) ^{n-1}  \right )\\
& \hspace{3em} + \left [\sum_{s=2}^{a} \frac{1}{a}\frac{1 - \left ( \frac{a-s+1}{a} \right )^{n-2}}{ \frac{s-1}{a}}+(n-2)\frac{1}{a} +1\right]\frac{1}{a}^{n-1}. \\
\end{split}
\end{equation}
After some more simplifications, we have
\begin{equation}
\begin{split}
&\E\left [\mathcal{D}_n^{a}\right ] = \frac{1}{a^{n-1}} \sum_{s=1}^{a-1}  \sum_{s'=1}^{a-1} \frac{1}{s} \left [(s'+1)^{n-1}  - {s'}^{n-1} \right ] \\
&\hspace{3em} - \frac{1}{a^{n-2}} \sum_{s=1}^{a-1}  \sum_{s'=1}^{a-1} \frac{1}{s} \frac{{s' }^{n-1}  - \left( (a-s)(s'+1)\right )^{n-1}}{ss' +s -a }\\
& \hspace{3em} + \frac{1}{a^n} \sum_{s=1}^{a-1}
(n-2)\left (s+1\right)^{n}- (n-1)\left ( s+1\right)^{n-1}s  + (s+1) s ^{n-1}\\
& \hspace{3em} + \frac{1}{a^{n-1}} \sum_{s=1}^{a-1} ( s+1) ^{n-1} -  s^{n-1}\\
& \hspace{3em} + \frac{1}{a^{n-1}} \left [ \frac{1}{a^{n-2}}\sum_{s=1}^{a-1} \frac{a^{n-2} - (a-s)^{n-2}}{ s}+\frac{n-2}{a} +1\right]. \\
\end{split}
\end{equation}
This in turn gives
\begin{equation}
\begin{split}
&\E\left [\mathcal{D}_n^{a}\right ] = \frac{1}{a^{n-1}} \sum_{s=1}^{a-1}  \frac{1}{s} \left ( a^{n-1} - 1 \right )\\
&\hspace{3em} - \frac{1}{a^{n-2}} \sum_{s=1}^{a-1}  \sum_{s'=1}^{a-1} \frac{1}{s} \frac{{s' }^{n-1}  - \left( (a-s)(s'+1)\right )^{n-1}}{ss' +s -a }\\
& \hspace{3em} + \frac{1}{a^n} \sum_{s=1}^{a-1}
(n-2)\left (s+1\right)^{n}- (n-1)\left ( s+1\right)^{n-1}s  + (s+1) s ^{n-1}\\
& \hspace{3em} + \frac{1}{a^{n-1}} \left( a^{n-1} - 1 \right )\\
& \hspace{3em} + \frac{1}{a^{n-1}} \left [ \frac{1}{a^{n-2}}\sum_{s=1}^{a-1} \frac{a^{n-2} - (a-s)^{n-2}}{ s}+\frac{n-2}{a} +1\right]. \\
\end{split}
\end{equation}
And after  some more simplifications
\begin{equation}
\begin{split}
&\E\left [\mathcal{D}_n^{a}\right ] = H_{a-1} - \frac{H_{a-1}}{a^{n-1}} \\
&\hspace{3em} - \frac{1}{a^{n-2}} \sum_{s=1}^{a-1}  \sum_{s'=1}^{a-1} \frac{1}{s} \frac{{s' }^{n-1}  - \left( (a-s)(s'+1)\right )^{n-1}}{ss' +s -a }\\
& \hspace{3em} + \frac{1}{a^n} \sum_{s=1}^{a-1}
(n-2)\left (s+1\right)^{n}- (n-1)\left ( s+1\right)^{n-1}s  + (s+1) s^{n-1}\\
& \hspace{3em} + 1 -\frac{1}{a^{n-1}} \\
& \hspace{3em} + \frac{H_{a-1}}{a^{n-1}}  + \frac{1}{a^{n-1}} + \frac{n-2}{a^{n}} + \frac{1}{a^{2n-3}}\sum_{s=2}^{a} \frac{ - (a-s+1)^{n-2}}{ s-1}. \\
\end{split}
\end{equation}
Finally we get
\begin{equation} \label{eq:ExpDepthART}
\begin{split}
&\E\left [\mathcal{D}_n^{a}\right ] = H_{a-1} +1  + \frac{n-2}{a^{n}} - \frac{1}{a^{2n-3}}\sum_{s=1}^{a-1} \frac{(a-s)^{n-2}}{ s}  \\
&\hspace{3em} - \frac{1}{a^{n-2}} \sum_{s=1}^{a-1}  \sum_{s'=1}^{a-1} \frac{1}{s} \frac{{s' }^{n-1}  - \left( (a-s)(s'+1)\right )^{n-1}}{ss' +s -a }\\
& \hspace{3em} + \frac{1}{a^n} \sum_{s=1}^{a-1}
(n-2)\left (s+1\right)^{n}- (n-1)\left ( s+1\right)^{n-1}s  + (s+1) s^{n-1}.\\
\end{split}
\end{equation}
For $n \to \infty$ Equation \ref{equ:DepthBRTninf} gives easily,
\begin{equation}\frac{\E\left [\mathcal{D}_n^{a} \right ]}{n} \xrightarrow {n \to \infty} \frac{1}{a}. \end{equation}
This can also be derived from the previous expression.

For the asymptotics when $a \to \infty$ we need more work. Since in the uniform case the expectation of the depth of $n$ is equal to $H_{n-1}$ we expect the same here. We will first look at the term in the second line of Equation \ref{eq:ExpDepthART}. In order to get its asymptotic value we will use the form it had before simplifying, as in Equation \ref{eq:ExpDepthARTLong}. Using this we have
\begin{equation} \label{eq:ProofDepthARTAsymp1}
\begin{split}
 \frac{1}{a^{n-2}} & \sum_{s=1}^{a-1}  \sum_{s'=1}^{a-1} \frac{1}{s} \frac{{s' }^{n-1}  - \left( (a-s)(s'+1)\right )^{n-1}}{ss' +s -a } \\
=& \sum_{s=2}^{a} \frac{1}{s-1} \sum_{s'=2}^{a} \frac{1}{a} \frac{\left (\frac{s'-1}{a} \right ) ^{n-1}  - \left (\frac{a-s+1}{a} \frac{s'}{a} \right )^{n-1}}{\frac{s'-1}{a} - \frac{a-s+1}{a}  \frac{s'}{a}  }\\
=& \sum_{s=1}^{a-1} \frac{1}{s} \sum_{s'=1}^{a-1} \frac{1}{a} \frac{\left (\frac{s'}{a} \right ) ^{n-1}  - \left (\frac{a-s}{a} \frac{s'+1}{a} \right )^{n-1}}{\frac{s'}{a} - \frac{a-s}{a}  \frac{s'+1}{a}  }\\
=&   \frac{1}{a} \sum_{s=1}^{a-1} \frac{1}{s} \sum_{s'=1}^{a-1} \left ( \frac{s'}{a} \right )^{n-2} \frac{1  - \left (\left (\frac{a-s}{a} \frac{s'+1}{a}\right )\left ( \frac{s'}{a} \right )^{-1} \right )^{n-1}}{1 - \frac{a-s}{a}  \frac{s'+1}{a}\left ( \frac{s'}{a} \right )^{-1}   }\\
=&   \frac{1}{a} \sum_{s=1}^{a-1} \frac{1}{s} \sum_{s'=1}^{a-1} \left ( \frac{s'}{a} \right )^{n-2} \sum_{\ell=0}^{n-2}  \left (\left (\frac{a-s}{a} \frac{s'+1}{a}\right )\left ( \frac{s'}{a} \right )^{-1} \right )^{\ell} \\
=& \sum_{\ell=0}^{n-2} \frac{1}{a^{n-1+\ell}} \sum_{s=1}^{a-1} \frac{1}{s} (a-s)^{\ell} \sum_{s'=1}^{a-1} {s'}^{n-2-\ell} (s'+1)^{\ell} .\\
\end{split}
\end{equation}
\newpage
We know look at the term containing only $s$ and the term containing only $s'$ separately.
First of all for $\ell=0, \dots, n-2$,
\begin{equation}
\begin{split}
\sum_{s=1}^{a-1} \frac{1}{s} (a-s)^{\ell} &= \sum_{s=1}^{a-1} \frac{1}{s} \sum_{k=0}^{\ell} (-1)^k \binom{\ell}{k} a^{\ell-k} s^k\\
 &= \sum_{k=0}^{\ell} (-1)^k \binom{\ell}{k}  a^{\ell-k} \sum_{s=1}^{a-1} s^{k-1}\\
&= \sum_{k=1}^{\ell} (-1)^k \binom{\ell}{k}  a^{\ell-k}\left [ \frac{a^k}{k} + \mathcal{O}(a^{k-1}) \right ] + a^\ell H_{a-1} \\
&= a^{\ell} \sum_{k=1}^{\ell} (-1)^k \binom{\ell}{k} \frac{1}{k} + \mathcal{O}(a^{l-1}) + a^\ell H_{a-1}. \\
\end{split}
\end{equation}
Now for the term depending on $s'$ only we get for $\ell=0, \dots, n-2$,
\begin{equation}
\begin{split}
\sum_{s'=1}^{a-1} {s'}^{n-2-\ell} (s'+1)^{\ell} &= \sum_{s'=1}^{a-1} {s'}^{n-2-\ell} \sum_{h=0}{\ell} \binom{\ell}{h} {s'}^{h} \\
&= \sum_{s'=1}^{a-1} {s'}^{n-2-\ell} \sum_{h=0}^{\ell} \binom{\ell}{h} {s'}^{h} \\
&= \sum_{h=0}^{\ell} \binom{\ell}{h} \sum_{s'=1}^{a-1} {s'}^{n+h-\ell-2}  \\
&= \sum_{h=0}^{\ell} \binom{\ell}{h}\left [ \frac{a^{n+h-\ell-1}}{n+h-\ell-1} + \mathcal{O}(a^{n+h- \ell -2}) \right ]. \\
\end{split}
\end{equation}
By inserting these expressions into Equation \ref{eq:ProofDepthARTAsymp1} we get 
\begin{equation}
\begin{split}
 \sum_{\ell=0}^{n-2} & \frac{1}{a^{n-1+\ell}} \sum_{s=1}^{a-1} \frac{1}{s} (a-s)^{\ell} \sum_{s'=1}^{a-1} {s'}^{n-2-\ell} (s'+1)^{\ell} \\
 &=  \sum_{\ell=0}^{n-2} \frac{1}{a^{n-1+\ell}} \left [ a^{\ell} \sum_{k=1}^{\ell} (-1)^k \binom{\ell}{k} \frac{1}{k} + \mathcal{O}(a^{l-1}) + a^\ell H_{a-1} \right ] \\
 & \hspace{2em} \cdot \sum_{h=0}^{\ell} \binom{\ell}{h}\left [ \frac{a^{n+h-\ell-1}}{n+h-\ell-1} + \mathcal{O}(a^{n+h- \ell -2}) \right ]   \\
&=  \sum_{\ell=0}^{n-2} H_{a-1} \sum_{h=0}^{\ell} \binom{\ell}{h}\left [ \frac{1}{a^{\ell-h}}\frac{1}{n+h-\ell-1} + \mathcal{O}\left (\frac{1}{a^{\ell-h+1}} \right ) \right ]  \\
& \hspace{1em} + \sum_{\ell=1}^{n-2} \left [  \sum_{k=1}^{\ell} (-1)^k \binom{\ell}{k} \frac{1}{k} + \mathcal{O}\left (\frac{1}{a}\right )\right ] \\
 & \hspace{2em} \cdot \sum_{h=0}^{\ell} \binom{\ell}{h}\left [ \frac{1}{a^{\ell-h}}\frac{1}{n+h-\ell-1} + \mathcal{O}\left (\frac{1}{a^{\ell -h+1}} \right ) \right ]   \\
&= \frac{1}{n-1} \sum_{\ell=0}^{n-2} H_{a-1}   + \mathcal{O}\left ( \frac{\ln(a)}{a}\right )\\
& \hspace{1em} +\frac{1}{n-1} \sum_{\ell=1}^{n-2}   \sum_{k=1}^{\ell} (-1)^k \binom{\ell}{k} \frac{1}{k} + \mathcal{O}\left (\frac{1}{a}\right )  \\
&= H_{a-1} - \frac{1}{n-1} \sum_{\ell=1}^{n-2}   H_{\ell} + \mathcal{O}\left ( \frac{\ln(a)}{a}\right ) \\
&= H_{a-1} - \frac{1}{n-1} \sum_{\ell=1}^{n-2}   (n-1-\ell)\frac{1}{\ell} + \mathcal{O}\left ( \frac{\ln(a)}{a}\right ) \\
&= H_{a-1} -  \sum_{\ell=1}^{n-2} \left ( \frac{1}{\ell} - \frac{1}{n-1} \right ) + \mathcal{O}\left ( \frac{\ln(a)}{a} \right ) \\
&= H_{a-1} -  H_{n-2} + \frac{n-2}{n-1} + \mathcal{O}\left ( \frac{\ln(a)}{a}\right ) \\
&= H_{a-1} -  H_{n-1} + 1 + \mathcal{O}\left ( \frac{\ln(a)}{a}\right ) \\
\end{split}
\end{equation}
where we again used Lemma \ref{lem:IdentityH}.

Now we look at the last term of Equation \ref{eq:ExpDepthART}:
\begin{equation}
\begin{split}
& \frac{1}{a^n} \sum_{s=1}^{a-1} (n-2)\left (s+1\right)^{n}- (n-1)\left ( s+1\right)^{n-1}s  + (s+1) s^{n-1}\\
&\hspace{1em} = \frac{1}{a^n} \sum_{s=1}^{a-1} (n-2) \sum_{\ell=0}^{n} \binom{n}{\ell} s^{\ell} - (n-1) \sum_{\ell=0}^{n-1} \binom{n-1}{\ell} s^{\ell+1}   + s^n+ s^{n-1}\\
&\hspace{1em} = \frac{1}{a^n} \sum_{s=1}^{a-1} (n-2) \sum_{\ell=0}^{n-2} \binom{n}{\ell} s^{\ell} - (n-1) \sum_{\ell=0}^{n-3} \binom{n-1}{\ell} s^{\ell+1}\\
&\hspace{1em} = \frac{1}{a^n} \left [  (n-2) \sum_{\ell=0}^{n-2} \binom{n}{\ell} \sum_{s=1}^{a-1} s^{\ell} - (n-1) \sum_{\ell=0}^{n-3} \binom{n-1}{\ell} \sum_{s=1}^{a-1} s^{\ell+1} \right ]\\
& \hspace{1em} = \mathcal{O}\left ( \frac{1}{a} \right ).
\end{split}
\end{equation}
Moreover we have
\begin{equation} \frac{n-2}{a^{n}} - \frac{1}{a^{2n-3}}\sum_{s=1}^{a-1} \frac{(a-s)^{n-2}}{ s} = \mathcal{O}\left ( \frac{1}{a} \right ) .\end{equation}
Thus in total we get
\begin{equation}
\begin{split}
\E\left [\mathcal{D}_n^{a}\right ] &= H_{a-1} +1  -H_{a-1} +  H_{n-1} - 1 + \mathcal{O}\left ( \frac{\ln(a)}{a}\right ) \\
& \xrightarrow{a \to \infty} H_{n-1}.
\end{split}
\end{equation}
\end{proof}
\end{document}